\documentclass[12pt]{amsart}

\usepackage{vmargin}
\usepackage[dvips]{graphicx}
\usepackage{color}
\usepackage{transparent}
\input{amssym.def}
\input{amssym}
\usepackage{a4wide}
\usepackage{supertabular}
\usepackage{tikz-cd}
\usepackage{array}
\usepackage{multicol}
\usepackage{bbm}
\usepackage{bbold}
\usepackage{epstopdf}

\usepackage{import}
\setmarginsrb{2cm}{2cm}{2cm}{2.5cm}{75pt}{20pt}{20pt}{30mm}
\def\derpar#1#2{\frac{\partial#1}{\partial#2}}
\def\R{\mathbb R}
\def\N{\mathbb N}

\def\C{\mathbb C}
\def\CD{{\rm CD}}
\def\Q{\mathbb Q}
\def\Z{\mathbb Z}
\def\P{\mathbb P}
\def\E{\mathbb E}
\def\X{\mathcal X}
\def\A{\mathbb A}
\def\G{\mathbb G}
\def\I{\mathbb I}
\def\HS{{\rm HS}}
\def\cI{({\rm I})}
\def\cII{({\rm II})}
\def\cIII{({\rm III})}
\def\S{\mathbb S}
\def\RP{\mathbb R \!\mathbb P}
    
\def\var{\varepsilon}

\def\pa{\partial}

\def\Om{\Omega}

\def\ov{\overline}
\def\cal{\mathcal}
\def\vphi{\varphi}
\def\hat{\widehat}

\def\tilde{\widetilde}

\def\tr{\mathop{{\rm tr}\,}}
\def\Var{\mathop{{\rm Var}\,}}
\def\Id{{\rm Id}\,}
\def\dps{\displaystyle}

\def\vol{{\rm vol}\,}

\def\<{\langle}
\def\>{\rangle}

\def\vect{{\rm Vect}}

\def\ddt#1{\left.\frac{d}{dt}\right|_{#1}}
\def\Ric{{\rm Ric}}
\def\law{{\rm law}}
\def\ddt#1{\left.\frac{d}{dt}\right|_{#1}}

\def\GGamma{\Gamma\Gamma}
\def\Lsharp{L_\#}
\def\Ksharp{K^\#}
\def\gtrsim{\ {\stackrel{>}{\sim}}} 

\newcommand{\bbnabla}{%
  \nabla\mkern-12mu\nabla%
}

\def\Cnk#1#2{\small \left(\!\begin{array}{c} #1 
                \\ #2 \end{array} \!\right)} 

\DeclareMathOperator*{\grad}{grad}

\def\med{\medskip}
\def\sm{\smallskip}
\def\bul{$\bullet$\ }

\def\begeq{\begin{equation}}
\def\endeq{\end{equation}}
\def\begar{\begin{eqnarray}}
\def\endar{\end{eqnarray}}
\def\begar*{\begin{eqnarray*}}
\def\endar*{\end{eqnarray*}}
\def\begal{\begin{align}}
\def\endal{\end{align}}
\def\begal*{\begin{align*}}
\def\endal*{\end{align*}}

\newtheorem{Thm}{Theorem}
\newtheorem{Lem}[Thm]{Lemma}

\newtheorem{Cor}[Thm]{Corollary}
\newtheorem{Prop}[Thm]{Proposition}

\newtheorem{Guess}[Thm]{Guess}

\numberwithin{equation}{section}
\numberwithin{Thm}{section}

\theoremstyle{definition}
\newtheorem{Def}[Thm]{Definition}
\newtheorem{Rk}[Thm]{Remark}
\newtheorem{Rks}[Thm]{Remarks}

\theoremstyle{remark}

\newtheorem{Ex}[Thm]{Example}

\newtheorem*{Thm*}{Theorem}
\newtheorem*{Lem*}{Lemma}
\newtheorem*{Conj*}{Conjecture}
\newtheorem*{Cor*}{Corollary}
\newtheorem*{Def*}{Definition}
\newtheorem*{Prop*}{Proposition}
\newtheorem*{Exo*}{Exercise}
\newtheorem*{Exs*}{Examples}
\newtheorem*{Ex*}{Example}
\newtheorem*{Rk*}{Remark}
\newtheorem*{Rks*}{Remarks}

\def\bibnotes{\medskip \noindent {\bf Bibliographical Notes }\sm \noindent}

\def\signcv{\vspace{1mm} \begin{center} {\sc C\'edric Villani\par\vspace{2mm}
Acad\'emie des Sciences\\
Universit\'e Claude Bernard Lyon I \\
Institut Camille Jordan\par
43 Bd du 11 Novembre 1918, 69100 Villeurbanne\par
FRANCE\par\vspace{3mm}
E-mail:} \tt{cv@cedricvillani.org} \end{center}}

\begin{document}

\title[Fisher information in kinetic theory]{Fisher information in kinetic theory}

\vspace*{-10mm}

\author{C. Villani}

\vspace*{-5mm}

\begin{abstract} These notes review the theory of Fisher information, especially its use in kinetic theory of gases and plasmas. The recent monotonicity theorem by Guillen--Silvestre for the Landau--Coulomb equation is put in perspective and generalised. Following my joint work with Imbert and Silvestre, it is proven that Fisher information is decaying along the spatially homogeneous Boltzmann equation, for all relevant interactions, and from this the once longstanding problem of regularity estimates for very singular collision kernels (very soft potentials) is solved.
\end{abstract}

\maketitle

\vspace*{-12mm}

\tableofcontents

\vspace*{-10mm}



\bigskip

These are the notes of the course which I gave at the summer school {\em Mathemata} in Chania, Crete in July 2024 as part of the Festum Pi 2024 festival.
My course was triggered by a spectacular theorem of Guillen \& Silvestre, which I wanted to present, expand and generalize. The core of these notes is based on my current research to extend this theorem, in collaboration with Cyril Imbert and Luis Silvestre. This course is also inspired by my long-term past collaborations in kinetic theory with Laurent Desvillettes, Giuseppe Toscani, Eric Carlen and Cl\'ement Mouhot; and by my long-term past interactions about functional inequalities with Michel Ledoux.
It is a pleasure to thank them all warmly, as well as the main organisers of the summer school,
Dionysios Dervis-Bournias and Bertrand Maury. These thanks extend to the talented students attending and participating in the course,
and to the Mediterranean Agronomic Institute of Chania (MAICH), which provided the inspiring location and outstanding support. My gratitude goes to St\'ephane Mischler for a very careful reading of these notes. Finally, part of this work was done as I was holding a joint chair between Universit\'e Claude Bernard Lyon 1 and Institut des hautes \'etudes scientifiques (IHES) in Bures-sur-Yvette, France, whose support is also gratefully acknowledged.
\med

In several respects these lectures are the sequel of my course at Institut Henri Poincar\'e in 2001,
{\em Entropy production and convergence to equilibrium} \cite{vill:ihp:01}, but
this is also my first research memoir after ten years of leave. So I consider it a key step bridging my
past and future research. Awkward as this may seem, on this occasion I consider it fair to thank the french politicians
(EM, PM, JLM) who ensured, willingly or unwillingly, my defeat at the 2022 French Parliamentary elections, 
even though by a statistically nonsignificant margin, making way for my comeback in research.
\med

This work is dedicated to the memory of the great mathematician Henry P. McKean,
visionary pioneer, among many other subjects, of Fisher information in kinetic theory.
Henry passed away in April 2024, as this work was in preparation.
Our encounter, organised by the late Paul Malliavin, remains a dear memory.

\section{Fisher information}
\label{sec:fisher}

\begin{Def} Let $(M,g,\nu)$ be a complete Riemannian manifold equipped with a Borel reference measure $\nu$.
Then for any Borel probability measure $\mu$ on $M$, define
\begeq \label{FI}
I_\nu(\mu) = 
\begin{cases}
\dps \int_M \frac{|\nabla f|^2}{f} \,d\nu \qquad \text{if $\mu = f\nu$},  \\[5mm]
+\infty \qquad \text{if $\mu$ is singular to $\nu$}.
\end{cases}\endeq
\end{Def}

This formally 1-homogeneous, convex, lower semi-continuous functional is always well-defined (by convexity),
even if $f$ is not differentiable. It has a long and rich history, appearing in particular in the following five
fields, roughly ranked by chronological order, all related to the combination of statistics and measurement (repeated measurement, or typical measurement).

\subsection{Statistics}

Fisher introduced the functional $I$ hundred years ago in relation to his notion of {\em efficient statistics}. I shall develop the idea a little bit,
both for historical reasons and for the intuition.

Suppose $X_1,\ldots, X_N,\ldots$ are independent identically distributed (i.i.d.) random variables with law $f\nu$ on some Polish space $\X$, where $f=f(x,\theta)$ is a parametric density with
parameter $\theta\in\Theta$. An estimator $\hat{\theta}$ is a map $\X^N\to \Theta$. The most famous such estimator
is the maximum likelihood estimate (MLE): $\theta$ such that the density of the observed sample is as large as possible;
in other words 
\[ \hat{\theta}_{\text{MLE}}(x_1,\ldots,x_N) \text{ maximizes } \quad  f(x_1,\ldots,x_N;\theta) = \prod_{i=1}^N f(x_i,\theta). \]
As $N\to\infty$ we expect $\hat{\theta}$ to converge almost surely (a.s.) to $\theta$ if $X_1,\ldots,X_N$ are drawn according to
parameter $\theta$; meaning that the estimator is asymptotically unbiased. 
But what about the fluctuation of $\hat{\theta}$ around $\theta$?

\begin{Prop}[Fisher's fluctuation estimate for MLE] \label{propIMLE}
Let $\hat{\theta}$ be an unbiased estimator valued in $\Theta$, which near $\theta$ is a differentiable manifold. Then under $f(x,\theta)$,
\[\text{As $N\to\infty$,}\quad \hat{\theta} (X_1,\ldots,X_N) \sim {\cal N} \left( \theta, \frac{\I_\theta^{-1}(f)}{N}\right), \]
meaning that $\sqrt{N}(\hat{\theta}-\theta) \sim {\cal N}(0,\I_\theta^{-1}(f))$, the centered normal (Gaussian) law whose
covariance matrix is the inverse of $\I$, where
\begeq\label{defIij}
\I_\theta(f)_{ij} = \int \frac{1}{f} \derpar{f}{\theta_i}\derpar{f}{\theta_j} (x,\theta)\, \nu(dx) 
= \int \derpar{}{\theta_i}(\log f)\, \derpar{}{\theta_j}(\log f)\,f(x,\theta)\,\nu(dx).
\endeq
In particular,
\begeq\label{Itens}
\I_\theta(f) = \E \Bigl[ \nabla_\theta \log f \otimes \nabla_\theta \log f\Bigr];
\endeq
\begeq\label{trI}
I_\theta(f) := \tr \I_\theta(f) = \int \frac{|\nabla_\theta f|^2}{f}\,\nu(dx)
\endeq
\end{Prop}

Before going further, note that if $f(x,\theta)=f(x-\theta)$ on $(\R^d, dx)$ (fixed profile, but centred at an unknown parameter),
then $I_\theta$ is just the standard Fisher information with respect to Lebesgue measure ($\nu = {\cal L}^d$ in \eqref{FI}), independently of $\theta$.

\begin{proof}[Sketch of proof of Proposition \ref{propIMLE}]
First note that 
\[\I_\theta(f) = - \int f \nabla^2\log f(x,\theta)\,\nu(dx).\]
(If $f(x,\theta) = f(x-\theta)$ this results from integration by parts, but the formula holds true
regardless of this assumption.)
Indeed,
\[ 0 = \derpar{}{\theta_i} \int f\, d\nu = \int \derpar{f}{\theta_i}\,d\nu = \int f\,\partial_{\theta_i}(\log f)\,d\nu. \]
So
\begin{align*}
0 = \derpar{}{\theta_j} \int f\,\partial_{\theta_i}(\log f)
& = \int \derpar{f}{\theta_j} \derpar{}{\theta_i}(\log f) + \int f\,\derpar{{}^2}{\theta_j\,\partial\theta_i} (\log f)\\
& = \int f\,\derpar{}{\theta_j}(\log f)\derpar{}{\theta_i}(\log f) + \int f\,\derpar{{}^2}{\theta_j\,\partial\theta_i} (\log f),
\end{align*}
as announced. Now, $\hat{\theta}_{{\rm MLE}}$ is obtained by solving the equation
\[ \forall k\in \{1,\ldots,d\},\qquad \derpar{}{\theta_k} \prod_i f(x_i,\theta) = 0,\]
where $d$ is the dimension of $\Theta$ (locally near $\theta$); by taking logarithm,
\begeq\label{defMLE}
\frac1{N} \sum_{i=1}^N \derpar{}{\theta_k} \bigl(\log f(x_i,\theta)\bigr) = 0. 
\endeq
Let $\ell(x,\theta') = \log f (x,\theta').$ Then
\[ \derpar{}{\theta_k} \ell(x,\theta') = \derpar{}{\theta_k} \ell(x,\theta)
+ \sum_{k,m} \derpar{{}^2}{\theta_k\,\partial \theta_m} \ell(x,\theta) (\theta'-\theta)_m
+  o(|\theta'-\theta|),\]
and for simplicity let us assume that the latter $o(|\theta'-\theta|)$ is uniform in $x$; so
\begeq\label{solog}
\frac1{N} \sum_i \partial_{\theta_k}\ell (x_i, \hat{\theta}) - \frac1{N} \sum_i \partial_{\theta_k}\ell(x_i,\theta)
= \frac1{N} \sum_{i,m} \derpar{{}^2}{\theta_k\,\partial\theta_m}\ell(x_i,\theta)\,(\hat{\theta}-\theta)_m
+ o(|\hat{\theta}-\theta|),
\endeq
where $\hat{\theta} = \theta(x_1,\ldots,x_N)$.
The first term on the left hand side of \eqref{solog} vanishes by \eqref{defMLE},
and by the law of large numbers and central limit theorem, the second term is distributed near
\[ -\E\, \partial_{\theta_k} \ell(X,\theta) = -  \int f(x,\theta)\, \partial_{\theta_k}\log f(x,\theta)\,\nu(dx) = 0,\]
with fluctuation $N^{-1/2} {\cal N}(0,C)$, where $C$ is the covariance matrix
\begeq\label{covmatrix}
C = \E \Bigl[ \nabla_\theta \log f(X,\theta) \otimes \nabla_\theta \log f(X,\theta)\Bigr] = \I_\theta(f).
\endeq
As for the first term in the right-hand side of \eqref{solog}, by the law of large numbers again it is approximately
\[\sum_m \left( \E\, \derpar{{}^2}{\theta_k\,\partial\theta_m} \ell(X,\theta) \right) (\hat{\theta}-\theta)_m = -\I_\theta(f)(\hat{\theta}-\theta)_k.\]
To summarise: Asymptotically, $-\sqrt{N} \I_\theta(f) (\hat{\theta}-\theta) \sim {\cal N}(0,\I_\theta)$,
so $\sqrt{N}(\hat\theta-\theta)\sim {\cal N}(0,\I_\theta^{-1})$.
\end{proof}

It turns out that this estimate is best possible, so that MLE is essentially ``the best'' estimator, one which fluctuates the least
around the mean. If the reader wonders why this estimator is not systematically used, the answer is easy: A major drawback of MLE is that in general it is so difficult to evaluate! In any case its asymptotic optimality is expressed by
the following uncertainty principle.
\begin{Prop}[Cramér--Rao inequality] \label{propCR}
If $\theta\in\R^d$, $\hat{\theta}:(\R^d)^N\to\R^d$, and $(X_i)_{i\in\N}$ are i.i.d. with law $f(x,\theta)\,\nu(dx)$, then
\[ \Var (\hat{\theta}) \geq \frac{ \Bigl( d + \nabla_\theta\cdot \bigl(\E \hat{\theta} - \theta\bigr)\Bigr)^2}{N\,I_\theta(f)} - \bigl| \E\hat{\theta}-\theta \bigr|^2.\]
\end{Prop}
In particular, if $\E\hat{\theta}- \theta = o(1/\sqrt{N})$ and $\nabla\cdot \E(\hat{\theta}-\theta)\longrightarrow 0$ as $N\to\infty$, then $\liminf N\Var(\hat{\theta})\geq d^2/I_\theta(f)$.

\begin{proof}[Sketch of proof of Proposition \ref{propCR}]
Write $f^N(x_1,\ldots,x_N;\theta) = f^N(x;\theta)=\prod_i f(x_i,\theta)$. The integral is with respect to $\nu^{\otimes N}$. Then for any $k \in \{1,\ldots,d\}$,
\begin{align*} 
\derpar{}{\theta_k} \int f^N(x;\theta) (\hat{\theta}-\theta)_k 
& = \int \partial_{\theta_k}f^N(x;\theta)\, (\hat{\theta}-\theta)_k - \int f^N(x;\theta)\\
& = \int \partial_{\theta_k}f^N(x;\theta)\, (\hat{\theta}-\theta)_k - 1.
\end{align*}
Summing up for $1\leq k\leq d$, we obtain
\begin{align*}
d + \nabla_\theta\cdot \E(\hat{\theta}-\theta)
 & = \sum_k \int \partial_{\theta_k} f^N(x;\theta) (\hat{\theta}-\theta)_k\\ 
 & = \sum_k \int f^N(x;\theta) \,\partial_{\theta_k} \log f^N(x,\theta)\, (\hat{\theta}-\theta)_k\\
& \!\!\!\!\!\!\!\!\!\!\!\! \leq \sum_k   \left(\int f^N(x;\theta) \bigl|\partial_{\theta_k} \log f^N(x;\theta)\bigr|^2 \right)^{1/2}
\left(\int f^N(x;\theta) (\hat{\theta}-\theta)_k^2\right)^{1/2}\\
&  \leq \left(\int f^N(x;\theta) \sum_k \bigl|\partial_{\theta_k} \log f^N(x;\theta)\bigr|^2\right)^{1/2} \left( \int f^N(x;\theta) \sum_k (\hat{\theta}-\theta)_k^2\right)^{1/2}\\
&  = I_\theta(f^N)^{1/2} \bigl(\E |\hat{\theta}-\theta|^2\bigr)^{1/2}.
\end{align*}
Then on the one hand,
\[ I_\theta(f^N) = N I_\theta(f);\]
this comes from the fact that $\log f^N = \sum \log f_j$, where $f_j=f(x_j;\theta)$ and in the Fisher information of $f^N$ the cross-products $\int f^N \nabla_\theta \log f_j \cdot \nabla_\theta \log f_\ell$
simplify into $\int f_j f_\ell \nabla_\theta\log f_j \cdot\nabla_\theta\log f_\ell = 
(\nabla_\theta\int f_j)\cdot (\nabla_\theta\int f_\ell) =0$.
On the other hand
\[ \E|\hat{\theta}-\theta|^2 = \Var \hat{\theta} + |\E\hat{\theta}-\theta|^2.\]
So in the end
\[ d+ \nabla_\theta\cdot \E (\hat{\theta}-\theta)
\leq \sqrt{N} I_\theta(f) \Bigl( \Var\hat{\theta} + |\E\hat{\theta}-\theta|^2\Bigr)^{1/2},\]
which leads to the desired conclusion.
\end{proof}

The meaning is the following, at least for an unbiased estimator: If we wish to estimate $\theta$ from $N$ independent observations, then however clever we are, we can be sure that fluctuations around the true value will be at least $d\,I_\theta(f)^{-1/2}/\sqrt{N}$.  The higher $I$, the most efficient the estimate may be.

While Fisher was a devout frequentist, this interpretation led Jeffreys to use the same functional in a Bayesian context. If $\Theta$, the set of parameters, is a differentiable manifold, let us bet that $\theta$ is all the more likely when it is easier to estimate. (This has a flavor of the maximum likelihood estimate, in which one bets that $\theta$ is all the more likely when it makes it easier to produce the observation.)
Then Jeffreys' procedure, similar to Riemannian geometry, is to use $\I_\theta(f)$ as a metric on $\Theta$, and the prior ${\cal J}(d\theta)$ will be the associated (diffeomorphism-invariant) volume: 
\[{\cal J}(d\theta_1\ldots d\theta_d) = \sqrt{\det(\I_\theta(f))}\,d\theta_1\ldots\,d\theta_d.\]
For example, if $f(x,\theta) = {\cal N}(\mu,\sigma^2)$ and\\

\bul $\sigma$ fixed, $\mu$ variable: ${\cal J}(d\mu) = d\mu$ (Lebesgue on $\R$);

\bul $\mu$ is fixed, $\sigma$ variable: ${\cal J}(d\sigma) = d\sigma/\sigma$.

Or,  if $f(x,\theta) = (\theta,1-\theta)$ is the Bernoulli trial, then ${\cal J}(d\theta) = d\theta/\sqrt{\theta(1-\theta)}$. (In this example the observation space is discrete, $\nu = \delta_0+\delta_1$, but the Fisher information makes sense as the parameter space is continuous.)

\subsection{Information theory}

Information theory is the science of complexity, signal processing and transmission. The most important pillar of that theory is the Boltzmann--Shannon information, introduced by Boltzmann in he nineteenth century to statistically found thermodynamics, and rediscovered by Shannon in his cult 1948 book on information theory.
\begin{Def} \label{HI} Let $(\X, d,\nu)$ be a complete, separable metric space equipped with a Borel reference measure. Then
\[
H_\nu(\mu) = 
\begin{cases}
\dps \int_\X f\log f\,d\nu \qquad \text{if $\mu = f\nu$},  \\[4mm]
+\infty \qquad \text{if $\mu$ is singular to $\nu$}.
\end{cases}\]
\end{Def}
Here $\log$ is the natural logarithm with base $e$; but other choices, like base 10, or more frequently base 2, are also used without any real trouble.

In the fifties, the functional $I$ became another hero of information theory, thanks primarily to {\bf de Bruijn's identity:} If the reference measure $\nu$ is Lebesgue on $\R^d$ and $\Delta$ is the usual Laplace operator, then
\begeq\label{debruijn}
\ddt{t=0} H\bigl( e^{t\Delta}\mu \bigr) = - I(\mu).
\endeq
More generally, if $\nu(dx) = e^{-V(x)}\,dx$ and ${\cal L}\mu = \Delta \mu + \nabla\cdot (\mu\nabla V)$ is the generator of the natural semigroup on probability measures having $\nu$ as an invariant measure, then
\begeq\label{dtHI}
\ddt{t=0} H_\nu \bigl( e^{t{\cal L}}\mu \bigr) = - I_\nu(\mu).
\endeq
Note that ${\cal L}(f\nu) = (Lf)\nu$ where $Lf= \Delta f - \nabla V\cdot\nabla f$ is the generator defined on functions (densities) rather than measures; going from one formula to the other is a popular game.

The intuition is as follows. In either Boltzmann's or Shannon's theory, $H$ measures how $\mu$ is exceptional and may carry significant information by itself: The higher $H$, the more the distribution is rare and precious. But any observation of $\mu$ will imply some dose of fluctuation. Adding a small Gaussian noise to the signal is a way to measure the size of these fluctuations. The volume of configurations is proportional to $e^{-N H_\nu(\mu)}$, and its logarithmic derivative along the heat flow is $N I_\nu(\mu)$: the higher $I$, the larger the increase. If Boltzmann's theory is about the volume in the space of microstates, then Fisher's information is about the surface.
The finite-dimensional analogue of the perturbation by small Gaussian noise would actually be the classical formula relating volume to surface via enlargement by a small ball:
\[ S(A) = \ddt{t=0} V(A+B_t),\]
where $B_t= B(0,t)$ is the ball of radius $t$ and the sum is in the sense of Minkowski.

By abuse of notation, I shall often write $I(f)$ for $I(f\,{\cal L}^d)$ (${\cal L}^d$= Lebesgue measure on $\R^d$) or $I(X)=I(\mu)$ where $X$ is a random variable with law $\mu$. 
The four most famous properties of $I$ in information theory are as follows; they all have counterparts for $H$.
\sm

{\bf (a) Scaling:} For all random variable $X$ and $\lambda>0$,
\[ I(\lambda X) = \frac1{\lambda^2}\, I(X);\]
equivalently, if $f_\lambda(x) = \lambda^{-d} f(x/\lambda)$, then
\[ I(f_\lambda) = \frac1{\lambda^2}\ I(f).\]
The entropic counterpart is 
\[ H(f_\lambda) = H(f) - d\log\lambda.\]
\sm

{\bf (b) Marginal superadditivity:} If $X_1,X_2$ are valued in $\R^{d_1},\R^{d_2}$ respectively, then
\[ I(X_1,X_2) \geq I(X_1) + I(X_2),\]
with equality when $X_1,X_2$ are independent. (There may be correlations between $X_1$ and $X_2$, making the joint observation easier to detect.) Equivalently, if $f(x_1,x_2)$ is a joint distribution,
\[ I(f) \geq I \left( \int f\,dx_2\right) + I \left(\int fdx_1\right).\]
The entropic counterpart is
 \[H(f) \geq H \left( \int f\,dx_2\right) + H \left(\int fdx_1\right).\]
\sm

{\bf (c) Gaussian minimisers:} Let $X\in\R^d$ be a random vector with finite variance, and let $G$ be a Gaussian vector with the same variance as $X$; then
\[ I(X) \geq I(G) = \frac{d^2}{\Var(G)}.\]
Alternatively, if $\gamma_\sigma(x) = (2\pi\sigma^2)^{-d/2} \exp(-|x|^2/2\sigma^2)$, and $\sigma$ is chosen to be the variance of $f$, which means
\[ \int f(x) |x-u|^2\,dx = \int \gamma_\sigma(x) |x|^2\,dx = d\,\sigma^2, \qquad u = \int f(x)\,x\,dx, \]
then
\[ I(f) \geq I(\gamma_\sigma) = \frac{d}{\sigma^2}.\]
Note that Gaussians also saturate the Cram\'er--Rao inequality.
The entropic counterpart is
\[ H(f) \geq H(\gamma_\sigma) = -\frac{d}{2} \log (2\pi e\sigma^2).\]
\sm

{\bf (d) Stam inequalities:} These are two functional inequalities expressing a subadditivity property; there is the Blachman--Stam inequality dealing with $I$, and the Shannon--Stam inequality dealing with $H$.

First, {\bf the Blachman--Stam inequality} states that: If $X$ and $Y$ are two independent random vectors in $\R^d$, then for all $\alpha \in [0,1]$,
\[ I(\sqrt{1-\alpha}\, X + \sqrt{\alpha}\, Y) \leq (1-\alpha)\, I(X) + \alpha\, I(Y).\]
Alternatively, for all $f,g,\alpha$,
\[ I\bigl (f_{\sqrt{1-\alpha}} \ast g_{\sqrt{\alpha}}\bigr) \leq (1-\alpha)\, I(f) + \alpha \, I(g).\]
By optimisation over $\alpha$ and the scaling property (b), there is an equivalent form for the Blachman--Stam inequality: For all $f,g$,
\[ \frac1{I(f\ast g)} \geq \frac1{I(f)} + \frac1{I(g)}.\]

Then, the entropic counterpart consists in either of the two forms of the {\bf Shannon--Stam inequality}: for all densities $f,g$ and $\alpha\in [0,1]$,
\[ H\bigl (f_{\sqrt{1-\alpha}} \ast g_{\sqrt{\alpha}}\bigr) \leq (1-\alpha)\, H(f) + \alpha \, H(g),\]
or equivalently for all probability densities $f,g$,
\[ \qquad {\cal N} (f\ast g) \geq {\cal N}(f) + {\cal N}(g), \qquad {\cal N}(f) = \frac{e^{-\frac2{d}H(f)}}{2\pi e};\]
here ${\cal N}$ is Shannon's {\bf entropy power} functional.
All these inequalities are saturated by Gaussian functions and then reduce to the additivity of the variance of independent random variables. There is a parallel between Shannon--Stam and Brunn--Minkowski inequalities, which was spectacularly explored by Dan Voiculescu in his theory of free probability.

Without providing here complete proofs of the above claims, three arguments will be worth sketching for future reference.

\begin{proof}[Gaussian minimisers: Sketch of proof]
Let $u$ and $\sigma$ be the mean and variance of $f$, and $g$ the Gaussian with same mean and variance. So
\[ \int f(x) \,x\,dx = u = \int g(x)\,x\,dx,\qquad \int f(x) | x-u|^2\,dx = d\,\sigma^2 = \int g(x) |x-u|^2\,dx,\]
\[ g(x) = \frac{e^{-|x-u|^2/(2\sigma^2)}}{(2\pi\sigma^2)^{d/2}}.\]
Then
\begin{align*}
0 & \leq \int f(x) \bigl| \nabla\log f - \nabla \log g \bigr|^2\,dx \\
& = \int f(x) |\nabla \log f|^2 + 2 \int f(x) \nabla \log f(x) \cdot \left(\frac{x-u}{\sigma^2}\right)\,dx
+ \int f(x) \frac{|x-u|^2}{\sigma^4}\,dx \\
& = I(f) + \frac{2}{\sigma^2} \int \nabla f(x)\cdot (x-u)\,dx + \frac1{\sigma^4} \int f(x) |x-u|^2\,dx \\
& = I(f) - \frac{2d}{\sigma^2} + \frac{d}{\sigma^2} = I(f) - \frac{d}{\sigma^2} = I(f)-I(g).
\end{align*}
\end{proof}

\begin{proof}[Blachman--Stam inequality: Sketch of proof]
Let $Q_\alpha(f,g) = f_{\sqrt{1-\alpha}}\ast g_{\sqrt{\alpha}}$. Then, distributing $\nabla$ on either function in the convolution product,
\begin{align*}
\nabla Q_\alpha(f,g) & = (1-\alpha) \, \nabla Q_\alpha(f,g) + \alpha \nabla Q_\alpha (f,g)\\
& = (1-\alpha)\,\nabla f_{\sqrt{1-\alpha}}\ast g_{\sqrt{\alpha}} + \alpha\, f_{\sqrt{1-\alpha}}\ast \nabla g_{\sqrt{\alpha}}\\
& = \int \left[ (1-\alpha)\, \frac{\nabla f_{\sqrt{1-\alpha}}(y)}{f_{\sqrt{1-\alpha}}(y)}
+ \alpha\, \frac{\nabla g_{\sqrt{\alpha}}(x-y)}{g_{\sqrt{\alpha}}(x-y)}\right]\,
f_{\sqrt{1-\alpha}}(y)\,g_{\sqrt{\alpha}}(x-y)\,dy
\end{align*}
(evaluated at $x$, of course). By Jensen's inequality, or Cauchy--Schwarz, for each $x\in\R^d$,
\begin{multline*}
\bigl| \nabla Q_\alpha(f,g)\bigr| (x) \leq
\left( \int \left| (1-\alpha)\,\frac{\nabla f_{\sqrt{1-\alpha}}(y)}{f_{\sqrt{1-\alpha}}(y)}
+ \alpha\, \frac{\nabla g_{\sqrt{\alpha}}(x-y)}{g_{\sqrt{\alpha}}(x-y)} \right|^2\,
f_{\sqrt{1-\alpha}}(y)\,g_{\sqrt{\alpha}}(x-y)\,dy \right)^{1/2}\\
\left(\int f_{\sqrt{1-\alpha}}(y)\,g_{\sqrt{\alpha}}(x-y)\,dy\right)^{1/2},
\end{multline*}
which is the same as
\[ \frac{\bigl| \nabla Q_\alpha(f,g)\bigr|^2}{Q_\alpha(f,g)}\, (x) \leq
\int \left|  (1-\alpha)\,\frac{\nabla f_{\sqrt{1-\alpha}}(y)}{f_{\sqrt{1-\alpha}(y)}}
+ \alpha\, \frac{\nabla g_{\sqrt{\alpha}}(x-y)}{g_{\sqrt{\alpha}}(x-y)} \right|^2\,
f_{\sqrt{1-\alpha}}(y)\,g_{\sqrt{\alpha}}(x-y)\,dy.\]
Upon integration,
\begin{align*}
I\bigl(Q_\alpha(f,g)\bigr) 
& \leq  \iint \left| (1-\alpha)\nabla\log f_{\sqrt{1-\alpha}}(y) + \alpha \nabla\log g_{\sqrt{\alpha}}(x-y)\right|^2\,
f_{\sqrt{1-\alpha}}(y)\,g_{\sqrt{\alpha}}(x-y)\,dx\,dy\\
& =  (1-\alpha)^2 \left( \int f_{\sqrt{1-\alpha}} |\nabla \log f_{\sqrt{1-\alpha}}|^2\right) \left(\int g_{\sqrt{\alpha}}\right)
  + \alpha^2 \left( \int f_{\sqrt{1-\alpha}}\right) \left( \int g_{\sqrt{\alpha}} |\nabla \log g_{\sqrt{\alpha}}|^2\right)\\
& \qquad\qquad + 2\,\alpha(1-\alpha) \left( \int f_{\sqrt{1-\alpha}}\nabla \log f_{\sqrt{1-\alpha}}\right)
\left(\int g_{\sqrt{\alpha}}\nabla\log g_{\sqrt{\alpha}}\right).
\end{align*}

The last term vanishes since $\int f \nabla\log f = \int \nabla f = 0$, and of course $\int g\nabla\log g=0$ as well (as in Cram\'er--Rao, the key is to apply Cauchy--Schwarz {\em before} expanding the Fisher information) and we are left with
\[ I\bigl(Q_\alpha(f,g)\bigr) \leq (1-\alpha)^2 I(f_{\sqrt{1-\alpha}}) + \alpha^2 I(g_{\sqrt{\alpha}}) = (1-\alpha) I(f) + \alpha I(g),\]
after a final application of the scaling property.
\end{proof}

\begin{proof}[Shannon--Stam inequality: Sketch of proof]
While the proof of Blachman--Stam involved symmetries and scalings, the proof of Shannon--Stam is completely different and invokes a semigroup argument. First note that the heat semigroup commutes with $Q_\alpha$: Indeed, $e^{t\Delta} f = \gamma_{\sqrt{2t}}\ast f$,
\begin{align*}
Q_\alpha\bigl(e^{t\Delta} f, e^{t\Delta}g\bigr) & = (\gamma_{\sqrt{2t}}\ast f)_{\sqrt{1-\alpha}} \ast (\gamma_{\sqrt{2t}}\ast g)_{\sqrt{\alpha}} \\
& = (\gamma_{\sqrt{(1-\alpha)2t}}\ast \gamma_{\sqrt{\alpha 2t}})\ast (f_{\sqrt{1-\alpha}}\ast g_{\sqrt{\alpha}}) = \gamma_{\sqrt{2t}}\ast Q_\alpha(f,g) = e^{t\Delta}Q_\alpha(f,g).
\end{align*}
(This is even easier to see with random variables.) In particular, by \eqref{debruijn},

\begin{multline*} -\frac{d}{dt} H \Bigl( Q_\alpha (e^{t\Delta} f, e^{t\Delta} g)\Bigr) = 
-\frac{d}{dt} H\bigl(e^{t\Delta} Q_\alpha(f,g)\bigr) = I \bigl(Q_\alpha(e^{t\Delta}f, e^{t\Delta}g)\bigr)\\
\leq (1-\alpha) I(e^{t\Delta}f) + \alpha I(e^{t\Delta} g) = - \frac{d}{dt} \Bigl[ (1-\alpha) H(e^{t\Delta}f) + \alpha H(e^{t\Delta}g)\Bigr].
\end{multline*}

The idea then is to integrate this inequality from $t=0$ to $t=\infty$.
However integration in $t$ is a bit tricky since $I(e^{t\Delta}f)$ is typically $O(1/t)$ as $t\to\infty$ along the heat flow, and likewise for the other terms. So it is better to replace the heat semigroup $e^{t\Delta}$ by its variant localised by rescaling, the semigroup $(S_t)_{t\geq 0}$ solving the linear drift-diffusion equation
\begeq\label{FP}
\derpar{f}{t} = \Delta f + \nabla\cdot (fx).
\endeq
Equation \eqref{FP} is known as the linear Fokker--Planck, or adjoint Ornstein--Uhlenbeck semigroup, and its behaviour as $t\to\infty$ is convergence to the Gaussian $\gamma=\gamma_1$. This semigroup preserves the class of Gaussian distributions and commutes with $Q_\alpha$. Then the inequality we obtained can be rewritten
\begeq\label{ddtStHI}
-\frac{d}{dt} H\bigl(Q_\alpha(S_tf,S_tg)\bigr) \leq 
-\frac{d}{dt} \Bigl[(1-\alpha)\, H(S_tf) + \alpha H(S_tg)\Bigr],
\endeq
and all three quantities $H(Q_\alpha (S_tf,S_tg)) = H(S_tQ_\alpha(f,g))$, $H(S_tf)$ and $H(S_tg)$ converge to $H(\gamma)$ as $t\to\infty$ (actually the convergence is exponentially fast), so there is no problem in integrating \eqref{ddtStHI} from $t=0$ to $\infty$, and the Shannon--Stam inequality follows. 
\end{proof}

Stam's inequalities are precious because they have the right scaling with respect to summation of independent variables (this is no accident, as the correct rescaling is the one which preserves the variance). For instance, if $X_1,X_2,\ldots$ are i.i.d. with finite variance, then $S_n=(X_1+\ldots+X_n)/\sqrt{n}$ satisfies 
\[ I(S_{2n})\leq I(S_n),\qquad H(S_{2n})\leq H(S_n).\]
In other words, $H$ and $I$ are Lyapunov functionals along the rescaled summation operation, at least along powers of 2 ($n=2^k$). Linnik and Barron and others used this and related estimates to get explicit rates of convergence for the central limit theorem. Even for the intuition it is precious, identifying the universal Gaussian fluctuation profile as a consequence of Stam's inequalities and the Gaussian minimisation property.

\subsection{Large deviations}

To present this line of thought it will be useful to compare the meanings of $H$ and $I$ in the spirit of large deviations. If Boltzmann's $H$ functional quantifies how rare a distribution function is and how difficult it is to realize, Fisher's $I$ functional is about how visible this distribution is and how difficult it is to measure. 

Boltzmann's historical justification of the $H$ function was the following. Take $N$ particles occupying $K$ possible states; if the frequencies $f_k = N_k/N$ are given ($N_k$ = number of particles in state $k$), then how many ways are there to realize them? The result is
\[ W_N(f) = \frac{N!}{(Nf_1)!\ldots (Nf_K)!}.\]
(Obviously each $f_k$ has to be an integer multiple of $1/N$ for the problem to have a solution.) Then
\[ \frac1{N} \log W_N(f) \xrightarrow[]{N\to\infty} - \sum_{k=1}^K f_k\log f_k. \]
This is Boltzmann's celebrated formula, $S=k\log W$, in its barest expression.
Here is a neat quantitative bound, recast by Cover and Thomas in the language of large deviations: 
If all particles are independent and drawn according to $\nu = (f_1,\ldots,f_K)$, and
$\hat{\mu} = (\hat{f}_1,\ldots,\hat{f}_K)$ is the collection of observed frequencies,
then $\hat{\mu}\longrightarrow \nu$ as $N\to\infty$ (law of large numbers) and for any $\var>0$,
\begeq\label{Pexp}
\P_{\nu^{\otimes N}} \bigl [ H_\nu(\hat{\mu})\geq \var\bigr]
\leq e^{-N \left[\var - (K-1)\frac{\log (N+1)}{N}\right]}.
\endeq
Note that $H_\nu(\mu)\geq 0$ if $\mu,\nu$ are both probability measures, since $H_\nu(\mu) = 
\int f\log f\,d\nu \geq (\int f\,d\nu)\log (\int f\,d\nu) = 0$.
Inequality \eqref{Pexp} is an instance of {\bf Sanov's theorem}, which states roughly speaking that
\[\P_{\nu^{\otimes N}} \bigl[\hat{\mu}^N \in {\cal O}\bigr]
\simeq \exp \left[ -N \inf_{\mu\in {\cal O}} H_\nu(\mu)\right]\qquad \text{as $N\to\infty$} \]
(large deviations of the empirical measure), under various assumptions and refinements.

A more precise formulation would be: If $(\varphi_k)_{k\in\N}$ is a dense sequence of observables (test functions in some well-chosen function space, for instance bounded and going to zero at infinity, or bounded Lipschitz), then
\begeq\label{rigSanov}
\lim_{K\to\infty} \ \lim_{\var\to 0} \ \lim_{N\to\infty}
-\frac1{N} \log \P_{\nu^{\otimes N}}\left[ \forall k\leq K,
\quad \left| \frac1{N} \sum_{n=1}^N \varphi_k(X_n) - \int \varphi\,d\mu\right| \leq \var\right] = H_\nu(\mu).
\endeq
(Actually one should write inequalities involving $\limsup$ and $\liminf$, but let me remain sloppy in this broad presentation.)

Now, what about the large deviation meaning of Fisher? One intuition is that fluctuation of particles states will inevitably blur the observation of the empirical measure, and Fisher's information gives a lower bound for that. Another intuition pertains to reconstruction of trajectories. To make a statement, we need a dynamical model, so let us assume that particles are subject to Gaussian fluctuations and feel a potential field $V$: this gives the standard Langevin stochastic differential equation
\begeq\label{SDE} \frac{dX}{dt} = \sqrt{2}\,\frac{dB}{dt} - \nabla V(X_t), \endeq
where $(B_t)$ is the standard Brownian motion and the equilibrium measure is $\nu(dx) = e^{-V(x)}\,dx/Z$, with $Z$ a normalising constant. The transition kernel would be something like 
\[ a_t(x,y) \simeq \frac{e^{-| y-(x-t\nabla V(x))|^2/4t}}{(4\pi t)^{d/2}}.\]
The likelihood of a time-discretised trajectory may be proportional to 
\[ f(x_0)\, a(x_0,x_1)\ldots a(x_{N-1},x_N)
= f(x_0) \prod e^{\log a(x_i,x_{i+1})},\]
so the log likelihood of that trajectory, after dividing by time, should be something like
\[ \frac1{N} \sum \frac{\Bigl| x_{i+1}- \bigl(x_i - \delta t \nabla V(x_i)\bigr)\Bigr|^2}{4\,\delta t}. \]

When we try to reconstruct $V$, we might assume $dX_t = \sqrt{2}\,dB_t + \xi(X_t)\,dt$, 
\[ \frac1{4N} \sum (\delta t) \left| \frac{x_{i+1}-x_i}{\delta t} + \nabla V(x_i)\right|^2
\simeq \frac1{4N} \sum \delta t \bigl | \xi(x_i) + \nabla V(x_i) \bigr|^2 + \quad \text{ fixed divergent Brownian part}\]
(the cross product involving $(B_{t_{i+1}}-B_{t_i})\cdot \nabla V(X_i)$ will disappear by martingale property).
All in all, the variational problem which emerges is
\[ \inf \left\{ \frac1{4T} \int_0^T f(t,x) \bigl| \xi(t,x) + \nabla V(x) \bigr|^2\,dx; \qquad \derpar{f}{t} = \Delta f - \nabla \cdot (\xi f)\right\}. \]
This tells us about the probability to observe a time-dependent profile $f(t,x)$ which, by some small probability event, would not be the solution of $\pa_t f = \Delta f +\nabla\cdot (f\nabla V)$. Let us particularise this to the static case. As $T\to\infty$ the average of particle properties along particle trajectories will certainly look like those of the equilibrium measure $e^{-V}$, but it may be that we are tricked by an exponentially small probability event and observe another profile $f_\infty$, then the infimum should be over $\xi$ such that
$0 = \Delta f_\infty - \nabla\cdot (\xi f_\infty) = \nabla\cdot (f_\infty (\nabla\log f_\infty-\xi))$. The infimum of $\int f_\infty |\xi+\nabla V|^2$
is obtained for $\xi=\nabla\log f_\infty$, and the value is $\int f_\infty |\nabla \log f_\infty + \nabla V|^2 = I_{e^{-V}}(f_\infty\,dx)$.
After this long chain of approximations and reductions, we finally arrived at a specific instance of the Donsker--Varadhan large deviation principle, which is parallel to Sanov's problem, but now the problem is to empirically estimate the equilibrium distribution not from independent samples but from trajectories: when particles satisfy the stochastic differential equation \eqref{SDE},
\[ \lim_{K\to\infty} \ \lim_{\var\to 0} \ \lim_{T\to\infty}
- \frac1{T} \log \P\left[\ \forall  k\leq K, \quad 
\left| \frac1{T} \int_0^T \varphi_k(X_t)\,dt - \int \varphi\,d\mu \right| \leq \var \right] = \frac14\, I_\nu(\mu). \]

To summarise, Fisher's information appears in large deviation estimates when the problem is the reconstruction of the density by the trajectorial empirical mean. Compare with Sanov's formula \eqref{rigSanov}.

\subsection{Logarithmic Sobolev inequalities}

Introduced by Leonard Gross as a substitute for Sobolev inequalities in infinite dimension, these inequalities read as follows: 
Given $(M,g,\nu)$ a complete Riemannian manifold with a reference measure $\nu$, for all $u:M\to \R$,
\[ \int_M |u|^2 \log |u|^2\, d\nu \leq
\frac{2}{K} \int_M |\nabla u|^2\,d\nu + \left( \int_M |u|^2\,d\nu\right) \log \left(\int_M |u|^2\,d\nu\right).\]
Or equivalently, if $P(M)$ stands for the space of Borel probability measures on $M$,
\begeq\label{LSIK} \forall \mu\in P(M),\qquad H_\nu(\mu) \leq \frac1{2K}\, I_\nu(\mu).
\qquad \text{(LSI($K$))}  \endeq
This property may be true or not (more rigorously, the optimal $K$ may be positive or~0), depending on $\nu$. If true, this imposes stringent localisation properties on $\nu$: at least Gaussian decay at $\infty$.
The archetype of such a property is the Stam--Gross logarithmic Sobolev inequality (first proven by Stam in an equivalent form, rediscovered by Gross) : If $\gamma$ is the standard Gaussian with identity covariance matrix on $\R^d$, then
\[ \forall \mu\in P(\R^d), \qquad H_\gamma(\mu) \leq \frac{I_\gamma(\mu)}{2}, \]
independently of the dimension $d$. 

Logarithmic Sobolev inequalities play tremendous service in the study of diffusion processes, large systems of particles, concentration of measure and isoperimetry, in particular. So it is a classical and important problem to know convenient criteria for them to hold. The most important of those is due to Bakry \& \'Emery: If $\nu(dx) = e^{-V(x)}\,\vol(dx)$ on $(M,g)$ and $\nabla^2V + \Ric \geq K g$ for some $K>0$, with $\vol$ being the Riemannian volume, $\nabla^2$ the Riemannian Hessian and $\Ric$ the Ricci curvature, then $\nu$ satisfies LSI($K$). Their argument uses a semigroup again: Identifying the density $f$ and the measure $f\,\vol$, writing as before $S_t$ for the semigroup solving $\pa_t f = \Delta f + \nabla\cdot (f\nabla V)$, they prove
\[ -\frac{d}{dt} I_\nu(S_t\mu) \geq 2 K\, I_\nu (S_t\mu) = - 2K \, \frac{d}{dt} H_\nu(S_t\mu),\]
and integration from $t=0$ to $\infty$ yields \eqref{LSIK}.

I shall come back in more detail to log Sobolev inequalities and to the Bakry--\'Emery argument later in this course.

\subsection{Optimal transport theory}

The Monge--Kantorovitch theory of optimal transport is about the most economical way to rearrange a mass distribution from initial to final state. At the end of the nineties, a dynamical version has been explored at length by Brenier--Benamou and Otto, in the formalism of fluid mechanics and Riemannian geometry. If an infinitesimal variation of measure, $\pa_t\mu$, is given, then the associated quadratic cost is
\[ \|\pa_t\mu\|^2 = \inf \left\{ \int |\xi|^2\,d\mu; \quad \pa_t\mu + \nabla\cdot(\mu\xi) = 0 \right\}. \]
This provides a formal Riemannian structure on $P(M)$, for which geodesics are dynamical realisations of the optimal transport problem with quadratic cost function. (If one wishes to draw an analogy with Jeffreys' prior, parameterize $\mu\simeq \mu_0$ by vector fields $\xi$, each vector field $\xi$ on the time interval $[0,\var]$ generates a flow $T_\var$ which pushes $\mu_0$ forward to $\mu$. Then take quotient to define a geometry on the space of probability measures.)

A major insight by Felix Otto is that this procedure identifies the heat flow (on probability measures) as the gradient flow for the $H$ functional. Moreover, it is not difficult to identify Fisher's information with the square norm of the gradient (in the sense of Riemannian geometry) of the $H$ functional, 
\[ I_\nu (\mu) = \bigl\| \grad H_\nu (\mu)\bigr\|^2.\]
This was the start of a series of works, in which I was strongly involved around the turn of the previous century, providing new interpretations and discoveries in the interplay between Riemannian geometry, entropy, Fisher information and optimal transport. For instance the HWI inequality reinforces Bakry--\'Emery's theorem as follows: If $\nu$ is a probability measure on $M$ and $K>0$ then
\begeq\label{HWI}
\nabla^2V + \Ric \geq K g \quad \Longrightarrow \forall \mu\in P(M),\quad 
H_\nu(\mu) \leq W_2(\mu,\nu) \sqrt{I_\nu(\mu)} - \frac{K}2\, W_2(\mu,\nu)^2,
\endeq
where $W_2$ is the $2$-Wasserstein distance,
\[ W_2(\mu,\nu) = \Bigl\{\inf \E d(X_0,X_1)^2 ; \quad \law(X_0)=\mu, \quad \law(X_1)=\nu \Bigr\}^{1/2}. \]
Note that by $ab - Ka^2/2 \leq b^2/(2K)$, the conclusion of \eqref{HWI} implies \eqref{LSIK}.
\medskip

\subsection{And then...}

After this overview, it is time to turn to the use of Fisher information in mathematical physics. There are at least two fields in which it plays an important role. One is quantum chemistry: if $\psi$ is a wave function, then $\rho = |\psi|^2$ is a probability density.
Writing $\psi = \sqrt{\rho}\,e^{i\varphi}$, we have 
\[ \int |\nabla\psi|^2 = \int \Bigl( |\nabla\sqrt{\rho}|^2 + \rho\, |\nabla\varphi|^2 \Bigr) = \frac{I(\rho)}{4} + \int \rho\, |\nabla\varphi|^2,\]
so that Fisher's information is naturally involved in the computation of the ``kinetic energy'' of the particle.

Another field of powerful application for the Fisher information is the kinetic theory of gases and plasmas, and this will be the focus of all the remaining of these notes. As we shall see, in particular, Fisher's information is monotonous (nonincreasing) along solutions of the most fundamental model in collisional kinetic theory, and this has important consequences.

\bibnotes

Some of the important founding papers are those by Fisher \cite{fisher:25}, Shannon and Weaver \cite{shannonweaver:book}, Jeffreys \cite{jeffreys:46}, Stam \cite{stam:59}, Linnik \cite{linnik:59}, Donsker--Varadhan \cite{donskvar:markov:75}, Gross \cite{gross:log:75}, Bakry--\'Emery \cite{bakem:hyperc:85}, Brenier--Benamou \cite{BB:semigeo:98}, Jordan--Kinderlehrer--Otto \cite{JKO:FP:98}, Otto \cite{otto:geom:01}, Otto--Villani \cite{OV:00}.

The time-honoured reference course on information theory is the book by Cover and Thomas \cite{CT:book}, containing itself hundreds of references. The discussion on large deviations and Boltzmann's formula there is illuminating. My version of the classical Cramér--Rao bound is a very slight generalisation of the classical bound, for coherence of the presentation.

The parallel between Brunn--Minkowski and information theory is analysed by Dembo, Cover and Thomas \cite{DCT:inf:91}. Free probability is exposed in Voiculescu's lecture notes \cite{voiculescu:stflour}. From there Voiculescu and Szarek \cite{szarekvoic:shannon:00} also devised a brilliant proof of Shannon--Stam resting on a large-dimensional analogue of Brunn--Minkowski (a proof which is considerably more complicated than the one by Stam, but very meaningful).

Explicit central limit theorems were developed by Barron \cite{barron:entropy:86} and further authors.

A classical course on large deviations is the book by Dembo and Zeitouni \cite{DZ:2}. Some quantitative variants of Sanov's theorem are in my work with Bolley and Guillin \cite{BGV:07}. The Donsker--Varadhan theorem can be found in various levels of generality \cite{rezakhanlou:LDP} but it is not so easy to find a neat reference for the instance presented here, and it would be good to have some quantitative estimates.

Logarithmic Sobolev inequalities are the object of many works, among which the synthesis book by the Toulouse research group \cite{toulouse:sobolog} and the one by Bakry--Gentil--Ledoux \cite{BGL:book}.

My own books can be consulted for optimal transport theory \cite{vill:TOT:03,vill:oldnew}. Chapter 21 in the latter reference is also largely devoted to logarithmic Sobolev inequalities.

An introduction to Fisher's information in variational problems related to quantum physics can be found in Lieb--Loss \cite{liebloss:book:2}.

\section{Core kinetic theory}

Classical kinetic theory was born from the heroic efforts of Maxwell and Boltzmann to statistically describe assemblies of particles through their time-dependent distribution function in position and velocity variables: $f=f(t,x,v)$ where, say, $t\geq 0$, $x\in \Om_x \subset\R^d$ or $x\in \R^d/\Z^d$, and $v\in\R^d$. If particles are subject to a macroscopic force field $F=F(t,x)$ (possibly induced or partially induced by the particles themselves) and interact through localised binary microscopic ``collisions'' then the model is the Boltzmann equation
\begeq\label{BE}
\derpar{f}{t} + v\cdot\nabla_x f + F\cdot\nabla_v f = Q(f,f)
\endeq
and Boltzmann's collision operator $Q$ reads
\begeq\label{Qff}
Q(f,f) = \int_{\R^d} \int_{\S^{d-1}} \bigl( f'f'_* - ff_*\bigr)\, B\,d\sigma\,dv_*, 
\endeq
where $f'=f(t,x,v')$ and so on, and
\begeq\label{v'v'star}
v'=\frac{v+v_*}2 + \frac{|v-v_*|}2\, \sigma, \qquad v'_* = \frac{v+v_*}2 - \frac{|v-v_*|}2 \,\sigma, 
\endeq
$B=B(v-v_*,\sigma)$ only depends on $|v-v_*|$ (relative velocity) and $\frac{v-v_*}{|v-v_*|}\cdot\sigma$ (cosine of the deviation angle). Think of $(v,v_*) \rightarrow (v',v'_*)$ as a sudden variation in the velocities of a pair of colliding particles:
Conservation of momentum induces $v+v_* = v'+v'_*$, conservation of energy induces $|v|^2+|v_*|^2 = |v'|^2 + |v'_*|^2$, these $d+1$ conservation laws leave room for $d-1$ parameters for solutions, and that is the role of $\sigma = (v'-v'_*)/|v'-v'_*|$, the direction of the relative postcollisional velocities.
Equation \eqref{BE} makes sense only with boundary conditions, but let me skip this issue here, or just think that $x$ lives in $\R^d/\Z^d$ so that there is no boundary in the physical domain.

In one of the greatest artworks of mathematical physics ever, Maxwell and Boltzmann set up the basis of classical kinetic theory in relation to the then-speculative atomistic theory. In particular they showed that
\medskip

{\bf (1)} Under an assumption of molecular chaos at initial time, and in a regime in which each particle collides at least a few times per unit of time, \eqref{BE} is a plausible model for the evolution of the density of particles;
\medskip

{\bf (2)} If $H(f) = \iint f\log f\,dx\,dv$ then in the absence of force and under appropriate boundary conditions, $dH/dt\leq 0$ (Boltzmann's $H$ Theorem, turning the Second Law of thermodynamics into a plausible {\em theorem} for this particular but fundamental system). Boltzmann's proof is a classical gem: applying the two basic ways to exchange variables, which are $(v,v') \longleftrightarrow (v_*,v'_*)$ and $(v,v_*,\sigma) \longleftrightarrow (v',v'_*,k)$ ($k=\frac{v-v_*}{|v-v_*|}$), using the additivity of the logarithm, the fact that the transport operator $v\cdot \nabla_x$ preserves all integral functions $\int C(f)\,dv\,dx$, the additivity property of the log and its increasing property,
\begin{align*}
H'(f)\cdot Q(f,f) &
= \iint \left( \iint (f'f'_* - ff_*)\, B\,dv_*\,d\sigma\right)(\log f+1)\,dv\,dx\\
& = \frac14 \iiiint (f'f'_* - ff_*) \bigl(\log f  + \log f_* - \log f'- \log f'_* \bigr)\, B\,dv\,dv_*\,d\sigma\,dx\\
& = -\frac14 \iiiint (f'f'_* - ff_*) \bigl (\log ff_* - \log f'f'_*\bigr)\, B\,dv\,dv_*\,d\sigma\,dx \leq 0.
\end{align*}
From this computation emerges {\bf Boltzmann's dissipation functional} (dissipation of $H$, or more rigorously its integrand in the $x$ variable)
\begeq\label{DB}
D_B(f) = \frac14 \iiint (f'f'_* - ff_*) \bigl(\log f'f'_* - \log ff_*\bigr)\,B\,dv\,dv_*\,d\sigma.
\endeq
This fundamental functional quantifies the strength of the dissipation process at work in \eqref{BE}.
\medskip

{\bf (3)} Equilibria for \eqref{BE}, in the absence of macroscopic force and except in presence of certain spatial symmetries, are of the form
\begeq\label{Mglo}
M(x,v) = M_{\rho u T}(v) = \rho\,\frac{e^{-|v-u|^2/(2T)}}{(2\pi T)^{d/2}},
\endeq
where $\rho\geq 0$, $u\in\R^d$, $T\geq 0$ are constants. In particular these equilibria are space-homogeneous. These Gaussian distributions are called Maxwellians in this context. (There are additional equilibria under certain symmetries, for instance axisymmetric domains with specular reflection of particles at the boundary.)
\medskip

{\bf (4)} As $t\to\infty$ solutions of \eqref{BE} converge to equilibrium.
\medskip

{\bf (5)} In a suitable regime of many collisions per unit of time, solutions of \eqref{BE} are well approximated by local equilibria $M_{\rho u T}(v)$, where now $\rho, u, T$ are functions of $t$ and $x$ and satisfy certain hydrodynamic equations for density, velocity and temperature. 
\medskip

In this way Maxwell and Boltzmann recovered and discovered some properties of fluids, sometimes consistent with hydrodynamics and sometimes not. The first two striking such discoveries by Maxwell were that the viscosity of a rarefied gas is independent of its density (something that he himself could not believe at first), and that near the boundary a gas may flow from low temperature to high temperature (Maxwell's paradoxical {\em thermal creep}). 
By all means the program  was a complete success, and a milestone in the discovery of atoms.

Still, even after 150 years and in spite of thousands of works, Boltzmann's equation retains some deep mathematical mysteries. Most famously, we are in want of a large time convergence proof from particle systems to \eqref{BE}. (Lanford's theorem, even after half a century of corrections and improvements, only works out for about a fraction of a collision time, which is of the order of $10^{-9} s$ under usual conditions for the air around us.) This and other works suggest that a theory of regular solutions is needed, but as of today the latter is nowhere to be seen, except in particular regimes, for instance very close to equilibrium.
Equation \eqref{BE} contains several major difficulties:

\bul the complexity of the operator $Q$;

\bul the degenerate nature of the equation, involving a conservative transport operator $v\cdot\nabla_x$ (or $v\cdot\nabla_x + F\cdot\nabla_v$) and a dissipation, but only in $v$ variable, the collision operator -- this is akin to a hypoelliptic or hypocoercive issue;

\bul the quadratic nature of $Q$, whose amplitude is proportional to $\rho^2$ ($\rho = \int f\,dv$ = spatial density), making a priori estimates a challenge;

\bul the singular nature of the hydrodynamic approximation, in a regime when $B$ becomes very large.
\medskip

In this dire but exciting situation, it is legitimate to study all difficulties in parallel. A good way to better understand the structure of $Q$ is to focus on the {\em spatially homogeneous Boltzmann equation},
\begeq\label{SHBE}
\derpar{f}{t} = Q(f,f),
\endeq
and this will be the main focus of this course.

\bibnotes

The founding papers of kinetic theory are those of Maxwell \cite{maxw:67} and Boltzmann \cite{boltz:weitere:72}. Maxwell's creep thermal effect is from \cite{maxw:79}. Boltzmann's book \cite{boltz:book} is a bright synthesis of the whole theory and had considerable influence, as shown for instance in Perrin's classic book {\em Les Atomes} \cite{perrin:book}. Albert Einstein, Erwin Schr\"odinger, Max Planck all built from Boltzmann's work to develop their theories of atomistic laws. The story is described in various books of history of science, like Lindley's \cite{lindley:book}.

Modern survey books on Boltzmann's theory from the past thirty years are those by Cercignani \cite{cer:book:88,cer:book:00}, Cercignani--Illner--Pulvirenti \cite{CIP:book:94}, Sone \cite{sone:book}, and my own review of collisional kinetic theory \cite{vill:handbook:02}. The most elaborate version of Lanford's theorem is the book by Bodineau, Saint-Raymond and Gallagher \cite{GSRT:N2B}. The recent PhD by Corentin Le Bihan provides a good review of the derivation problem \cite{lebihan:PhD}.

\section{Fisher information into kinetic theory} \label{secFIK}

Fisher information was imported into kinetic theory by Henry P. McKean in a remarkable paper of 1966. His motivation came from the problem raised by Kac on the speed of approach to equilibrium, starting with a one-dimensional caricature of the Boltzmann equation. The idea was sharp and clear: Draw a parallel between the convergence to equilibrium in Boltzmann's theory, and the central limit theorem in classical probability theory and statistics.
Repeated interactions tend to push the distribution closer to a Maxwellian (or Gaussian); by writing the solution as a superposition of contributions from all interaction histories, represented by interaction trees, show that most trees are ``deep'' and thus the corresponding terms are close to Maxwellian.

McKean borrowed from Linnik the use of Fisher's functional $I$, which had been instrumental in getting some results of explicit convergence for the central limit theorem. He worked on Kac's simplified equation and (a) proved that $I(f)$ is nonincreasing for this model, (b) established the first result of exponential convergence for the distribution function (not just for the solution of the linearised equation) as $t\to\infty$. Note that in dimension 1, the bound on $I$ is already a strong regularity estimate, implying H\"older continuity of $f$.

Since that time, McKean's program has been extended and refined through a number of tools -- functional inequalities,  probability metrics, entropy methods, Fourier transform, linearisation... Some notable authors in this program have been Arkeryd, Bobylev, Carlen, Carvalho, Cercignani, Dolera, Desvillettes, Gabetta, Loss, Regazzini, Tanaka and myself. In particular, in 2013 Dolera and Regazzini achieved a very close parallel between Boltzmann equilibration and central limit theorem, and established that the rate of exponential convergence of the spatially homogeneous Boltzmann equation with Maxwell kernel and cutoff (that is, when $B$ only depends on $k\cdot\sigma$ and $\int B\,d\sigma<\infty$) is given precisely by the spectral gap of the linearised equation. 
Eventually they did not use the Fisher information for that, relying rather on Fourier transform. But in the meantime, Fisher information had proven useful in a number of related problems in kinetic theory. Here are five examples:
\med

(a) {\em Equilibration for the Fokker--Planck equation:} Consider the basic linear Fokker--Planck equation
\[ \derpar{f}{t} = \Delta f + \nabla\cdot (fv),\qquad f=f(t,v),\qquad v\in\R^d.\]
Then by the Stam--Gross logarithmic Sobolev inequality
\[ -\frac{d}{dt} H_\gamma(f) = I_\gamma(f),\qquad
-\frac{d}{dt} I_\gamma(f) \geq 2 I_\gamma(f),\qquad H_\gamma(f) \leq \frac12 I_\gamma(f),\]
which readily implies the exponential convergence estimates
\[ H_\gamma(f(t)) \leq e^{-2t} H_\gamma(f_0),\qquad I_\gamma(f(t))\leq e^{-2t} I_\gamma(f_0).\]
Using classical inequalities in information theory, this provides a convergence of $f(t)$ to $\gamma$ like $O(e^{-t})$. Note: That rate of convergence is the same as in the linear theory (going through the spectral analysis of $-\Delta -v\cdot\nabla$ in $L^2(\gamma\,dv)$), but assumptions are much less stringent since the linear theory corresponds to $f/\gamma\in L^2(\gamma)$, or equivalently $\int e^{|v|^2/2} f(v)^2\,dv <\infty$, while here the only requirement is that $f_0=f_0(v)$ satisfying either $H_\gamma(f_0)<\infty$ or the stronger condition $I_\gamma(f_0)<\infty$; and one may even relax those conditions through a regularisation study.
\med

(b) {\em Entropy production estimates for Boltzmann's collision operator:} At the end of the nineties, Toscani and I established bounds of the form
\[ \forall f\qquad D_B(f) \geq K(B,\var,f)\, H_\gamma(f)^{1+\var},\]
where $\var\in (0,1)$ can be arbitrarily small 
($\var=0$ is impossible in general, as shown by counterexamples of Bobylev and Cercignani) and $K(B,\var,f)$ only depends on $B$ through a polynomial lower bound 
(like $B\geq a \min (|v-v_*|^\alpha, |v-v_*|^{-\alpha})$, $a,\alpha>0$) 
and on $f$ through regularity bounds (upper control of moments and Sobolev norms of $f$ of large enough order, lower bound on $f$ like $f\geq K_0\, e^{-A_0 |v|^{q_0}}$). The result does not use Fisher information but the proof does;
it involves a semigroup argument \`a la Stam, using the Fokker--Planck semigroup $(S_t)_{t\geq 0}$, and goes roughly like this:
\begin{align*}
D_B(f)  & \gtrsim \int_0^{+\infty} e^{-2\lambda t}\, D_L(S_tf)\,dt\\
&  \gtrsim \int_0^{+\infty} e^{-2\lambda t}\, I_\gamma(S_tf)\,dt\\
& \geq \int_0^{+\infty} I_\gamma\bigl( S_{(1+\lambda)t} f\bigr)\,dt = \frac{H_\gamma(f)}{1+\lambda},
\end{align*}
where $D_L$ is Landau's dissipation functional, which I will present in detail later in these notes.
\med

(c) {\em Inhomogeneous entropy estimates:} Around 2000, Desvillettes and I analysed the (hypocoercive) equilibration for the linear inhomogeneous kinetic Fokker--Planck equation with potential, white noise forcing and friction:
\begeq\label{kFP}
\derpar{f}{t} + v\cdot\nabla_x f - \nabla V(x)\cdot\nabla_v f
= \Delta_v f + \nabla_v\cdot (fv)\qquad f=f(t,x,v)\geq 0.
\endeq
Then the equilibrium is $f_\infty(x,v) = e^{-V(x)}\gamma(v)$ (assume that $V$ is normalised so that $\int e^{-V}=1$).  
We found out that the basic Gronwall-type inequality $dH_\gamma(f)/dt \leq -2 H_\gamma(f)$, typical of the spatially homogeneous regime, should be replaced by the more sophisticated system
\[ 
\begin{cases}
\dps -\frac{d}{dt} H_{f_\infty}(f) \geq H_{\rho\gamma} (f) \qquad \rho = \int f\,dv\\[5mm]
\dps \frac{d^2}{dt^2} H_{\rho\gamma}(f) \geq I_{e^{-V}}(\rho) 
- O \Bigl( \|f-f_\infty\|\, \|f-\rho\gamma\|\Bigr),
\end{cases}\]
leading to long-time estimates on how $f$ approaches $f_\infty$. (Here I am cheating a bit, but the spirit is correct.) This was later extended in my memoir on Hypocoercivity and other works.
\med

(d) {\em Entropic hypocoercive estimates:} About twenty years ago, I used a mixed $(x,v)$-Fisher information to establish the first explicit rates of exponential convergence to equilibrium for the kinetic Fokker--Planck equation in spaces of finite entropy : the functional was
\[ \tilde{I}(f) = \int f\, \Bigl\< A \nabla_{x,v} f, \nabla_{x,v} f\Bigr\> \,dx\,dv, \]
where $\nabla_{x,v} = (\nabla_x , \nabla_v)$ and 
\[ A = \left( \begin{matrix} a I & b I \\ bI & cI \end{matrix} \right),  \qquad
a>b>c, \quad b< \sqrt{ac},\]
$a,b,c$ being well-chosen. Coupled with global hypoelliptic regularisation estimates, this approach also led to equilibration results from measure initial data with enough moments. These results were later recovered and improved by Mouhot and collaborators.
\med

(e) {\em Refined notions of propagation of chaos}: In various works, Carrapataso, Fournier, Hauray, Mischler have shown how to use Fisher information bounds to refine the usual notion of chaos (which says that a sequence of symmetric probability measures $\mu^N$ in $N$ variables is close, in weak sense, to a tensor product) into the stronger notion of entropic chaos (which says, in addition, that there is convergence of the entropy). This is a kind of infinite-dimensional interpolation: convergence in weak sense and bound in Fisher sense imply convergence in entropy sense; the HWI inequality allows to do that when the reference measure is log concave.
\med

Still, a more basic development of McKean's paper was almost left aside as a curiosity: When does $I$ decrease along solutions of Boltzmann's equation?
The transport flow does not leave $I$ invariant (unlike $H$) so it seemed natural to focus on the effect of just $Q$, that is the spatially homogeneous case.
McKean seemed to think that the decreasing property of $I$ was a specific feature of dimension 1, but in 1992 Toscani generalised it to the 2-dimensional Boltzmann equation with Maxwell collisions; and in 1998, during my PhD I obtained the generalisation to all dimensions (a question asked by my advisor on the first day of my PhD!). I actually established the following analogues of Stam's inequalities, from which the decreasing property follows at once: If 
\[ Q^+(f,g) = \int_{\R^d} \int_{\S^{d-1}} f'g'_* b(k\cdot\sigma)\,d\sigma,\qquad
\int b(k\cdot\sigma)\,d\sigma =1,\quad b(k\cdot\sigma) = b(-k\cdot\sigma),\]
then for all distributions $f,g$ with finite energy,
\[ I\bigl( Q^+(f,g)\bigr) \leq \frac12 \bigl[ I(f) +I(g)\bigr], \qquad
H\bigl( Q^+(f,g)\bigr) \leq \frac12 \bigl[ H(f) +H(g)\bigr].\]
This seemed like a satisfactory answer and (contrary to Toscani) I did not expect $I$ to be decreasing for more general Boltzmann kernels. In the twenty-five years to follow, there was no extension of this result (save for an alternative proof by Matthes and Toscani of a particular calculation in my paper).

Before I turn to the rest of the story, let me mention that McKean's 1966 paper also contained some other questions or conjectures; some were disproved, but others are still standing, concerning either higher order functionals or higher order time derivations.

Another area of research opened in my PhD, on the other hand, was well identified as incomplete, and frustratingly so: the regularity of solutions for ``very soft'' potentials, that is, when the collision kernel has a strong singularity as $|v-v_*|\to 0$, say like $O(|v-v_*|^\gamma)$ for $\gamma<-2$ (here $\gamma$ is just a number, has nothing to do with the Gaussian distribution). In dimension 3 this corresponds to power law forces like $O(1/r^s)$ with $2<s\leq 7/3$. The most important motivation for such singularities is the Landau--Coulomb model, a limit case for Boltzmann's equation describing near encounters in plasma physics: then $\gamma=-3$. 
In such a singular regime, usual notions of weak solutions seemed untractable, and I had constructed a class of ``very weak' solutions using a time-integrated formulation and the functional $D_B$, or its Landau counterpart $D_L$; they were called $H$-solutions and the subject of my very first international seminar (Pavia, 1997).

Later it was found that when the kernel is {\em singular enough} in the angular variable, there is enough time-integrated regularity in the entropy production estimate to reinforce the notion of $H$-solutions in a more classical notion of weak solutions. It could also be shown, under the same assumption of angular singularity, that there is conditional regularity: If the solutions satisfy some integrability bounds ($L^p$ for $p$ large enough) then they are actually quite regular. But in the following years, nobody could establish whether there was actually regularity, or whether singularity could get in the way, despite enormous work by Golse, Gualdani, Imbert and Vasseur who identified more and more precise blow-up criteria.  

Before I go on with my story, let me describe more precisely the taxonomy of collision kernels and their associated equations. This will be the subject of the next section.

\bibnotes

The founding papers are those of Kac \cite{kac:foundations} and McKean \cite{mck:kac:65}.
The most accomplished parallel between central limit theorem and Boltzmann theory is achieved, for spatially homogeneous Maxwell collisions, by Dolera and Regazzini \cite{doleraregazzini:13}, which also includes a review of the field since McKean.

It took enormous effort before one could put Kac's program on a sound footing, estimating the propagation of chaos in such a way that the limit $N\to\infty$ and the limit $t\to\infty$ could commute; this was achieved by Mischler and Mouhot, building on a number of previous authors and ideas \cite{mischlermouhot:kac:13}.

Entropy production estimates for Boltzmann's collision operator are in my works \cite{TV:entropy:99}, jointly with Toscani, and \cite{vill:cer:03}. Before that, Carlen and Carvalho \cite{carlencarv:entropy:92,carlencarv:physic:94} had obtained the first entropy production estimates relating $D$ and $H$ (but with a much worse dependence), also using Fisher's information. Slightly before Carlen and Carvalho, there had been works by Desvillettes \cite{desv:entr:89,desv:eq:90} establishing stable entropy production bounds and convergence for the spatially homogeneous Boltzmann equation, albeit not quantitative.  

I worked with Desvillettes on inhomogeneous models, conditional to regularity estimates \cite{DV:FP:01,DV:boltz:05}. I wrote several lecture notes on this \cite{vill:bari,vill:icm,vill:ihp:01} and a memoir on hypocoercivity \cite{vill:hypoco}.

The combination of Fisher information bounds and chaos property was examined by Hauray and Mischler \cite{hauraymischler:chaos} and Carrapataso \cite{carrapataso:chaos}, and applied to mean-field limits of particle systems by Fournier, Hauray and Mischler \cite{FHM:chaos} (for the two-dimensional viscous vortex model) and by Carrapataso \cite{carrapataso:chaos} (for the Boltzmann equation).

Toscani \cite{tosc:linnik:92} proved decay of the Fisher information for the spatially homogeneous Boltzmann equation with Maxwellian molecules; this was in 1992 and revived the interest for the problem, showing that McKean's 1966 result was not limited to the one-dimensional case. A few years after Toscani, I established Stam's inequalities for the Boltzmann collision kernel in \cite{vill:fisher}, then Matthes and Toscani found an alternative road to a key lemma in \cite{matthestoscani:Bobvill}.

Several types of questions have been called ``McKean's conjectures'' after McKean's original paper \cite{mck:kac:65}. Some are about successive derivatives of the entropy along the Boltzmann flow: Is $H$ a convex function of time? or even a completely monotone function of time? The latter is known to be false even for the most simple models of Boltzmann equation, but the former might be true for the spatially homogeneous Boltzmann equation. Other conjectures are about the functionals obtained by successive derivation of the entropy and Fisher information along the heat flow: do they have alternate signs, are they optimised by Gaussians under constraint of fixed variance? Some comments are in my review \cite[Chapter 4, Section 4.3]{vill:handbook:02} and a much more up-to-date discussion, including recent results by Gao, Guo, Yuan, Wang, can be found in Ledoux \cite{ledoux:conjectures}.

I introduced $H$-solutions in \cite{vill:new:98} to handle the Cauchy problem for the spatially homogeneous Boltzmann equation with a strong singularity in the relative velocity; the refinement to time-integrated weak solutions was proven as a consequence of the fine study of entropy production which we did with Alexandre,  Desvillettes and Wennberg \cite{ADVW}. The control of the growth of moments for these solutions was done by Carlen, Carvalho and Lu \cite{CCL:soft:09}. The partial or conditional regularity of these solutions, in the particular Landau--Coulomb case was the subject of a series of works by Golse, Gualdani, Imbert, Vasseur, e.g. \cite{GGIV:landau, GIV:landau}.

In those works, entropy production is exploited to overcome the strong singularity of the collision kernel. Note that already at the end of the eighties, entropy production had played a key role in the inhomogeneous theory of weak solutions by DiPerna and Lions \cite{DPL:boltz:89}; but that was to overcome the different problem of the genuinely quadratic nature of the collision operator in the full space-inhomogeneous configuration.

\section{Taxonomy of Boltzmann collisions} \label{sectax}

Let me recollect notation. If $f:\R^d\to\R_+$ is a probability density, define $Q(f,f)$ as a function of $v$ by
\begeq\label{Qagain}
Q(f,f) = 
\int_{\R^d} \int_{\S^d} 
\bigl[ f(v')f(v'_*) - f(v) f(v_*)\bigr]\, B(v-v_*,\sigma)\,dv_*\,d\sigma,
\endeq
where $v,v_*,v',v'_*$ all belong in the {\em collision sphere} of diameter $[v,v_*]$, and $B(v-v_*,\sigma) = B(|v-v_*|,\cos\theta)$ (by abuse of notation I use the same symbol $B$ for both writings), and $\theta\in[0,\pi]$ is the angle between $k$ and $\sigma$, the unit vectors directing $v-v_*$ and $v'-v'_*$ respectively; see Fig.~\ref{figcollision}. The collision kernel is related to the {\bf cross-section} $\Sigma$ by the formula $B(z,\sigma) = |z| \Sigma(z,\sigma)$.

\begin{figure}
\caption{The collision sphere: a collision changes velocities $v,v_*$ into $v',v'_*$ and {\em vice versa}; the relative velocity axis is changed from $k$ to $\sigma$; and $\theta\in [0,\pi]$ is the deviation angle. It is often convenient to express $\sigma$ as a combination of $k$ and $\phi$, a $(d-2)$-dimensional unit vector in $k^\bot$.}
 \label{figcollision}
\def\svgwidth{0.5\textwidth}
\vspace*{-20mm}
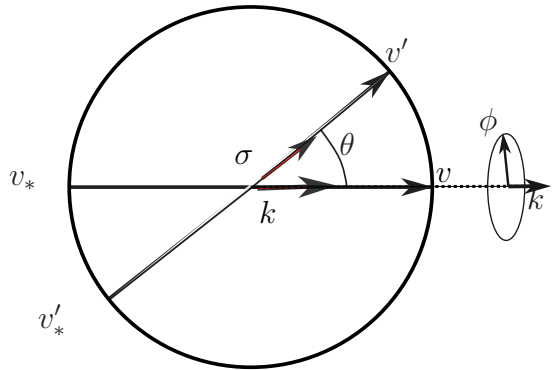
\vspace*{-25mm}
\end{figure}

For intermolecular forces deriving from a repulsive central potential $\psi$, Maxwell established the formula for the cross-section;  this amounts to solving a classical scattering problem with given relative velocity $z$ and impact parameter $p\geq 0$. Working in the referential of center of mass (fixed during the collision), the picture is as in Fig. \ref{figscat}.

\begin{figure} 
\caption{Maxwell's classical scattering problem: In the referential of the center of mass, with velocity $(v+v_*)/2$, one of the particles arrives from the right with asymptotic relative velocity $z/2=(v-v_*)/2$ and goes away to the left with asymptotic relative velocity $z'/2=(v'-v'_*)/2$, the impact parameter is $p$, the deviation angle is $\theta$.}
\label{figscat}
\def\svgwidth{0.5\textwidth}
\begingroup%
  \makeatletter%
  \providecommand\color[2][]{%
    \errmessage{(Inkscape) Color is used for the text in Inkscape, but the package 'color.sty' is not loaded}%
    \renewcommand\color[2][]{}%
  }%
  \providecommand\transparent[1]{%
    \errmessage{(Inkscape) Transparency is used (non-zero) for the text in Inkscape, but the package 'transparent.sty' is not loaded}%
    \renewcommand\transparent[1]{}%
  }%
  \providecommand\rotatebox[2]{#2}%
  \newcommand*\fsize{\dimexpr\f@size pt\relax}%
  \newcommand*\lineheight[1]{\fontsize{\fsize}{#1\fsize}\selectfont}%
  \ifx\svgwidth\undefined%
    \setlength{\unitlength}{595.27559055bp}%
    \ifx\svgscale\undefined%
      \relax%
    \else%
      \setlength{\unitlength}{\unitlength * \real{\svgscale}}%
    \fi%
  \else%
    \setlength{\unitlength}{\svgwidth}%
  \fi%
  \global\let\svgwidth\undefined%
  \global\let\svgscale\undefined%
  \makeatother%
  \begin{picture}(1,1.41428571)%
    \lineheight{1}%
    \setlength\tabcolsep{0pt}%
    \put(0,0){\includegraphics[width=\unitlength,page=1]{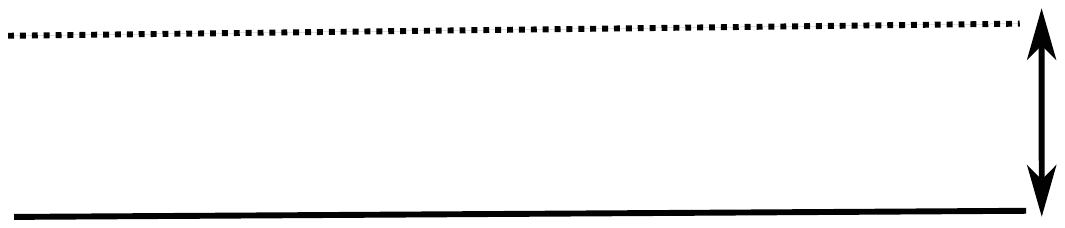}}%
    \put(0.9028539,0.89497495){\color[rgb]{0.01176471,0.01176471,0.01176471}\makebox(0,0)[lt]{\lineheight{1.89988542}\smash{\begin{tabular}[t]{l}$p$\end{tabular}}}}%
    \put(0,0){\includegraphics[width=\unitlength,page=2]{scattering.pdf}}%
    \put(0.70979258,0.91247184){\color[rgb]{0.00784314,0.00784314,0.00784314}\makebox(0,0)[lt]{\lineheight{1.89988542}\smash{\begin{tabular}[t]{l}$z/2$\end{tabular}}}}%
    \put(0,0){\includegraphics[width=\unitlength,page=3]{scattering.pdf}}%
    \put(0.10603479,1.3316802){\color[rgb]{0.00784314,0.00784314,0.00784314}\makebox(0,0)[lt]{\lineheight{1.89988542}\smash{\begin{tabular}[t]{l}$z'/2$\end{tabular}}}}%
    \put(0,0){\includegraphics[width=\unitlength,page=4]{scattering.pdf}}%
    \put(0.23266411,1.01661362){\color[rgb]{0.00784314,0.00784314,0.00784314}\makebox(0,0)[lt]{\lineheight{1.89988542}\smash{\begin{tabular}[t]{l}$\theta$\end{tabular}}}}%
  \end{picture}%
\endgroup%

\vspace*{-35mm}
\end{figure}

The deviation angle $\theta$ is computed from $p$ and $z$ as
\begeq\label{thetapz}
\theta(p,z) = \pi -  2 p \int_{r_0}^{\infty} \frac{dr/r^2}{\sqrt{1-\frac{p^2}{r^2}-\frac{4\psi(r)}{|z|^2}}},
\endeq
where $r_0$ is the positive root of
\begeq\label{defr0}
1-\frac{p^2}{r_0^2} -  \frac{4\psi(r_0)}{|z|^2} = 0.
\endeq
The interaction being effectively a 2-dimensional problem,  this formula is independent of $d\geq 2$. What does depend on dimension, though, is the volume of particles in phase space being prone to collide with deviation angle $\theta$ per unit of time. Assuming that there are no correlations between the distributions of $v$ and $v_*$, for given $z$ and $p$ this volume will be $f(v) f(v_*) |C_{z,p,dp}|\,dt$, where $C_{z,p,dp}$ is the cylinder of axis $z$, radius $p$ and with $dp$,  that is the set of all $tz +[p,p+dp]\phi$, with $0\leq t\leq 1$, and $\phi$ lying in the $(d-2)$-dimensional sphere unit orthogonal to $z$; see Fig.~\ref{figcylinder}.
\bigskip

\begin{figure} 
\caption{The cylinder of collisions: the lateral crust of the cylinder is where collision partners will be found with impact parameter $p$ (up to $dp$) and relative velocity $z$, during the time interval $[t,t+dt]$.}
\label{figcylinder}
\def\svgwidth{0.5\textwidth}
\vspace*{-20mm}
\begingroup%
  \makeatletter%
  \providecommand\color[2][]{%
    \errmessage{(Inkscape) Color is used for the text in Inkscape, but the package 'color.sty' is not loaded}%
    \renewcommand\color[2][]{}%
  }%
  \providecommand\transparent[1]{%
    \errmessage{(Inkscape) Transparency is used (non-zero) for the text in Inkscape, but the package 'transparent.sty' is not loaded}%
    \renewcommand\transparent[1]{}%
  }%
  \providecommand\rotatebox[2]{#2}%
  \newcommand*\fsize{\dimexpr\f@size pt\relax}%
  \newcommand*\lineheight[1]{\fontsize{\fsize}{#1\fsize}\selectfont}%
  \ifx\svgwidth\undefined%
    \setlength{\unitlength}{595.27559055bp}%
    \ifx\svgscale\undefined%
      \relax%
    \else%
      \setlength{\unitlength}{\unitlength * \real{\svgscale}}%
    \fi%
  \else%
    \setlength{\unitlength}{\svgwidth}%
  \fi%
  \global\let\svgwidth\undefined%
  \global\let\svgscale\undefined%
  \makeatother%
  \begin{picture}(1,1.41428571)%
    \lineheight{1}%
    \setlength\tabcolsep{0pt}%
    \put(0,0){\includegraphics[width=\unitlength,page=1]{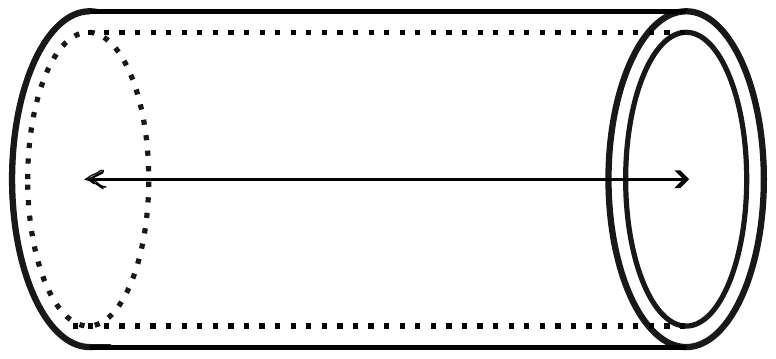}}%
    \put(0.40787935,0.71964395){\color[rgb]{0.01176471,0.01176471,0.01176471}\makebox(0,0)[lt]{\lineheight{1.89988542}\smash{\begin{tabular}[t]{l}$z$\end{tabular}}}}%
    \put(0.70505589,0.73218573){\color[rgb]{0.01176471,0.01176471,0.01176471}\makebox(0,0)[lt]{\lineheight{1.89988542}\smash{\begin{tabular}[t]{l}$p$\end{tabular}}}}%
    \put(0.74000716,0.79011456){\color[rgb]{0.01176471,0.01176471,0.01176471}\makebox(0,0)[lt]{\lineheight{1.89988542}\smash{\begin{tabular}[t]{l}$dp$\end{tabular}}}}%
    \put(0,0){\includegraphics[width=\unitlength,page=2]{cylindre.pdf}}%
  \end{picture}%
\endgroup%

\vspace*{-30mm}
\end{figure}

So $B(|z|,k\cdot\sigma)\,d\sigma = |C_{z,p,dp}|= p^{d-2} |\S^{d-2}|\,|z|\,dp\,d\phi$.
Since $d\sigma = |\S^{d-2}|\,\sin^{d-2}\theta\,d\theta\,d\phi$ (usual polar coordinates around $z$),
we end up with
\begeq\label{Bzcos}
B(|z|, \cos\theta) = \left(\frac{p}{\sin\theta}\right)^{d-2} \left|\frac{dp}{d\theta}\right|\, |z|.
\endeq

\noindent {\bf Main example: Inverse power laws.} Assume that the force is an inverse power law of exponent $s$, so $-\psi'(r) = (s-1)\psi_0/r^s$ for some $s>1$, or
\[ \psi(r) = \frac{\psi_0}{r^{s-1}}.\]
In practice $\psi_0$ will involve physical constants quantifying the strength of interaction.
Then 
\[ \theta\bigl( |z|^{-\frac{2}{s-1}}p, z\bigr) = \theta(p,1). \]
So for given $\theta$ the product $p|z|^{2/(s-1)}$ is independent of $|z|$. It follows that
$B(|z|,\cos\theta) = |z|^\gamma B(1,\cos\theta)$ with $\gamma = (s-(2d-1))/(s-1)$. Let us denote
\begeq\label{gammab}
B(v-v_*,\sigma) = |v-v_*|^\gamma b(\cos\theta),\qquad \gamma = \frac{s-(2d-1)}{s-1}.
\endeq
As for the function $b(\cos\theta)$, it is nonexplicit and in general involves transcendantal integrals. To study it more explicitly let us assume that the choice of units is such that
\[ \psi_0 = \frac14.\]
Writing $\theta(p)=\theta(p,1)$, we have
\begeq\label{thetap}
b(\cos\theta) = \left(\frac{p}{\sin\theta(p)}\right)^{d-2} \left| \frac{dp}{d\theta}\right|,
\quad  
\theta(p) = \pi -  2 p \int_{r_0}^{\infty} \frac{dr/r^2}{\sqrt{1-\frac{p^2}{r^2}-\frac{1}{r^{s-1}}}}.\endeq
There are two cases which lead to explicit computation: $s=2$, $s=3$.
For the particular exponent $s=2$ the computation can be carried out completely and yields, in dimension 3, {\bf Rutherford's cross-section formula}:
\begeq\label{rutherford} B_2(|z|, \cos\theta) = \frac1{16} \frac{|z|}{\left( |z| \sin \frac{\theta}{2}\right)^4}.
\endeq
More generally, in dimension~$d$,
\begeq\label{rutherford_d} B_2^{(d)}(|z|, \cos\theta) = \frac{|z|}{\left( 2 |z| \sin \frac{\theta}{2}\right)^{2(d-1)}}.
\endeq
For more general $s$ things remain implicit, but at least the behaviour for $\theta\to 0$ can be estimated, though. For this we need to recast the formula for $\theta$ in \eqref{thetap} as an integral with fixed boundaries. This computation will be important in the sequel, so let us take some care to perform this. 

A first possibility is to choose as new variable
\[ u = 1 - \frac{p^2}{r^2} - \frac{1}{r^{s-1}}, \]
and noting $p/r\simeq \sqrt{1-u}$ as $p\to\infty$ (i.e. $\theta\to 0$), we rewrite
\begin{align} \theta & = \pi - \int_0^1 \frac{\sqrt{1 + \frac{1}{p^2 r^{s-3}}}} {\left( 1+\frac{2(s-1)\psi_0}{p^2 r^{s-3}}\right)}
\, \frac{du}{\sqrt{u(1-u)}} \nonumber \\
& \simeq \frac{2\psi_0 (s-2)}{p^{s-1}} \int_0^1 (1-u)^{\frac{s-4}{2}}\, \frac{du}{\sqrt{u}} \qquad \text{as $p\to\infty$}. \label{integrals-4}
\end{align}
Thus
\[ p \simeq \left(\frac{C(s)\,\psi_0}{\theta}\right)^{\frac1{s-1}}  \text{as $\theta\to 0$},\]
and eventually
\begeq\label{bcostheta}
b(\cos\theta)\,\sin^{d-2}\theta \sim_{\theta\to 0} \left( \frac{\bigl(C(s)\,\psi_0)^{\frac{d-1}{s-1}}}{s-1}\right)\, \theta^{-(1+\nu)},
\qquad \nu = \frac{d-1}{s-1}.
\endeq
Here $\nu$ quantifies the degree of nonintegrability of the collision kernel in the $\theta$ variable. That formula converges for $s>2$; as $s\to 2$ the factor $s-2$ vanishes, but the integral \eqref{integrals-4} goes to $\infty$, so it is indeterminate.

Another change of variable is $y=r_0/r$. Then, using $1-p^2/r_0^2 = 1/r_0^{s-1}$, we obtain, as $p\to\infty$,
\begin{align*}
\theta(p) & = \pi - 2\int_0^1 \frac{dy}{\sqrt{\frac{r_0^2}{p^2} - y^2 - \frac{y^{s-1}}{r_0^{s-3}p^2}}}\\
& = 2 \int_0^1 \left( \frac1{\sqrt{1-y^2}}
- \frac1{\sqrt{\frac{r_0^2}{p^2} - y^2 - \frac{y^{s-1}}{r_0^{s-3} p^2}}}\right) \\
& \simeq \frac{r_0^2}{p^2}
\int_0^1 \frac1{(1-y^2)^{3/2}} 
\left( \frac1{r_0^{s-1}} - \frac{y^{s-1}}{r_0^{s-1}}\right)\,dy\\
& \simeq \frac1{r_0^{s-3}p^2} \int_0^1 \frac1{(1-y^2)^{3/2}} (1-y^{s-1})\,dy\\
& \simeq \left(\int_0^1 \frac1{(1-y^2)^{3/2}} (1-y^{s-1})\,dy\right) \frac1{p^{s-1}} \\
& = \left( \frac{\sqrt{\pi}\,\Gamma\left(\frac{s}{2}\right)}{\Gamma \left(\frac{s-1}{2}\right)}\right)\frac1{p^{s-1}}.
\end{align*}
That formula applies to the whole range $s>1$.

\begin{Rks}

\bul $\nu\to 0$ as $s\to\infty$, whatever the dimension, and likewise $\gamma\to 1$. This limit corresponds to the hard sphere kernel, which is better expressed in terms of $\omega$, the direction of the centers of mass of colliding particles (with obvious notation, $\omega =(x-x_*)/|x-x_*|$):
\[ \tilde{B}_{HS}(z,\omega) = |z\cdot \omega|\, 1_{k\cdot\omega <0}, \]
The directions $\sigma$ and $\omega$ are related through $|k\cdot\omega| = \sin (\theta/2)$,
and the Jacobian is $|d\sigma/d\omega| = (2\sin(\theta/2))^{d-2}$. So eventually, in dimension $d$,
\begeq\label{BHS} B^{(d)}_{HS} (z,\sigma) = \frac{|z|}{2^{d-2} \sin^{d-3}(\theta/2)}. \endeq
When $d=2$ the angular dependence of that kernel is proportional to $\sin(\theta/2)$ and when $d\geq 4$ to the singular function $(\sin(\theta/2))^{-(d-3)}$ (the singularity is integrable and may be seen as an artifact of the change of variables $\omega\to\sigma$). Thus in the limit $s\to\infty$, $B$ becomes integrable over $\S^{d-1}$ even though for any finite $s$ it is not. In the particular case $d=3$, Hard Spheres correspond to a kernel which is just proportional to the relative velocity ({\em constant cross-section}) $B(z,\sigma) = |z|/2$.

\bul Power law forces satisfy $\gamma+ 2\nu =1$. 

\bul Fortunately the Boltzmann equation still makes sense if $B$ is nonintegrable, provided that the cross-section for momentum transfer $M=M(z)$ is well defined:
\begeq\label{defM} M(z) = \int_{\S^{d-1}} B(z,\sigma)\, (1-k\cdot\sigma)\,d\sigma.  \endeq
In fact $M(z)|z|$ is the typical change of momentum in a collision of relative velocity $z$; by symmetry $M$ only depends on $|z|$ and I shall often write it as $M(|z|)$. Since $1-k\cdot\sigma \sim \theta^2/2$ as $\theta\to 0$, the finiteness of $M$ requires $\nu<2$. We shall see in the next section that $\nu=2$ can be handled as a limit case.

\bul Coulomb interaction in dimension $d$ is for $s=d-1$, so $\gamma = -d/(d-2)$, $\nu = (d-1)/(d-2)$. In particular the singularity in velocity for Coulomb interactions is $-3$ for $d=3$, $-2$ for $d=4$, $-5/3$ for $d=5$, and converges to $-1$ as $d\to\infty$. In other words,

-  for $d=2$, Boltzmann's equation definitely does not make sense for Coulomb potential;

- for $d=3$ it is borderline integrable in $\theta$ ($\nu=2$) and also borderline integrable in $z$ ($\gamma = -3$);

- for $d\geq 4$ it is well defined with $\gamma$ varying from $-2$ up to $-1$ and $\nu$ from $3/2$ down to $1$.
\end{Rks}
\med

More generally, as a reference family for collision kernels, we may consider the factorised kernels with exponents $\gamma,\nu$ with the following behaviour as $\theta\to 0$:
\begeq\label{collkfact}
B(v-v_*,\sigma) = |v-v_*|^\gamma \, b(\cos\theta),\qquad
b(\cos\theta)\,\sin^{d-2}\theta \sim_0 \beta_0\, \theta^{-(1+\nu)}\quad \text{for some $\beta_0>0$}.
\endeq
This defines a $(\gamma,\nu)$ parameter space, where $0\leq \nu\leq 2$ (let us agree that $\nu=0^-$ means cutoff and $\nu=2$ is the diffusive limit, explored in the next section) and $\gamma\geq -\min(4,d)$. For $\gamma<-4$ or $\gamma<-d$ there are great mathematical difficulties, and not much physical motivation. On the other hand, on physical grounds there is not much motivation for powers $\gamma>2$, neither for potential more singular than Coulomb, since the latter's slow decay is already quite a challenge. To summarise, the whole range $0\leq\nu\leq 2$ is interesting, but on $\gamma$ there is a tough mathematical restriction to $\gamma\geq-\min (4,d)$ and a natural physical restriction to $-d/(d-2) \leq\gamma\leq 2$. For inverse power laws this means: $s\geq 5/3$ if $d=1$, $s\geq 2$ if $d=3$, $s\geq d-1$ if $d\geq 4$.

Decades of research have explored the various regions in this diagram, first putting much emphasis on hard potentials with cutoff. See Fig.~\ref{figtaxonomy} for a synthetic presentation of that taxonomy.

\begin{Rk} In the presentation and in Fig.~\ref{figtaxonomy} I am simplifying a bit by pretending to forget that sometimes it is the behaviour of $B$ as $|z|\to\infty$ which matters, and sometimes the behaviour as $|z|\to 0$; to be very rigorous we could use a pair of exponents to allow for different behaviours in those regimes, but this is already subtle enough.
\end{Rk}
\bigskip

\begin{figure} 

\caption{This diagram presents the landscape of the various regimes for the Boltzmann equation, according to the behaviour of the cross section with respect to relative velocity ($\gamma$ = power law) and deviation angle ($\nu$ = inverse power law singularity). Read as follows.
HS= Hard spheres: $B=\sin^{(d-3)}(\theta/2)|z|$ ($\gamma=1$, $\nu=0^-$);
SHS= Super hard spheres: $B=|z|^2$ ($\gamma=2$, $\nu=0^-$);
MM=  Maxwellian molecules ($\gamma=0$);
HP= Hard potentials ($\gamma>0$);
SP= Soft potentials ($\gamma<0$);
VSP= Very soft potentials ($\gamma<-2$); 
(co)= cutoff ($\nu=0^-$);
(nco)= noncutoff ($0<\nu<2$);
(d)= diffusive ($\nu=2$); LE = Landau equation;
(PL)= (Inverse) Power Laws= the line $(\gamma +2\nu = 1)$. Coulomb interactions are distinguished members of this family; C3 is the Coulomb interaction for $d=3$ ($\gamma=-3, \nu =2$), C4 for $d=4$ and so on, accumulating in large dimension at C$\infty$ ($\gamma=-1,\nu=1$);
EG= Entropic gap, {\em i.e.} when $D_B(f)\geq K[H(f)-H(M)]$ for all $f$: above the line $(\gamma+\nu= 2)$;
SG= Spectral gap, {\em i.e.} when $-\<Lh,h\>_{L^2(M)} \geq K\, \|h-1\|^2_{L^2(M)}$ for all $h$: above the line $(\gamma+\nu = 0)$;
CR= Conditional regularity, {\em i.e.} when an a priori estimate in $L^p$ large enough actually implies that $f$ is a priori regular: above the line $(\gamma+\nu= -2)$, restricted to $\gamma\geq -d$; in the spatially homogeneous case away from equilibrium, the new estimates based on Fisher information monotonicity allow to prove that above this line (and with some structure conditions covering all natural cases of interest) one can work with smooth solutions, while below this line only $H$-solutions have been constructed, whose regularity is unknown.
}
\label{figtaxonomy}
\def\svgwidth{0.8\textwidth}
\vspace*{-30mm}
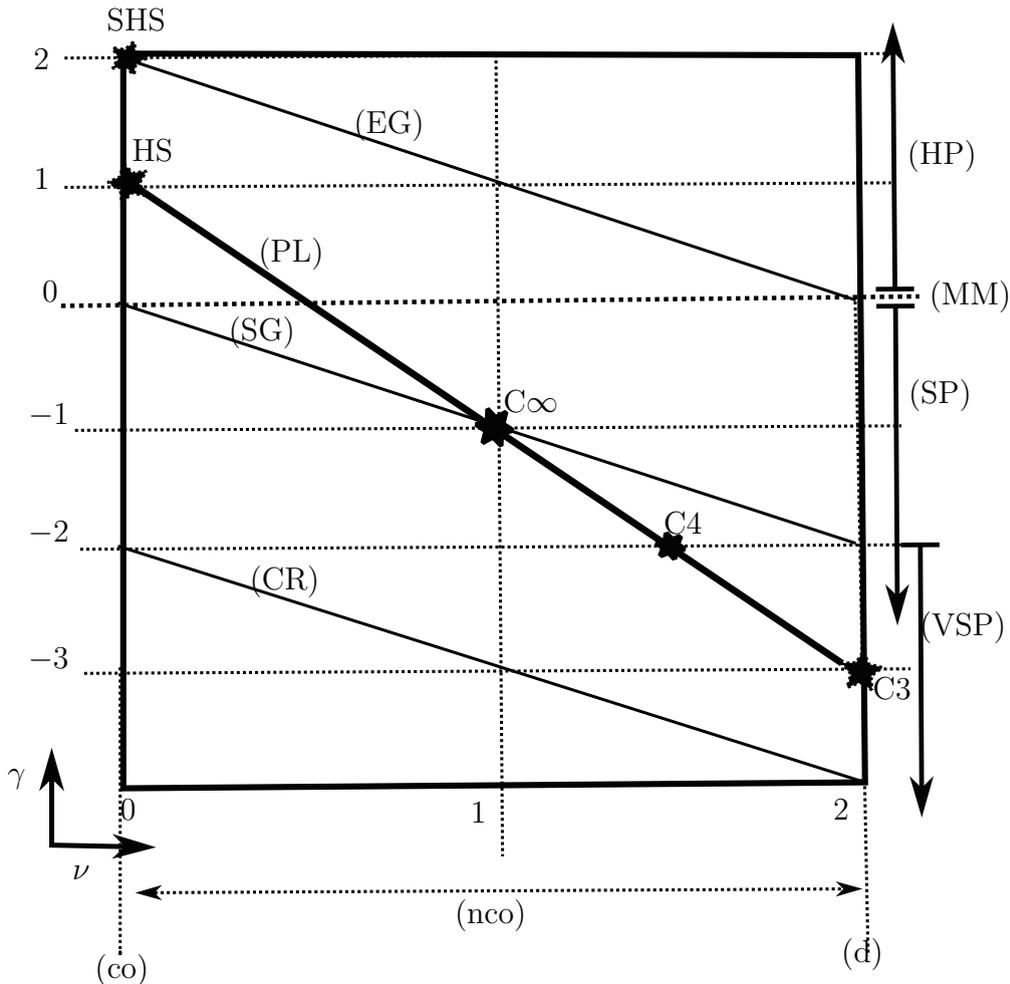
\vspace*{-35mm}
\end{figure}

\bibnotes

The scattering formulas \eqref{thetapz} are already in Maxwell's founding paper \cite{maxw:67} for $d=3$. Rutherford derived his cross-section formula \eqref{rutherford} as part of his epoch-making work on scattering of $\alpha$ particles \cite{rutherford:1911}, leading to the discovery of nuclei. It is not difficult to generalise his calculation to any dimension.

The first mathematical study of the Boltzmann equation as an evolution problem for the probability density is due to Carleman \cite{carleman} and focused on hard spheres under an assumption of spatial homogeneity. Grad initiated an ambitious program starting from the linearised theory and insisting on the integrability of Boltzmann's kernel, which is never satisfied for long-range interactions; so that assumption of integrability is often called ``Grad's angular cutoff assumption''. The classical study of hard potentials with cutoff ($\gamma>0$, integrable $b$) started with Arkeryd \cite{ark:I+II:72} (after a preliminary work by Povzner \cite{povz:65}) and was followed by Elmroth, Gustafsson, Wennberg, Desvillettes, Mischler, Mouhot and others. Soft potentials without cutoff ($s<5$ for $d=3$) were studied by Arkeryd \cite{ark:infini:81}, Goudon \cite{goud:graz:97} and I \cite{vill:new:98}, followed by a number of others. A review can be found in my survey \cite{vill:handbook:02}. That was more than twenty years ago and back then, hardly anything was known about very soft potentials, which is when $\gamma<-2$ ($2\leq s<7/3$ for $d=3$); but this part of the diagram will be the main focus in Sections \ref{secreg} and \ref{seceq}, see also the more complete bibliographical notes in these upcoming sections.

Maxwellian molecules ($s=5$ for $d=3$), or more generally Maxwellian collision kernels ($\gamma=0$) were noticed to have special properties, already by Maxwell himself. There is a long history of particular tools in this special case, involving in particular Kac \cite{kac:foundations}, McKean \cite{mck:kac:65} Tanaka \cite{tanaka:boltz:78}, Toscani \cite{tosc:linnik:92} and an enormous series of works by Bobylev, most of which is summarised in his review \cite{bob:theory:88}.

Super hard spheres have been used as a phenomenological model in kinetic theory, for modelling issues. From a theoretical point of view, they appear in my work with Toscani on the entropic gap \cite{TV:entropy:99,vill:cer:03}. The entropic gap for the Landau equation with Maxwell kernel ($\nu=2$, $\gamma=0$) was established in \cite{vill:FIL:00}. Counterexamples for hard spheres and Maxwellian molecules are due to Bobylev and Cercignani \cite{bobcer:entropy:99}. The fact that the condition $\gamma+\nu\geq 2$ defines the region of entropic gap is in unpublished notes of Desvillettes and Mouhot. The mere problem of entropic gap has been loosely posed by Cercignani in 1982 \cite{cer:Hthm:82} and has long been dubbed ``Cercignani conjecture''. 

It has long been known that the Boltzmann equation with cutoff has spectral gap for $\gamma\geq 0$ and not for $\gamma<0$ \cite{cafl:soft:80}. That the line $\gamma+\nu\geq 0$ is the correct separation for spectral gap without the cutoff assumption is due to Mouhot and Strain \cite{mouhot:sg,mouhotstrain:sg}.

\section{Asymptotics of grazing collisions} \label{secgrazing}

When plasma physics took off, it was important, both theoretically and practically, to model the correction to the Vlasov--Poisson system which is due to ``collisions'', or rather near encounters of electrons. Landau provided the desired recipe as early as 1936 by working out a suitable converging approximation from the diverging Boltzmann operator. Later, plasma specialists like Balescu and Bogoljubov proposed other, more subtle, collision operators, bypassing Boltzmann's formula and going back to the analysis of Coulomb particle systems; but Landau's approximation is still fine for the vast majority of purposes.

Naturally, we first focus on the Coulomb interaction in dimension $d=3$: the interaction potential is $\psi(r)=\psi_0/r$; fix $\psi_0>0$ such that the Rutherford cross section reads
\[ \frac{B_2(|z|,\cos\theta)}{|z|} = \frac1{\left( |z|\,\sin\frac{\theta}{2}\right)^{4}}.\]
For this kernel $B$ the cross-section for momentum transfer, $M$, is infinite, making the Boltzmann equation meaningless. This is due to the large tails of the Coulomb potential. A classical remedy is to tame $\psi$ at large distances, replacing $\psi$ by $\psi_\lambda(r) = \psi(r) e^{-r/\lambda}/r$, where $\lambda$ is Debye's ``screening length'. The resulting collision kernel does make sense, the price to pay is that now it is not explicit. A decent approximation is
\[ B_\lambda(|z|,\cos\theta) \simeq \frac{|z|}{\left( \left(|z|\,\sin\frac{\theta}{2}\right)^2 + \frac1{\lambda}\right)^2}. \]
Then
\[ M_\lambda (|z|) \simeq \frac{8\pi\,\log\lambda}{|z|^3} \left(1+O\left(\frac1{\lambda}\right)\right).\]
Of course this diverges as $\lambda\to\infty$, but ``not much' for practical purpose, since only logarithmically. The approximation prefactor is called the Coulomb logarithm. 

In the above, the kernel $B_\lambda/(8\pi\log \lambda)$ is an approximation, so to speak, of $|z|^{-3}\delta''_{\theta=0}$. This can be generalised into a procedure of concentrating the whole influence of $B$ on infinitesimally small deviation angles $\theta$, whose effect dwarves all contributions of $\theta\geq\theta_0$ for any fixed $\theta_0>0$. Starting from the nineties, this concentrating procedure was formalised in increasing generality, ending up with the following framework by Alexandre and myself. Handling the borderline singularity $\gamma=-d$ requires a new kernel, which will also turn out to be useful for other purpose.

\begin{Def}[Compensated adjoint Boltzmann kernel]
If $B=B(|z|,\cos\theta)$ is a given collision kernel, define $S=S(|z|)$ by
\begeq\label{eqdefS} S (z) = |\S^{d-2}| \int_0^\pi
\left[\frac1{\cos^{d}(\theta/2)}\,  B \left(\frac{|z|}{\cos(\theta/2)},\cos\theta\right) - B(|z|,\cos\theta)
\right]\,\sin^{d-2}\theta\,d\theta.
\endeq
\end{Def}

The meaning of this kernel is the following: $S\ast f = Q(f,1) = \iint (f'_*-f_*)\,B\,dv_*\,d\sigma$. For the next definition, the other meaningful kernel is the cross-section for momentum transfer $M$ defined in \eqref{defM}, or equivalently the momentum transfer kernel, $|z|M(|z|)$.

\begin{Def}[Asymptotics of grazing collisions, AGC]
A sequence $(B_n)_{n\in\N}$ of collision kernels concentrates on grazing collisions if

(a) The associated compensated adjoint and momentum transfer kernels, $S_n(|z|)$ and $|z|M_n(|z|)$, as measures on $\R^d$, are bounded in total variation on compact sets, uniformly in $n$;

(b) $\forall \theta_0>0$, $S_n^{\theta_0}(|z|)$ and $|z| M_n^{\theta_0}(|z|)$, defined in a similar way with $B_n$ replaced by $B_n\,1_{\theta\geq \theta_0}$, converge to as $n\to\infty$, locally uniformly in $\R^d\setminus \{0\}$;

(c) There is a measurable $M_\infty:\R_+\to\R_+$ such that $z\,M_n(|z|) \longrightarrow z M_\infty(|z|)$, locally weakly in the sense of measures.
\end{Def}

The short meaning is that $M_n\to M_\infty$ and only $\theta\simeq 0$ count. 
Recalling the definition of $M_n$, actually
\[ M_n = \int_{\S^{d-1}} (1-\cos\theta)\,B_n\,d\sigma \simeq 
\int_{\S^{d-1}} \frac{\theta^2}{2}\,B_n\,d\sigma \xrightarrow[AGC]{} M_\infty.\]
In the sequel of these notes I shall not worry about regularity issues and often omit the index $n$, using the symbols AGC to stand for this limit procedure.

\begin{Prop}[Landau approximation] \label{propLandau}
Under the asymptotics of grazing collisions,

(i) $Q  (f,f) \xrightarrow[AGC]{} Q_L(f,f)$ where
\begeq \label{QL} 
Q_L(f,f) = \nabla_v\cdot \left( \int_{\R^d} a(v-v_*)\, \bigl[ f(v_*)\nabla f(v) - f(v)\nabla f(v_*)\bigr]\,dv_* \right)
\endeq
is Landau's collision operator, and the matrix-valued function $a=a(z)$ is defined by
\begeq\label{aPi}
a(z) = \Psi(|z|)\,\Pi_{z^\bot},\qquad \Pi_{z^\bot} = I_d - \frac{z\otimes z}{|z|^2},
\qquad \Psi(|z|) = \frac{|z|^{2} M_\infty(|z|)}{4(d-1)};
\endeq

(ii) $D_B(f) \xrightarrow[AGC]{}  D_L(f,f)$, Landau's dissipation functional ($H$-dissipation, or entropy production, functional), defined by
\begeq\label{DL}
D_L(f) = \frac12 \iint_{\R^d\times\R^d} ff_* \Psi(|v-v_*|)\,
\Bigl| \Pi_{k^\bot} \bigl( \nabla \log f - (\nabla \log f)_*\bigr) \Bigr|^2\,dv\,dv_*,
\endeq
with the shorthands $f=f(v), f_*=f(v_*), k = \frac{v-v_*}{|v-v_*|}$.
\end{Prop}

The proof rests on the useful

\begin{Lem}[Averages on subspheres] \label{lemav}
Let $A$ be a $d\times d$ symmetric matrix, $k\in \S^{d-1}$ and $\S^{d-2}_{k^\bot}$ be the $(d-2)$ unit sphere in the plane $k^\bot$. Then 
\begeq\label{av1}
\frac1{|\S^{d-2}|} \int_{\S^{d-2}_{k^\bot}} \< A\phi,\phi\>\,d\phi = \frac{\tr (A \Pi_{k^\bot})}{d-1}.
\endeq
As a particular case, for any vector $\xi\in\R^d$,
\begeq\label{av2}
\frac1{|\S^{d-2}|} \int_{\S^{d-2}_{k^\bot}} (\xi\cdot\phi)^2\,d\phi = \frac{|\Pi_{k^\bot}\xi|^2}{d-1}.
\endeq
and for any two vectors $\xi,\eta\in R^d$,
\begeq\label{av3}
\frac1{|\S^{d-2}|} \int_{\S^{d-2}_{k^\bot}} (\xi\cdot\phi)\,(\eta\cdot\phi) \,d\phi = 
\frac{\< \Pi_{k^\bot}\xi, \Pi_{k^\bot}\eta\>}{d-1}.
\endeq
\end{Lem}

\begin{proof}[Proof of Lemma \ref{lemav}] $\<A\phi,\phi\> = \<\Pi_{k^\bot} A \Pi_{k^\bot}\phi,\phi\>$ and $\Pi_{k^\bot} A \Pi_{k^\bot}$ is a symmetric endomorphism of $k^\bot$, of the form $\sum \lambda_i e_i\otimes e_i$, where $(e_i)_{1\leq i\leq d-1}$ is an orthonormal basis of $k^\bot$. Then $\tr(A\Pi_{k^\bot}) = \sum \lambda_i$. So it suffices to check that if $e$ is a unit vector in $k^\bot$, then
\begeq\label{sufficeav}
\frac1{|\S^{d-2}|} \int_{\S^{d-2}} (e\cdot\phi)^2\,d\phi = \frac1{d-1}.
\endeq
This will also imply \eqref{av2} since $\xi\cdot\phi = (\Pi_{k^\bot}\xi)\cdot\phi
= |\Pi_{k^\bot}\xi|\,e\cdot\phi$, $e=(\Pi_{k^\bot}\xi)/|\Pi_{k^\bot\xi}|$; and \eqref{av3} will follow also by polarisation.

To prove \eqref{sufficeav}, write $\phi = (\cos\chi)e + (\sin\chi)\psi$, where
$\psi \in \S^{d-3}_{e^\bot}$ ($e^\bot$ is a subspace of $k^\bot$), so the left-hand side of \eqref{sufficeav} is
\[ \frac{\int_0^\pi \cos^2\chi\,\sin^{d-3}\chi\,d\chi}
{\int_0^\pi \sin^{d-3}\chi\,d\chi}
= \left(\frac1{d-2}\right) \frac{\int \sin^{d-1}\chi\,d\chi}{\int\sin^{d-3}\chi\,d\chi} = 
\left(\frac1{d-2}\right)\left(\frac{d-2}{d-1}\right) = \frac1{d-1},\]
thanks to the classical Wallis integral formula.
\end{proof}

\begin{proof}[Sketch of proof of Proposition \ref{propLandau}]
Depending on the needs, (i) can be proven either directly on the operator $Q$, or from its weak formulation $\int Qh$. I shall go for the latter. Testing $Q$ against a smooth $h(v)$ and using $(v,v_*,\sigma)\leftrightarrow (v',v'_*,k)$,
\begin{multline} \label{endupQB}
\int Q(f,f)\,h\,dv 
= \iiint B (f'f'_* - ff_*)\,h
= \iiint B ff_* (h'-h)\\
= \iint ff_* \left( \int B (h'-h)\,d\sigma\right)\,dv\,dv_*.
\end{multline}
On the other hand,
\begin{multline*}
\int Q_L(f,f)h = 
- \iint a(v-v_*) (\nabla-\nabla_*) ff_* \cdot \nabla h \\
= \iint ff_* (\nabla-\nabla_*)\cdot \bigl[ a(v-v_*)\nabla h\bigr]\,dv\,dv_*.
\end{multline*}
Next, using 
\[ (\nabla-\nabla_*) \cdot \Pi_{k^\bot} = - (\nabla-\nabla_*)(k\otimes k) = -2(d-1)\,k,\]
and $\nabla_*\nabla h=0$, $\Pi_{k^\bot}k=0$, $(\nabla-\nabla_*)\Psi(|v-v_*|)=0$,
we find
\begin{multline*}
(\nabla-\nabla_*)[a(v-v_*)\nabla h]   = (\nabla-\nabla_*)\cdot \bigl [\Psi(|v-v_*|)\,\Pi_{k^\bot} \nabla h\bigr ] \\
 =  \Psi (|v-v_*|)\,(\nabla-\nabla_*)\cdot \Pi_{k^\bot}\nabla h + \Psi(|v-v_*|)\,\Pi_{k^\bot}: \nabla^2 h.
\end{multline*} 
So
\begeq\label{endupQL}
\int Q_L(f,f)h = 
\iint ff_* \Psi \left( \Pi_{k^\bot}: \nabla^2h
- 2(d-1) \frac{(v-v_*)}{|v-v_*|^2}\cdot\nabla h\right)\,dv\,dv_*.
\endeq

Now for the Boltzmann term. Write
\begeq\label{37ter} v'-v = \frac{v_*-v}{2} + \frac{|v-v_*|}{2} \sigma = \frac{|v-v_*|}{2} (\sigma-k).\endeq
Decompose
\begeq\label{37bis} \sigma = (\cos\theta) k + (\sin\theta)\phi,\qquad \phi \in \S^{d-2}_{k^\bot},\endeq
\begeq\label{37} \sigma - k = (\cos\theta -1) k + (\sin\theta) \phi = \theta\phi + O(\theta^2).\endeq
So

\begin{align*}
h'-h 
& = \nabla h(v)\cdot (v'-v)
+ \frac12 \bigl\< \nabla^2h(v)\cdot v'-v,v'-v \bigr\> + o(|v'-v|^2) \\
& = -\frac{|v-v_*|}{2} k\cdot\nabla h\, (1-\cos\theta)
+ \frac{|v-v_*|}{2} \sin\theta\, (\phi\cdot \nabla h)\\
& \qquad\qquad\qquad\qquad + \frac{|v-v_*|^2}{8} \bigl\<\nabla^2h \cdot\phi,\phi\bigr\>\, \theta^2
+ o (|v-v_*|^2 \theta^2).
\end{align*}
This gives rise to four terms in \eqref{endupQB}, let us consider them separately and check that they match the corresponding integral terms in \eqref{endupQL}:
\begin{align*}
- \int_{\S^{d-1}} \frac{|v-v_*|}{2}
k\cdot \nabla h\,  (1-\cos\theta)\,B\,d\sigma
& = - \frac{|v-v_*|}{2} (k\cdot\nabla h)
\int_{\S^{d-1}}(1-\cos\theta)\,B\,d\sigma \\
& \longrightarrow - \frac{|v-v_*|}2 (k\cdot\nabla h) M_\infty
= - \frac{(v-v_*)}{2}\cdot\nabla h\, M_\infty,
\end{align*}

\[ \frac{|v-v_*|}{2} 
\int_{\S^{d-1}} \sin\theta\, (\phi\cdot \nabla h) B\,d\sigma = 0
\qquad\text{by symmetry $\phi\to -\phi$},\]

\begin{align*}
\int_{\S^{d-1}} \frac{|v-v_*|^2}{8} \bigl\< \nabla^2h \cdot\phi,\phi\bigr\> \,\theta^2\, B\,d\sigma
& = \frac{|v-v_*|^2}{4(d-1)}\tr \bigl(\nabla^2h\,\Pi_{k^\bot}\bigr) \int_{\S^{d-1}} \frac{\theta^2}{2}\,B\,d\sigma \\
& \longrightarrow \frac{|v-v_*|^2}{4(d-1)} M_\infty(|v-v_*|)\,\tr (\nabla^2h\,\Pi_{k^\bot}),
\end{align*}

\[ \int o(|v-v_*|^2\theta^2)\,B\,d\sigma \longrightarrow 0.\]
This concludes the proof of (i). As for (ii), first note that
\begin{align*}
D_L(f)
& = -\int \nabla\cdot \left( \int \Psi \Pi_{k^\bot} \bigl(f_*\nabla f- f(\nabla f)_*\bigr) \,dv_*\right) (\log f+1)\,dv \\
& = \iint ff_* \Psi \Pi_{k^\bot} \left(\frac{\nabla f}{f} - \Bigl(\frac{\nabla f}{f}\right)_*\Bigr)\cdot \frac{\nabla f}{f}\,dv\,dv_*\\
& = \frac12 \iint ff_*\Psi \Pi_{k^\bot} \left( \frac{\nabla f}{f} - \Bigl(\frac{\nabla f}{f}\Bigr)_*\right)\cdot
\left( \frac{\nabla f}{f} - \Bigl(\frac{\nabla f}{f}\Bigr)_*\right)\,dv\,dv_*\\
& = \frac12 \iint ff_* \Psi \, \left| \Pi_{k^\bot} \left( \frac{\nabla f}{f} - \Bigl(\frac{\nabla f}{f}\Bigr)_*\right)\right|^2\,dv\,dv_*.
\end{align*}
Next the infinitesimal velocity variations
\begeq\label{infv} v'_*-v_* =v-v' = - \frac{|v_*-v|}{2}\theta \phi + o (|v-v_*|\theta) \endeq
imply, by Taylor expansion of $f(v')$ around $f(v)$ and $f(v'_*)$ around $f(v_*)$,
\[ f'f'_* - ff_*= \frac{|v-v_*|}{2} \bigl( f_*\nabla f - f (\nabla f)_*\bigr)\cdot\theta\phi + o(|v-v_*| \theta). \]
As $b-a\to 0$ we have $(a-b)(\log a-\log b) \simeq (a-b)^2/a$, so
\begin{align*}
(f'f'_* - ff_*) (\log f'f'_* - \log ff_*)
&  = \frac{(f'f'_* - ff_*)^2}{ff_*} + o(|v-v_*|^2\theta^2) \\
& = \frac{|v-v_*|^2}{4} ff_* \left| \left[\frac{\nabla f}{f} - \Bigl(\frac{\nabla f}{f}\Bigr)_*\right]\cdot\phi \right|^2 \,\theta^2 + o (|v-v_*|^2 \theta^2).
\end{align*}
Thus $D_B(f)$ is
\begin{align*}
\frac14 & \iiint (f'f'_* - ff_*) (\log f'f'_* - \log ff_*)\,B\,d\sigma\,dv\,dv_* \\
& \simeq \frac14 \iint \frac{|v-v_*|^2}{4(d-1)}\, ff_* 
\left| \Pi_{k^\bot} \left( \frac{\nabla f}{f} - \Bigl(\frac{\nabla f}{f}\Bigr)_*\right) \right|^2\, 2 M_\infty\,dv\,dv_* \\
& = \frac12 \iint \Psi ff_* \left | \Pi_{k^\bot} \ \left( \frac{\nabla f}{f} - \Bigl(\frac{\nabla f}{f}\Bigr)_*\right) \right|^2\,dv\,dv_*
= D_L(f).
\end{align*}

\end{proof}

The classification of Boltzmann's kernels leads, via the AGC, to a similar taxonomy for Landau's equation: if $\Psi(z) = K |z|^{\gamma+2}$, one talks of hard potentials for $\gamma>0$, Maxwell kernel for $\gamma=0$, soft potentials for $\gamma<0$ and very soft potentials for $\gamma<-2$.

\bibnotes

The founding paper is Landau \cite{landau:Coulomb:36}. Background on plasma physics can be found in Landau--Lipschitz \cite{LL:kin:81} and Delcroix-Bers \cite{delcbers:book:94}. Detailed discussions of kinetic models for plasmas are in the books by Bogoljubov \cite{bogol:book} and Balescu \cite{balescu:charged:63}. The so-called Balescu--Lenard model (which should rather be called Bogoljubov--Balescu) appears in those works; it makes sense even without the Debye cut, but its popularity has remained limited, since its complexity is much higher and its usefulness only appears in very special cases.

Asymptotics of grazing collision was mathematically considered, with increasing generality, by Arsen'ev and Buryak \cite{arsenbur:landau:91}, Degond and Lucquin-Desreux \cite{DLD:FPL:92}, Desvillettes \cite{desv:landau:92}, Goudon \cite{goud:graz:97}, Villani \cite{vill:new:98}, Alexandre and Villani \cite{AV:landau:04}. The more recent of these works prove not only the convergence of the operator, but also the convergence of solutions, with increasing generality, until \cite{AV:landau:04} covers the most general mathematical setting that one may hope for: namely ``renormalised'' solutions in the sense of DiPerna--Lions, a concept of very weak solutions for inhomogeneous solutions of the Boltzmann equation \cite{DPL:boltz:89} and Landau equation \cite{lions:landau:94,vill:landau:96}. In \cite{AV:landau:04} one has also sought the most general conditions on the cross-section, not assuming any factorisation property for instance. Along the way, as there was progress in the understanding of weak solutions for the Boltzmann equation with nonintegrable collision kernels, the AGC served both as a motivation and a consistency check for any new such theory. Background on the physical meaning of the asymptotics is also included in \cite{AV:landau:04}.

Many authors have studied the properties of the Landau equation itself. For the purpose of this course, the most relevant works are those handling qualitative properties (smoothness, moments, lower bound, equilibration) in the spatially homogeneous situation: my study of the Maxwellian case ($\gamma=0)$ \cite{vill:maxw:98} and my study of hard potentials ($\gamma>0$) with Desvillettes \cite{DV:landau:1,DV:landau:2}. Then the special situation $d=3$, $\gamma=-3$ was considered with special care, but up to recently there were only partial or conditional results \cite{GGIV:landau,GIV:landau}.

There is another way from the Boltzmann to the Landau equation, distinct from the AGC, which goes via the study of fluctuations in the derivation from particle system; then $\Psi$ is always of Coulomb type (proportional to $1/|z|$ if $d=3$) and the Boltzmann collision kernel only enters via a multiplicative constant; see Spohn \cite{spohn:book} and references therein.

\section{The Guillen--Silvestre theorem}

Let us resume the story which I started at the end of Section \ref{sectax}. Around 1998 the picture was the following:

\bul Fisher information was proven decreasing for Maxwellian kernels ($\gamma=0$) and it seemed (at least to me) the best to hope for;

\bul the Landau--Coulomb equation ($d=3,\gamma=-3$) was resisting attempts to study its smoothness or singularity.  Actually there was, at first, reason to suspect spontaneous development of singularities. Indeed, for $\Psi(z) = 1/|z|$ we have a special form: writing $\ov{a} = a\ast f$,
\begin{align} \label{LCspecial}
\derpar{f}{t} & = \nabla\cdot (\ov{a}\nabla f- (\nabla\cdot \ov{a})f)\\
& \nonumber = \ov{a} : \nabla^2 f - (\nabla^2:\ov{a}) f \\
& \nonumber
= \ov{a}:\nabla^2 f + 8\pi f^2,
\end{align}
where I used $\nabla^2: a = - 8 \pi \delta_0$.
If $f$ is smooth, then at best $\ov{a}= a \ast f$ is positive, smooth, degenerate at infinity, so this equation does not seem more regularising than the model equation
\begeq\label{nlheat}
\derpar{f}{t} = \Delta f + 8 \pi f^2,
\endeq
which is known to generally blow up in finite time in $L^p$ for all $p> 3/2 $ in dimension $d=3$, even starting from smooth initial conditions.

So it was tempting to believe that there could be blow-up in such $L^p$ spaces, and still the solution would remain in $L\log L$ (by $H$-Theorem) or maybe in $L^p$ for some $p<3/2$; in this way weak solutions would still be well-defined but smooth solutions would break up. This was my initial guess, and that of my PhD advisor as well, as I was starting to work on the Landau equation. But on precisely this topic I remember a heated discussion in Oberwolfach with Sacha Bobylev, who criticized the whole idea of writing the equation in nonconservative form.

It turned out that Bobylev was right (as often): shortly after,  numerical simulations arrived -- it was the time at which the progress of computing power and numerical analysis opened the path to deterministic schemes for collisional kinetic equations, yielding much more convincing qualitative insight than probabilistic schemes -- and they showed clearly that the behaviour of \eqref{LCspecial} and that of \eqref{nlheat} were completely different. Refining the analysis and eager to test their numerical methods on various qualitative problems, Christophe Buet and St\'ephane Cordier observed in those days that (a) solutions of \eqref{LCspecial} do not seem to exhibit any degradation of smoothness, (b) Fisher information seems to be decreasing along solutions of \eqref{LCspecial} even it is far from the Maxwellian case, and (c) numerical schemes accurate enough to capture this monotonicity were also the more stable. This was intriguing and I mentioned this as food for thought in my habilitation dissertation. 

And then on this front nothing occurred for 25 years!

But after this quarter-of-Sleeping Beauty period, in 2023 Guillen and Silvestre proved:

\bul $I$ is nonincreasing along solutions of \eqref{LCspecial} in dimension $d=3$;

\bul As a consequence, \eqref{LCspecial} does not exhibit blowup, and in fact as soon as the initial condition $f_0=f_0(v)$ on $\R^d$ satisfies $I(f_0)< \infty$ one can solve the Cauchy problem with a smooth solution;

\bul The monotonicity property is true for a whole range of exponents $\gamma$, positive or negative (provided that the Cauchy problem makes sense, which is not clear for $\gamma< -3$ when $d= 3$, but that is a different story): they proved it for $|\gamma| \leq \sqrt{19}$, and since then that bound has been slightly improved to $\sqrt{22}$.
\sm

Observe that with this theorem at hand the figure of known parameters for the monotonicity of $I$ was now the union of two segments in Fig.~\ref{figtaxonomy}, one horizontal in the middle, and one vertical on the right, like a $\dashv$ symbol, calling for extension and clarification. 

Upon learning of this spectacular result, I first started to dig in the manuscript of Guillen and Silvestre, then I eventually resumed my old papers and worked out on the Boltzmann equation again, trying to improve the known results up to a point where the Guillen--Silvestre theorem would be included. At the same time, symmetrically, Imbert and Silvestre were working to generalise the Guillen--Silvestre approach to the Boltzmann equation. In the Summer of 2024 we joined forces and after a particularly fruitful season we obtained some quite general results covering a large part of the parameter space of collision kernels; this will be the central focus in the rest of these notes. Stated informally, our central result is

\begin{Thm} Fisher's information is nonincreasing along solutions of the spatially homogeneous Boltzmann equation, for the vast majority of kernels of interest, including all power law forces between Coulomb and hard spheres, for all dimensions.
\end{Thm}

\bibnotes

The description of singularities for the nonlinear heat equation with quadratic nonlinearity is a classical topic, addressed for instance by Zaag \cite{zaag:PhD}.

Deterministic numerical models for the Landau equation and the evaluation of the Fisher information were performed by Buet and Cordier \cite{buetcordier:conserv:98}.

Guillen and Silvestre proved their result in \cite{GS:landaufisher}. (Silvestre reminded me that, back in 2009, upon hearing him lecture in IAS Princeton about jump diffusion processes, I encouraged him to start working on the Boltzmann equation. In retrospect my encouragement was an investment with extraordinary payback.) Their method to handle differentiation of the Fisher information is based on a vector field formalism, contrary to my older approach in \cite{vill:fisher}, which is in the spirit of connections. (I prefer the latter but I am obviously biased.) Our joint work with Imbert and Silvestre is in \cite{ISV:fisher}.

\section{Tensorisation} \label{sectens}

As it models the effect of binary collisions, the Boltzmann operator takes the form of the marginal of a linear operator in the tensor product $f\otimes f$. Already Maxwell and Boltzmann put this to good use in their discussion of equilibria: The Gaussian arises here because it is the only tensor product with spherical symmetry. Now Guillen and Silvestre observed that this can also be explicitly useful for the computation of the time-evolution of the Fisher information.

\begin{Def}[joint linear Boltzmann operator] \label{defcalB}
Let ${\cal B}$ be the linear operator, acting on distributions $F=F(v,v_*)$ on $\R^{2d}$, by
\[
({\cal B}F)(v,v_*) = \int_{\S^{d-1}} \bigl[ F(v',v'_*)- F(v,v_*)\bigr]\,B(v-v_*,\sigma)\,d\sigma.
\]
\end{Def}
Obviously ${\cal B}$ still makes sense when applied to a function defined on the collision sphere with diameter $[v,v_*]$. Equally obviously,
\begeq\label{QintB}
Q(f,f) = \int_{\R^d} {\cal B}(f\otimes f)\,dv_*.
\endeq
I shall note $V=(v,v_*)$, $V'=(v',v'_*)$, $z(V)=v-v_*$.

\begin{Prop}\label{propdoubling}
For any probability distribution $f$ on $\R^d$,
\begeq\label{Idoubling}
I(f) =\frac12 I(f\otimes f),\qquad
I'(f)\,Q(f,f) = \frac12 I'(f\otimes f)\, {\cal B}(f\otimes f).
\endeq
\end{Prop}

Practical consequence:  To show that $dI/dt\leq C\,I$ along the Boltzmann equation ($C\in\R$) it is sufficient to prove it along the flow of the linear Boltzmann equation $\pa_t F = {\cal B}F$, at least when $F$ has the special form $f\otimes f$.

\begin{Rk} $e^{t{\cal B}}$ does not preserve the tensor product structure (it would need ${\cal B}(f\otimes f) = Q(f,f)\otimes f + f\otimes Q(f,f)$, a very special property wich probably forces either the interaction to be trivial, or $f$ to be Maxwellian). So $e^{tQ}$ and $e^{t{\cal B}}$ are not equivalent. This is what makes Property \ref{propdoubling} interesting.
\end{Rk}

\begin{proof}[Proof of Proposition \ref{propdoubling}] Let us write $\bbnabla=\nabla_V = [\nabla_v,\nabla_{v_*}]$. I use the same letter $I$ for $I(f) = \int f |\nabla \log f|^2$ and $I(F) = \iint F |\bbnabla \log F|^2$. 

The first equality in \eqref{Idoubling} is well-known: If $F=f\otimes f$ then $|\bbnabla \log F|^2 = | \nabla\log f|^2 + |(\nabla\log f)_*|^2$, so
$I(ff_*) = \iint ff_* |\nabla \log f|^2 + \iint ff_* |(\nabla \log f)_*|^2 = 2 \int f|\nabla \log f|^2$.

Next,
\begeq\label{firstI'Q} 
I'(f)\, Q(f,f) = \int Q(f,f) |\nabla\log f|^2 + 2 \int f \nabla \log f \cdot \nabla \left(\frac{Q(f,f)}{f}\right).
\endeq
On the other hand,
\begin{multline*}
I'(f\otimes f)\,{\cal B}(f\otimes f) 
= \int {\cal B}(f\otimes f)\, |\bbnabla \log f\otimes f|^2\,dV \\
+ 2 \int f\otimes f\, \bbnabla \log (f\otimes f)\cdot \bbnabla \left(\frac{{\cal B}(f\otimes f)}{f\otimes f}\right)\,dV 
\end{multline*}
\begin{multline*}
= \iiint (f'f'_* - ff_*) B \, \bigl( |\nabla \log f|^2 + |(\nabla \log f)_*|^2\bigr)\,d\sigma\,dv\,dv_*\\
+ 2 \iint ff_* \nabla \log f\cdot \nabla \left(\frac1{ff_*}
\int (f'f'_* - ff_*)\,B\,d\sigma\right)\,dv\,dv_*\\
+ 2 \iint ff_* (\nabla \log f)_*\cdot \nabla_* \left(\frac1{ff_*}
\int (f'f'_*-ff_*)\,B\,d\sigma\right)\,dv\,dv_*.
\end{multline*}
Obviously $(1/ff_*)\int (f'f'_*-ff_*)\,B\,d\sigma$ is symmetric in $(v,v_*)$ so the last two integrals are identical; and multiplication by $1/f_*$ commutes with $\nabla$, so eventually
\begin{align*}
 I'(f\otimes f){\cal B} (f\otimes f) & = 2 \iiint (f'f'_*-ff_*)\,B\,|\nabla \log f|^2\,d\sigma\,dv\,dv_*\\
&\qquad\qquad
+ 4 \iint ff_* \nabla \log f \cdot \nabla \left(\frac1{ff_*} \int (f'f'_*-ff_*)\,B\,d\sigma\right)\,dv\,dv_*\\
& = 2 \int \left(\iint (f'f'_*-ff_*)\,B\,d\sigma\,dv_*\right) |\nabla\log f|^2\,dv\\
&\qquad\qquad + 4 \iint f\nabla\log f\cdot \nabla\left(\frac1{f}\int (f'f'_*-ff_*)B\,d\sigma\right)\,dv\,dv_*.
\end{align*}
In the end, this is
\begin{multline*} 2 \int\left( \iint (f'f'_*-ff_*)\,B\,d\sigma\,dv_*\right) |\nabla\log f|^2\,dv\\
 + 4 \int f\nabla\log f\cdot \nabla \left(\frac1{f} \iint (f'f'_*-ff_*)\,B\,d\sigma\,dv_*\right)\,dv,
\end{multline*}
whence the result upon comparison with \eqref{firstI'Q}.
\end{proof}

\begin{Rk} The proof above would fail for any generic functional of the form $\int A(F,\bbnabla F)\,dV$. But it works for the particular choice $A(F,\bbnabla F)=|\bbnabla F|^2/F$. The same can be said of Boltzmann's entropy: only for $A(F) = F\log F$ can one show an interesting relation between $\int A(F)$ and $\int A(f)$.
Actually, it is easy to establish the entropic counterpart of Proposition \ref{propdoubling}:
\[ H(f\otimes f) = 2H(f),\qquad
H'(f\otimes f)\, {\cal B}(f\otimes f) = 2 H'(f)\, Q(f,f).\]
\end{Rk} 

\begin{Rk}\label{rkHThrevisited}
Let us recover the classical $H$-Theorem with the help of those identities:
\begin{multline*}
-H'(F)\,({\cal B}F) = - \int (\log F+1)\, {\cal B}F = - \iint (F'-F)\,B\,d\sigma (\log F+1)\,dV \\
= - \iint F (\log F'-\log F)\,B\,d\sigma\,dV = \iint F \log\frac{F}{F'}\,B\,d\sigma\,dV.
\end{multline*}
Now $C:(a,b)\longmapsto a\log (a/b)$ is jointly convex on $\R_+^2$ and $1$-homogeneous, so we conclude that $-H'(F)\cdot {\cal B}F\geq 0$ by invoking Jensen's inequality, in the form
\[ \int (F-G)\,d\nu = 0 \Longrightarrow \int C(F,G)\,d\nu \geq 0,\]
with $F=F(V)$, $G=G(V')$, $d\nu = B\,d\sigma\,dV$. (The homogeneity makes it unnecessary to impose $\int d\nu = 1$, and an approximation argument also covers the case when $\int d\nu = \infty$.)
Now, estimating $I'(F)\,({\cal B}F)$ will turn out to be much more complicated that $H'(F)\,({\cal B}F)$!
\end{Rk}

\bibnotes

The use of the tensor product structure of the Boltzmann collision operator appears already in Maxwell \cite{maxw:67}, Boltzmann \cite{boltz:book}, or in the treatment of Boltzmann's dissipation functional by Toscani and I \cite{TV:entropy:99,vill:cer:03}, but Guillen and Silvestre seem to be the first to have explicitly shown its full power \cite{GS:landaufisher}. 

McKean \cite{mck:kac:65} makes some remarks about the properties of those particular nonlinearities appearing in $H$ and $I$; actually in information theory one often starts {\em a priori} from desired identities under tensor product, to justify the introduction of $H$.
\cite{CT:book}.

\section{Playing $\Gamma$ calculus with Boltzmann and Fisher} \label{secGamma}

A cornerstone of the study of diffusion processes, everywhere present in the theory of log Sobolev inequalities, is the structure of the infinitesimal commutator between the diffusion under scrutiny and a local nonlinearity, typically the integrand of an integral functional of interest. It captures the ``genuinely dissipative part' in some sense. As the most classical example, starting from the quadratic functional $\int f^2$, given a linear operator $L$, define
\begeq\label{Gammaf} \Gamma(f,f) = \Gamma_1(f,f) = \frac12 \bigl( L f^2 - 2 f Lf\bigr). \endeq
Then
\[ - \left.\frac{d}{2\,dt}\right|_{t=0} \int f^2 = \int \Gamma(f,f),\]
but $\Gamma$ is a local quantity, which may be technically useful, and computationally more elementary, than its integral version. Notice, if $ Lf = \Delta f + \xi\cdot\nabla f$, for any vector field $\xi$, then $\Gamma(f,f) = |\nabla f|^2$. In certain circles, $\Gamma$ is called the {\em carr\'e du champ}. Then $\Gamma$ can be extended into a bilinear operator $\Gamma(f,g)$, and the procedure can be iterated, replacing $fg$ by $\Gamma(f,g)$:
\begeq\label{Gamma2f} 
\Gamma_2 (f,f) = \frac12\bigl( L \Gamma(f,f)- 2\Gamma(f,Lf)\bigr),
\endeq
whence
\[ -\left.\frac{d}{2\,dt} \right|_{t=0} \int \Gamma(f,f) = \int \Gamma_2(f,f).\]

\begin{Rk}
As another warmup, let us rewrite Remark \ref{rkHThrevisited} in that way. Introduce the functional $\Gamma_{H,{\cal B}}$ by
\[ \Gamma_{H,{\cal B}}(F) 
 = {\cal B} (F \log F) - (\log F+1)\, {\cal B}F\]
and then compute
\begin{align*} \Gamma_{H,{\cal B}}(F) 
& = {\cal B} (F \log F) - (\log F+1)\, {\cal B}F\\
& = \int \bigl(F'\log F'-F\log F\bigr)\,B\,d\sigma
- (\log F+1) \int (F'-F)\,B\,d\sigma \\
& = \int \left(F'\log\frac{F'}{F} - F'+F\right)\,B\,d\sigma\\
& = F \int \left(\frac{F'}{F} \log \frac{F'}{F} - \frac{F'}{F} + 1\right)\,B\,d\sigma \geq 0,
\end{align*}
since $x\log x - x + 1\geq 0$ for all $x$. This implies the $H$ Theorem. Note that this computation is in a way simpler than the usual one, and even Remark \ref{rkHThrevisited}, since (a) integration in $V$ was not used, (b) the pre-post-collisional change of variables is not explicitly used, it enters only via $\int {\cal B}G=0$ after the computation is performed, (c) neither is Jensen's inequality. Of course all three ingredients are implicit and useful for the $H$-Theorem itself, but in the $\Gamma$ computation the algebraic structure for the monotonicity of $H$ is reduced to its core.
\end{Rk}

In this section I shall consider as nonlinearity the integrand of Fisher's information, in the form ${\cal F}(F,\bbnabla F) = F|\bbnabla \log F|^2$, and take ${\cal B}$ as linear operator. As before, $\bbnabla = [\nabla,\nabla_*]$.

\begin{Def}[Gamma calculus for Fisher and Boltzmann] If $B$ is a Boltzmann kernel and ${\cal B}$ the associated joint linear operator as in Definition \ref{defcalB}, let $\Gamma_{I,{\cal B}}$ be defined by
\begeq\label{GIBF}
\Gamma_{I,{\cal B}}(F) = {\cal B} \bigl( F |\bbnabla \log F|^2\bigr)
- \left[ ({\cal B}F) |\bbnabla \log F|^2 + 2 \bbnabla F \cdot \bbnabla \left(\frac{{\cal B}F}{F}\right)\right].
\endeq
\end{Def}

(No factor $1/2$ here since ${\cal F}$ is $1$-homogeneous in $F$.) By construction, using $\int {\cal B} G =0$,
\begeq\label{GIBI'B}
\int \Gamma_{I,{\cal B}}(F)\,dV = - I'(F) ({\cal B}F).
\endeq
So in view of Proposition \ref{propdoubling}, any positivity estimate on $\Gamma_{I,{\cal B}}$ will imply a decay for $I$ along the Boltzmann equation; and will even more precise, as it will allow for multiplication by a function of $V$, for whatever purpose.

We go on with the computation of $\Gamma_{I,{\cal B}}(F)$. Obviously it is needed to compute $\bbnabla {\cal B}F$. I start with three computational lemmas, using the same formalism as in older work of mine on Fisher information and Boltzmann equation.

\begin{Lem} \label{lemBL1}
$\dps \bbnabla (F') = \A_{\sigma k} (\bbnabla F)'$
where, for any $k,\sigma$ in $\S^{d-1}$,
\[ \A_{\sigma k} = \frac12 \left[\begin{matrix} I + k\otimes\sigma & I -k\otimes\sigma\\
I - k\otimes\sigma & I + k\otimes\sigma \end{matrix}\right], \]
and by convention $k\otimes\sigma(x) = (x\cdot\sigma)k$.
\end{Lem}

For the next one, recall that $B=B(|z|,\cos\theta)$.
\begin{Lem}\label{lemBL2}
$\dps \bbnabla B = [\nabla B, -\nabla B]$,
where
\[ \nabla B = \left(\derpar{B}{|z|}\right) k + \frac1{|z|} \left(\derpar{B}{(\cos\theta)}\right) \Pi_{k^{\bot}}\sigma.\]
\end{Lem}

\begin{Lem}\label{lemBL3}
For any differentiable functions $g:\R^d\to \R$, $b:(-1,1)\to\R$, any $k\in\S^{d-1}$,
\[ \int_{\S^{d-1}}b'(k\cdot\sigma) g(\sigma) \bigl(\Pi_{k^\bot}\sigma\bigr)\,d\sigma
= \int b(k\cdot\sigma) M_{\sigma k}\nabla g(\sigma)\,d\sigma,\]
where $\dps M_{\sigma k}(x) = (k\cdot\sigma) x - (x\cdot k)\sigma$.
\end{Lem}

(I will write indifferently $M_{\sigma k}(x)$ or $M_{\sigma k}x$.) The proofs of these lemmas will be given after the main results of this section, namely the computation of $\bbnabla {\cal B}F$ and that of $\Gamma_{I,{\cal B}}(F)$.

\begin{Prop}\label{propnablaBF}
\[ \bbnabla {\cal B}F = \int_{\S^{d-1}} \bigl[
\G_{\sigma k}(\bbnabla F)' - \bbnabla F \bigr]\,B\,d\sigma
+ \left( \int_{\S^{d-1}}(F'-F)\,\derpar{B}{|z|}\,d\sigma\right) [k,-k],\]
where
\begeq\label{Gsk} \G_{\sigma k} = \frac12 \left[ \begin{matrix}
I + P_{\sigma k} & I - P_{\sigma k} \\ I-P_{\sigma k} & I+P_{\sigma k} \end{matrix}\right] 
\endeq
and $P_{\sigma k} = k\cdot \sigma + k\otimes \sigma - \sigma\otimes k$, or explicitly
\begeq\label{Psk}
P_{\sigma k} (x) = (k\cdot\sigma)x + (\sigma\cdot x) k - (x\cdot k)\sigma .
\endeq
\end{Prop}

\begin{Thm} \label{thmGamma}
\begin{multline}
\label{fmlGamma}
\Gamma_{I,{\cal B}}(F) = 
\int_{S^{d-1}} F'\Bigl\< \C_{\sigma k} \bigl[ \bbnabla \log F, (\bbnabla \log F)'\bigr],
\bigl[ \bbnabla \log F, (\bbnabla \log F)'\bigr]\Bigr\>\,B\,d\sigma \\
- 2 \left(\int (F'-F)\,\derpar{B}{|z|}\,d\sigma\right) k\cdot z(\bbnabla\log F),
\end{multline}
where $z(v,v_*) = v-v_*$ and for any $k,\sigma$ in $\S^{d-1}$, $\C_{\sigma k}$ is the $(4d)\times (4d)$ linear operator defined by
\[ \C_{\sigma k} = 
\left[\begin{matrix} I & \G_{\sigma k} \\ \G_{k \sigma} & I \end{matrix} \right]= 
\left[ \begin{matrix} \left(\begin{matrix} I & 0 \\ 0 & I\end{matrix} \right) &
\frac12 \left( \begin{matrix} I+P_{\sigma k} & I-P_{\sigma k} \\
I -P_{\sigma k} & I+P_{\sigma k}\end{matrix}\right) \\ \\
\frac12 \left( \begin{matrix} I+P_{k \sigma} & I-P_{k \sigma} \\
I -P_{k \sigma} & I+P_{k \sigma}\end{matrix}\right) &
\left(\begin{matrix} I & 0 \\ 0 & I\end{matrix} \right) \end{matrix} \right].\]

In particular, with the notation
\[ \xi = -\nabla\log f, \qquad z = v-v_*, \qquad k = \frac{v-v_*}{|v-v_*|},\qquad
|y-x|_{k,\sigma}^2 = |x|^2 + |y|^2 - 2 (P_{k\sigma}x)\cdot y,\]
we have the following formula when $F=f\otimes f$:
\begin{align} \label{Gammaff}
\Gamma_{I,{\cal B}}(f\otimes f)
& = \frac12 \int f'f'_*\, \bigl| (\xi'+\xi'_*) - (\xi+\xi_*)\bigr|^2\,B\,d\sigma\\
& \nonumber + \frac12 \int f'f'_*\, \bigl| (\xi'-\xi'_*) - (\xi-\xi_*)\bigr|^2_{k,\sigma}\,B\,d\sigma\\
& \nonumber + 2 \int (f'f'_* - ff_*)\,k\cdot (\xi-\xi_*)\,\derpar{B}{|z|}\,d\sigma.
\end{align}
\end{Thm}

\begin{Cor}\label{cormaxw} $\Gamma_{I,{\cal B}}(f\otimes f)\geq 0$ if $B$ does not depend on $|z|$ (Maxwellian kernel), in particular $I$ is nonincreasing along the spatially homogeneous Boltzmann equation with such a kernel.
\end{Cor}

Some words of explanation about the notation $|y-x|_{k,\sigma}^2$. One way to think of $P_{k\sigma}$ is a recipe to map $k^\bot=T_k\S^{d-1}$ onto $\sigma^\bot=T_\sigma\S^{d-1}$ (a connection, in pedantic words, although much simpler than the Riemannian one; see Figure \ref{figPks}). This allows to compare vectors which are tangent to the sphere (e.g. gradients) defined at $k$ and $\sigma$ respectively. Note carefully that when I write $|y-x|_{k,\sigma}^2$, there is no such thing as $y-x$, it is the whole expression $|y-x|_{k,\sigma}^2$ which makes sense.

\begin{Ex} When $d=2$, then under the identification $\S^1= 2\pi\R/\Z$, $|y-x|^2_{k,\sigma}$ is just the usual squared norm $|y-x|^2$ in $\R$.
\end{Ex}

\begin{Ex} When $d=3$, we may use spherical coordinates $(\alpha,\phi)$ (with the notation $\alpha$ rather than $\theta$ to avoid confusion with the deviation angle); if $k$ has coordinates $(\alpha,\phi)$ and $\sigma$ has coordinates $(\alpha',\phi')$, then
\begin{multline*} |\nabla f(\sigma)-\nabla f(k)|_{k,\sigma}^2 = 
(\partial_\alpha f)^2 + (\partial_\phi f)^2 + (\partial_\alpha f)'^{2} + (\partial_\phi f)'^2\\
- 2 \cos(\phi-\phi') (\pa_\alpha f) (\pa_\alpha f)' - 2 \bigl (\sin\alpha\sin\alpha'+ \cos\alpha\cos\alpha'\cos(\phi-\phi')\bigr) (\pa_\phi f)(\pa_\phi f)' \\
+ 2 \sin(\phi'-\phi) \bigl ( \cos\alpha'(\pa_\alpha f)(\pa_\phi f)' - \cos\alpha (\pa_\alpha f)'(\pa_\phi f)\bigr ).
\end{multline*}
So already in dimension 3 this explicit formula becomes cumbersome and it will be more convenient, in practice, to use the abstract, intrinsic expression.
\end{Ex}

\bigskip

\begin{figure}
\caption{The operator $P_{k\sigma}$ sends $k$ onto $\sigma$, and the tangent plane $k^\bot$ onto the tangent plane $\sigma^\bot$, thereby allowing to compare vectors tangent to the sphere at $k$ and $\sigma$ respectively.}
\label{figPks}
\def\svgwidth{0.5\textwidth}
\begingroup%
  \makeatletter%
  \providecommand\color[2][]{%
    \errmessage{(Inkscape) Color is used for the text in Inkscape, but the package 'color.sty' is not loaded}%
    \renewcommand\color[2][]{}%
  }%
  \providecommand\transparent[1]{%
    \errmessage{(Inkscape) Transparency is used (non-zero) for the text in Inkscape, but the package 'transparent.sty' is not loaded}%
    \renewcommand\transparent[1]{}%
  }%
  \providecommand\rotatebox[2]{#2}%
  \newcommand*\fsize{\dimexpr\f@size pt\relax}%
  \newcommand*\lineheight[1]{\fontsize{\fsize}{#1\fsize}\selectfont}%
  \ifx\svgwidth\undefined%
    \setlength{\unitlength}{595.27559055bp}%
    \ifx\svgscale\undefined%
      \relax%
    \else%
      \setlength{\unitlength}{\unitlength * \real{\svgscale}}%
    \fi%
  \else%
    \setlength{\unitlength}{\svgwidth}%
  \fi%
  \global\let\svgwidth\undefined%
  \global\let\svgscale\undefined%
  \makeatother%
  \begin{picture}(1,1.41428571)%
    \lineheight{1}%
    \setlength\tabcolsep{0pt}%
    \put(0,0){\includegraphics[width=\unitlength,page=1]{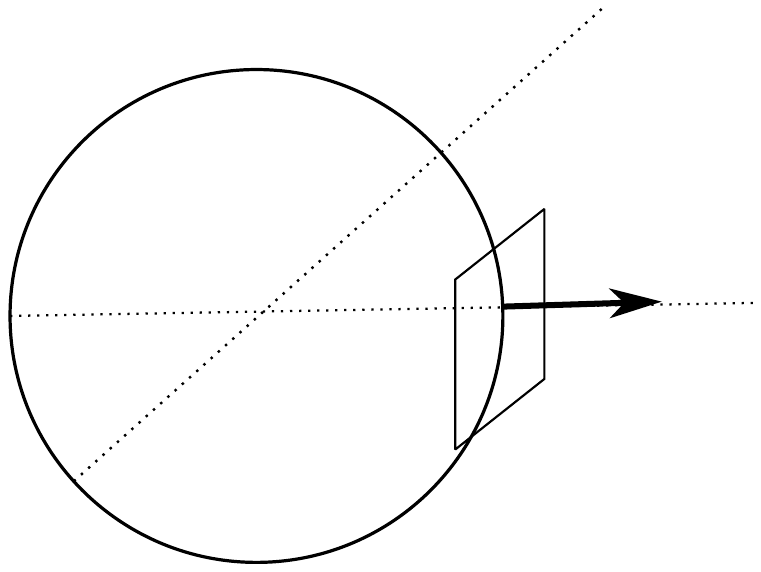}}%
    \put(0.72282896,1.01312078){\color[rgb]{0,0,0}\makebox(0,0)[lt]{\lineheight{1.89988983}\smash{\begin{tabular}[t]{l}$k$\end{tabular}}}}%
    \put(0.61350919,1.22717852){\color[rgb]{0,0,0}\makebox(0,0)[lt]{\lineheight{1.89988983}\smash{\begin{tabular}[t]{l}$\sigma$\end{tabular}}}}%
    \put(0.83839162,1.19164088){\color[rgb]{0,0,0}\makebox(0,0)[lt]{\lineheight{1.89988971}\smash{\begin{tabular}[t]{l}$P_{k\sigma}$\end{tabular}}}}%
    \put(0,0){\includegraphics[width=\unitlength,page=2]{Pks.pdf}}%
  \end{picture}%
\endgroup%

\vspace*{-60mm}
\end{figure}
\bigskip

In the end of this section, I shall provide a list of useful properties for the operators $P_{\sigma k}$ (Proposition \ref{propPsk}). For the moment let me note that
\begeq\label{v2ksigma} 
\bigl| (v'-v'_*) - (v-v_*) \bigr|_{k,\sigma}^2 = \bigl( |v'-v'_*| - |v-v_*| \bigr)^2.
\endeq
(Check that identity directly or apply Proposition \ref{propPsk}.)

\begin{Rk}
All in all, \eqref{Gammaff} has a rather neat formal structure:

\bul The first two terms in the right-hand side of \eqref{Gammaff} are related to the two conservation laws of elastic collisions: conservation of momentum and conservation of energy. Indeed, they both take the form $\int f'f'_*\, C_i(\xi,\xi_*,\xi',\xi'_*)\,d\sigma$, $i=1,2$, where $C_1 = |(\xi'+\xi'_*)- (\xi+\xi_*)|^2$ and $C_2 = |(\xi'-\xi'_*)-(\xi-\xi_*)|^2_{k,\sigma}$ (the latter expression depends also on $k,\sigma$) and conservation of momentum reads $C_1(v,v_*,v',v'_*)=0$ while in view of \eqref{v2ksigma} conservation of energy reads $C_2(v,v_*,v',v'_*)=0$. In particular, if $f$ were Maxwellian, then $\xi= av+b$ for some $a>0$, $b\in\R^d$, and the integrands in those integral expressions would both vanish.

\bul The last term in the right-hand side of \eqref{Gammaff} combines the non-Maxwellianity of $f$ (that is, by how much $f'f'_*-ff_*$ is nonzero) and the non-Maxwellianity of $B$ (that is, by how much $B$ depends on $|z|$).
\end{Rk}

Now let us prove the basic computational lemmas.

\begin{proof}[Proof of Lemma \ref{lemBL1}]
By definition, 
\[ F'= F\left( \frac{v+v_*}2 + \frac{|v-v_*|}2 \sigma, \frac{v+v_*}2 - \frac{|v-v_*|}2\sigma\right),\]
so with $\nabla=\nabla_v$, $\nabla_*=\nabla_{v_*}$ we have
\[ \nabla (F') = \frac12\bigl(\nabla F + k (\sigma\cdot\nabla F)\bigr) (V')
+ \frac12\bigl(\nabla_*F - k(\sigma\cdot\nabla_* F)\bigr) (V'),\]
and symmetrically
\[ \nabla_* (F') = \frac12 \bigl(\nabla F - k (\sigma\cdot\nabla F)\bigr)  (V')
+ \frac12\bigl(\nabla_*F + k(\sigma\cdot\nabla_* F)\bigr)(V').\]
\end{proof}

\begin{proof}[Proof of Lemma \ref{lemBL2}]
By assumption,
\[ B = B\left( |v-v_*|, \frac{v-v_*}{|v-v_*|}\cdot\sigma\right),\]
so writing this as a function of two variables, $B(|z|,\cos\theta)$, and using chain rule,
\[ \nabla B = \frac{v-v_*}{|v-v_*|}\,\derpar{B}{|z|}
+ \frac1{|v-v_*|}\left(\sigma - \left(\frac{v-v_*}{|v-v_*|}\cdot\sigma\right) \frac{v-v_*}{|v-v_*|} \right)
\derpar{B}{(\cos\theta)},\]
and $\nabla_* B = - \nabla B$.
\end{proof}

\begin{proof}[Proof of Lemma \ref{lemBL3}]
First, let $u:\R^d\to\R$ and $\varphi=\varphi(r)$ concentrating on $r=1$ in such a way that $\int \varphi(|x|)\,dx=1$. By polar change of coordinates $x=r\sigma$,
\begin{align*}
\int \nabla u(x)\,\varphi(|x|)\,dx 
& = \iint \nabla u(r\sigma)\,\varphi(r)\,r^{d-1}\,dr\,d\sigma \\
& \simeq \left(\int \nabla u(\sigma)\,d\sigma\right)\left(\int\varphi(r)\,r^{d-1}\,dr\right).
\end{align*}
On the other hand, $\nabla(u\varphi) = \nabla u(x)\,\varphi(r)+ u\sigma\varphi'(r)$,
so
\begin{align*}
\int \nabla u(x)\,\varphi(|x|)\,dx
& = - \iint u(r\sigma) \sigma\varphi'(r)\,r^{d-1}\,dr\,d\sigma\\
& \simeq -\left( \int u(\sigma)\,\sigma\,d\sigma\right) \left(\int \varphi'(r)\,r^{d-1}\,dr\right)\\
& = \left(\int u(\sigma)\,\sigma\,d\sigma\right) (d-1) \left(\int \varphi(r)\,r^{d-2}\,dr\right)\\
& \simeq  \left(\int u(\sigma)\,\sigma\,d\sigma\right) (d-1) \left(\int \varphi(r)\,r^{d-1}\,dr\right),
\end{align*}
where I used again integration by parts and the concentration of $\varphi$ on $(r=1)$.
The conclusion of this first step is
\begeq\label{intnablau}
\int_{\S^{d-1}} \nabla u(\sigma)\,d\sigma = (d-1) \int_{\S^{d-1}} u(\sigma)\,\sigma\,d\sigma.
\endeq

This is true for any $u:\R^d\to\R$, so also for any vector-valued function $\xi:\R^d\to\R^d$, so
\begeq\label{intnablaxi}
\int_{\S^{d-1}} \nabla\xi(\sigma)\,d\sigma = (d-1) \int_{\S^{d-1}} \xi(\sigma)\otimes \sigma\,d\sigma.
\endeq
As corollaries, here are some formulas of integration by parts on the sphere:
For any functions $u$, $v$ and vector field $\xi$ defined in a neigborhood of $\S^{d-1}$,
for any $k\in\S^{d-1}$,
\begeq\label{ippS1}
 \int_{\S^{d-1}} u(\sigma)\,\nabla v(\sigma)\,d\sigma
 = - \int_{\S^{d-1}} \nabla u(\sigma) \, v(\sigma)\,d\sigma
 + (d-1) \int_{\S^{d-1}} u(\sigma)\,v(\sigma)\,\sigma\,d\sigma;
 \endeq
 \begeq\label{ippS2}
 \int_{\S^{d-1}} u(\sigma)\, [k\cdot\nabla v(\sigma)]\,d\sigma
 = - \int_{\S^{d-1}} k\cdot\nabla u(\sigma)\,v(\sigma)\,d\sigma 
 + (d-1) \int_{\S^{d-1}} u(\sigma)\,v(\sigma)\,(k\cdot\sigma)\,d\sigma
 \endeq
 \begeq\label{ippS3}
 \int_{\S^{d-1}} \xi(\sigma)\,[k\cdot\nabla v(\sigma)]\,d\sigma
 = -\int_{\S^{d-1}}[k\cdot\nabla\xi(\sigma)]\,v(\sigma)\,d\sigma
 +(d-1) \int_{\S^{d-1}} \xi(\sigma)\,v(\sigma)\,(k\cdot\sigma)\,d\sigma.
 \endeq
 (Recall $(k\cdot\nabla u)_i = (k\cdot\nabla)u_i$ by convention.)
 The next preliminary is the elementary formula
 \begeq\label{pisigma}
 \Pi_{k^\bot}\sigma = (\Pi_{\sigma^\bot}k\cdot k)\sigma - (k\cdot\sigma) \Pi_{\sigma^\bot} k,
 \endeq
 indeed the right-hand side is $[k-(k\cdot\sigma)\sigma]\cdot k \sigma - (k\cdot\sigma)(k-(k\cdot\sigma)\sigma) = \sigma- (k\cdot\sigma)^2 \sigma - (k\cdot\sigma)k + (k\cdot\sigma)^2\sigma
 = \sigma - (\sigma\cdot k)k$.
 
 It follows from \eqref{pisigma} that
 \begin{multline}\label{lhsbks}
 \int_{\S^{d-1}} b'(k\cdot\sigma) \Pi_{k^\bot} \sigma\,g(\sigma)\,d\sigma
 = \int b'(k\cdot\sigma) \bigl(\Pi_{\sigma^\bot}k\cdot k\bigr)\,\sigma\,g(\sigma)\,d\sigma \\
 - \int (k\cdot\sigma)\,b'(k\cdot\sigma)\, \Pi_{\sigma^\bot}k\,g(\sigma)\,d\sigma.
 \end{multline}
To transform the first integral on the right hand side, choose $\xi(x) = \frac{x}{|x|}g(x)$, $v(x) = b(k\cdot\frac{x}{|x|})$, apply \eqref{ippS3} and $(k\cdot\nabla g) \sigma = (\sigma\otimes \nabla g(\sigma))k$ to obtain
\begeq\label{afirstrhs}
-\int \bigl(\Pi_{\sigma^\bot}k \,g(\sigma) + (\sigma\otimes\nabla g(\sigma)\bigr)\,k\, b(k\cdot\sigma)\,d\sigma
+ (d-1) \int g(\sigma) b(k\cdot\sigma)\,(k\cdot\sigma)\,\sigma\,d\sigma.
\endeq
The second integral on the right hand side of \eqref{lhsbks} is handled by setting $b_1(s) = s b(s)$, so $sb'(s) = b'_1(s) -b(s)$, and letting $v(x) = b_1(k\cdot\frac{x}{|x|})$, so $\nabla v(x) = b'_1(k\cdot \frac{x}{|x|}) \Pi_{\sigma^\bot} k$; then letting $u(x)=g(x)$ and applying \eqref{ippS2}:
\begin{multline*}
\int (k\cdot\sigma) b'(k\cdot\sigma)\,\Pi_{\sigma^\bot}k\,g(\sigma)\,d\sigma 
 = \int \bigl[b'_1(k\cdot\sigma) - b(k\cdot\sigma)\bigr]\Pi_{\sigma^\bot}k\,g(\sigma)\,d\sigma\\
 = - \int \nabla g(\sigma)\,b_1(k\cdot\sigma)\,d\sigma
+ (d-1) \int b_1(k\cdot\sigma)\,g(\sigma)\,\sigma\,d\sigma
- \int b(k\cdot\sigma)\Pi_{\sigma^\bot} k\,g(\sigma)\,d\sigma\\
= - \int (k\cdot\sigma) b(k\cdot\sigma)\,\nabla g(\sigma)\,d\sigma
+ (d-1) \int b(k\cdot\sigma)(k\cdot\sigma)\,g(\sigma)\,\sigma\,d\sigma
- \int b(k\cdot\sigma) \Pi_{\sigma^\bot}k\,g(\sigma)\,d\sigma.
\end{multline*}
Subtracting this from \eqref{afirstrhs}, one sees that \eqref{lhsbks} is equal to
\[ \int_{\S^{d-1}} \bigl[ (k\cdot\sigma)\nabla g(\sigma) - (k\cdot\nabla g(\sigma))\sigma \bigr]\,b(k\cdot\sigma)\,d\sigma, \]
as announced in Lemma \ref{lemBL3}.
\end{proof}
 
\begin{proof}[Proof of Proposition \ref{propnablaBF}] 
By Lemma \ref{lemBL1},
\begin{align*}
\bbnabla ({\cal B}F) 
& = \bbnabla \int_{\S^{d-1}}(F'-F)\,B\,d\sigma \\
& = \int \bigl[ \bbnabla (F') - \bbnabla F\bigr]\,B\,d\sigma 
+ \int (F'-F)\,\bbnabla B\,d\sigma \\
& = \int \bigl[ \A_{\sigma k}(\bbnabla F)'-\bbnabla F\bigr]\,B\,d\sigma
+ \int (F'-F)\, [\nabla B, -\nabla B]\,d\sigma.
\end{align*}
By Lemmas \ref{lemBL2} and \ref{lemBL3},
\begin{align*}
\int (F'-F)\,\nabla B\,d\sigma 
& = \int (F'-F)\,\left( \derpar{B}{|z|}\, k + \frac1{|z|}\derpar{B}{(\cos\theta)}\,\Pi_{k^\bot}\sigma\right)\,d\sigma\\
& = \left(\int (F'-F)\,\derpar{B}{|z|}\,d\sigma \right)k 
+ \frac1{|z|} \int M_{\sigma k}\nabla_\sigma (F'-F)\,B\,d\sigma.
\end{align*}
The second term involves $\nabla_\sigma (F') = (|v-v_*|/2) (\nabla F - \nabla_* F)'$ (here $\sigma$ lives in $\R^d$), so
\[ \int (F'-F)\,\nabla B\,d\sigma
= \left(\int (F'-F)\,\derpar{B}{|z|}\,d\sigma \right)k 
+ \frac12 \int \bigl[ M_{\sigma k}(\nabla F)'- M_{\sigma k}(\nabla_*F)'\bigr]\,B\,d\sigma.\]
All in all, 
\begin{multline*}
\nabla ({\cal B}F)
= \int \Bigl[ \frac12 (\nabla F)'+ \frac12 \sigma\cdot (\nabla F)'k+ \frac12 (\nabla_*F)'
- \frac12\sigma\cdot(\nabla_*F)'k
+ \frac12 (k\cdot\sigma)(\nabla F)'- \frac12 k\cdot(\nabla F)'\sigma \\ 
- \frac12 (k\cdot\sigma) (\nabla_*F)' + \frac12 k\cdot(\nabla_*F)\,\sigma - \nabla F \Bigr]\,B\,d\sigma
+ \left( \int (F'-F)\,\derpar{B}{|z|}\,d\sigma\right)k,
\end{multline*}
and by symmetry there is a similar formula for $\nabla_* ({\cal B}F)$, whence the result.
\end{proof}

\begin{proof}[Proof of Theorem \ref{thmGamma}]
Let us compute one after the other the various terms appearing in $\Gamma_{I,{\cal B}}$:

\[ {\cal B} \bigl( F|\bbnabla \log F|^2\bigr)
= \int \Bigl[ \bigl(F |\bbnabla \log F|^2\bigr)'- \bigl(F |\bbnabla \log F|^2\bigr)\Bigr]\,B\,d\sigma;\]

\[ ({\cal B} F) |\bbnabla \log F|^2
= \int \Bigl( F'|\bbnabla \log F|^2 - F | \bbnabla \log F|^2\Bigr)\,B\,d\sigma;\]

\begin{multline*}
2F \bbnabla \log F\cdot\bbnabla \left(\frac{{\cal B}F}{F}\right) 
= 2 \bbnabla\log F \cdot\bbnabla {\cal B}F
- 2 F \bbnabla \log F\cdot \left({\cal B}F\,\frac{\bbnabla F}{F^2}\right)\\
= 2 \int \bbnabla \log F\cdot \bigl[ \G_{\sigma k} (\bbnabla F)'- \bbnabla F\bigr]\,B\,d\sigma 
- 2 \int \bbnabla \log F \cdot (F'-F) \frac{\bbnabla F}{F}\, B\,d\sigma\\
\qquad\qquad\qquad\qquad\qquad\qquad
+ 2 \left(\int (F'-F) \derpar{B}{|z|}\,d\sigma\right)\, [k,-k]\,\cdot\bbnabla \log F \\
= 2 \int F'\bbnabla \log F \cdot\G_{\sigma k}(\bbnabla \log F)'B\,d\sigma
- 2 \int F'|\bbnabla \log F|^2 B\,d\sigma \\
\qquad\qquad\qquad\qquad
+ 2 \left(\int (F'-F)\,\derpar{B}{|z|}\,d\sigma\right)\,[k,-k] \cdot \bbnabla \log F.
\end{multline*}
All in all,
\begin{multline*}
\Gamma_{I,{\cal B}}(F) = 
\int F'
\Bigl( |\bbnabla\log F|^2+ |(\bbnabla\log F)'|^2 - 2 (\bbnabla\log F)\cdot \G_{\sigma k} (\bbnabla\log F)'
\Bigr)\,B\,d\sigma \\
- 2 \left(\int (F'-F)\,\derpar{B}{|z|}\,d\sigma\right)\, [k,-k]\cdot\bbnabla \log F,
\end{multline*}
which is the same as \eqref{fmlGamma} upon noting that $(P_{\sigma k})^* = P_{k\sigma}$ (Proposition \ref{propPsk}(iii) below). 

Now let us turn to \eqref{Gammaff}. Plug $F=f\otimes f=ff_*$ into \eqref{fmlGamma}. Then
\[ \bigl[ \bbnabla\log F, (\bbnabla \log F)\bigr]
= - \bigl[ \xi, \xi_*, \xi', \xi'_*\bigr].\]
Then the integrand in the integral expression of $\Gamma$ will be $f'f'_*B$ multiplied by
\begin{multline*}
|\xi|^2 + |\xi_*|^2 + |\xi'|^2 + |\xi'_*|^2
- \bigl\< (I+P_{k\sigma})\xi,\xi'\bigr\> - \bigl\< (I-P_{k\sigma})\xi,\xi'_*\bigr\>
- \bigl\< (I-P_{k\sigma})\xi_*,\xi'\bigr\> - \bigl\< (I+P_{k\sigma})\xi_*,\xi'_*\bigr\> \\
= |\xi|^2 + |\xi_*|^2 + |\xi'|^2 + |\xi'_*|^2
- \bigl( \xi'\cdot\xi + \xi\cdot\xi'_* + \xi_*\cdot\xi'+ \xi_*\cdot\xi'_*\bigr)
- P_{k\sigma}\xi\cdot\xi' + P_{k\sigma}\xi\cdot\xi'_* + P_{k\sigma}\xi_*\cdot\xi'
- P_{k\sigma}\xi_*\cdot\xi'_*\\
= \frac12 \Bigl| (\xi'+\xi'_*) - (\xi+\xi_*)\Bigr|^2
+\frac12 \Bigl( |\xi'-\xi'_*|^2 + |\xi-\xi_*|^2 - 2 P_{k\sigma}(\xi-\xi_*)\cdot (\xi'-\xi'_*)\Bigr),
\end{multline*}
which yields the desired expression.
\end{proof}

I end up this section with the promised list of properties for $P_{\sigma k}$. (Actually for future use it will be more convenient to swap $\sigma$ and $k$.)

\begin{Prop} \label{propPsk}
For any $k,\sigma$ in $\S^{d-1}$, let $P_{k\sigma}(x) = (k\cdot\sigma)x + (x\cdot k)\sigma - (\sigma\cdot x)k$, or equivalently $P_{k\sigma} = (k\cdot\sigma) I  + \sigma\otimes k - k\otimes \sigma$. Then

(0) $P_{(-k) \sigma} = P_{k (-\sigma)} = - P_{k \sigma}$;

(i) $P_{k k} = \Id$;

(ii) $P_{k\sigma} k = \sigma$, \quad $P_{k\sigma}(k^\bot) \subset \sigma^\bot$;

(iii) $P_{k\sigma}^* = P_{\sigma k}$;

(iv) $(P_{k\sigma} \sigma)\cdot k = 2 (k\cdot\sigma)^2 -1 = (P_{\sigma k}k)\cdot\sigma$;

(v) If $k\neq \pm\sigma$ then the restriction of $P_{k\sigma}$ to $\vect (k,\sigma)$, the plane generated by $k$ and $\sigma$, is the rotation sending $k$ onto $\sigma$, and its restriction to $\vect(k,\sigma)^\bot$ is the multiplication by $k\cdot\sigma$;
\sm

\noindent Moreover, for all $x\in\R^d$, 

(vi) If $\Pi_{\{k,\sigma\}}^\bot$ stands for the orthogonal projection on $\vect(k,\sigma)^\bot$, then
\begin{align*} |x|^2 - |P_{k\sigma}x|^2 
& = \bigl[1-(k\cdot\sigma)^2\bigr] \bigl| \Pi_{\{k,\sigma\}}^\bot x\bigr|^2\\
& = \bigl(1-(k\cdot\sigma)^2\bigr)|x|^2 - (x\cdot k)^2 - (x\cdot\sigma)^2
+ 2 (k\cdot\sigma) (x\cdot k) (x\cdot\sigma)\\
& = |\Pi_{k^\bot}x|^2 - |P_{k\sigma} \Pi_{k^\bot}x|^2, 
\end{align*}

(vii) $|P_{k\sigma} x| \leq |x|$ with equality only if $x\in\vect(k,\sigma)$ or $k=\pm\sigma$;

(viii) $P_{k\sigma} x = 0$ if and only if $(k,\sigma,x)$ are pairwise orthogonal;

(ix) $\sigma\cdot (P_{k\sigma} x) = k\cdot x$;
\sm

(x) In addition, for all $x,y \in \R^d$ and with the notation $|y-x|_{k,\sigma}^2 = |x|^2 + |y|^2 - 2 (P_{k \sigma}x)\cdot y$,\\
$|y-x|_{k,\sigma}^2 = \bigl| \Pi_{\sigma^\bot}y - \Pi_{k^\bot}x\bigr|_{k,\sigma}^2
+ (y\cdot\sigma-x\cdot k)^2$,\\ 
$\qquad \dps |y-x|_{k,\sigma}^2 = 
( |x|^2 - |P_{k\sigma}x|^2) + |y-P_{k\sigma}x|^2 \geq |y-P_{k\sigma}x|^2,$\\
$\qquad|y-P_{k\sigma}x|^2 = \bigl| \Pi_{\sigma^\bot}y - P_{k\sigma} \Pi_{k^\bot}x\bigr|^2 + (y\cdot \sigma-x\cdot k)^2$,\\
$\qquad\dps |y-x|_{k,\sigma}^2 = \bigl| \Pi_{\sigma^\bot}y - P_{k\sigma} \Pi_{k^\bot}x\bigr|^2 + 
( |x|^2 - |P_{k\sigma}x|^2) + (y\cdot \sigma-x\cdot k)^2$,\\
$\qquad \dps |y-x|_{k,\sigma}^2 = |x-y|_{\sigma,k}^2$.
\end{Prop}

\begin{proof}[Proof of Proposition \ref{propPsk}]
(0), (i), (ii), (iii), (iv) are immediate. 

To prove (v), let $k^\bot$ and $\sigma^\bot$ be obtained from $k$ and $\sigma$ by the action of the rotation of angle $\pi/2$ in $P=\vect(k,\sigma)$. (Choose an arbitrary orientation of $P$.) Then $x=x_P + x_\bot$, where $x_P$ is the projection of $x$ onto $P$ and $x_\bot$ on its orthogonal, so that $x_\bot \bot k,\sigma,k^\bot$. Then we have the orthogonal decomposition
\[ x = x_\bot + (x\cdot k)k + (x\cdot k^\bot)k^\bot,
\qquad |x|^2 = |x_\bot|^2 + (x\cdot k)^2 + (x\cdot k^\bot)^2,\]
and
\begin{align*} 
P_{k\sigma} x 
& = P_{k\sigma}x_\bot + (x\cdot k) P_{k\sigma} k + (x\cdot k^\bot) P_{k\sigma} k^\bot\\
& = (k\cdot\sigma) x_\bot + (x\cdot k)\sigma + (x\cdot k^\bot)P_{k\sigma}k^\bot\\
& = (k\cdot\sigma)x_\bot + (x\cdot k)\sigma
+ (x\cdot k^\bot) \bigl[ (k\cdot \sigma)k^\bot - (\sigma\cdot k^\bot)k\bigr]\\
& = (k\cdot\sigma) x_\bot + (x\cdot k)\sigma + (x\cdot k^\bot)\sigma^\bot.
\end{align*}

This implies (v), the first equality of (vi), (vii) and (viii). As for the second equality of (vi), write
\begin{multline*}
|x|^2 - |P_{k\sigma}x|^2
= |x|^2 - \Bigl( (k\cdot\sigma)^2 |x|^2 + (k\cdot x)^2 +(\sigma\cdot x)^2 
+ 2 (k\cdot\sigma) (x\cdot k) (\sigma\cdot x) \\ 
\qquad\qquad\qquad\qquad\qquad - 2 (k\cdot \sigma) (k\cdot x) (\sigma\cdot x)
- 2 (k\cdot\sigma) (\sigma\cdot x)(x\cdot k) \Bigr)\\
= |x|^2 - \Bigl( (k\cdot\sigma)^2|x|^2 + (k\cdot x)^2 + (\sigma\cdot x)^2- 2 (k\cdot\sigma)(\sigma\cdot x) (x\cdot k)\Bigr).
\end{multline*}
Then to prove the third (and last) equality of (vi), insert the decomposition $x = \Pi_{k^\bot}x + (k\cdot x)k$ into the previous formula for $|x|^2 - |P_{k\sigma}x|^2$ and see that it simplifies into $(1-(k\cdot\sigma)^2)|\Pi_{k^\bot}x|^2- ((\Pi_{k^\bot}x)\cdot\sigma)^2$, which is the same as $|\Pi_{k^\bot}x|^2 - |P_{k\sigma}\Pi_{k^\bot}x|^2$.

Identity (ix) is immediate. It remains to prove (x).
Start from
\begeq\label{yyxx} y = \Pi_{\sigma^\bot} y + (\sigma\cdot y) \sigma,\qquad 
x = \Pi_{k^\bot}x + (k\cdot x)k,
\endeq
and deduce
\begin{multline} \label{intermx}
|y-x|_{k,\sigma}^2= |x|^2 + |y|^2 - 2 P_{k\sigma}x \cdot y \\
= |\Pi_{k^\bot}x|^2 + (x\cdot k)^2 + |\Pi_{\sigma^\bot}y|^2 + (y\cdot \sigma)^2
- 2 (P_{k\sigma}\Pi_{k^\bot}x)\cdot\Pi_{\sigma^\bot}y 
- 2 (P_{k\sigma}k)\cdot (x\cdot k) \Pi_{\sigma^\bot}y
\\
- 2 P_{k\sigma} (\Pi_{k^\bot}x)\cdot (y\cdot \sigma)\sigma
- 2 P_{k\sigma}k \cdot (x\cdot k) (y\cdot \sigma)\sigma.
\end{multline}
Now, 
\[ (P_{k\sigma} k)\cdot (x\cdot k) \Pi_{\sigma^\bot}y = (x\cdot k) \sigma\cdot\Pi_{\sigma^\bot} y = 0,\]
\[ P_{k\sigma} (\Pi_{k^\bot}x)\cdot (y\cdot \sigma) \sigma = (y\cdot\sigma) (\Pi_{k^\bot}x)\cdot P_{\sigma k} \sigma = (y\cdot\sigma) (\Pi_{k^\bot}x)\cdot k =0,\]
\[ P_{k\sigma} k\cdot(x\cdot k) (y\cdot\sigma) \sigma = (x\cdot k) (y\cdot \sigma).\]
All of this reduces \eqref{intermx} to
\begin{multline*} 
|\Pi_{k^\bot}x|^2 + |\Pi_{\sigma^\bot}y|^2 - 2 (P_{k\sigma} \Pi_{k^\bot} x)\cdot \Pi_{\sigma^\bot} y
+ (x\cdot k)^2 + (y\cdot \sigma)^2 - 2 (x\cdot k) (y\cdot \sigma)\\
 =
\bigl|\Pi_{k^\bot}x - \Pi_{\sigma^\bot}y\bigr|_{k,\sigma}^2 + (x\cdot k-y\cdot\sigma)^2,
\end{multline*}
which yields the first part of (x).

The second part of (x) is immediate from algebra and (vii).

For the third part of (x), using \eqref{yyxx} again and (ii),
\begin{align*}
|y-P_{k\sigma}x|^2
& = \Bigl| \bigl( \Pi_{\sigma^\bot}y + (y\cdot\sigma)\sigma\bigr)
- P_{k\sigma} \bigl(\Pi_{k^\bot} x + (x\cdot k)k\bigr)\Bigr|^2 \\
& = \Bigl| \bigl( \Pi_{\sigma^\bot}y- P_{k\sigma}\Pi_{k^\bot}x\bigr) + (y\cdot \sigma-x\cdot k)\sigma\Bigr|^2\\
& = \bigl| \Pi_{\sigma^\bot} y - P_{k\sigma} \Pi_{k^\bot}x\bigr|^2
+ (y\cdot\sigma - x\cdot k)^2.
\end{align*}
Then, starting again from
\[ |y-P_{k\sigma} x|^2 = |y-x|^2_{k,\sigma} - \bigl( |x|^2 - |P_{k\sigma}x|^2\bigr) \]
and likewise
\[ \bigl| \Pi_{\sigma^\bot}y - P_{k\sigma}\Pi_{k^\bot} x\bigr|^2 
= \bigl| \Pi_{\sigma^\bot} y - \Pi_{k^\bot}x\bigr|^2_{k,\sigma} 
- \bigl( |\Pi_{k^\bot}x|^2 - |P_{k\sigma}\Pi_{k^\bot}x|^2\bigr),\] 
we see that the first and third parts of (x) are equivalent provided that
\begeq\label{xPksx}
|x|^2 - |P_{k\sigma}x|^2 = |\Pi_{k^\bot}x|^2 - \bigl|P_{k\sigma}\Pi_{k^\bot}x\bigr|^2.
\endeq
Here is a direct proof of \eqref{xPksx}: Write $x=\Pi_{k^\bot}x + (x\cdot k)k$, then
\begin{align*}
|x|^2 - |P_{k\sigma}x|^2
& = | \Pi_{k^\bot}x|^2 + (x\cdot k)^2 - |P_{k\sigma} \Pi_{k^\bot}x|^2
- |(x\cdot k) P_{k\sigma}k|^2
- 2 (k\cdot\sigma) (P_{k\sigma}\Pi_{k^\bot}x)\cdot (P_{k\sigma}k)\\
& = |\Pi_{k^\bot}x|^2 + (x\cdot k)^2 - |P_{k\sigma}\Pi_{k^\bot}x|^2 - (x\cdot k)^2
- 2 (k\cdot \sigma) (P_{k\sigma}\Pi_{k^\bot}x)\cdot \sigma\\
& = |\Pi_{k^\bot}x|^2 - |P_{k\sigma}\Pi_{k^\bot}x|^2
- 2(k\cdot\sigma) (\Pi_{k^\bot}x\cdot P_{\sigma k}\sigma)
\end{align*}
and the last term is 0 since $P_{\sigma k}\sigma =k$. This concludes the proof of \eqref{xPksx} and the but-to-last bit of (x) as well. The final bit of (x) follows at once from (iii).
\end{proof}

\bibnotes

The {\em carr\'e du champ it\'er\'e} $\Gamma_2$ appears in Bakry--\'Emery \cite{bakem:hyperc:85} and many further works by these authors and the research community which has developed from there. As for the {\em carr\'e du champ} itself, it is present in innumerable works.

In the context of kinetic theory, the local nature of $\Gamma$ it is exploited in \cite{vill:cer:03} to establish an entropic gap.

Lemmas \ref{lemBL1} to \ref{lemBL3} and Proposition \ref{propnablaBF} are reformulations of my older work on the Fisher information for Boltzmann equation \cite{vill:fisher}, while Corollary \ref{cormaxw} captures the main result of that work. The proof of Lemma \ref{lemBL3} given here is more intrinsic than in \cite{vill:fisher}, but there is certainly an even more synthetic formulation in the language of intrinsic differential geometry. Of course the main novelty in this section is the more precise calculation leading to Theorem \ref{thmGamma}, allowing $B$ to depend on $|z|$.

Expressions like $(F+G) | \nabla\log F - \nabla\log G|^2$, underlying the estimates in this section, also play a key role in my older estimates, with Toscani, of Boltzmann's dissipation functional \cite{TV:entropy:99,vill:cer:03}. This coincidence is not surprising: As noticed in \cite[Section 8]{TV:entropy:99}, for Maxwellian kernels the dissipation of $D_B$ by the Fokker--Planck equation coincides, by commutation of the flows, with the dissipation of $I$ by the Boltzmann equation.

\section{Qualitative discussion} \label{secqualit}

The intricate nature of $\Gamma_{I,{\cal B}}$ is the result of the interaction of three ``geometries': tensorisation, collision spheres, and Fisher information. Integration, Proposition \ref{propdoubling} and Theorem \ref{thmGamma} yield
\[ -I'(f)\cdot Q(f,f) = \frac12 \int \Gamma_{I,{\cal B}}(f\otimes f)\,dv\,dv_* 
= \cI + \cII + \cIII, \]
where
\begeq\label{c1}
\cI = \frac14 \iiint f'f'_* \bigl| (\xi'+\xi'_*) - (\xi+\xi_*)\bigr|^2\, B\,d\sigma\,dv\,dv_*
\endeq
\begeq\label{c2}
\cII = \frac14 \iiint f'f'_* \bigl| (\xi'-\xi'_*) - (\xi-\xi_*)\bigr|^2_{k,\sigma}\,B\,d\sigma\,dv\,dv_*
\endeq
\begin{align} \label{c3}
\cIII & = \iiint (f'f'_*- ff_*) \, k\cdot(\xi-\xi_*) \,\derpar{B}{|z|} \,d\sigma\,dv\,dv_* \\
& \nonumber = \frac12 \iiint (f'f'_*- ff_*) \bigl[ k\cdot(\xi-\xi_*)-\sigma\cdot(\xi'-\xi'_*)\bigr]
\,\derpar{B}{|z|}\,d\sigma\,dv\,dv_*.
\end{align}
By symmetry, in $\cI$ and $\cII$ one may replace $f'f'_*$ by $ff_*$ or $(f'f'_*+ff_*)/2$. In a way $\cI$ is associated with momentum conservation ($(v'+v'_*) - (v+v_*)=0$), $\cII$ with energy conservation, keeping collisions on the sphere ($|v'-v'_*| = |v-v_*|$ implies $|(v'-v'_*)-(v-v_*)|_{k,\sigma}^2 =0$).

Now I will choose to write $\cII$ with $ff_*$ in the integrand, and decompose, according to Proposition \ref{propPsk}(x),
\[ |y-x|^2_{k,\sigma} =
 \bigl| \Pi_{\sigma^\bot}y - P_{k\sigma}\Pi_{k^\bot}x\bigr|^2
+ \bigl( |x|^2 - |P_{k\sigma}x|^2\bigr)
+ (y\cdot\sigma - x\cdot k)^2 \]
with $y= (\xi'-\xi'_*)$ and $x=(\xi-\xi_*)$.
This yields
\[ \cII = \cII_1 + \cII_2 + \cII_3,\]
where
\begeq\label{cII1}
\cII_1 = \frac14 \iiint ff_* \Bigl| \Pi_{\sigma^\bot} (\xi'-\xi'_*) - P_{k\sigma} \Pi_{k^\bot}(\xi-\xi_*)\Bigr|^2\,B\,d\sigma\,dv\,dv_*
\endeq
\begeq\label{cII2}
\cII_2 = \frac14 \iiint ff_* \Bigl( |\xi-\xi_*|^2 - |P_{k\sigma}(\xi-\xi_*)|^2 \Bigr)\,B\,d\sigma\,dv\,dv_*
\endeq
\begeq\label{cII3}
\cII_3 = \frac14 \iiint ff_* \Bigl [ (\xi'-\xi'_*)\cdot\sigma - (\xi-\xi_*)\cdot k \Bigr]^2 \, B\,d\sigma\,dv\,dv_*.
\endeq
To summarise: four nonnegative terms, $\cI$, $\cII_1$, $\cII_2$, $\cII_3$, and one unsigned term $\cIII$. The monotonicity property (or not) of $I$ will depend on the possibility to control the unsigned term by the nonnegative ones. Among the latter, $\cII_2$ may be simplified using the following lemma.

\begin{Lem} \label{lemcurv}
If $X \in\R^d$ and $k\in \S^{d-1}$ are given, then
\[ \int_{\S^{d-1}} \bigl( |X|^2 - |P_{k\sigma}X|^2\bigr)\,B\,d\sigma
= \left(\frac{d-2}{d-1}\right) \left( \int_{\S^{d-1}} \bigl[1-(k\cdot\sigma)^2\bigr]\,B\,d\sigma\right) \bigl|\Pi_{k^\bot}X\bigr|^2.\]
\end{Lem}

\begin{Rk} 
$|\int [1-(k\cdot\sigma)^2]\,B\,d\sigma \leq 2 \int (1-k\cdot\sigma)\,B\,d\sigma = 2 M(z)$ is finite even when $B$ is nonintegrable.
\end{Rk}

\begin{proof}[Proof of Lemma \ref{lemcurv}] Write $\cos\theta = k\cdot\sigma$.
By Proposition \ref{propPsk}(vi),
\begin{align*}
|X|^2 - |P_{k\sigma}X|^2 
& = \bigl[ 1-(k\cdot\sigma)^2\bigr] |X|^2 - (k\cdot X)^2 - (\sigma\cdot X)^2
+ 2 (k\cdot X) (\sigma\cdot X)(k\cdot\sigma)\\
& = \sin^2\theta |X|^2 - (k\cdot X)^2 - (\sigma\cdot X)^2 + 2 \cos\theta (k\cdot X)(\sigma\cdot X).
\end{align*}
Decompose $\sigma = (\cos\theta)k + (\sin\theta)\phi$, $\phi\in \S^{d-2}_{k^\bot}$. Then
\begin{multline*} |X|^2 - |P_{k\sigma}X|^2 \\
= \sin^2 \theta |X|^2 - (k\cdot X)^2
- \cos^2\theta (k\cdot X)^2 - \sin^2 \theta (\phi\cdot X)^2 - 2 (\cos\theta)(\sin\theta) (k\cdot X)(\phi\cdot X)\\
+ 2 \cos^2 \theta (k\cdot X)^2 + 2\cos\theta \sin\theta (k\cdot X) (\phi\cdot X).
\end{multline*}
By Lemma \ref{lemav},
\begin{multline*} \frac1{|\S^{d-2}|} 
\int_{S^{d-2}_{k^\bot}} \bigl( |X|^2 - |P_{k\sigma} X|^2\bigr)\,d\phi \\
= \sin^2\theta |X|^2 - (k\cdot X)^2 - \cos^2\theta (k\cdot X)^2 - \frac{\sin^2\theta}{d-1} |\Pi_{k^\bot}X|^2
+ 2\cos^2\theta (k\cdot X)^2.
\end{multline*}
(Note that $\int \phi\cdot X\,d\phi = 0$ by symmetry.) All in all, using again $|X|^2 = (k\cdot X)^2 + |\Pi_{k^\bot}X|^2$,
\begin{align*} & \frac1{|\S^{d-2}|} 
\int_{S^{d-2}_{k^\bot}} \bigl( |X|^2 - |P_{k\sigma} X|^2\bigr)\,d\phi \\
& = (k\cdot X)^2 \bigl(\sin^2\theta-1+\cos^2\theta\bigr) + | \Pi_{k^\bot}X|^2 \left(\sin^2\theta - \frac{\sin^2\theta}{d-1}\right) \\
& = \left(\frac{d-2}{d-1}\right) \sin^2\theta\, |\Pi_{k^\bot}X|^2,
\end{align*}
and the result follows.
\end{proof}

Lemma \ref{lemcurv} implies the following simplified formula for $\cII_2$:
\begeq\label{cII2'}
\cII_2 = \frac14\left(\frac{d-2}{d-1}\right)
\iint ff_*\, \Sigma(|v-v_*|)\, \left| \Pi_{k^\bot} \left(\frac{\nabla f}{f} - \Bigl(\frac{\nabla f}{f}\Bigr)_*\right) \right|^2\,dv\,dv_*,
\endeq
where
\begeq\label{Sigma}
\Sigma(|z|) = \int_{\S^{d-1}} \sin^2\theta\,B(z,\sigma)\,d\sigma
= \int_{\S^{d-1}} [1-(k\cdot\sigma)^2]\,B(z,\sigma)\,d\sigma.
\endeq
This formula resembles Landau's dissipation functional, and suggests that a better intuition for the various terms in $-I'(f)\,Q(f,f)$ will be obtained by first going to the asymptotics of grazing collisions. By Proposition \ref{propLandau} this will also amount to considering the problem of monotonicity of Fisher information along the spatially homogeneous Landau equation. In the Proposition below,
$\|M\|_{\HS}^2 = \tr (M^*M)$ will stand for the Hilbert--Schmidt square norm of a matrix $M$ (nothing to do with Hard Spheres...).

\begin{Prop}[Asymptotics of grazing collisions for the Fisher information dissipation]
\label{propAGCFI}
With the notation
\[  \xi= -\nabla \log f,\qquad A = - \nabla^2\log f, \qquad \Psi = \frac{M_\infty(v-v_*)\, |v-v_*|^2}{4(d-1)}\]
as in Proposition \ref{propLandau}, one has the following limits in the AGC:
\[ \cI \xrightarrow[AGC]{} \frac12 \iint ff_* \Psi\, \bigl\| (A-A_*)\Pi_{k^\bot}\bigr\|_{\HS}^2\,dv\,dv_*\]
\[ \cII_1 \xrightarrow[AGC]{} \frac12\iint ff_* \Psi\,
\left\|\left( (A+A_*) - 2k\cdot \frac{(\xi-\xi_*)}{|v-v_*|} I\right)\,\Pi_{k^\bot} \right\|_{\HS}^2\,dv\,dv_*\]
\[ \cII_2 \xrightarrow[AGC]{} 2(d-2) \iint ff_* \frac{\Psi\,}{|v-v_*|^2}\, \bigl| \Pi_{k^\bot} (\xi-\xi_*)\bigr|^2\,dv\,dv_*\]
\[ \cII_3 \xrightarrow[AGC]{} \frac12\iint ff_* \Psi\,
\left| \Pi_{k^\bot} \left( (A+A_*)k + 2 \frac{\xi-\xi_*}{|v-v_*|}\right)\right|^2\,dv\,dv_*\]
\[ \cIII \xrightarrow[AGC]{} - \iint ff_* \left(\Psi'- \frac{2\Psi}{|v-v_*|}\right)
\left\< \Pi_{k^\bot}(\xi-\xi_*),\Pi_{k^\bot} \left[ (A+A_*)k + 2\frac{(\xi-\xi_*)}{|v-v_*|}\right]\right\>\,dv\,dv_*.
\]
(where $\Psi'$ is just the usual derivative of the function $\Psi(r)$ applied to $r=|v-v_*|$).
\end{Prop}

Before giving the proof, here are some comments on these various terms. At several places, one computes $\|M\Pi_{k^\bot}\|_\HS^2 = \tr (\Pi_{k^\bot} M^*M \Pi_{k^\bot})$, in particular with $M=A+A_*$; note, this is larger than $\|\Pi_{k^\bot} M\Pi_{k^\bot}\|_{HS}^2$, the Hilbert--Schmidt norm of the operator induced on $k^\bot$ by $M$ (because $A+A_*$ acting on $k^\bot$ in general has a component in the $k$ direction). Terms $\cI$ and $\cII_1$ take the form of an integrated squared $\log$ Hessian: roughly speaking they look like $\int \|(\nabla^2-\nabla^2_*) \log (ff_*)\|^2_\HS \, ff_*\,dv\,dv_*$ and $\int \|(\nabla^2+\nabla^2_*) \log (ff_*)\|^2_\HS\,ff_*\,dv\,dv_*$, respectively; some kind of (weighted) tensorised version of the familiar term $\int \|\nabla^2 \log f/f_\infty \|^2_\HS\,f$ which appears in the theory of log Sobolev inequalities. And pursuing this analogy, the term $\cII_2$, which looks something like $(d-2)\int |(\nabla-\nabla_*)\log (ff_*)|^2\,ff_*\,dv\,dv_*$, is akin to a {\em curvature} term, something like a tensorised version of $(d-2) \int f |\nabla\log f|^2\,dv$, which is exactly the first order term popping out in Bakry--\'Emery's famous computation of the derivative of the Fisher information along the Fokker--Planck equation. (Note that $d-2$ is exactly the value of the Ricci curvature on the sphere, up to multiplication by the metric of course.) Next, to understand $\cII_3$, consider the following: If $ff_*$ were intrinsincally defined on the collision sphere (not on the whole space), then $(\xi-\xi_*)\cdot k=0$ would hold throughout the whole sphere, and differentiation would yield
\[ \Pi_{k^\bot} (\nabla-\nabla_*) \bigl( (\xi-\xi_*)\cdot k \bigr) = 0, \]
but this is the same as
\[ \Pi_{k^\bot} \left( (A+A_*)k + 2 \frac{\xi-\xi_*}{|v-v_*|}\right) = 0.\]
In other words, $\cII_3$ is a kind of {\em extrinsic term} measuring the extent to which $ff_*$ fails to be preserved along the collision sphere. Finally $\cIII$ is the only ``bad'' term in the monotonicity property, and in it one can appreciate the effect of the non-Maxwellian nature of the interaction (if $\Psi$ is a power law $|z|^{\gamma+2}$ then $|\Psi'-2\Psi/|z||$ is just $|\gamma| |z|^{\gamma+1}$, hence vanishes only for $\gamma=0$), and the amount of variation of $ff_*$ along the collision spheres (recall that a tensor product invariant along collision spheres is a Maxwellian equilibrium).

This intuition behind the various terms in \eqref{fmlGamma} is the main outcome of this section. But also we can be a straightforward derivation of a suboptimal version of Guillen--Silvestre theorem.

\begin{Thm}[A rough version of the Guillen--Silvestre theorem] \label{thmGS0}
Assume $|r\Psi'(r)/\Psi(r) -2| \leq 2\sqrt{d-2}$; then the Fisher information is monotonous along solutions of the spatially homogeneous Landau equation with function $\Psi$.
\end{Thm}

\begin{proof}[Proof of Theorem \ref{thmGS0}]
Let $\ov{\gamma}$ be an upper bound for $|r\Psi'(r)/\Psi(r)-2|$. Then from the asymptotic expressions in Proposition \ref{propAGCFI}
\begin{align*}
| \cIII | & \leq 
\ov{\gamma} \iint ff_* \,\frac{\Psi(|v-v_*|)}{|v-v_*|}
\, \bigl| \Pi_{k^\bot} (\xi-\xi_*) \bigr|\, \Bigl| \Pi_{k^\bot} \Bigl[ (A+A_*)k + 2 \frac{\xi-\xi_*}{|v-v_*|} \Bigr]\Bigr|\,dv\,dv_* \\
& \leq \frac12 \iint ff_*  \Psi \left|\Pi_{k^\bot} \Bigl[ (A+A_*)k + 2 \frac{\xi-\xi_*}{|v-v_*|} \Bigr]\right|^2\,dv\,dv_* \\
& \qquad\qquad\qquad   + \frac{(\ov{\gamma})^2}{2}
\iint ff_*\,\frac{\Psi}{|v-v_*|^2}\, \bigl| \Pi_{k^\bot}(\xi-\xi_*)\bigr|^2\,dv\,dv_*\\
& = \cII_3 + \frac{(\ov{\gamma})^2}{4(d-2)} \cII_2.
\end{align*}
So if $\ov{\gamma}\leq 2\sqrt{d-2}$, this is bounded above by $\cII_3+\cII_2$, and decay follows.
\end{proof}

\begin{Rk} In dimensions 6 and higher this covers already all mathematical and physical cases in which one would reasonably be interested. If one restricts to potentials decaying faster than $d$-dimensional Coulomb, as I chose to do, this even covers them for $d=4,5$. But it misses the range $\gamma=[-3,-2)$ in dimension $d=3$ (which is crucial since it contains the true Landau--Coulomb equation) and it says nothing about $d=2$. It is clear from the proof that to improve this theorem one will need to use the Hessian terms.
\end{Rk}

It remains to sketch the proof of Proposition \ref{propAGCFI}, which is a tricky exercise in Taylor expansions, based on repeated use of \eqref{infv}.

\begin{proof}[Sketch of proof of Proposition \ref{propAGCFI}]
First with $\cI$. From \eqref{37ter} and \eqref{37},
\[ \xi'-\xi = \nabla\xi (v)(v'-v) + o(\theta |v-v_*|)
= \frac{|v-v_*|}2 \nabla\xi(v) \cdot \theta\phi + o (\theta |v-v_*|)\]
and likewise
\[ \xi'_* - \xi_* = - \frac{|v-v_*|}2 \nabla\xi(v_*)\cdot\theta\phi + o(\theta |v-v_*|),\]
so that
\[ \bigl| (\xi'+\xi'_*) - (\xi+\xi_*)\bigr|^2
= \frac{|v-v_*|^2}4 \bigl| (A-A_*) \phi \bigr|^2 \theta^2 + o(\theta^2 |v-v_*|^2).\]
In the AGC, using Lemma \ref{lemav},
\[ \int \bigl| (\xi'+\xi'_*) - (\xi+\xi_*)\bigr|^2\,B\,d\sigma
\longrightarrow \frac{|v-v_*|^2}{4} \frac1{d-1} \tr \bigl[ (A-A_*)\Pi_{k^\bot}\bigr]^2 (2M_\infty),\]
thus
\[ \frac14 \int \bigl| (\xi'+\xi'_*) - (\xi+\xi_*)\bigr|^2\,B\,d\sigma 
\longrightarrow \frac12\Psi(|v-v_*|)\, \bigl\| (A-A_*)\Pi_{k^\bot} \bigr\|_\HS^2.\]

Now for $\cII$. From
\[ \xi'-\xi'_* = (\xi-\xi_*) + \frac{|v-v_*|}2 (A+A_*)\theta\phi + o(\theta |v-v_*|)\]
one obtains, recalling \eqref{37},
\begin{multline} \label{xixixi1} (\xi'-\xi'_*) - P_{k\sigma}(\xi-\xi_*) \\ = 
(1-k\cdot\sigma)(\xi-\xi_*) -k\cdot (\xi-\xi_*)\sigma- \sigma\cdot (\xi-\xi_*) k  \\
\qquad\qquad\qquad\qquad\qquad 
+ \frac{|v-v_*|}2 (A+A_*)\theta\phi + o(\theta |v-v_*|)\\
\qquad\qquad =
 (1-k\cdot\sigma)(\xi-\xi_*)- \bigl[ k\cdot(\xi-\xi_*) (\sigma-k) + k\cdot (\xi-\xi_*)k \bigr]
+ \bigl[ (\sigma-k)\cdot (\xi-\xi_*)k + k\cdot(\xi-\xi_*)k\bigr] \\
\qquad\qquad\qquad\qquad\qquad 
+ \frac{|v-v_*|}2 (A+A_*)\theta\phi + o(\theta |v-v_*|)\\
= -k\cdot (\xi-\xi_*)\theta\phi + \theta\phi\cdot (\xi-\xi_*) k +\frac{|v-v_*|}2 (A+A_*)\phi\theta
+ o(\theta |v-v_*|).
\end{multline}
On the other hand, using \eqref{37bis},
\begin{multline} \label{sxxkxx}
\sigma\cdot (\xi'-\xi'_*) - k\cdot (\xi-\xi_*) \\
= \bigl[ (\cos\theta)k + (\sin\theta)\phi\bigr] \cdot 
\Bigl[ (\xi-\xi_*) + \frac{|v-v_*|}2 (A+A_*) \theta \phi
+ o ( \theta |v-v_*|)\Bigr] - k\cdot (\xi-\xi_*) \\
= - (1-\cos\theta)\, k\cdot (\xi-\xi_*) +\theta\phi\cdot (\xi-\xi_*)
+ \frac{|v-v_*|}2 \Bigl( \theta^2 \< (A+A_*)\phi,\phi\> +
\theta \< (A+A_*)\phi,k\>\Bigr) + o(\theta |v-v_*|)\\
= \frac{|v-v_*|}2 \theta \left( \<(A+A_*)\phi,k\> + 2 \left\< \phi, \frac{\xi-\xi_*}{|v-v_*|} \right\> \right) 
+ o(\theta |v-v_*|),
\end{multline}
combining this with $\sigma = k + O(\theta)$ it follows that
\[ \bigl[\sigma\cdot (\xi'-\xi'_*) - k\cdot (\xi-\xi_*)\bigr]\sigma
= \frac{|v-v_*|}2 \theta \left( \<(A+A_*)\phi,k\> + 2 \left\< \phi, \frac{\xi-\xi_*}{|v-v_*|} \right\> \right) k
+ o(\theta |v-v_*|).\]
But the left-hand side is also $\Pi_\sigma (\xi'-\xi'_*) - P_{k\sigma} \Pi_k(\xi-\xi_*)$.
Subtracting this expression from \eqref{xixixi1} one obtains
\begin{multline} \label{PsbPksxi}
\Pi_{\sigma^\bot} (\xi'-\xi'_*) - P_{k\sigma}\Pi_{k^\bot}(\xi-\xi_*) \\
= \frac{|v-v_*|}2 \theta
\left[ \left( (A+A_*) - 2 k\cdot \frac{(\xi-\xi_*)}{|v-v_*|} I \right) \phi 
- \<(A+A_*)\phi, k\> k\right] + o(\theta |v-v_*|).
\end{multline}
Since $\phi\cdot k=0$, 
\begin{multline} \label{Psxixi2} 
\Bigl| \Pi_{\sigma^\bot} (\xi'-\xi'_*) - P_{k\sigma}\Pi_{k^\bot}(\xi-\xi_*)\Bigr|^2 \\
= \frac{|v-v_*|^2\theta^2}4 
\left( \left| 
 \left( (A+A_*) - 2 k\cdot \frac{(\xi-\xi_*)}{|v-v_*|} I \right) \phi \right|^2 
+ \<(A+A_*)\phi, k\>^2 \right) + o(|v-v_*|\theta) .
\end{multline}
Let $C = (A+A_*)-2k\cdot (\xi-\xi_*)/|v-v_*|$. Note that $\<(A+A_*)\phi,k\> = \<C\phi,k\>$.
Applying Lemma \ref{lemav} to \eqref{Psxixi2},
\begin{align*} &  \frac1{|\S^{d-2}|} \int_{S^{d-2}_{k^\bot}}
\Bigl| \Pi_{\sigma^\bot} (\xi'-\xi'_*) - P_{k\sigma}\Pi_{k^\bot}(\xi-\xi_*)\Bigr|^2\,d\phi \\
& = \frac1{d-1} \| \Pi_{k^\bot}C\Pi_{k^\bot}\|_\HS^2 + \|\Pi_{k^\bot} (Ck)\|^2 + o(\theta^2 |v-v_*|^2)\\
& = \frac1{d-1} \|\Pi_{k^\bot} C\|_\HS^2\ +o(\theta^2 |v-v_*|^2) .
\end{align*}
From this the asymptotic behaviour of $\cII_1$ follows at once.

Handling $\cII_2$: just start from \eqref{Sigma}, note that in the AGC, $\sin^2\theta = (1+\cos\theta)(1-\cos\theta) \simeq 2(1-\cos\theta)$, and apply Lemma \ref{lemcurv}.

Now for $\cII_3$: Starting again from \eqref{sxxkxx},
\[ \Bigl[ \sigma\cdot (\xi'-\xi'_*) - k\cdot (\xi-\xi_*)\Bigr]^2
= \frac{|v-v_*|^2}{4}\theta^2 \left[ \left( (A+A_*)k + 2 \frac{\xi-\xi_*}{|v-v_*|}\right)\cdot\phi\right]^2
+ o(\theta^2 |v-v_*|^2),\]
so by Lemma \ref{lemav} again,
\begin{multline*} \frac1{|\S^{d-2}|}
\int_{\S^{d-2}_{k^\bot}} 
\bigl[\sigma\cdot (\xi'-\xi'_*)-k\cdot (\xi-\xi_*)\bigr]^2\,d\phi \\
= \frac{|v-v_*|^2\theta^2}{4(d-1)}
\left| \Pi_{k^\bot} \left( (A+A_*)k + 2\frac{\xi-\xi_*}{|v-v_*|}\right) \right|^2 + o(\theta^2 |v-v_*|^2),
\end{multline*}
and the asymptotics of $\cII_3$ follows.

Finally $\cIII$. Combine \eqref{sxxkxx} with 
\begin{align*} f'f'_* - ff_* & = \frac{|v-v_*|^2}2 \bigl(f_* \nabla f - f(\nabla f)_*\bigr) \cdot \theta\phi + o(|v-v_*| \theta)\\
& = \frac{|v-v_*|^2}2 (\nabla-\nabla_*) (ff_*) \cdot \theta\phi + o(|v-v_*| \theta),
\end{align*}
to get
\begin{align*}
& \frac1{|\S^{d-2}|} \int_{\S^{d-2}_{k^\bot}}
(f'f'_* - ff_*)\, \bigl[ k\cdot (\xi-\xi_*) - \sigma\cdot (\xi'-\xi'_*)\bigr]\,d\phi + o(|v-v_*|^2 \theta^2)\\
= & \frac{|v-v_*|^2}{4 |\S^{d-2}|} \int_{\S^{d-2}_{k^\bot}}
\bigl( (\nabla-\nabla_*)(ff_*) \cdot\phi \bigr) \left( \left[ (A+A_*)k + 2 \frac{(\xi-\xi_*)}{|v-v_*|} \right]\cdot\phi\right)\,d\phi + o(|v-v_*|^2 \theta^2)\\
= & \frac{|v-v_*|^2}{4 (d-1)}
\Pi_{k^\bot} \bigl( (\nabla-\nabla_*) (ff_*)\bigr) \cdot \Pi_{k^\bot} \left[ (A+A_*)k + 2 \frac{(\xi-\xi_*)}{|v-v_*|} \right] + o(|v-v_*|^2\theta^2)\\
= & \frac{|v-v_*|^2}{4 (d-1)} ff_*
\Pi_{k^\bot} \bigl( (\nabla-\nabla_*) (\log ff_*)\bigr) \cdot \Pi_{k^\bot} \left[ (A+A_*)k + 2 \frac{(\xi-\xi_*)}{|v-v_*|} \right] + o(|v-v_*|^2\theta^2).
\end{align*}

To conclude, it only remains to relate $\pa_{|z|}B$ and $\pa_{|z|}\Psi$ in the AGC. For this,
\begin{multline*} \lim_{AGC} \frac{|z|^2}{4(d-1)} \int_{\S^{d-1}}
\left(\derpar{B}{|z|}\right)\frac{\theta^2}{2}\,d\sigma
= \frac{|z|^2}{4(d-1)} \derpar{}{|z|} \left( \lim_{AGC} \int_{\S^{d-1}} B \frac{\theta^2}{2}\,d\sigma\right)\\
= |z|^2 \derpar{}{|z|} \left(\frac{\Psi(|z|)}{|z|^2}\right) 
= \Psi'(|z|) - 2 \frac{\Psi(|z|)}{|z|}.
\end{multline*}
\end{proof}

\bibnotes 

I prepared the conjunction of AGC and Fisher information monotonicity for the purpose of this course. Precursors are the study of the AGC {\em per se} \cite{AV:landau:04}, the relation between Boltzmann and Landau dissipation functionals \cite{TV:entropy:99,vill:cer:03}, and of course the Guillen--Silvestre theorem \cite{GS:landaufisher}.

\section{Reduction to the sphere} \label{seccrit}

Section \ref{secGamma} has revealed how to compute $I'(f)\,Q(f,f)$. In this section we shall see that the ``bad'' term is controlled by the ``good'' terms as soon as a certain functional inequality holds for even functions on the sphere. I shall start with an intermediate result.

\begin{Prop} \label{propcriterion} If, for all probability distributions $f$ on $\R^d$,
\begeq\label{criterionC}
\iiint \frac{(f'f'_*-ff_*)^2}{ff_*+f'f'_*}\,\frac1{B} \left(\derpar{B}{|z|}\right)^2\,d\sigma\,dv\,dv_*
\leq \frac12 \iiint ff_* \Bigl | \Pi_{\sigma^\bot} (\xi'-\xi'_*) - \Pi_{k^\bot}(\xi-\xi_*) \Bigr|_{k,\sigma}^2\,B\,d\sigma\,dv\,dv_*,
\endeq
then $I$ is nonincreasing along solutions of the spatially homogeneous Boltzmann equation with kernel $B$.
\end{Prop}

\begin{proof}[Proof of Proposition \ref{propcriterion}]
Start again from the results of Sections \ref{secGamma} and \ref{secqualit}; the issue is to control $\cIII$ in \eqref{c3} by $\cI+\cII$ in \eqref{c1}--\eqref{c2}. It turns out that $\cII$ will suffice. (I do not know whether a more general result can be obtained by exploiting \cI.) By Cauchy--Schwarz,
\begin{align*}
 |\cIII | & \leq \frac12
\left( \iiint (f'f'_*+ff_*) \bigl[ \sigma\cdot (\xi'-\xi'_*) - k\cdot (\xi-\xi_*)\bigr]^2\, B\,d\sigma\,dv\,dv_*\right)^{1/2}\\
& \qquad\qquad\qquad\qquad\qquad\qquad
\left(\iiint \frac{(f'f'_*-ff_*)^2}{ff_* + f'f'_*}\, \frac1{B} \left(\derpar{B}{|z|}\right)^2\,d\sigma\,dv\,dv_*\right)^{1/2}\\
& = \frac12 (8 \cII_3)^{1/2} 
\left(\iiint \frac{(f'f'_*-ff_*)^2}{ff_* + f'f'_*}\, \frac1{B} \left(\derpar{B}{|z|}\right)^2\,d\sigma\,dv\,dv_*\right)^{1/2}\\
& \leq \cII_3 + \frac12 \iiint \frac{(f'f'_*-ff_*)^2}{ff_* + f'f'_*}\, \frac1{B} \left(\derpar{B}{|z|}\right)^2\,d\sigma\,dv\,dv_*.
\end{align*}
By applying \eqref{criterionC},
\[ |\cIII | \leq \cII_3 + \frac14  \iiint ff_* \Bigl | \Pi_{\sigma^\bot} (\xi'-\xi'_*) - \Pi_{k^\bot}(\xi-\xi_*) \Bigr|_{k,\sigma}^2\,B\,d\sigma\,dv\,dv_*.\] 
By the but-to-last formula of Proposition \ref{propPsk}(x) and the last bit of Proposition \ref{propPsk}(vi), 
\[ |\Pi_{\sigma^\bot}(\xi'-\xi'_*) - \Pi_{k^\bot}(\xi-\xi_*)|_{k,\sigma}^2 = 
|\Pi_{\sigma^\bot}(\xi'-\xi'_*) - P_{k\sigma}\Pi_{k^\bot} (\xi-\xi_*)|^2
+ \bigl( |\xi-\xi_*|^2 - |P_{k\sigma}(\xi-\xi_*)|^2\bigr). \]
Thus the latter integral coincides with $\cII_1+\cII_2$. Thus $|\cIII|\leq \cII$, so $\Gamma_{I,{\cal B}}\geq 0$ and the proof is complete.
\end{proof}

To go from \eqref{criterionC} to a more reduced integral criterion on the sphere, let us perform a change of variables
\[ (v,v_*,\sigma) \longrightarrow \left(\frac{v+v_*}{2}, |v-v_*|, \frac{v-v_*}{|v-v_*|},\sigma\right)
= (c,r,k,\sigma).\]
Then
\[ dv\,dv_*\,d\sigma = | \S^{d-1}| \,r^{d-1} \,dc\,dr\,dk\,d\sigma.\]
For each $c,r$ write
\[ F_{c,r}(\sigma) = f\left( c + \frac{r}{2} \sigma\right) f\left(c - \frac{r}{2} \sigma \right),\]
note that $F_{c,r}$ is an even function of $\sigma$, in the sense that $F_{c,r}(-\sigma) = F_{c,r}(\sigma)$. Then the left hand side in \eqref{criterionC} is
\begeq\label{critClhs} |\S^{d-1}|
\iiiint \frac{[F_{c,r}(\sigma)-F_{c,r}(k)]^2}{F_{c,r}(k) + F_{c,r}(\sigma)}\,
\frac1{B} \left(\derpar{B}{|z|}\right)^2(r,k\cdot\sigma)\, r^{d-1}\,dc\,dr\,dk\,d\sigma.
\endeq

On the other hand, treating $\sigma$ as an independent variable in $\R^d$,
\[ \nabla_\sigma \log F_{c,r}(\sigma) = \frac{r}{2} \left[ (\nabla\log f) \left(c + \frac{r}{2}\sigma\right)
- (\nabla\log f) \left(c-\frac{r}{2}\sigma\right) \right] = - \frac{r}{2} (\xi'-\xi'_*).\]
So the integral in the right-hand side of \eqref{criterionC} is
\begeq\label{critCrhs}
\frac{|\S^{d-1}|}{2} \iiiint F_{c,r}(k) \frac{4}{r^2}
\Bigl| \Pi_{\sigma^\bot}\nabla\log F_{c,r}(\sigma) - \Pi_{k^\bot} \nabla\log F_{c,r}(k)\Bigr|^2_{k,\sigma}\,
B(r,k\cdot\sigma)\,r^{d-1}\,dc\,dr\,dk\,d\sigma.
\endeq

Requiring the integrand in \eqref{critClhs} to be bounded above by the integrand in \eqref{critCrhs} {\em for all $c,r$}, we are left with the following criterion: For all $r,c$,
\begin{multline*}
\iint_{\S^{d-1}\times\S^{d-1}} \frac{[F(c,r)(\sigma)-F_{c,r}(k)]^2}{F_{c,r}(k)+F_{c,r}(\sigma)}\, \frac1{B} \left(\derpar{B}{|z|}\right)^2 (r,k\cdot\sigma)\,dk\,d\sigma \\
\leq 
\frac{2}{r^2}
\iint F_{c,r}(k) \Bigl| \Pi_{\sigma^\bot}\nabla\log F_{c,r}(\sigma) - \Pi_{k^\bot} \nabla\log F_{c,r}(k)\Bigr|^2_{k,\sigma}\,
B(r,k\cdot\sigma)\,r^{d-1}\,dk\,d\sigma.
\end{multline*}
The integral in the right-hand side involves the gradient $\nabla F(\sigma)$ only via its projection on the tangent space $\sigma^\bot$; so we might as well assume that $F_{c,r}$ is defined intrinsically on the sphere. The inequality is also translation-invariant in $c$. Eventually we arrive at the sufficient compact condition: For all even functions $F:\S^{d-1}\to\R_+$,
\begin{multline*} \iint_{\S^{d-1}\times\S^{d-1}} \frac{[F(\sigma)-F(k)]^2}{F(k)+F(\sigma)}\, \frac1{B} \left(\derpar{B}{|z|}\right)^2 (r,k\cdot\sigma)\,dk\,d\sigma \\
\leq \frac2{r^2}
\iint_{\S^{d-1}\times\S^{d-1}} F(k) \Bigl| \nabla\log F(\sigma) - \nabla\log F(k)\Bigr|^2_{k,\sigma}\,
B(r,k\cdot\sigma)\,r^{d-1}\,dk\,d\sigma.
\end{multline*}
To turn this into a more tractable criterion, we may use two simplifications: 

\bul Assume that $|\pa_{|z|}B(r,k\cdot\sigma)| \leq \ov{\gamma}(r) B(r,k\cdot\sigma)/r$; then $\ov{\gamma}(r)$ is something like the maximum singularity/growth admissible in $r$; and of course, we may simplify even further by taking $\ov{\gamma}$ to be constant;

\bul Replace the nonlinearity $(b-a)^2/(a+b)$ by $(\sqrt{b}-\sqrt{a})^2$, in view of
\[ \frac{(F(\sigma) - F(k))^2}{F(k)+F(\sigma)} \leq 2 \bigl( \sqrt{F(\sigma)}- \sqrt{F(k)} \bigr)^2.\]

Let us summarise all the above discussion in the next

\begin{Thm} \label{thmcriterion}
Let $B=B(|z|,\cos\theta)$ be a collision kernel. Let
\begeq\label{ovgammaB} 
\ov{\gamma}(B,r) = \sup_{-1\leq\cos\theta\leq 1} \left( \frac{|z|}{B}\, \derpar{B}{|z|}\right) (r,\cos\theta),
\qquad \ov{\gamma}(B) = \sup_{r\geq 0} \ov{\gamma}(B,r)
\endeq
Let us write
\[ \beta_r(\cos\theta) = B(r,\cos\theta),\]
and say that $\beta:[-1,1]\to\R_+$ is a section of $B$ if there is $r\geq 0$ such that $\beta(\cos\theta) = B(r,\cos\theta)$ for all $\theta$.
Then a sufficient condition for $I$ to be nonincreasing along solutions of the Boltzmann equation with kernel $B$ is the following: For all even $F:\S^{d-1}\to\R_+$, for all $r\geq 0$,
\begin{multline}\label{critfinal1}
\iint_{\S^{d-1}\times\S^{d-1}}
\frac{(F(\sigma)-F(k))^2}{F(k)+F(\sigma)}\,\beta_r(k\cdot\sigma)\,dk\,d\sigma\\
\leq \frac{2}{\ov{\gamma}(B,r)^2} \iint_{\S^{d-1}\times\S^{d-1}}
F(k) \,\Bigl| \nabla\log F(\sigma) - \nabla \log F(k)\Bigr|_{k,\sigma}^2\,\beta_r(k\cdot\sigma)\,dk\,d\sigma.
\end{multline}
This is true in particular if:
For all even $F:\S^{d-1}\to\R_+$, for all $r\geq 0$,
\begin{multline}\label{critfinal2}
\iint_{\S^{d-1}\times\S^{d-1}}
\bigl(\sqrt{F}(\sigma) - \sqrt{F}(k) \bigr)^2 \,\beta_r(k\cdot\sigma)\,dk\,d\sigma\\
\leq \frac{1}{\ov{\gamma}(B,r)^2} \iint_{\S^{d-1}\times\S^{d-1}}
F(k)\, \Bigl| \nabla\log F(\sigma) - \nabla \log F(k)\Bigr|_{k,\sigma}^2\,\beta_r(k\cdot\sigma)\,dk\,d\sigma.
\end{multline}
The latter is true in particular if: For all even $F:\S^{d-1}\to\R_+$, for all sections $\beta$ of $B$,
\begin{multline}\label{critfinal3}
\iint_{\S^{d-1}\times\S^{d-1}}
\bigl(\sqrt{F}(\sigma) - \sqrt{F}(k) \bigr)^2 \,\beta (k\cdot\sigma)\,dk\,d\sigma\\
\leq \frac{1}{\ov{\gamma}(B)^2} \iint_{\S^{d-1}\times\S^{d-1}}
F(k)\, \Bigl| \nabla\log F(\sigma) - \nabla \log F(k)\Bigr|_{k,\sigma}^2\,\beta(k\cdot\sigma)\,dk\,d\sigma.
\end{multline}
\end{Thm}

\begin{Rk} Inequalities \eqref{critfinal1}--\eqref{critfinal3} look familiar -- control of an integral difference by an integral gradient, like in Poincar\'e or log Sobolev -- and at the same time weird -- a nonlocal nonlinearity with gradients is unusual, and we have a coupling by $\beta$ instead of a reference measure. Most importantly and possibly most disturbingly, the {\em value} of the optimal constant in these inequalities will determine the range of exponents covered by the theorem.
\end{Rk}

In the sequel I will focus on the simplest inequality \eqref{critfinal3}, and the question is to find conditions under which it may hold, or not. But first of all, in view of the turn taken by the discussion, it will be good to devote some time to equations on the sphere.

\bibnotes

Theorem \ref{thmcriterion} is proven by Imbert, Silvestre and myself \cite{ISV:fisher}.

\section{Games of Boltzmann and Laplace on the sphere} \label{secBLsphere}

Let $\beta=\beta(\cos\theta)\geq 0$ be an angular collision kernel, satisfying the minimum integrability condition 
\begeq\label{minintb}
\mu_\beta = \int_{\S^{d-1}} (1-k\cdot\sigma)\, \beta(k\cdot\sigma)\,d\sigma <\infty.
\endeq
For $f:\S^{d-1}\to\R$, define
\begeq\label{Lbeta}
L_\beta f(k) = \int_{\S^{d-1}} [f(\sigma)-f(k)]\,\beta(k\cdot\sigma)\,d\sigma. 
\endeq
This is a linear Boltzmann operator on $L^1(\S^{d-1})$, a kind of diffusion on the sphere. It is actually the same object that I denoted by ${\cal B}$ in Section \ref{sectens} but let us keep notation different since the latter was acting on functions of $(v,v_*)$ and the current one on functions of just one variable. The operator $L_\beta$ acts on all $L^p(\S^{d-1})$, in particular on $L^2$. Thanks to the exchange $k\cdot\sigma$,
\begin{align*}
\<L_\beta f, g\>_{L^2(\S^{d-1})}
& = \iint [f(\sigma) -f(k)]\, g(k)\,\beta(k\cdot\sigma)\,d\sigma\,dk\\
& = -\frac12 \iint [f(\sigma)-f(k)]\, [g(\sigma)-g(k) ]\,\beta(k\cdot\sigma)\,d\sigma\,dk.
\end{align*}
So $L_\beta$ is self-adjoint and nonpositive,
\begeq\label{Lbetasa}
- \< L_\beta f, f\>_{L^2} = \frac12 \iint_{\S^{d-1}\times\S^{d-1}} [f(\sigma)-f(k)]^2\,\beta(k\cdot\sigma)\,dk\,d\sigma.
\endeq
Further, along $e^{tL_\beta}$ all convex integrals decay: if $C$ is any differentiable convex function, bounded below by a linear function, then ${\cal C}(f) = \int C(f)\,d\sigma$ satisfies
\begin{align*} {\cal C}(f)\cdot L_\beta f & = \iint_{\S^{d-1}\times \S^{d-1}} C'(f(k)) [ f(\sigma) - f(k)]\,\beta(k\cdot\sigma) \, dk\,d\sigma \\
& = -\frac12 \iint_{\S^{d-1}\times \S^{d-1}} \Bigl( C'(f(\sigma)) - C'(f(k))\Bigr)  \bigl( f(\sigma) - f(k) \bigr) \leq 0.
\end{align*}

Of course the most famous dissipation on the sphere is the heat equation, whose generator, the Laplace--Beltrami operator $\Delta$, may be introduced in several ways. For instance, $\Delta f$ at $x$ is the usual Laplacian of $f(x/|x|)$, evaluated at $x$. Or $\Delta = \sum_{ij} (x_i\pa_j - x_j\pa_i)^2/2$ (an average of ``pure'' second derivatives along rotations in the planes $(i,j)$, thus yielding $\Delta$ in any dimension from $\Delta$ in dimension~1). In any case, the explicit expression in coordinates is
\begin{align*}
\Delta f & = \sum_{ij} (\delta_{ij}-x_i x_j) \derpar{{}^2f}{x_i\,\pa x_j} - (d-1)\, x\cdot\nabla f \\
& = \tr (\Pi_{x^\bot} \nabla^2 f) - (d-1)\, x\cdot\nabla f,
\end{align*}
where $\nabla^2$ and $\nabla$ are usual differential operators in $\R^d$, and $f$ is extended outside $\S^{d-1}$ in any way. Then $\Delta$ is also self-adjoint and
\[ - \int_{\S^{d-1}} (\Delta f) g\,d\sigma = \int_{\S^{d-1}} \nabla f\cdot \nabla g\,d\sigma. \]
In the sequel, $\nabla$ and $\nabla^2$ will be the intrinsic gradient and Hessian on the sphere: when $f$ is 0-homogeneous ($f=f(x/|x|)$) they coincide with the usual expressions in $\R^d$; when defined in an intrinsic way, $\nabla f(\sigma)$ is always orthogonal to $\sigma$ and $\nabla^2 f$ is a symmetric operator on $\sigma^\bot$.

The usual $\Gamma$ procedure yields
\begin{Prop} \label{propGammaDelta}
If $\Delta$ is the Laplace operator on $\S^{d-1}$ then \eqref{Gammaf} and \eqref{Gamma2f} reduce to the following expressions, for any $f:\S^{d-1}\to\R$:
\[ \Gamma(f)=\Gamma(f,f) = |\nabla f|^2,\]
which will be denoted also $\Gamma_1(f)$ or $\Gamma_1(f,f)$; and
\[ \Gamma_2(f)=\Gamma_2 (f,f) = \Delta \frac{|\nabla f|^2}{2} - \nabla f \cdot\nabla \Delta f 
= \|\nabla^2 f\|_\HS^2 + (d-2) |\nabla f|^2\]
(a particular case of Bochner's formula).
\end{Prop}

Note carefully, the expressions above only apply to a function defined on $\S^{d-1}$; if $f$ is defined on, say, $\R^d$ then for instance $\nabla f(\sigma)$ should be replaced by $\Pi_{\sigma^\bot} \nabla f(\sigma)$, the orthogonal component of the gradient. 

It turns out that $\Delta$ is part of the Boltzmann family on the sphere; not surprisingly, its limit in the asymptotics of grazing collisions. Let us convene that a sequence $(\beta_n)_{n\in\N}$ concentrates on grazing collisions if $\mu_n = \int (1-k\cdot\sigma) \beta_n(k\cdot\sigma)\,d\sigma$ has a finite limit $\mu>0$, while for any $\theta_0>0$, $\int 1_{k\cdot\sigma \leq \cos\theta_0} \beta_n(k\cdot\sigma)\,d\sigma \longrightarrow 0$. (Here $k$ is any fixed unit vector.) Then

\begin{Prop}\label{propBD}
(i) For any two kernels $\beta_0$ and $\beta_1$, $[L_{\beta_0}, L_{\beta_1}]=0$;

(ii) In the AGC, if $\int (1-k\cdot\sigma) \beta(k\cdot\sigma) \longrightarrow \mu >0$, then
\[ L_\beta \longrightarrow \left(\frac{\mu}{d-1}\right)\Delta;\]

(iii) For any kernel $\beta$, $[L_\beta,\Delta]=0$.
\end{Prop}

\begin{proof}[Proof of Proposition \ref{propBD}]
We start with (ii), which follows the same pattern as the AGC for Boltzmann, using $k\cdot\sigma=\cos\theta$, and the identities \eqref{37bis} and \eqref{37}. As $\theta\simeq 0$,
\begin{multline*}
f(\sigma) - f(k) \simeq (\sigma-k) \cdot \nabla f(k) 
+ \frac12 \nabla^2 f(k)\cdot (\sigma-k,\sigma -k) \\
 = \left(-\frac{\theta^2}{2} k + \theta\phi\right)\cdot \nabla f(k)
+ \frac12 \bigl\< \nabla^2 f(k)\cdot\phi,\phi\bigr\> \theta^2 + o(\theta^2).
\end{multline*}
So from Lemma \ref{lemav},
\begin{align*} \int [f(\sigma) -f(k)]\, \beta(k\cdot\sigma)\,d\sigma
& = - k\cdot \nabla f(k) \int \frac{\theta^2}{2} \,\beta(k\cdot\sigma)\,d\sigma
+ \left( \int \phi\,\beta(k\cdot\sigma)\,d\sigma\right)\cdot\nabla f(k)\\
& \qquad\qquad\qquad\qquad\qquad\qquad
+ \int \<\nabla^2 f(k)\cdot\phi,\phi\>\,\frac{\theta^2}{2}\,\beta(k\cdot\sigma)\,d\sigma + o(1)\\
& = -k\cdot\nabla f(k)\,\mu + 0 + \frac1{d-1} \tr \bigl(\nabla^2 f(k)\,\Pi_{k^\bot}\bigr)\,\mu + o(1).
\end{align*}

Of course (iii) will result from (i) and (ii). So it only remains to check (i). Let us assume that $\int \beta_0$, $\int \beta_1$ are finite; the general case will follow by approximation. Then, for any $\ell\in\S^{d-1}$,
\begin{multline*}
L_{\beta_1} L_{\beta_0} f(\ell) 
= \int \Bigl[ (L_{\beta_0}f)(k) - (L_{\beta_0}f)(\ell)\Bigr]\,\beta_1(\ell\cdot k)\,dk \\
= \iint f(\sigma)\,\beta_0(k\cdot\sigma)\,\beta_1(\ell\cdot k)\,dk\,d\sigma
- \iint f(k) \,\beta_0(k\cdot\sigma)\,\beta_1(\ell\cdot k)\,dk\,d\sigma \\
- \iint f(\sigma)\,\beta_0(\ell\cdot\sigma)\,\beta_1(\ell\cdot k)\,dk\,d\sigma
+ \iint f(\ell) \,\beta_0(\ell\cdot\sigma)\,\beta_1(\ell\cdot k)\,dk\,d\sigma.
\end{multline*}
Let us try to exchange $\beta_0$ and $\beta_1$ in the latter expression made up of four integrals.

The last of the four integrals is the easiest: exchanging $k$ and $\sigma$,
\begin{align*}
f(\ell) \iint \beta_0(\ell\cdot\sigma)\,\beta_1(\ell\cdot k)\,dk\,d\sigma
& = f(\ell) \iint \beta_0(\ell\cdot k)\,\beta_1(\ell\cdot\sigma)\,dk\,d\sigma\\
& = f(\ell) \iint \beta_1(\ell\cdot\sigma)\,\beta_0(\ell\cdot k)\,dk\,d\sigma.
\end{align*}

Now turn to the first of the four integrals. Here is the trick (which I learnt from Sacha Bobylev) to use in such a situation: For any fixed $\sigma$ we may find a change of variables $k\to \tilde{k}$ by an isometry of $\S^{d-1}$ which exchanges $\sigma$ and $\ell$. Then
\begin{align*}
\int \beta_0(k\cdot\sigma)\,\beta_1(\ell\cdot k)\,dk
& = \int \beta_0(\tilde{k}\cdot\sigma)\,\beta_1(\ell\cdot\tilde{k})\,dk\\
& = \int \beta_0(k\cdot \ell)\,\beta_1(\sigma\cdot k)\,dk,
\end{align*}
whence
\[ \iint f(\sigma)\,\beta_0(k\cdot\sigma)\,\beta_1(\ell\cdot k)\,dk\,d\sigma 
= \iint f(\sigma) \,\beta_1(k\cdot\sigma)\,\beta_0(\ell\cdot k).\]

Finallly the same trick will apply to {\em the sum} of the two central integrals:
\begin{align*}
& \iint f(k)\,\beta_0(k\cdot\sigma)\,\beta_1(\ell\cdot k)\,dk\,d\sigma 
& + \iint f(\sigma) \,\beta_0(\ell\cdot\sigma)\,\beta_1(\ell\cdot k)\,dk\,d\sigma\\
& = \iint f(k) \,\beta_0(\ell\cdot\sigma)\,\beta_1(\ell\cdot k)\,dk\,d\sigma
& + \iint f(\sigma)\,\beta_0(\ell\cdot\sigma)\,\beta_1(\sigma\cdot k)\,dk\,d\sigma\\
& \text{(change of variable: $\sigma$, exchange: $k$ \& $\ell$)}
&\text{(c.o.v. $k$, exch. $\ell$ \& $\sigma$)}\\
& = \iint f(\sigma)\,\beta_0(\ell\cdot k)\,\beta_1(\ell\cdot\sigma)\,dk\,d\sigma
& + \iint f(k)\,\beta_0(\ell\cdot k)\,\beta_1(\sigma\cdot k)\,dk\,d\sigma.\\
& \text{(exchanging $k$ and $\sigma$)} & \text{(idem)}
\end{align*}
\end{proof}

Now let us go on with the $\Gamma$ formalism for $L_\beta$ and $\Delta$:

\begin{Def}[Gamma commutators involving linear Boltzmann] \label{defcomLB}
Given a Boltzmann kernel $\beta$, let
\[ \Gamma_\beta(f,f) = \frac12 \Bigl( L_\beta (f^2) - 2 f\, L_\beta f\Bigr); \]
\[ \Gamma_{1,\beta}(f,f) = \frac12 \Bigl[ L_\beta\,\Gamma_1(f,f) - 2 \Gamma_1 (f,L_\beta f)\Bigr].\]
\end{Def}

It turns out that by commutation of $L_\beta$ and $\Delta$, it does not matter in which order we let the two operators act: also
\[ \Gamma_{1,\beta}(f,f) = \frac12 \Bigl[ \Delta\,\Gamma_\beta(f,f) - 2 \Gamma_\beta (f,\Delta f)\Bigr].\]
Explicit computations yield
\begeq\label{Gammabeta}
\Gamma_\beta (f,f) (k) = \frac12 \int_{\S^{d-1}}
(f(\sigma)-f(k))^2\,\beta(k\cdot\sigma)\,d\sigma
\endeq
\begeq\label{Gamma1beta}
\Gamma_{1,\beta}(f,f)(k) = \frac12 
\int_{\S^{d-1}} \bigl|\nabla f(\sigma)-\nabla f(k)\bigr|_{k,\sigma}^2\,\beta(k\cdot\sigma)\,d\sigma,
\endeq
where as before $|y-x|_{k,\sigma}^2 = |x|^2 + |y|^2 - 2 (P_{k\sigma}x)\cdot y$ and 
$P_{k\sigma}$ is defined as in Proposition \ref{propPsk}. Notice: since $\nabla f(\sigma)\cdot\sigma =0$, here $M_{\sigma k} \nabla f (\sigma) = P_{\sigma k}\nabla f(\sigma)$.

Eventually there is also a nice formula for $\nabla L_\beta$:
\begeq\label{nablaLbeta}
\nabla L_\beta (f) = \int_{\S^{d-1}} \bigl[ P_{\sigma k} \nabla f(\sigma) - \nabla f(k)\bigr]\,\beta(k\cdot\sigma)\,d\sigma.
\endeq
Further, in the AGC, as a consequence of the commutation,
\begeq\label{Gamma1b2}
\Gamma_{1,\beta}(f,f) \xrightarrow[AGC]{} \frac{\mu}{d-1} \,\Gamma_2(f,f).
\endeq
And of course, we may also study the evolution of the Fisher information $I$ under the Boltzmann flow and define as in \eqref{GIBF}
\begeq\label{GIBIF'} \Gamma_{I,L_\beta} (f) = L_\beta \bigl( f |\nabla \log f|^2\bigr)
- \left[ (L_\beta f) |\nabla \log f|^2 + 2 \nabla f \cdot \nabla \left(\frac{L_\beta f}{f}\right)\right],
\endeq
so that
\begeq\label{-I'Lf}
-I'(f)\cdot L_\beta f = \int_{\S^{d-1}} \Gamma_{I,L_\beta}(f);
\endeq
and we find, as in Section \ref{secGamma} and with the natural simplifications,
\begeq\label{GILBF} 
\Gamma_{I,L_\beta}(f)\,(k) = \int_{\S^{d-1}} f(\sigma) \bigl| \nabla\log f(\sigma)-\nabla\log f(k)\bigr|^2_{k,\sigma}\,\beta(k\cdot\sigma)\,d\sigma
\endeq
a limit case being
\begeq\label{GIDF} 
\Gamma_{I,\Delta}(f) = 2 f\, \Gamma_2(\log f).
\endeq

A final subject we should explore is, quite naturally in this linear context, the spectrum. This will deserve a section on its own.

\bibnotes

Linear Boltzmann equations are classical in neutronics \cite{cer:book:88}; but also, interacting particle systems are naturally described by a linear Boltzmann equation in a phase space for several particles, shaped by the conservation laws. Kac \cite{kac:foundations} took the jist of this approach by letting his ``particles'' be described by a point in a high-dimensional sphere. But in this section the sphere may be low-dimensional, and as we have seen the linear Boltzmann equation on this space can help understanding the nonlinear Boltzmann equation.

I am not aware of any reference where the subject of this section is explicitly treated, but at least part of it must be folklore. I learnt the important trick for the commutation of the linear Boltzmann equations in Bobylev's survey paper on the Boltzmann equation with Maxwellian kernels \cite{bob:theory:88}, he used this to prove the commutation between (nonlinear) spatially homogeneous Maxwellian Boltzmann and Fokker--Planck semigroups. This is a particular case of much more general theorems used in representation theory and spectral analysis.

\section{Apparition spectrale}

Eigenfunctions of the Laplace--Beltrami operator constitute an orthonormal basis of $L^2(\S^{d-1})$, generalising to higher dimensions the decomposition of $L^2(\S^1)$ in Fourier series. These functions are the spherical harmonics on $\S^{d-1}$ and they have been studied and classified for more than two centuries, especially on $\S^2$. There are several ways to introduce them, one is that spherical harmonics are restrictions to $\S^{d-1}$ of homogeneous polynomials. By definition, a polynomial $P:\R^d\to\R$ is $\ell$-homogeneous ($\ell\in\N_0$) if $P(\lambda x) = \lambda^\ell P(x)$ for all $x\in\R^d$ and $\lambda>0$. Then the restriction of $P$ to $\S^{d-1}$ satisfies 
\[ \Delta P = \lambda\,P \qquad \lambda = \ell (\ell+d-2).\]
(Think that $\Delta |x|^\ell = \ell (\ell +d-2)$, even though that is a mnemo rather than an example, since the restriction of $|x|^\ell$ is trivial.) So the eigenfunctions of $-\Delta$ are $0, d-1, 2d, 3(d+1)$, etc.

Spherical harmonics have ``explicit' expressions in terms of special functions: as we see in Wikipedia (why not), they can be written $Y_{\ell,m}^{(d)}(\theta,\phi)$, where $\ell \in\N_0$, $m$ is a $(d-2)$-index, $m=(m_1,\ldots,m_{d-2})$, $|m_1|\leq m_2\leq \ldots\leq m_{d-2}\leq \ell$, and $x$ is represented by its $(d-1)$-dimensional spherical coordinates: $\theta\in [0,\pi]$, $\phi\in \S^{d-2}$ through $x= (\cos\theta) e + (\sin\theta)\phi$, $e$ fixed arbitrarily, $\phi \in S^{d-2}_{e^\bot}$, and recursively $\phi =\phi_{d-2}$,
$\phi_j = (\cos\psi_j)e_j + (\sin\psi_j)\phi_{j-1}$ with $e_j$ fixed arbitrarily and $\phi_{j-1}\in \S^{j-1}_{e_j^\bot}$, $0\leq \psi_j\leq\pi$, and in the end $\phi$ is identified with its coordinates $(\psi_1,\ldots,\psi_{d-2})$. Then, in complex notation,
\[ Y_{\ell,m}(\theta,\phi) = \frac{e^{i m_1 \psi_1}}{\sqrt{2\pi}}
\ov{P}_{m_2}^{m_1, (2)} (\psi_2) \ov{P}_{m_3}^{m_2, (3)}(\psi_3)\ldots \ov{P}_{\ell}^{m_{d-2},(d-1)}(\theta),\]
where 
\[ P_L^{\ell, (j)}(\psi) = \sqrt{ \left(\frac{2 L + j-1}{2}\right) \frac{(L+\ell+j-2)!}{(L-\ell)!}}\,
(\sin^{\frac{2-j}{2}}\psi)\, P_{L+\frac{j-2}{2}}^{-\left(\ell+\frac{j-2}2\right)} (\cos\psi),\]
and $P_\lambda^\mu$ is Legendre's special function, one of the solutions of the second-order real singular differential linear equation
\[ (1-x^2) y''(x) - 2 x y'(x) + \left( \lambda(\lambda+1) - \frac{\mu^2}{1-x^2}\right) y =0\]
(itself a particular instance of the hypergeometric function); then the associated eigenvalue is $\lambda_\ell^{(d)} = \ell (\ell+d-2)$. (Phew!)

I will not need these explicit formulae, and for the sequel the spherical harmonics of given degree $\ell$ may be ordered in an arbitrary way. But the following facts about spherical harmonics $Y_{\ell,m}^{(d)}$ in dimension $d\geq 2$ will be useful. (The case $\ell=0$ is omitted: obviously, the eigenspace is just made of constant functions. Note, I write $\N = \{1,2,3,\ldots\}$.)

\begin{Prop}[Spherical harmonics on $\S^{d-1}$] \label{propspher}
(i) For each $\ell\in\N$, the spherical harmonics $(Y_{\ell,m}^{(d)})_{1\leq m\leq N(d,\ell)}$ constitute an orthonormal basis of $H_\ell$, the eigenspace of $-\Delta$ associated with the eigenvalue $\lambda_\ell = \ell(\ell+d-2)$, with multiplicity
\[ N(d,\ell) = \left(\frac{2\ell+d-2}{\ell}\right) \, \Cnk{\ell+d-3}{\ell-1};\]

(ii) For any $\ell\in\N$ there is a unique harmonic $\ell$-homogeneous polynomial $P_\ell = P_\ell^{(d)}$ (sometimes denoted ${}_{d-1}\!P_\ell^{0}$) on $[-1,1]$ such that for all $k,\sigma$ in $\S^{d-1}$,
\[ P_\ell (k\cdot\sigma) = \frac{|\S^{d-1}|}{N(d,\ell)}
\sum_{m=1}^{N(d,\ell)} Y_{\ell,m} (k) Y_{\ell,m}(\sigma).\]
This is the {\em addition formula}, and $P_\ell$ is the Legendre polynomial of order $\ell$ in dimension $d$. Legendre polynomials contain most of the relevant information about spherical harmonics. In particular, the space of spherical harmonics of order $\ell$ which admit an axis of symmetry $e\in \S^{d-1}$ is a one-dimensional space generated by $\sigma\longmapsto P_\ell(\sigma\cdot e)$. Similarly, there are vectors $(\eta_m)_{1\leq m \leq N(d,\ell)}$ such that for any $Y\in H_\ell$ there are coefficients $(a_m)_{1\leq m\leq N(d,\ell)}$ such that one may write
$Y(\sigma) = \sum_{m=1}^{N(d,\ell)} a_m\,P_\ell(\sigma\cdot\eta_m)$.

(iii) Legendre polynomials $P_\ell^{(d)}$ are even for even $\ell$ and odd for odd $\ell$; they satisfy $P_\ell(1)=1$ and $|P_\ell(x)|\leq 1$ for all $x\in [-1,1]$. Moreover they satisfy the following four important identities:

\bul Rodrigues formula:
\[ P_\ell(x) = \frac{(-1)^\ell}{2^\ell \left(\ell + \frac{d-3}{2}\right)_\ell} (1-x^2)^{-\left(\frac{d-3}{2}\right)}
\frac{d}{dx}^\ell\! (1-x^2)^{\ell+\frac{d-3}{2}},\]
where $a_\ell = a (a-1)\ldots (a-\ell+1)$.

\bul Hypergeometric differential equation:
\begeq
\label{hypergeom} (1-x^2) P''_\ell(x) - (d-1) xP'_\ell(x)  + \ell (\ell+d-2)\,P_\ell(x) = 0.
\endeq

\bul Recursion formula: $P_0(x)=1$, $P_1(x)=x$, and
\[ (\ell+d-2) P_{\ell+1}(x) - (2\ell + d-2) x P_\ell (x) + \ell P_{\ell-1}(x)=0.\]

\bul Representation by axisymmetric averaging:
\begeq\label{Pl2} P_\ell^{(2)} (x) = \Re (x+i\sqrt{1-x^2})^\ell
\endeq
and, for $d\geq 3$,
\begeq\label{Plaxid} P_\ell (x) = \frac{|\S^{d-3}|}{|\S^{d-2}|}
\int_{-1}^{1} \Bigl( x + i\,s\, \sqrt{1-x^2} \Bigl)^\ell (1-s^2)^{(d-4)/2}\,ds.
\endeq
\end{Prop}

\begin{Ex} \label{exPol}
For $d=2$, an element in $\S^1$ can be identified to an angle $\theta\in \R/(2\pi\Z)$; then $\Delta$ is the usual operator $d^2/d\theta^2$, the eigenspace $H_\ell$ is $2$-dimensional for any $\ell\in\N$ and an orthonormal basis is ($\pi^{-1/2} \cos (n\theta)$, $\pi^{-1/2}\sin(n\theta)$); then the addition formula turns into
\[ P_\ell (\cos (\theta'-\theta)) = \cos (\ell\theta)\cos(\ell\theta') + \sin(\ell\theta)\sin (\ell\theta'),\]
which by the usual addition formula for trigonometric functions is just $\cos(\ell(\theta'-\theta))$; so $P_\ell$ is just the Tchebyshev polynomial defined by the identity
\[ P_\ell(\cos\alpha) = \cos(\ell \alpha).\]
In complex notation, we immediately see that $P_\ell(\cos\alpha) = \Re e^{i\ell\alpha} 
= \Re (\cos\alpha + \sqrt{1-\cos^2\alpha})^\ell$, which amounts to \eqref{Pl2}.
\end{Ex}

\begin{Rk}\label{rkleg}
Formula \eqref{Plaxid} shows that, in some sense, all $P_\ell$ can be obtained from the 2-dimensional case; indeed, this formula amounts to writing a harmonic polynomial in two variables, say $x_1,x_2$, and symmetrize it around the axis $e_1$. Note that the integral in \eqref{Plaxid} is real by symmetry $y\to -y$ and the operation applied to the polynomial is indeed an averaging, since
\[ \int_{-1}^1 {\sqrt{1-s^2}\,}^{d-4}\,ds = \int_0^\pi (\sin\theta)^{d-3}\,d\theta = \frac{|\S^{d-2}|}{|\S^{d-3}|}.\]
\end{Rk}

Legendre polynomials have since long been used to describe the spectrum of the linearised Boltzmann equation on $\R^d$. But they will also be useful to describe the spectrum of the linear Boltzmann equation on $\S^{d-1}$.

\begin{Thm}[Spectrum of the linear Boltzmann equation] \label{thmspectrum}
If a collisional kernel $\beta$ is given on $\S^{d-1}$, then the linear operator $-L_\beta$ has eigenvalues $(\lambda_\ell)_{\ell\in\N}$, and for each $\ell\in\N$, $H_\ell$ is an invariant subspace of dimension $N(d,\ell)$ associated with $\nu_\ell$, where
\begin{align} \label{nuell}
\nu_\ell & = \int_{\S^{d-1}} \bigl[ 1-P_{\ell}^{(d)}(k\cdot\sigma)\bigr]\,\beta(k\cdot\sigma)\,d\sigma \\
& \nonumber 
= |\S^{d-2}| \int_0^\pi \bigl[ 1-P_\ell^{(d)}(\cos\theta)\bigr]\,\beta(\cos\theta)\,\sin^{d-2}\theta\,d\theta.
\end{align}
\end{Thm}

\begin{Rk} It may be that some eigenvalues $\nu_\ell$ coincide, depending on the choice of $\beta$; so the multiplicity may not be exactly $N(d,\ell)$.
\end{Rk}

\begin{proof}[Proof of Theorem \ref{thmspectrum}]
Let me skip the technical issue of domain of definition (handled by many authors in the much more tricky case of $\R^d$) and jump directly to the calculation of the spectrum. Since $\Delta$ and $L_\beta$ commute, we know that the basis of spherical harmonics $(Y_{\ell,m})$ is a basis of diagonalisation for $L_\beta$. The question is to find the associated eigenvalues 
\[ \nu(Y) = -\< L_\beta Y,Y\> = \frac12 \iint \bigl[ Y(\sigma)-Y(k)\bigr]^2\,\beta(k\cdot\sigma)\,dk\,d\sigma.\]

A particular choice for $Y$ is $Y(k)=c\, P_\ell(k\cdot\eta)$ for some $\eta\in\S^{d-1}$, where $c$ is a normalisation constant. If $\eta$ and $\zeta$ are two such unit vectors, again using an isometry of $\S^{d-1}$ which exchanges $\eta$ and $\zeta$, we find
\[ \frac12 \iint \bigl[ P_\ell(\sigma\cdot\eta)-Y(k\cdot\eta)\bigr]^2\,\beta(k\cdot\sigma)\,dk\,d\sigma 
= \frac12 \iint \bigl[ P_\ell(\sigma\cdot\zeta)-Y(k\cdot\zeta)\bigr]^2\,\beta(k\cdot\sigma)\,dk\,d\sigma;\]
in other words the eigenvalue $\nu$ associated with $P_\ell(k\cdot\eta)$ is independent of the choice of $\eta$. Since all spherical harmonics $Y$ are linear combinations of such elementary functions, we deduce that the eigenvalue $\nu(Y)$ is independent of $Y$ and only depends on $\ell$. It remains to compute this number.

Consider the elementary Boltzmann equation operator $L^\theta$ where $\beta$ is concentrated at a single deviation angle $\theta$: 
\[ L^\theta f(k) = \int_{\S^{d-1}} [f(\sigma)-f(k)]\,1_{k\cdot\sigma = \cos\theta}\,d\sigma.\]
In this way
\begeq\label{Lbetatheta} L_\beta = \int_0^\pi L^\theta \,\beta(\cos\theta)\, \sin^{d-2}\theta\,d\theta. \endeq

Each $\sigma\in\S^{d-1}$ such that $k\cdot\sigma =\cos\theta$ is the endpoint of a unit-speed geodesic $\gamma:[0,\theta]\to\S^{d-1}$ such that $\gamma(0)=k$, $\dot{\gamma}(0)=\phi\in \S^{d-2}_{k^\bot}$. In other words, $\sigma = \exp_k (\theta\phi)$. Thus
\[ L^\theta f(k) = \int_{\S^{d-2}_{k^\bot}} \Bigl[ f\bigl(\exp_{k}\theta\phi \bigr) - f(k)\Bigr]\,d\phi.\]
Again $L^\theta$ commutes with $\Delta$, and the eigenvalues are given by the numbers
\[ \nu_\ell(\theta) = -\< L_\theta Y,Y\>,\]
where $Y$ is any normalised spherical harmonic in $H_\ell$. Then
\begin{align*} 
\nu_\ell(\theta) & = \int_{\S^{d-1}} Y(k) \left(\int_{\S_{k^\bot}^{d-2}}
\bigl[ Y(k) - Y(\exp_k \theta\phi)\bigr]\,d\phi\right) dk\\
& = \int_{\S^{d-1}} Y(k)^2\, |\S^{d-2}_{k^\bot}|\,dk - \int_{\S^{d-1}} \int_{\S^{d-2}_{k^\bot}} Y(k)\, Y(\exp_k\theta\phi)\,d\phi\,dk \\
& = |\S^{d-2}| - \int_{\S^{d-1}} \int_{\S^{d-2}_{k^\bot}} Y(k)\, Y(\exp_k\theta\phi)\,d\phi\,dk,
\end{align*}
and as a consequence $\iint Y(k) Y(\exp_k\theta\phi)\,d\phi\,dk$ does not depend on the choice of $Y$.

By the addition formula of Proposition \ref{propspher} (ii), for all $k$ and $\sigma=\exp_k\theta\phi$,
\[ P_\ell^{(d)}(k\cdot\sigma) = \frac{|\S^{d-1}|}{N(d,\ell)}
\sum_{m=1}^{N(d,\ell)} Y_{\ell,m}(k)\, Y_{\ell,m}(\sigma), \]
so upon integration $d\phi\,dk$ and substitution $k\cdot\sigma=\cos\theta$,
\[ P^{(d)}_\ell (\cos\theta)\,|\S^{d-2}|\,|\S^{d-1}| 
= \frac{|\S^{d-1}|}{N(d,\ell)}
\sum_{m=1}^{N(d,\ell)} \iint Y_{\ell,m}(k)\,Y_{\ell,m}(\exp_k\theta\phi)\,d\phi\,dk.\]
On the right hand side, all terms in the sum are equal to $|\S^{d-2}|-\nu_\ell(\theta)$. So 
\[ P^{(d)}_\ell (\cos\theta)\,|\S^{d-2}|\,|\S^{d-1}|
= |\S^{d-1}|\,\bigl(|\S^{d-2}| - \nu_\ell(\theta)\bigr).\]
All in all,
\begeq\label{nultheta}
\nu_\ell(\theta) = |\S^{d-2}|\, \bigl( 1-P^{(d)}_\ell(\cos\theta)\bigr).
\endeq
Then the conclusion follows from \eqref{Lbetatheta} and its corollary $\nu_\ell(\beta) = \int_0^\pi \nu_\ell(\theta)\,\beta(\cos\theta)\,\sin^{d-2}\theta\,d\theta$.
\end{proof}

\bibnotes

About spherical harmonics, among innumerable references (the majority of which insist on dimension $d=3$) I recommend the exceptionally clear lecture notes by Frye and Efthimiou \cite{frye-efthimiou:legendre:book} considering general dimensions.

The application of special functions and symmetries to the computation of the spectrum of the Boltzmann equation with Maxwellian kernels goes back at least to Wang Chang and Uhlenbeck in the sixties \cite{wchanguhl:spectrum}; general background is in the reference books like Cercignani \cite{cer:book:88}. More recently, noting that the relevant material is scattered in a number of places, Dolera \cite{dolera:spectrum} provides a self-contained presentation and review. The eigenvalues are also given by integral expressions involving the kernel and Legendre polynomials.

\section{Reduction to the sphere, reformulated}

Using the language of Section \ref{secBLsphere} it is possible to rewrite the criterion from Section \ref{seccrit} in a more synthetic way. For simplicity I will only consider product collision kernels:
$B(v-v_*,\sigma) = \Phi(|v-v_*|)\,\beta(\cos\theta)$; it is not difficult to adapt the statements to more general collision kernels. The associated Landau function $\Psi$ is proportional to $|v-v_*|^2 \Phi(|v-v_*|)$. Theorem \ref{thmcriterion}, in its simplest version, can now be rewritten

\begin{Thm}[Criterion for decay of the Fisher information] \label{thmcritsyn}
Consider the spatially homogeneous Boltzmann equation with kernel $B(v-v_*,\sigma) = \Phi(|v-v_*|)\,\beta(\cos\theta)$, or the Landau equation with function $\Psi(|v-v_*|)$ proportional to $|v-v_*|^2\Phi(|v-v_*|)$. Let
\[ \ov{\gamma} = \sup_{r>0} \frac{r\,|\Phi'(r)|}{\Phi(r)} = \sup_{r>0} \left| \frac{r\Psi'(r)}{\Psi(r)} -2 \right|.\]
Then, the Fisher information $I$ is nonincreasing along solutions of equation, as soon as
\[ \ov{\gamma}\leq \sqrt{K_*},\]
where $K_*$ is the best constant in the following inequality, required for all $F:\S^{d-1}/\{\pm I\}\to \R_+$,

\bul For Boltzmann:
\begeq\label{eqBcrit}
\int_{\S^{d-1}} \Gamma_\beta(\sqrt{F}) \leq \frac1{K_*} \int_{\S^{d-1}} F\,\Gamma_{1,\beta}(\log F),
\endeq
where the expressions of $\Gamma_\beta$ and $\Gamma_{1,\beta}$ are in Definition \ref{defcomLB};

\bul For Landau:
\begeq\label{eqLcrit}
\int_{\S^{d-1}} \Gamma_1(\sqrt{F}) \leq \frac1{K_*} \int_{\S^{d-1}} F\,\Gamma_2(\log F),
\endeq
where the expressions of $\Gamma_1$ and $\Gamma_2$ are  in Proposition \ref{propGammaDelta}.
\end{Thm}

The constant $K_*$ a priori depends on whether we are considering Boltzmann or Landau, and a priori depends on $\beta$. Here $I$ is the identity on $\S^{d-1}$. Note that in the definition of $F$, taking quotient by $\{\pm I\}$ is the same as requiring $F$ to be even. In differential geometry, $\S^{d-1}/ \{\pm I\}$ is called the projective space in dimension $d$ (that is, the set of lines in $\R^d$) and denoted $\RP^{d-1}$.

Obviously these criteria can be refined as in Section \ref{seccrit}, considering non-factorised kernels $B$, but also replacing the square roots by the more precise nonlinearity there. The advantage of the formulation above is that it is compact and synthetic. Clearly \eqref{eqLcrit} is the AGC of \eqref{eqBcrit}. The remarkable achievement of Guillen--Silvestre is the explicit reduction of the decay of $I$ to an optimal functional inequality of logarithmic Sobolev type. Before going more deeply in this issue, let us consider the linearisation.

\section{Linearisation} \label{seclinearisation}

Let $M(v) = (2\pi)^{-d/2} e^{-|v|^2/2}$. Linearising the Boltzmann equation around $M$ consists in writing $f=M(1+\var\,h)$ and letting $\var\to 0$, then $Q(f,f) \simeq \var {\cal L}_B h$, where
\begin{align*} {\cal L}_B h & = \frac1{M} \bigl[ Q(M,Mh) + Q(Mh,M)\bigr]
\\ & = \int \bigl[ h(v')+ h(v'_*) - h(v) - h(v_*) \bigr] \, M(v_*)\, B(v-v_*,\sigma)\,dv_*\,d\sigma.
\end{align*}
Of course the linearised Boltzmann equation is $\pa_t h = {\cal L}_B h$.
Without loss of generality I assume $\int h(v)\,dv =0$, $\int h(v)\,v\,dv =0$, $\int h(v)\, |v|^2\,dv=0$.
Then as $\var\to 0$, 
\[
\begin{cases} \dps H(M(1+\var h)) \simeq H(M) + \frac{\var^2}2 \|h\|_{L^2(M)}^2, \\
\dps I(M(1+\var h)) \simeq I(M) + \var^2 \|\nabla h\|_{L^2(M)}^2, \\
\dps \Gamma_1((\log (1+\var M))
\simeq \var^2 \Gamma_1(h).
\end{cases}
\]
So Theorem \ref{thmcritsyn} admits the following linearised version involving the Dirichlet form
\[ {\cal I}(h) = \int_{\R^d} |\nabla h(v)|^2\,M(v)\,dv.\]

\begin{Thm}[Decay of the Dirichlet form along linearised Boltzmann] \label{thmcritlin}
With the same notation as in Theorem \ref{thmcritsyn}, for ${\cal I}$ to be nonincreasing along either the spatially homogeneous linearised Boltzmann or linearised Landau equation, it is sufficient that $\ov{\gamma}\leq \sqrt{4P_*}$, where $P_*$ is the best constant in the following inequality: for all even $h:\S^{d-1}/\{\pm I\}\to \R$ with $\int h\,dv =0$,

\bul For Boltzmann:
\[ \int_{\S^{d-1}} \Gamma_\beta(h) \leq \frac1{P_*} \int_{\S^{d-1}} \Gamma_{1,\beta}(h).\]

\bul For Landau:
\[ \int_{\S^{d-1}} \Gamma_1(h) \leq \frac1{P_*} \int_{\S^{d-1}} \Gamma_2(h).\]
\end{Thm}

\begin{Rk} The range of $\gamma$ may possibly be further improved by restricting to functions $h$ satisfying not only $\int h =0$, but also $\int h\,v\,dv =0$ and $\int h |v|^2\,dv =0$.
\end{Rk}

The constant $P_*$ a priori depends on whether we are considering Boltzmann or Landau, and a priori depends on $\beta$. For the Landau case, it cannot be larger than $\lambda_1$, which is thus the best uniform lower bound we may hope for. As we shall see later, that best possible bound does hold true: $P_*\geq\lambda_1$, and in general there is equality. Before that we need some reminders.

\bibnotes

The relation between entropic inequalities and Poincar\'e (spectral gap) inequalities is very well known, see e.g. \cite{BGL:book,vill:oldnew}. In the special case $\gamma=0$ (Maxwellian molecules), the monotonicity of ${\cal I}$ follows from the commutation property with $\Delta$; apart from that special case the general result of Theorem \ref{thmcritlin} is new as far as I know.

\section{Reminders on diffusion processes}

Some reminders from the theory of diffusion of processes. Since Bakry and \'Emery, the theory has been built upon the concept of {\bf curvature-dimension condition} $\CD(K,N)$: Ricci curvature bounded below by $K$, dimension bounded above by $N$. Expressed in terms of the Laplace--Beltrami operator, this condition, set on a Riemannian manifold $(M,g)$, reads
\[\forall f:M\to\R,\qquad  \Gamma_2 (f) \geq \frac{(\Delta f)^2}{N} + K \Gamma_1(f).\]
(Here $\Gamma_1(f)=\Gamma_1(f,f)$, etc.)
For instance the sphere satifies $\CD(d-2,d-1)$:
\begeq\label{CDdd}
\forall f:\S^{d-1}\to\R,\qquad
\Gamma_2(f) \geq \frac{(\Delta f)^2}{d-1} + (d-2)\, \Gamma_1(f).
\endeq
Keeping only the latter term, this obviously yields
\begeq\label{fGamma2}
\int f \Gamma_2(\log f) \geq (d-2) \int f \Gamma_1(\log f) = (d-2) I(f).
\endeq
Along the heat flow, $dH/dt=-I$, $dI/dt = -2 \int f\Gamma_2(\log f)$; so \eqref{fGamma2} is the same as
\[ -\frac{d}{dt} I(e^{t\Delta} f) \geq - 2 (d-2) \frac{d}{dt} H(e^{t\Delta} f)\geq 0. \]
Upon integrating from $0\to \infty$ {\em \`a la} Stam, we discover
\begeq\label{LSIsub} H(f)  \leq \frac1{2(d-2)} I(f), \endeq
that is, the log Sobolev inequality with constant $d-2$. This is good, but not optimal. On the one hand, in dimension 2 the unit sphere $\S^1$ has no curvature, but still satisfies log Sobolev, with constant 1. On the other hand, linearisation of \eqref{LSIsub} yields a Poincar\'e inequality with constant $d-2$, but we know from the discussion of spherical harmonics that the spectral gap for $-\Delta$ on $\S^{d-1}$ is $d-1$, not $d-2$. It turns out that $d-1$ is indeed optimal, not only for Poincar\'e but also for log Sobolev: For all probability densities $f$ on $\S^{d-1}$, for all functions $h$ on $\S^{d-1}$ with $\int h =0$,
\begeq\label{Lichne} \int h^2 \leq \frac1{d-1} \int \Gamma_1 (h), \endeq
\begeq\label{optBE} H(f) \leq \frac1{2(d-1)} I(f),\endeq
\begeq\label{diffBE} \int f \,\Gamma_1(\log f) \leq \frac1{d-1} \int f\,\Gamma_2(\log f).\endeq
\med

To summarise: There are three functional inequalities of interest for us: 

\bul Optimal Poincar\'e (spectral gap): For all functions $h$,
\begeq\label{optP} \int_{M} h^2 - \left(\int_M h\right)^2 
\leq \frac1{\lambda_1(M)} \int_{M}|\nabla h|^2 \endeq

\bul Optimal log Sobolev inequality: For all probability densities $f$,
\begeq\label{optLS}
\int_M f \log f \leq \frac1{2L(M)} \int_M f |\nabla\log f|^2
\endeq

\bul Optimal differential log Sobolev inequality, or optimal differential Bakry--\'Emery inequality: For all probability densities $f$,
\begeq\label{optDLS}
\int_M f |\nabla \log f|^2 \leq \frac1{L_*(M)} \int f\Gamma_2 (\log f);
\endeq
and always 
\[ L_* \leq L \leq \lambda_1,\]
while
\[ L_*(\S^{d-1}) = L(\S^{d-1}) = \lambda_1(\S^{d-1}) = d-1.\]

When this is applied to Theorem \ref{thmcritsyn} and Theorem \ref{thmcritlin}, taking into account $\int \Gamma_1(\sqrt{F}) = (1/4) I(F)$, this immediately yields a range of exponents for decay of the Fisher information along the Landau equation: $K_* = 4 L_* = 4(d-1)$, $P_* = \lambda_1 = d-1$, and in both cases $\ov{\gamma} \leq 2\sqrt{d-1}$. This is already good, and better than the crude bounds in Theorem \ref{thmGS0}. But it still fails to cover the most important case, physically speaking, $\gamma=-3$ in dimension $d=3$, in fact with this bound $\ov{\gamma}$ can only go up to $2\sqrt{2}$, which is $< 3$ (not by much, but that is life!).

However there is a crucial information which was not yet exploited: Since $F$ is assumed even in the criteria of Theorems \ref{thmcritsyn} and \ref{thmcritlin}, this can only improve the constants; in fact we are looking for optimal constants on $\RP^{d-1}$, rather than $\S^{d-1}$.
For the spectral gap, we already know the answer: considering even functions amounts to rule out spaces $H_\ell$ with odd $\ell$, so
\begeq\label{lambda1RP} \lambda_1(\RP^{d-1}) = \lambda_2 (\S^{d-1}) = 2d  \endeq
(much better than $d-1$). Also, for $d=2$ one can work out completely
\[ L_*(\RP^1) = L(\RP^1) = 4 \quad (= \lambda_1(\RP^1) ). \]

For higher dimensions the constants $L$ and $L_*$ are not as good: In fact
\[ d\geq 3\Longrightarrow\qquad  L_*(\RP^{d-1}) \leq L(\RP^{d-1}) < 2d. \]
However, the excellent spectral gap already implies at least an improved bound on $L$: Indeed, a theorem by Rothaus shows that if $M$ satisfies $\CD(K,N)$ and $\lambda_1> KN/(N-1)$ then $L(M) \geq [K(1-1/N)+ (4/N)\lambda_1]/(1+1/N)^2$; in our case this yields
\begeq\label{LRPd} L(\RP^{d-1}) \geq (d-1) \left(\frac{d+2}{d}\right)^2. \endeq
(Note that it does capture the optimal constant $4$ when $d=2$.) In dimension~3 it yields an excellent bound $60/9\simeq 5.55$ for $L$. While this gives no lower bound for $L_*$, it suggests at least that the latter is also improved. In fact, one can establish
\begeq\label{boundLstar2} 5.5 \leq L_*(\RP^2) < 5.74. \endeq
In higher dimensions, the bound gets worse and worse, but still one has at least that $2d = \lambda_1(\RP^{d-1})\geq L_*(\RP^{d-1})\geq L_*(\S^{d-1})= d-1$.
\medskip

Now we are ready to rework Theorems \ref{thmcritsyn} and \ref{thmcritlin}, and obtain

\begin{Thm}[Guillen--Silvestre theorem, full version] \label{thmfullGS}
Assume that the function $\Psi$ appearing in Landau's collision operator satisfies $\ov{\gamma} = \sup |r\Psi'(r)/\Psi(r) -2 | \leq 4$. Then

(i) Fisher's information can only go down with $t$ along the spatially homogeneous Landau equation; 

(ii) The Dirichlet form $\int |\nabla h|^2\,M$ can only go down with $t$ along the spatially homogeneous linearised Landau equation.
\end{Thm}

\begin{Rk} The choice of the constant 4 in Theorem \ref{thmfullGS} is because the Landau equation theory assumes at least $\gamma\geq -4$. \end{Rk}

\begin{proof}[Proof of Theorem \ref{thmfullGS}]
For (i): In view of Theorem \ref{thmcritsyn}, suffices to check $L_*(\RP^{d-1})\geq 4$. This is postponed to Section \ref{secdiffusive}: There I shall prove that $L_*(\RP^1) = 4$, $L_*(\RP^{2})\geq 4.5$ (a bit short of the better bounds $4.75$ and $5.5$ obtained respectively by Guillen--Silvestre and by Ji, but quite sufficient for this theorem), and for higher dimensions $L_*(\RP^{d-1}) \geq 6 d/(d+1)$ (quite bad in large dimension since actually $L_*\geq d-1$, but also quite sufficient for this theorem).

Then, (ii) is simpler: It follows directly from \eqref{lambda1RP} and Theorem \ref{thmcritlin}; actually any $\ov{\gamma}\leq \sqrt{8d}$ will do.
\end{proof}

Now what about the Boltzmann equation, that is, what to do of the inequalities appearing in the first half of Theorems \ref{thmcritsyn} and \ref{thmcritlin}? These ones do not fall within the range of usual functional inequalities, but they are related as the preceding chapters demonstrate. The next chapter will be devoted to the linearised case: a spectral gap problem. The subsequent four chapters will consider the nonlinear inequality.

\bibnotes

There are two lines of thought behind curvature-dimension bounds: operators and geodesics. The first one goes via functional inequalities involving $\Delta$, as in the works of Bakry, Ledoux and others. The other one involves the behaviour of geodesics and the way they distort measure. This dichotomy can be seen as an Eulerian/Lagrangian duality; see \cite[Chapter~14]{vill:oldnew}. The second approach has been at the core of the synthetic theory of curvature-dimension bounds \cite[Part III]{vill:oldnew}. In the smooth setting, the Bochner formula relates the two approaches, which are therefore equivalent in a smooth context.

For functional inequalities in log Sobolev style, Bakry--Gentil--Ledoux \cite{BGL:book} is a standard rich introduction to the subject; the differential Bakry--\'Emery inequality with optimal constants is Eq. (5.7.5) in that book. Another excellent introduction is the collective work \cite{toulouse:sobolog}, in french. 

Inequality \eqref{Lichne} for $\CD(d-2,d-1)$ manifolds is due to Lich\'erowicz, while \eqref{optBE} and \eqref{diffBE} are part of the Bakry--\'Emery theory. The optimal inequality \eqref{optBE} is discussed in Ledoux \cite{ledoux:hyperc:92}, who develops a unified approach to Poincar\'e and log Sobolev inequalities.

Michel Ledoux has kindly provided references on the problem of constants on $\RP^{d-1}$, with the help of Fabrice Baudoin and Dominique Bakry.  The theorem by Rothaus is from \cite{rothaus:hyperc:86}. Saloff-Coste \cite{lsc:rate:94} considers the log Sobolev and Poincar\'e constants on $\RP^{d-1}$ and notices that the log Sobolev and Poincar\'e inequality differ (Theorem 9.2 therein), and that the ratio should go down to $1/2$ as $d\to\infty$ (Remark 3), so that in large dimension the sphere and projective space would have nearly similar $L$ and $L_*$ constants. Details do not seem to have been written down before. A lecture by Fontenas \cite{fontenas:Jacobi:98} addresses this and related topics. Sehyun Ji worked on this issue and presented preliminary results in the summer school Mathemata (July 2024), then announced \cite{ji:BE} a lower bound $L_*(\RP^{d-1})\geq d+3-1/(d-1)$.

\section{Spectral gap}

This section addresses the problem of monotonicity for the Dirichlet form ${\cal I}$ along the spatially homogeneous linearised Boltzmann equation; that is, the linearised version of the initial quest.

Consider the following commutative diagram, where $h$ is a function on $\S^{d-1}$ with zero average.

\hspace{50mm}\begin{tikzcd}
\dps \frac12 \int_{\S^{d-1}} h^2 \arrow[r,"{-\ddt{L_\beta}}"]
\arrow[d,"-\ddt{\Delta}"] 
& {\dps \int_{\S^{d-1}} \Gamma_\beta(h)}  \arrow[d, "-\ddt{\Delta}"]\\
{\dps \int_{\S^{d-1}}\Gamma_1(h)} \arrow[r,"-\ddt{\Delta}"] 
& {\dps 2\int_{\S^{d-1}} \Gamma_{1,\beta}(h)}
\end{tikzcd}

Knowing that $-\ddt{\Delta} \int h^2 \geq \lambda_1 \int \Gamma_1(h)$ does not imply in itself that $\ddt{L_\beta} \ddt{\Delta} \frac12\int h^2 \geq -\lambda_1 \ddt{L_\beta}\int \Gamma_1(h)$; it is actually the other way round. But these are symmetric operators with a common basis of eigenfunctions (spherical harmonics $e_k$), so we may write $(\lambda_k)$ and $(\nu_k)$ for the eigenvalues of $-\Delta$ and $-L_\beta$, with repetition, and if $h=\sum x_k\,e_k$, assuming $\int h =0$,
\[ \frac12 \int h^2 = \frac12 \sum_{k\geq 1} x_k^2,\qquad
\int \Gamma_1(h) = \sum_{k\geq 1} \lambda_k x_k^2 \]
\[ \int \Gamma_\beta(h) = \sum_{k\geq 1} \nu_k x_k^2\qquad
2 \int \Gamma_{1,\beta}(h) = 2 \sum_{k\geq 1} \lambda_k \nu_k x_k^2.\]
Then 
\[ 2 \int \Gamma_{1,\beta}(h) \geq 2 \lambda_1 \sum_{k\geq 1}\nu_k x_k^2 = 2 \lambda_1 \int\Gamma_\beta(h).\]
Let $P_*$ be as in Theorem \ref{thmcritlin} (first forget about the parity requirement); then the above computation shows that $P_* \geq \lambda_1$, whatever $\beta$, and also in the AGC. It is also immediate from those formulae that $\lambda_1$ is achieved for $x=e_1$, except if $\nu_1=0$ (which could arise from a peculiar choice of $\beta$; recall formula \eqref{nuell}). When only even functions are considered, as in the statement of Theorem \ref{thmcritlin}, the discussion is the same with $\lambda_1$ and $\nu_1$ replaced by $\lambda_2$ and $\nu_2$.

The above discussion is summarised into the

\begin{Prop} \label{propPstar}
Whatever the collision kernel $\beta$, it holds $P_* \geq \lambda_1(\RP^{d-1}) = \lambda_2(\S^{d-1})$, with equality as soon as $\nu_2(\beta)\neq 0$. \end{Prop}

\begin{Rk} It is enlightening to work out directly the one-dimensional case. When $d=2$, identify $\S^1$ with $\R/(2\pi\Z)$, write $\dot{g}$ for the usual derivative of $g$, then the problem becomes
\begeq\label{Pstar1d}
\iint
\bigl(g(\alpha')-g(\alpha)\bigr)^2\,\beta(\cos(\alpha-\alpha'))\,d\alpha\,d\alpha'
\leq \frac1{P_*} 
\iint \bigl(\dot{g} (\alpha')-\dot{g}(\alpha)\bigr)^2\,\beta(\cos(\alpha-\alpha'))\,d\alpha\,d\alpha'.
\endeq
Let us prove \eqref{Pstar1d} by elementary means. Upon first sight, the optimal constant $P_*$ will depend on $\beta$. But rewrite everything in Fourier series. Writing by abuse of notation $\hat{\beta}$ for $\hat{\beta\circ\cos}$, 
\[ \beta(\alpha) = \sum_{k\in\Z} \hat{\beta}(k)\,\cos(k\alpha),\qquad
g(\alpha) = \sum_{k\in\Z} \hat{g}(k) \,e^{i k\alpha}, \]
where
\[ g(\alpha) = \frac1{2\pi} \int_0^{2\pi} g(t)\,e^{-i\, k t}\,dt.\]
By approximation we may assume that $\beta$ is integrable. Then, using standard formulas of Fourier transform,
\begin{align*}
\frac12 \iint & \bigl( g(\alpha') - g(\alpha)\bigr)^2\,\beta(\cos(\alpha'-\alpha))\,d\alpha\,d\alpha'\\
& = \int g(\alpha)^2\,d\alpha \int \beta(\cos\alpha')) \,d\alpha''
- 2 \int g(\alpha)\,g(\alpha')\,\beta(\cos(\alpha'-\alpha))\,d\alpha\,d\alpha'\\
& = \|g\|_{L^2}^2 \left(\int\beta\right) - \int g (g\ast\beta)\\
& = 2\pi^2 \Bigl[ \bigl( \sum_k |\hat{g}(k)|^2\bigr)\hat{\beta}(0) - \Re \sum_k \hat{g}(k) (\hat{g}(k)\hat{\beta}(k))^* \Bigr] \\
& = 2\pi^2 \Bigl[ \sum_k |\hat{g}(k)|^2 \bigl( \hat{\beta}(0) - \Re \hat{\beta}(k)\bigr)\Bigr]\\
& = 2\pi^2 \Bigl[ \sum_k |\hat{g}(k)|^2 \bigl( \hat{\beta}(0) - \hat{\beta}(k)\bigr)\Bigr].
\end{align*}
So \eqref{Pstar1d} is equivalent to
\begeq\label{Pstar1deq}
\sum_{k\in\Z} |\hat{g}(k)|^2 \bigl(\hat{\beta}(0)-\hat{\beta}(k)\bigr)
\leq \frac1{P_*} \sum_{k\in\Z} |k|^2 |\hat{g}(k)|^2 \bigl(\hat{\beta}(0)-\hat{\beta}(k)\bigr).
\endeq
The coefficients $\hat{\beta}(0)-\hat{\beta}(k)$ are nonnegative since $\beta$ is, and vanish for $k=0$. So \eqref{Pstar1deq} holds with $P_*=1$. Now if we recall the additional condition that $g$ is even, i.e. $g(\alpha+\pi) = g(\alpha)$, only even coefficients $k$ remain, and $\hat{g}(\pm 1)=0$, so the constant $P_*$ improves from 1 to $4$. This handles the one-dimensional case. \qed
\end{Rk}

The Boltzmann analogue of Theorem \ref{thmfullGS} (ii) now comes as an easy consequence of the previous estimates:

\begin{Thm}[Decay of the Dirichlet form along linearised Boltzmann] \label{thmdecayDLB}
With the notation \eqref{ovgammaB}, let $B=B(|v-v_*|,\cos\theta)$ be a Boltzmann kernel satisfying $\ov{\gamma}(B)\leq 4$. Then the Dirichlet form $\int |\nabla h|^2 M$ can only go down with $t$ along the spatially homogeneous linearised Boltzmann equation.
\end{Thm}

\begin{proof}[Proof of Theorem \ref{thmdecayDLB}]
Combining Theorem \ref{thmcritlin} with Proposition \ref{propPstar} yields the decay for $\ov{\gamma}(B)\leq \sqrt{4\lambda_2(\S^{d-1})} = \sqrt{8d}$, which is always at least 4.
\end{proof}

\section{Local monotonicity criterion via curvature} \label{seccurv}

Now we are back to the study of Criterion \ref{thmcritsyn} for the nonlinear Boltzmann equation. 
Recall the expression of $\Gamma_\beta$ and $\Gamma_{1,\beta}$ from \eqref{Gammabeta}--\eqref{Gamma1beta}. By analogy with the diffusive case, we may call inequality \eqref{eqBcrit} a $\beta$-nonlocal differential log Sobolev inequality. As in Section \ref{secqualit}, let us separate $\Gamma_{1,\beta}$ into two terms, one involving variations of the gradient and the other being ``curvature-like'', pointwise in the gradient:
\begin{align} \label{Gamma1b}
\Gamma_{1,\beta}(f) 
& = \frac12 \int \bigl|\nabla f(\sigma) - P_{k\sigma}\nabla f(k)\bigr|^2\,\beta(k\cdot\sigma)\,d\sigma
+ \frac12 \int \bigl( |\nabla f(k)|^2 - |P_{k\sigma}\nabla f(k)|^2\bigr)\,\beta(k\cdot\sigma)\,d\sigma\\
& \nonumber 
= \frac12 \int \bigl|\nabla f(\sigma) - P_{k\sigma} \nabla f(k)\bigr|^2\,\beta(k\cdot\sigma)\,d\sigma
+ (d-2)\,\Sigma(\beta)\,|\nabla f(k)|^2,
\end{align}
where I used the second identity in Proposition \ref{propPsk}(x) and Lemma \ref{lemcurv}, and
\begeq\label{Sigmabeta}
\Sigma(\beta) = \frac1{2(d-1)} \int \bigl[1-(k\cdot\sigma)^2\bigr]\, \beta(k\cdot\sigma)\,d\sigma.
\endeq
In particular,
\[ \Gamma_{1,\beta}(f) \geq (d-2)\,\Sigma(\beta)\, |\nabla f|^2. \]

The first lowest hanging fruit from Criterion \ref{thmcritsyn} is to rely on just the curvature term:
\begin{align} \label{lowesthanging}
\int_{\S^{d-1}} F\, \Gamma_{1,\beta} (\log F) 
& \geq (d-2)\,\Sigma(\beta) \int_{\S^{d-1}} F(k) |\nabla \log F(k)|^2\,dk \\
& = 4(d-2)\,\Sigma(\beta) \int_{\S^{d-1}} |\nabla \sqrt{F}|^2. \nonumber
\end{align}

Remarkably, $\Sigma$ turns out the be {\em exactly} the optimal constant relating $\int \Gamma_\beta(g)$ and $\int \Gamma_1(g)$.

\begin{Prop} \label{propGbG1} For any even function $g:\S^{d-1}\to\R$,
\begeq\label{GbG1}
\int_{\S^{d-1}} \Gamma_\beta(g) \leq \Sigma(\beta) \int_{\S^{d-1}} \Gamma_1(g),
\endeq
and the constant $\Sigma(\beta)$ is optimal (lowest).
\end{Prop}

Before proving Proposition \ref{propGbG1}, I will state some striking consequences. First, combined with \eqref{lowesthanging}, Proposition \ref{propGbG1} implies a universal (independent of $\beta$) bound:

\begin{Cor}[Curvature-based nonlocal differential log Sobolev] \label{corNLDLS}
For any even function $F:\S^{d-1}\to\R_+$, for any $\beta: [-1,1]\to\R_+$ with $\int \theta^2\beta(\cos\theta)\sin^{d-2}\theta\,d\theta <\infty$,
\begeq\label{GbG1'}
\int_{\S^{d-1}}\Gamma_{1,\beta}(\log F) \geq 4(d-2) \int_{\S^{d-1}}\Gamma_\beta(\sqrt{F}).
\endeq
\end{Cor}

A very simple decay criterion follows:

\begin{Thm}[Curvature-induced decay of $I$ for the Boltzmann equation] \label{thmcurvI}
If the collision kernel $B=B(|v-v_*|,\cos\theta)$ satisfies 
\[  \sup_{r>0,\ 0\leq\theta\leq\pi}
\left\{ \frac{r}{B} \left|\derpar{B}{r}\right| (r,\cos\theta)\right\} 
\leq 2 \sqrt{d-2},\]
then $I$ is nonincreasing along solutions of the spatially homogeneous Boltzmann equation.
\end{Thm}

That covers for instance $|\gamma| \leq 2$ in any dimension $d\geq 3$.

\begin{proof}[Proof of Theorem \ref{thmcurvI}]
As in Corollary \ref{corNLDLS}, for any section $\beta_r (\cos\theta) = B(r,\cos\theta)$, we have
\begin{align*} \int_{\S^{d-1}} F \Gamma_{1,\beta_r} (\log F) & \geq 
(d-2) \Sigma(\beta_r) \int_{\S^{d-1}} F|\nabla \log F|^2 
\\ & = 4 (d-2)\, \Sigma(\beta_r) \int_{\S^{d-1}}\Gamma_1(\sqrt{F}) \\
& \geq 4(d-2)\, \int_{\S^{d-1}} \Gamma_{\beta_r}(\sqrt{F}),
\end{align*}
and Theorem \ref{thmcritsyn} (or actually the slightly generalised version coming from Theorem \ref{thmcriterion} when $B$ is not necessarily in product form and \eqref{ovgammaB} is used) applies with $\ov{\gamma}(B,r)=\sqrt{4(d-2)}$.
\end{proof}

It remains to prove Proposition \ref{propGbG1}. I shall provide two arguments. The first one carries some intuition of Poincar\'e inequalities but fails to achieve the sharp bound. The second one carries no intuition at all, but provides the sharp bound.

\begin{proof}[Proof of Proposition \ref{propGbG1}, nonsharp bound]
Here I will prove
\begin{multline} \label{GbG1nonsharp}
\frac12 \iint \bigl[g(\sigma)-g(k)\bigr]^2 \,\beta(k\cdot\sigma)\,dk\,d\sigma \\
\leq \frac1{2(d-1)} \left(|\S^{d-2}| \int_0^{\pi} \theta^2 \beta(\cos\theta)\, \sin^{d-2}\theta\,d\theta\right)
\int_{\S^{d-1}} |\nabla g|^2.
\end{multline}

Replacing $\beta$ by $[\beta(\cos\theta) + \beta(-\cos\theta)] 1_{\cos\theta\geq 0}$ does not change the integral since $g(-\sigma) = g(\sigma)$. So we may assume that $\beta$ is supported in $\theta\in [0,\pi/2]$. Then \eqref{GbG1nonsharp} differs from \eqref{GbG1} by a factor at most $\sup(\theta^2/\sin^2\theta)= \pi^2/4$.

By linearity it is enough to prove \eqref{GbG1nonsharp} when $\beta$ is concentrated on some fixed angle $\theta$. Then the estimate becomes
\[ \iint \bigl[g(\sigma)-g(k)\bigr]^2 \,1_{k\cdot\sigma=\cos\theta}\, dk\,d\sigma
\leq \frac{\theta^2}{d-1} \int_{\S^{d-1}} |\nabla g|^2.\]
Given $k$, each $\sigma$ such that $k\cdot\sigma = \cos\theta$ is the endpoint of a geodesic $[0,\theta]\to \S^{d-1}$ coming from $k$ with initial velocity $\phi\in S^{d-2}_\bot$, the $(d-2)$ orthonormal sphere centered at $k$ and orthogonal to it. In other words, $\phi$ lies in the unitary tangent space to $\S^{d-1}$ at $k$. So the estimate to prove really is
\[ \iint_{U\S^{d-1}} \bigl[g(\exp_k \theta\phi)-g(k)\bigr]^2 \,d\phi\,d\sigma
\leq \frac{\theta^2}{d-1} \int_{\S^{d-1}} |\nabla g|^2,\]
where $U\S^{d-1}$ is the unitary tangent bundle. 
For each $k$ and $\phi$ we write $\gamma_{k,\phi}(t) = \exp_k(t\phi)$, and
\[ g(\exp_k(\theta \phi)) - g(k) = \int_0^\theta \nabla g(\gamma_{k,\phi}(t))\cdot \dot{\gamma}_{k,\phi}(t)\,dt\]
so
\[ \Bigl( g(\exp_k(\theta \phi)) - f(k) \Bigr)^2
= \left( \int_0^\theta \nabla g(\gamma_{k,\phi}(t))\cdot \dot{\gamma}_{k,\phi}(t)\,dt\right)^2
\leq \theta \int_0^\theta \bigl( \nabla g(\gamma_{k,\phi}(t))\cdot\dot{\gamma}_{k,\phi}(t)\bigr)^2\,dt.\]
Upon integration and using the fact that for each $t$ the exponential map is measure preserving from the unitary tangent bundle to itself,

\begin{align*}  \iint_{U\S^{d-1}} \Bigl( g(\exp_k(\theta \phi)) - g(k) \Bigr)^2\,dk\,d\phi
& \leq \theta \int_0^\theta \iint_{ U\S^{d-1}} \bigl( \nabla g(\gamma_{k,\phi}(t))\cdot\dot{\gamma}_{k,\phi}(t)\bigr)^2\,dt\,dk\,d\phi \\
& = \theta \int_0^\theta \iint_{ U\S^{d-1}} \bigl(\nabla g(x)\cdot v\bigr)^2\,dt\,dx\,dv\\
& = \theta^2 \iint_{ U\S^{d-1}} \bigl(\nabla g(x)\cdot v\bigr)^2\,dx\,dv\\
& = \frac{\theta^2}{d-1} \int_{\S^{d-1}} |\nabla g(x)|^2\,dx.
\end{align*}
\end{proof}

\begin{proof}[Proof of Proposition \ref{propGbG1}, sharp bound]
Now the goal is 
\begin{multline}\label{GbG1sharp}
\frac12 \iint \bigl[g(\sigma)-g(k)\bigr]^2 \,\beta(k\cdot\sigma)\,dk\,d\sigma \\
\leq \frac1{2(d-1)} \left(|\S^{d-2}| \int_0^{\pi} \sin^2\theta\, \beta(\cos\theta)\, \sin^{d-2}\theta\,d\theta\right)
\int_{\S^{d-1}} |\nabla g|^2.
\end{multline}
In other words,
\[ \<-L_\beta g, g\> \leq \Sigma(\beta)\,\<-\Delta g,g\>.\]
It suffices to check this for all even spherical harmonics; in other words, that for all $\ell\in 2\N$,
\[ \nu_\ell(\beta) \leq \Sigma(\beta)\, \lambda_\ell^{(d)},\]
where $\nu_\ell(\beta)$ is the eigenvalue of $-L_\beta$ in $H_\ell$, and $\lambda_\ell^{(d)}(\beta)$ the one of $-\Delta$. This is the same as
\[ \int_{\S^{d-1}} 
\bigl[ 1-P_\ell(k\cdot\sigma)\bigr]\,\beta(k\cdot\sigma)\,d\sigma
\leq \frac{\ell (\ell+d-2)}{2(d-1)} \int_{\S^{d-1}} \bigl[ 1-(k\cdot\sigma)^2\bigr]\,\beta(k\cdot\sigma)\,d\sigma.\]
The proof is concluded by Lemma \ref{lemLeg} below.
\end{proof}

\begin{Lem} \label{lemLeg}
For all $d\in\N$, for all $\ell\in 2\N$,
\begeq\label{ineqLeg}
\sup_{-1\leq x\leq 1} \left( \frac{1-P_\ell^{(d)}(x)}{1-x^2}\right) \leq \frac{\ell (\ell+d-2)}{2(d-1)},
\endeq
with equality for $\ell = 2$ (for all $x$) and for $x\to\pm 1$ (for all $\ell$).
\end{Lem}

\begin{proof}[Proof of Lemma \ref{lemLeg}]
Throughout the proof, superscripts $(d)$ are implicit. First some reformulations: Since $\lambda_\ell = \ell (\ell +d-2)$ and $P_2(x) = (dx^2-1)/(d-1)$, the inequality amounts to
\[ \text{ $\ell$ even } \Longrightarrow 
\forall x,\quad \frac{1-P_\ell(x)}{\lambda_\ell} \leq \frac{1-P_2(x)}{\lambda_2}. \]
Since $P_\ell(\pm 1) = 1$ and $P'_\ell(\pm 1) = \pm \lambda_\ell$ from \eqref{hypergeom}, the inequality also says that $P_\ell$ remains above the parabola which is tangent to its graph at $x=\pm 1$.

Now for the proof of the Lemma itself. For pedagogical reasons, first consider $d=2$. Then $\lambda_2=4$, $\lambda_{2n} = 4n^2$, and from the interpretation of $P_\ell$ (Example \ref{exPol}), \eqref{ineqLeg} is the same as
\[ \forall n\in\N,\qquad \frac{1-\cos (2n\alpha)}{1-\cos^2\alpha}
\leq 2 n^2.\]
Equivalently, since $1-\cos(2 n\alpha) = 2 \sin^2 (n\alpha)$ and $1-\cos^2\alpha=\sin^2\alpha$, the inequality amounts to just
\begeq\label{sinnn} |\sin(n\alpha)| \leq n\,\sin\alpha, \endeq
which is geometrically obvious and can also be proven by induction using $| \sin(n\alpha+\alpha) |\leq |\cos(n\alpha)|\sin\alpha+ |\sin(n\alpha)||\cos\alpha| \leq \sin\alpha+ |\sin (n\alpha)|$.

To handle the general case, recall that Formula \ref{Plaxid} reduces the general case to the 2-dimensional situation. In view of that formula and Remark \ref{rkleg},
\begeq\label{1-Pl-av}
1-P_\ell(x) = \frac{|\S^{d-3}|}{|\S^{d-2}|}
\int_{-1}^1 \bigl[ 1- (x+is\sqrt{1-x^2})^{\ell}\bigr]\,(1-s^2)^{\frac{d-4}{2}}\,ds.
\endeq
We turn to estimate the real part of the integrand (the imaginary part vanishes by symmetry).
Given $s$ and $x$ both $[-1,1]$, let $\rho = \sqrt{x^2+ s^2(1-x^2)}$,
$\cos\alpha = x/\rho$, $\sin\alpha = s\sqrt{1-x^2}/\rho$, so that $x+i s \sqrt{1-x^2} = \rho e^{i\alpha}$.
\sm

\noindent Claim:  
\begeq \label{rhoalpha}
\forall n\in\N_0,\ \rho \in [0,1],\ \alpha \in [0,\pi]\qquad
\Re \bigl[ 1-(\rho e^{i\alpha})^{2n}\bigr] \leq
n(1-\rho^2) + 2 n^2 \rho^2\sin^2\alpha. \endeq
The proof of \eqref{rhoalpha} is by induction again. For $n=0$ the identity is obvious. If true at step $n$, then
\begin{align*}
\Re \bigl[ 1-(\rho e^{i\alpha})^{2n+2}\bigr]
& = \Re \bigl[ 1- (\rho e^{i\alpha})^{2n}\bigr] + \Re \bigl[ (\rho e^{i\alpha})^{2n} - (\rho e^{i\alpha})^{2n+2)}\bigr] \\
& = \Re \bigl[ 1- (\rho e^{i\alpha})^{2n}\bigr]  + \rho^{2n} (1-\rho^2)\cos(2n\alpha) 
+ \rho^{2n+2} [\cos 2n\alpha - \cos(2n+2)\alpha]\\
&  = \Re \bigl[ 1- (\rho e^{i\alpha})^{2n}\bigr]  + \rho^{2n} (1-\rho^2)\cos(2n\alpha) 
- 2\rho^{2n+2} \sin (2n+1)\alpha\,\sin\alpha.
\end{align*}
Using both the induction hypothesis and the bound \eqref{sinnn}, the above is bounded by
\[ \bigl(n(1-\rho^2) + 2n^2\rho^2\sin^2\alpha\bigr)
+ (1-\rho^2) + 2 (2n+1)\rho^2 \sin^2\alpha \leq (n+1)(1-\rho^2)+ 2(n+1)^2\rho^2\sin^2\alpha,\]
which propagates the induction hypothesis and proves \eqref{rhoalpha}. So
\begin{align*}
\Re \bigl[ 1-(x+i s\sqrt{1-x^2})^{2n}\bigr]
& \leq \bigl[ n(1-s)^2+ 2n^2s^2\bigr]\,(1-x^2)\\
& \leq \bigl [n(1-2n)(1-s^2) + 2n^2\bigr] (1-x^2).
\end{align*}
Plugging this back in \eqref{1-Pl-av}, we get
\begin{align*}
\frac{1-P_\ell(x)}{1-x^2}
& \leq \frac{|\S^{d-3}|}{|\S^{d-2}|}
\int_{-1}^1 \bigl[ n(1-2n)(1-s^2)^{\frac{d-2}{2}} + 2n^2 (1-s^2)^{\frac{d-4}{2}}\bigr]\,ds\\
& = \frac{|\S^{d-3}|}{|\S^{d-2}|} \left( n(1-2n) \frac{|\S^{d}|}{|\S^{d-2}|} + 2n^2 \frac{|\S^{d-2}|}{|\S^{d-3}|}\right)\\
& = n (1-2n) \frac{|\S^{d-3}|\,|\S^{d}|}{|\S^{d-2}|\,|\S^{d-1}|} + 2 n^2\\
& = n (1-2n) \left(\frac{\int_0^{\pi} \sin^{d-1}\psi\,d\psi}{\int_0^\pi \sin^{d-3}\psi\,d\psi} \right) + 2n^2\\
& = n(1-2n) \left(\frac{d-2}{d-1}\right) + 2n^2 = \frac{2n (n+d-2)}{d-1},
\end{align*}
as required, after application of the Wallis integrals again.
\end{proof}

\bibnotes

Poincar\'e estimates in positive curvature are classical and appear e.g. in \cite[Chapter 18]{vill:oldnew}, but the setting here allows for much more explicit constants, as Proposition \ref{propGbG1} shows. All this section (except the nonoptimal argument for Proposition \ref{propGbG1}) is taken from \cite{ISV:fisher}.

\section{(Counter)examples} \label{secex}

Theorem \ref{thmcurvI} provides a general answer, even if nonoptimal, in dimension $d\geq 3$, but says nothing of dimension~2. This may seem awkward as dimension~2 is usually simpler for the theory of the Boltzmann equation, and as a matter of fact the monotonicity property of the Fisher information for Maxwellian kernels was established in dimension~2 several years before higher dimensions.

It turns out that this is in the order of things. The general criterion established in Section \ref{seccrit} fails for dimension~2, at least for certain classes of singular kernels. At the opposite side, when the collision kernel is constant in the angular variable (which is arguably the most regular that can be!) then the criterion holds even in dimension~2, and this will allow for a first positive result for non-Maxwell kernels in that dimension.

\begin{Thm} \label{thmcex} (i) Consider a measure $\beta$ on $[0,\pi]$ defined by
\[ \beta(d\theta) = \sum_{i=0}^{N} \beta_i\,\delta_{\theta_i},\quad
\theta_i = \frac{i\pi}{N},\ N\in\N, \ \beta_i\geq 0,\ \sum_{i=1}^{N-1}\beta_i >0.\]
Then
\begeq\label{inf=0}
\inf \left\{ \frac{\dps \int_{\S^{1}} F\, \Gamma_{1,\beta}(\log F)}{\dps \int_{\S^{1}} \Gamma_\beta(\sqrt{F})};\quad
F: \S^{1}/\{\pm I\} \to\R_+ \right\} = 0.
\endeq

(ii) If $\beta$ is constant on $\S^{d-1}$, then for any $d\geq 2$,
\begeq\label{inf4d}
\inf \left\{ \frac{\dps \int_{\S^{d-1}} F\, \Gamma_{1,\beta}(\log F)}{\dps \int_{\S^{d-1}} \Gamma_\beta(\sqrt{F})};\quad
F: \S^{d-1}/\{\pm I\} \to\R_+ \right\} \geq 4d.
\endeq

(iii) For the symmetrised hard spheres kernel in dimension 2,
\begeq\label{infHS2d}
\inf \left\{ \frac{\dps \int_{\S^1} F\, \Gamma_{1,\beta}(\log F)}{\dps \int_{\S^1} \Gamma_\beta(\sqrt{F})};\quad
F: \S^{1}/\{\pm I\} \to\R_+ \right\} \geq 4\sqrt{2}.
\endeq
In particular, $I$ is nonincreasing along solutions of the 2-dimensional spatially homogeneous Boltzmann equation for hard spheres.
\end{Thm}

\begin{Rks}
1. In Part (i), $\delta_0$ and $\delta_\pi$ play no role in either $\Gamma_\beta$ or $\Gamma_{1,\beta}$ ($\delta_0$ obviously corresponds to no deviation in the collision, and so does $\delta_\pi$ when applied to $\pi$-periodic functions). So we may work directly with $\theta_1,\ldots\theta_{N-1}$. By setting some coefficients $\beta_i$ equal to~0, we see that the theorem covers any kernel with finite support in $\Q\pi$. Extension to angles which are irrational to $\pi$ is a more difficult problem. Another question left open by this theorem is whether one can use this negative result to construct $f=f(v)$ on $\R^2$ such that $I'(f)\cdot Q(f,f)>0$, say for $B= |v-v_*|\,\beta$. Theorem \ref{thmIpos} below will show that Part (i) would be false if $\beta$ would not vanish.

(ii) Part (ii) implies that for $B(v-v_*,\sigma) = |v-v_*|^\gamma$, $|\gamma|\leq 2 \sqrt{d}$, $I$ is nonincreasing along solutions of the spatially homogeneous Boltzmann equation. But only the case $\gamma=1$, $d=3$ corresponds to a classical interaction considered in physics, namely three-dimensional hard spheres (already covered by Theorem \ref{thmcurvI}). Recall that in dimension 2 the hard spheres kernel vanishes for $\theta=0$ and in dimensions 4 and higher it has an integrable singularity as $\theta\to 0$. Still (iii) tells us that hard spheres in dimension 2 can be handled, this will be done by using the obvious perturbation lemma below.
\end{Rks}

\begin{Lem}[Perturbation lemma for the integral criterion] \label{lemperturb}
If $\beta_0$ and $\beta$ are two collision kernels, and $K_*(\beta_0)$, $K_*(\beta)$ are the respective associated optimal constants in the criterion from Theorem \ref{thmcritsyn}. If there are constants $m,M>0$ such that
\[ \forall \theta\in [0,\pi],\qquad m\,\beta_0(\cos\theta) \leq \beta(\cos\theta) \leq M \,\beta_0(\cos\theta),\]
Then
\[ K_*(\beta) \geq \frac{m}{M}\, K_*(\beta_0).\]
\end{Lem}

\begin{proof}[Proof of Theorem \ref{thmcex} (i)]
Fix $0<\delta< 1/2$ and in $\R/\pi\Z$ let $I_i = [(i-\delta)\pi/N, (i+\delta)\pi/N]$, $1\leq i\leq N$, so that the intervals $I_i$ are disjoint neighbourhoods of $\theta_i$. Pick up two smooth functions $\R/\pi\Z\to\R_+$, $h$ and $\psi$, such that

\bul $h$ is constant on each $I_i$ and takes distinct integer values on all these intervals;

\bul $\psi$ is $\pi/N$-periodic, nonzero and supported over the intervals $I_i$.

Then for $A>0$ let $F = (1+A\psi)h$. Since $\psi$ is $\pi/N$-periodic, so is $\log (1+A\psi)^\cdot$, so
\[ (\log F)^\cdot (\alpha+\theta_i) - (\log F)^\cdot (\alpha)  = (\log h)^\cdot (\alpha+\theta_i) - (\log h)^\cdot (\alpha)\]
and thus
\[ \int F \Gamma_{1,\beta} (\log F) = 
\frac12 \sum_i \beta_i \int (1+A\psi(\alpha)) h(\alpha) \bigl[ (\log h)^\cdot(\alpha+\theta_i) - (\log h)^\cdot (\alpha)\bigr]^2\,d\alpha.\]
On the support of $\psi$, $h$ is constant, and so is $h(\cdot + \theta_i)$; thus $\psi [(\log h)^\cdot (\alpha+\theta_i) - (\log h)^\cdot (\alpha)]^2=0$. In the end,
\begin{align*} \int F \Gamma_{1,\beta}(\log F) 
& = \frac12 \sum_i \beta_i 
\int h(\alpha) \bigl[ (\log h)^\cdot(\alpha+\theta_i) - (\log h)^\cdot (\alpha)\bigr]^2\,d\alpha \\
& = \int h \Gamma_{1,\beta}(h).
\end{align*}
This is independent of $A$ and $\psi$.

On the other hand,
\begin{align*}
\Gamma_\beta(\sqrt{F}) 
& \geq \sum_{i,j} \beta_i \int_{I_j} \bigl[ \sqrt{F(\alpha+\theta_i)} - \sqrt{F(\alpha)}\bigr]^2\,d\alpha\\
& = A \left(\sum_{i} \beta_i \right) \left(\sqrt{\psi(I_j)}-\sqrt{\psi(I_i)}\right)^2 \left(\sum_j \int_{I_j} h\right)^2.
\end{align*}
\end{proof}

\begin{proof}[Proof of Theorem \ref{thmcex} (ii)]
Take $\beta\equiv 1$. Then
\begin{align*}
\int F \Gamma_{1,\beta} (\log F)
& = \frac12 \iint F(k) \bigl|\nabla \log F(\sigma)-\nabla\log F(k)\bigr|_{k,\sigma}^2\,dk\,d\sigma\\
& = \frac12 \iint F(k) |\nabla \log F(k)|^2\,dk\,d\sigma
+ \frac12 \iint F(k) |\nabla \log F(\sigma)|^2\,dk\,d\sigma\\
& \qquad\qquad\qquad
- \iint \nabla F(k) P_{\sigma k} \nabla \log F(\sigma)\,dk\,d\sigma\\
& = \frac{|\S^{d-1}|}2 \int F|\nabla \log F|^2 + \left(\int_{\S^{d-1}} F\right) \left(\int_{\S^{d-1}} |\nabla \sqrt{F}|^2\right) + 0,
\end{align*}
since $\int P_{\sigma k}\nabla\log F(\sigma)\,d\sigma =0$ by Lemma \ref{lemBL3} (recall $M_{\sigma k}\nabla f = P_{\sigma k}\nabla f$ if $\nabla f(\sigma) \cdot \sigma =0$).
So
\begeq\label{FGammageq} 
\int F \Gamma_{1,\beta} (\log F) \geq 2 |\S^{d-1}| \int |\nabla \sqrt{F}|^2.
\endeq

On the other hand,
\begin{align*}
\int \Gamma_\beta(\sqrt{F})
& = \frac12 \iint \bigl( \sqrt{F(\sigma)}-\sqrt{F(k)}\bigr)^2\,dk\,d\sigma\\
& = |\S^{d-1}| \int F - \left(\int\sqrt{F}\right)^2\\
& = |\S^{d-1}|^2 \left[ \frac{1}{|\S^{d-1}|} \int F - \left(\frac1{|\S^{d-1}|} \int\sqrt{F}\right)^2\right]\\
& \leq \frac{|\S^{d-1}|^2}{\lambda_1} \left( \frac1{|\S^{d-1}|} \int_{\S^{d-1}} |\nabla\sqrt{F}|^2\right),
\end{align*}
where $\lambda_1 = \lambda_1(\RP^{d-1}) = 2d$. The conclusion follows.
\end{proof}

\begin{proof}[Proof of Theorem \ref{thmcex} (iii)]
For $d=2$ the hard spheres kernel $b(\cos\theta)$ varies like $\sin(\theta/2)$. Its symmetric version $[b(\cos\theta) + b(-\cos\theta)]/2$ like $[\sin(\theta/2)+\cos(\theta/2)]/2$, which takes values in $[1/2,\sqrt{2}/2]$. So we may apply Lemma \ref{lemperturb} with $\beta_0=1$ and $M/m=\sqrt{2}$, leading to \eqref{infHS2d}.
\end{proof}

\bibnotes 

Counterexample (i) was shown to me by Luis Silvestre (for $N=2$). Estimates (i)-(ii)-(iii) are published in \cite{ISV:fisher}.

The perturbation lemma \ref{lemperturb} is obvious, but it is natural, in this context of log Sobolev type inequalities, to link it to the Holley--Stroock perturbation lemma \cite{holstroo:logsob:87}.

\section{Local-integral monotonicity criterion via positivity}

Combining the example of constant kernel in Section \ref{secex}, a comparison argument and the Poincar\'e inequality from Section \ref{seccurv}, is enough to show that $K_*>0$, thereby ruling out the situation of counterexample \ref{thmcex}(i) and proving the following monotonicity theorem.

\begin{Thm}[Positivity-induced decay of $I$ for the Boltzmann equation] \label{thmIpos}
If the collision kernel $B=B(|v-v_*|,\cos\theta)$ is such that for all $r>0$ and (any) $k\in\S^{d-1}$,
\[  \sup_{0\leq\theta\leq\pi}
\left\{ \frac{r}{B} \left|\derpar{B}{r}\right| (r,\cos\theta)\right\}
\leq \left(\frac{4(d-1) |\S^{d-1}| \, \left(\dps\min_{0\leq \theta\leq\pi} B(r,\cos\theta)\right)}{\dps\int_{\S^{d-1}} \bigl[1-(k\cdot\sigma)^2\bigr] B(r,k\cdot\sigma)\,d\sigma}\right)^{1/2}, \]
then $I$ is nonincreasing along solutions of the spatially homogeneous Boltzmann equation. Such is the case if $B = |v-v_*|^\gamma \beta(\cos\theta)$ with
\begeq\label{inegthmposit} 
|\gamma| \leq \left(\frac{4(d-1) |\S^{d-1}| (\min\beta)}{|\S^{d-2}| \dps \int_0^\pi \sin^d\theta\, \beta(\cos\theta)\,d\theta}\right)^{1/2}.
\endeq
\end{Thm}

\begin{proof}[Proof of Theorem \ref{thmIpos}]
Let $\beta_r(\cos\theta) = \beta(r,\cos\theta)$ and $m_r = \min (\beta_r)$. On the one hand, by \eqref{FGammageq},
\[
\int F \Gamma_{1,\beta_r}(\log F) \geq
\int F \Gamma_{1,m_r}(\log F) \geq 2 |\S^{d-1}| m_r \int |\nabla\sqrt{F}|^2,
\]
on the other hand, from Proposition \ref{propGbG1},
\[ \int \Gamma_{\beta_r}(\sqrt{F}) \leq  \Sigma(\beta_r) \int |\nabla\sqrt{F}|^2.\]
So for all even $F:\S^{d-1}\to\R_+$,
\[ K_* (\beta_r)  \geq \frac{2 |\S^{d-1}|\, m_r}{\Sigma(\beta_r)}\]
and the conclusion follows from Theorem \ref{thmcriterion} and the definitions of $m_r, \Sigma(\beta_r)$.
\end{proof}

While this approach provides a considerable list of kernels, even in dimension 2, for which the Fisher information is decaying, it is not powerful enough to handle all classical inverse-law potentials.

\begin{Ex} Assume $d=2$ and inverse power law forces like $1/r^s$. The collision kernel takes the form $B = |v-v_*|^\gamma b(\cos \theta)$ with $\gamma= (s-3)/(s-1)$. Maxwellian molecules correspond to $s=3$. This kernel is bounded from below, so the right-hand side of \eqref{inegthmposit} is a strictly positive number. As $s$ varies, $\gamma=\gamma(s)$ varies continuously, but also the right-hand side of \eqref{inegthmposit}. By continuity there is an interval around $s=3$ in which the inequality \eqref{inegthmposit} remains satisfied. Numerics can tell us exactly which is this range, but we already see that this method covers some of the classical inverse power law forces in dimension 2, which the curvature-based approach was unable to tackle.
\end{Ex}

\begin{Ex} Let us now consider the case $s=2$, still in dimension $d=2$. Then $\gamma=-1$ and $b(\cos\theta)$ is proportional to $1/\sin^2(\theta/2)$. (This is the 2-dimensional version of Rutherford's cross-section formula.) An explicit calculation shows that the right-hand side in \eqref{inegthmposit} is $1/\sqrt{2}$, which is not enough to squeeze $|\gamma|=1$ in the left-hand side. So Theorem \ref{thmIpos} is not general enough to cover $s=2$ in dimension $d=2$.
\end{Ex}

\section{Exploiting gradient variations I: the diffusive case} \label{secdiffusive}

In this and the next two sections the question is how to exploit the variation of gradients in $\Gamma_{1,\beta}$, that is mainly the term $\cII_1$ in \eqref{cII1}. In the diffusive case ($\nu=2$) this reduces to a Hessian term. Exploiting this Hessian is a classical topic: it corresponds to the difference between $\CD(K,\infty)$ and $\CD(K,N)$ bounds. It will be the subject of the present section. Unfortunately (or fortunately for later work), most if not all the recipes of the diffusive case will turn out difficult to adapt to the Boltzmann situation, as will be reviewed in Section \ref{secfailed}. However, in Section \ref{secheatk} a strategy will be presented, which is based on the reduction to the diffusive case and suffices to handle all the classical situations which are not covered by Theorem \ref{thmcurvI}.

\subsection{Dimension~2: The circle} Let us start with $d=2$, then there is no curvature, and Theorem \ref{thmcriterion} is about functions of one variable $\theta$ in $\S^1$ or $\S^1/\{\pm I\}$. In this diffusive context there is no risk of confusion with post-collisional velocities, so I shall use primes for regular derivatives in $\R$. The problem is the differential Bakry--\'Emery inequality
\begeq\label{diffBE2}
\int_{\S^1} [(\sqrt{f})']^2 \leq \frac1{K_*} \int_{\S^1} f \bigl[ (\log f)''\bigr ]^2.
\endeq
At this point there is a fascinating calculus combining nonlinearities and derivatives in dimension~1, which was discovered and rediscovered several times, in particular by McKean:
\begin{align*}
\int f [(\log f)'']^2 
& = \int f \left( \frac{f''}{f} - \frac{f'^2}{f^2}\right)^2\\
& = \int \frac{(f'')^2}{f} - 2 \int \frac{f''f'^2}{f^2} + \int \frac{f'^4}{f^3},
\end{align*}
and by integration by parts
\begeq\label{IPPmckean}
\int \frac{f''f'^2}{f^2} = \frac23 \int \frac{f'^4}{f^3}.
\endeq
So that
\begeq\label{flogf2}
\int f[(\log f)'']^2 
= \int \frac{f''^2}{f} - \frac13 \int \frac{f'^4}{f^3},
\endeq
an algebraic relation involving three quantities which all are natural ``higher order'' generalisations of the Fisher information.

\begin{Rk}
Choosing $\alpha\in (0,1)$ we find
\[ \frac1{\alpha} \frac{((f^\alpha)'')^2}{f^{2\alpha -1}}
= \frac{f''^2}{f} + (\alpha-1)^2 \frac{f'^4}{f^3} + 2 (\alpha-1) \frac{f''f'^2}{f^2},\]
so by \eqref{IPPmckean}
\[ \frac1{\alpha} \int \frac{(f^\alpha)''^2}{f^{2\alpha-1}} = 
\int \frac{f''^2}{ff} + \left[ (\alpha-1)^2 + \frac43 (\alpha-1) \right] \int \frac{f'^4}{f^3}.\]
In particular, the nonnegativity of the left hand side implies
\[ \int \frac{f'^4}{f^3} 
\leq \frac1{(1-\alpha)(\alpha+\frac13)} \int\frac{f''^2}{f}, \]
which is optimal for $\alpha=1/3$ and yields
\begeq\label{optimalf'4}
\int_{\S^1} \frac{f'^4}{f^3} \leq \frac94 \int_{\S^1} \frac{f''^2}{f}.
\endeq
Plugging this back in \eqref{flogf2} further yields
\[ \int f[(\log f)'']^2  \geq \frac14 \int \frac{f''^2}{f} \geq \frac19 \int \frac{f'^4}{f^3}.\]
\end{Rk}

Back to the proof of \eqref{diffBE2}, write $g=\sqrt{f}$, the goal is
\begeq\label{diffBE2ref}
\int g'^2 \leq \frac1{K_*} \int g^2 (\log g^2)''^2.
\endeq
But
\begin{align*}
\int g^2 (\log g^2)''^2
& = 4 \int g^2 \left(\frac{g''}{g} - \frac{g'^2}{g^2}\right)^2\\
& = 4 \int g''^2 - 8 \int \frac{g'' g'^2}{g} + 4 \int \frac{g'^4}{g^2}.
\end{align*}
Now the relevant integration by parts reads
\[ \int \frac{g'' g'^2}{g} = \frac13 \int \frac{g'^4}{g^2}, \]
so
\begeq\label{ineqflogf2}
\int f (\log f)''^2 = 4 \int g''^2 + \frac43 \int \frac{g'^4}{g^2} 
=  4 \int (\sqrt{f})''^2 + \frac1{12} \int \frac{f'^4}{f^3}.
\endeq
Combining this with the usual Poincar\'e inequality on $\S^1$ yields
\[ \int f \Gamma_2 (\log f) = \int f(\log f)''^2 \geq 4 \int (\sqrt{f})''^2
\geq 4 \int (\sqrt{f})'^2 = I(f),\]
which yields $K_* \geq 1$ and also implies the log Sobolev inequality with constant~1.
Combining this, instead, with the better Poincar\'e inequality for even functions (that is, when $f$ is $\pi$-periodic) yields
\[ \int f \Gamma_2(\log f) \geq 4 \int (\sqrt{f})''^2 \geq 16 \int (\sqrt{f})'^2 = 4 I(f), \]
which also implies log Sobolev with constant 4.

To summarise: 
\[ L_*(\RP^1) = L(\RP^1) = \lambda_1(\RP^1) = 4,\] 
and this implies the bound announced in the proof of Theorem \ref{thmfullGS}.

\subsection{Dimensions 3 and higher} \label{secdiff3d}

For $d\geq 3$ there are various ways to combine the Hessian and curvature estimates in order to estimate the best constant $K_*$ in the inequality
\begeq\label{optK*d3} \int \Gamma_1(\sqrt{f}) \leq \frac1{K_*} \int f \Gamma_2(\log f). \endeq
Here I will combine the optimal spectral gap with a change of nonlinearity. So the chain of inequalities will be
\begeq\label{lambdaK}
\int \Gamma_1(\sqrt{f}) \leq \frac1{\lambda_1} \int \Gamma_2(\sqrt{f})
\leq \frac1{\tilde{K} \lambda_1} \int f \Gamma_2(\log f).
\endeq
Note that $\int \Gamma_1(\sqrt{f}) = (1/4) I(f)$, so this implies log Sobolev with constant $\tilde{K}\lambda_1/4$. Since $L_*\leq \lambda_1$, $\tilde{K}$ can be at most 4.

Now
\[ \int \Gamma_2 (\log f) = \int f \|\nabla^2 \log f\|^2_{\HS} + (d-2) \int f |\nabla \log f|^2,\]
\[ \int \Gamma_2 (\sqrt{f}) = \int \|\nabla^2 \sqrt{f}\|^2_{\HS} + (d-2) \int |\nabla\sqrt{f}|^2,\]
and since $\int f|\nabla\log f|^2 = 4 \int |\nabla\sqrt{f}|^2$, it will be sufficient to show
\begeq\label{suffKtilde} 
\int \|\nabla^2 \sqrt{f}\|^2_{\HS} \leq \frac1{\tilde{K}} \int f \|\nabla^2\log f\|_{\HS}^2
\endeq
for some $\tilde{K}\leq 4$.

For this let us play with nonlinear changes of variables: writing $\vphi\circ f=\vphi(f)$, the three basic such formulas are 
\begeq\label{fmlnl1}
\chi'(f) \Delta f = \Delta \chi (f) - \chi''(f)\Gamma_1(f),
\endeq
\begeq\label{fmlnl2}
\Gamma_1(\vphi(f)) = (\vphi')^2(f)\, \Gamma_1(f),
\endeq
\begeq\label{fmlnl3}
\Gamma_2(\vphi(f)) = (\vphi')^2(f)\, \Gamma_2(f) + 2 \vphi'(f)\, \vphi''(f) \GGamma (f) + (\vphi'')^2(f)\, \Gamma_1(f)^2,
\endeq
where I write indifferently $\Gamma_i(f) = \Gamma_i(f,f)$, $\Gamma_1=\Gamma$, and
\[ \GGamma(f) = \frac12 \Gamma\bigl(f,\Gamma(f,f)\bigr) \]
(morally this is $\<\nabla^2 f\cdot \nabla f,\nabla f\>$).

Dominique Bakry has been virtuose in showing the power of these local identities, combined with a fourth, integral, formula:
\begeq\label{fmlnl4}
\int \psi(f) (\Delta f)^2 =
\int \Bigl[ \psi(f)\, \Gamma_2(f) + 3 \psi'(f)\,\GGamma(f) + \psi''(f)\, \Gamma(f)^2 \Bigr]^2.
\endeq
(Notice that $\psi(f) = 1/f$ in dimension~1 reduces to \eqref{IPPmckean}.) Let us apply the above formula \eqref{fmlnl3} with $f=g^2$ and $\vphi(g) = \log g$:
\[ \Gamma_2(\log g) = \frac1{g^2} \Gamma_2(g) - \frac2{g^3} \GGamma (g) + \frac1{g^4}\Gamma(g)^2;\]
or equivalently
\[ f \Gamma_2 (\log f) = 
4 g^2 \Gamma_2(\log g) = 4 \Gamma_2(g) - 8 \frac{\GGamma(g)}{g} + 4 \frac{\Gamma(g)^2}{g^2};\]
using $\GGamma(\log f) = 8\, \GGamma(\log g) = 8\,(\GGamma(g)/g^3 - \Gamma(g)^2/g^4)$ and $\Gamma(\log f)^2 =16\, \Gamma(\log g)^2$, we obtain
\begeq\label{Gamma2g} \int \Gamma_2(g) = \frac14 \int f \Gamma_2 (\log f)
+ \frac14 \int f \GGamma(\log f) + \frac1{16} \int f\Gamma(\log f)^2.
\endeq

From \eqref{fmlnl4},
\[ \int f (\Delta \log f)^2 =
\int f \Gamma_2(\log f) + 3 \int f \GGamma(\log f) + \int f \Gamma(\log f)^2,\]
so
\[ \int f \GGamma(\log f) = \frac13 \left( \int f \bigl[ (\Delta \log f)^2 - \Gamma_2(\log f) \bigr] \right)
- \frac13 \int f \Gamma(\log f)^2.\]
Plugging this back in \eqref{Gamma2g},
\begin{multline*}  \int \Gamma_2(\sqrt{f}) = 
\frac14 \int f \Gamma_2(\log f) + \frac1{12} \left( 
\int f \bigl [(\Delta \log f)^2 - \Gamma_2(\log f) \bigr] \right) - \frac1{12} \int f \Gamma(\log f)^2\\
+ \frac1{16} \int f \Gamma(\log f)^2;
\end{multline*}
that is
\begeq \label{intGamma2sqrt}
\int \Gamma_2(\sqrt{f}) = 
\frac14 \int f \Gamma_2(\log f) + \frac1{12} \left( 
\int f \bigl [(\Delta \log f)^2 - \Gamma_2(\log f) \bigr] \right) - \frac1{48} \int f \Gamma(\log f)^2.
\endeq

The $\CD(d-2,d-1)$ criterion yields
\[ \Gamma_2(u) \geq \frac{(\Delta u)^2}{d-1} + (d-2)\, \Gamma(u),\]
or
\[ (\Delta u)^2 - \Gamma_2(u) \leq (d-2) [ \Gamma_2(u) - (d-1)\Gamma(u)].\]
Plugging this back in \eqref{intGamma2sqrt} yields
\begin{align*}
\int \Gamma_2(\sqrt{f})
& \leq \left(\frac14 + \frac{d-2}{12}\right)
\int f \Gamma_2(\log f) - \frac{(d-2)(d-1)}{12}\int f\Gamma(\log f) 
- \frac1{48} \int f \Gamma(\log f)^2\\
& = \left(\frac{d+1}{12}\right)
\int f \Gamma_2(\log f) - \frac{(d-2)(d-1)}{12} \int f \Gamma(\log f) - \frac1{48} \int f \Gamma(\log f)^2.
\end{align*}
This shows that
\[ \tilde{K} \geq \frac{12}{d+1}, \qquad \forall d\geq 2.\]
Combining this in \eqref{lambdaK} with $\lambda_1(\RP^{d-1}) = 2d$ yields for the optimal constant in \eqref{optK*d3}
\begeq\label{myKstar}  K_* \geq \frac{24\, d}{d+1}  \endeq
As $d$ becomes large this is a poor estimate but for $d\leq 6$ it improves on the obvious bound $L_*(\RP^{d-1}) \geq L_*(\S^{d-1})$, and for $d=2$ it yields again $K_*=16$ (optimal), for $d=3$, $K_*\geq 18$ (not bad) and it is anyway bigger than 16 for any $d$, completing the proof of Theorem \ref{thmfullGS}.

\bibnotes

For $d=2$ the key integration by parts \eqref{IPPmckean} and its use appear in many places but the oldest reference that I am aware of is McKean \cite{mck:kac:65}. 

Formulas \eqref{fmlnl2}--\eqref{fmlnl4} can be found in various places of the literature on $\Gamma_2$ and logarithmic Sobolev inequalities; e.g. Bakry \cite{bakry:lnm:94}. These are also Lemmas 5.4.3 and 5.5.5 in \cite{toulouse:sobolog}.

For higher dimensions there are many works about the estimate of the log Sobolev and differential Bakry--\'Emery constants \cite{BGL:book} on $\S^{d-1}$, but estimates about best constants on $\RP^{d-1}$ are much more rare: Fore increasingly precise results, see Saloff-Coste \cite{lsc:rate:94}, Fontenas \cite{fontenas:Jacobi:98}, Guillen--Silvestre \cite{GS:landaufisher} and Ji \cite{ji:BE}.

\section{Exploiting gradient variations II: failed attempts} \label{secfailed}

How to adapt or extend the good theory and tools from the diffusive case, to more general Boltzmann equation, possibly of fractional diffusion nature? Reviewed here are several attempts which did not succeed, but may possibly be partly saved.

\subsection{Decomposing the inequality}

The desired inequality
\[ \int \Gamma_\beta(\sqrt{F}) \leq \frac1{K_*} \int F \Gamma_{1,\beta}(\log F)\]
involves at the same time a change of derivation order (1 more on the right hand side) and of nonlinearity ($\sqrt{\ }$ vs $\log$). If a direct approach is elusive, one can try and handle one change at a time, decomposing in any of the following two ways:

(a) $K_*\geq \ov{K} \tilde{K}$, where $\ov{K}$, $\tilde{K}$ are best in
\begeq\label{Kbar}
\int \Gamma_\beta(\sqrt{F}) \leq \frac1{\ov{K}} \int \Gamma_{1,\beta}(\sqrt{F})
\endeq
\begeq\label{Ktilde}
\int \Gamma_{1,\beta}(\sqrt{F}) \leq \frac1{\tilde{K}} \int F\, \Gamma_{1,\beta}(\log F).
\endeq

(b) $K_*\geq \dot{K} \hat{K}$, where $\dot{K}$, $\hat{K}$ are best in
\begeq\label{Kdot}
\int \Gamma_\beta(\sqrt{F}) \leq \frac1{\dot{K}} D_\beta(F)
\endeq
\begeq\label{Khat}
D_\beta(F) \leq \frac1{\hat{K}} \int F\,\Gamma_{1,\beta}(\log F).
\endeq
Here
\[ D_\beta(F) = \frac12 \iint\bigl(F(\sigma)-F(k)\bigr) \bigl(\log F(\sigma)-\log F(k)\bigr)\,\beta(k\cdot\sigma)\,dk\,d\sigma\]
is the dissipation of Boltzmann's $H$ functional along the operator $L_\beta$.
In all these inequalities, $F$ is assumed even.

It turns out that two of the above constants are easy to estimate. From the linear discussion we recall that $\ov{K}\geq \lambda_1 (\RP^{d-1}) = 2d$ (with equality in general). And from the elementary inequality $4(\sqrt{x}-\sqrt{y})^2 \leq (x-y)\log(x/y)$ we have $\dot{K}\geq 4$ (with equality in general). This reduces the problem to the estimates of either $\tilde{K}$ (as I have done in the diffusive case) or $\hat{K}$ (which in the diffusive case is just as tricky as estimating $K_*$).

\subsection{Nonlocal $\Gamma_2$ changes of variables}

Nonlinear changes of variables, both for $\Gamma$ and $\Gamma_2$, proved powerful in the previous section and it is natural to try and repeat them in the nonlocal context of the Boltzmann collision operator. A seemingly good news is that all four formulas \eqref{fmlnl1}--\eqref{fmlnl4} have nonlocal analogues. The analogue of \eqref{fmlnl1} is
\begeq\label{fmlnl1'} \chi'(f) L_\beta f = 
L_\beta (\chi(f)) - \frac12 \int \Bigl( \chi(f(\sigma)) - \chi(f(k)) - \chi'(f(k))\bigl[f(\sigma)-f(k)\bigr]\Bigr) \beta(k\cdot\sigma)\,d\sigma,
\endeq
note that inside the large brackets the expression looks like $\chi''(f) [f(\sigma)-f(k)]^2$. The analogue of \eqref{fmlnl2} is the rather tautological
\begeq\label{fmlnl2'} \Bigl( \vphi(f(\sigma)) - \vphi(f(k))\Bigr)^2
= \left( \frac{\vphi (f(\sigma)) - \vphi(f(k))}{f(\sigma)-f(k)}\right)^2
\bigl( f(\sigma)-f(k)\bigr)^2.
\endeq
The analogue of \eqref{fmlnl3} is the less obvious
\begin{multline} \label{fmlnl3'}
\Bigl| \Phi'(f(\sigma))\nabla f(\sigma) - \Phi'(f(k))\nabla f(k)\Bigr|^2_{k,\sigma}
 = \Phi'(f(k))^2\, \bigl| \nabla f(\sigma)-\nabla f(k)\bigr|^2_{k,\sigma}\\
 + 2 \Phi'(f(k)) \bigl[ \Phi'(f(\sigma)) -\Phi'(f(k))\bigr] \nabla f(\sigma)\cdot \bigl[ \nabla f(\sigma) - P_{k\sigma}\nabla f(k)\bigr]\\
 + \bigl[\Phi'(f(\sigma)) - \Phi'(f(k))\bigr]^2\, |\nabla f(\sigma)|^2.\end{multline}

Finally the analogue of \eqref{fmlnl4} is
\begin{multline}\label{fmlnl4'}
\int \psi(f)\,L_\beta f \Delta f =
\frac12 \iint \psi(f(k)) \bigl|\nabla f(\sigma) -\nabla f(k)\bigr|^2_{k,\sigma} \beta(k\cdot\sigma)\,dk\,d\sigma 
\\ 
+ \frac12 \iint \Bigl [ 2\psi'(f(\sigma)) \bigl [ f(\sigma)-f(k)\bigr ]^2
+ \bigl(\psi(f(\sigma))-\psi(f(k))\bigr) \Bigr] \\
\left( \frac{|\nabla f(\sigma)|^2}2 - \frac{|\nabla f(k)|^2}2 \right)\,\beta(k\cdot\sigma)\,dk\,d\sigma\\
+ \frac12 \iint \bigl[ \psi'(f(\sigma))-\psi'(f(k))\bigr]\,
\bigl(f(\sigma)-f(k)\bigr)\, |\nabla f(k)|^2\,\beta(k\cdot\sigma)\,dk\,d\sigma.
\end{multline}

At first sight this looks promising, and also one can hope to exploit 
\[ \int (\Delta g) (L_\beta g) = \int \Gamma_{1,\beta}(g)^2. \]
But notice that expressions which were similar in the local case, in particular in \eqref{fmlnl3} and \eqref{fmlnl4}, are now different, as in \eqref{fmlnl3'} and \eqref{fmlnl4'}. So arranging the nonlocal formulas together necessarily comes at a cost in the constants. For a classical problem of partial differential equation estimate this might not be a big issue, but here the value of constants does matter, and the problem is rather tight! In the end I did not manage to put this strategy up and alive, but who knows.

\subsection{Semigroup one step beyond}

What about going for a semigroup argument to estimate
\[ I'_\beta(F) = \int F \bigl|\nabla \log F(\sigma)-\nabla\log F(k)\bigr|^2_{k,\sigma}\,\beta(k\cdot\sigma)\,dk\,d\sigma\]
from its dissipation via the heat equation on $\S^1$?

By commutation of semigroups, this amounts to estimate the dissipation of
\[ J(F) = \int F (\log F)''^2 \]
via the Boltzmann semigroup on $\S^1$. (The Boltzmann-dissipation of the heat-dissipation of Fisher's information is the same as the heat-dissipation of the Boltzmann-dissipation of Fisher's information.)

Let us start with $d=2$. Similarly to what was done in Section \ref{secGamma}, let us define
\[ \Gamma_{J,\beta}
(F) = \bigl( F (\log F)''^2\bigr) - (Lf) (\log f)''^2 - 2 f (\log f)''\left(\frac{Lf}{f}\right)''\]
in the hope to find a good leading nonnegative term. It turns out that there is a good algebra and the result is neat:
\[ \Gamma_{J,\beta}(F) (x)
= \int F(y) \Bigl( \bigl[\ell''(y) - \ell''(x) \bigr]^2 - 2 \ell''(x) \bigl( \ell'(y)-\ell'(x)\bigr)^2\Bigr)\,\beta(\cos(x-y))\,dy,\]
where $\ell = \log f$. But then, what to do? It is normal that log concavity improves the estimates, but on $\S^1$ functions cannot be log concave everywhere. Is there a nice inequality involving this functional? It seems, we have reduced a tricky problem to another one. Furthermore, in view of the counterexample in Section \ref{secex} this can turn out nicely only if $\beta$ has some regularity, and so far nothing has been used. Should one look for a clever integration by parts involving $\beta$? Again, who knows.

\subsection{Interpolation}

Consider the family of inverse power laws interactions, as in Section \ref{sectax}. In dimension 3, we are able to handle the ``right'' end $\nu=2$ (diffusive case, Section \ref{secdiffusive}); and also the ``left'' end $\nu=0$ (hard spheres) thanks to either Theorem \ref{thmcurvI} or Theorem \ref{thmcex}(ii). It would be nice if we could interpolate between both ends in the inequality
\[ \int \Gamma_\beta(\sqrt{F}) \leq \frac{1}{K_*} \int F \Gamma_{1,\beta}(\log F).\]
and cover all power laws with exponents $s \in (2,\infty)$. More generally, once the covered area comprises $\gamma=0$, $\nu=0$, $\nu=2$ (at least part of a ``H'' in the diagram), it is tempting to hope that interpolation fills the whole region.

But obstacles are not easy to overcome. The functionals involved are homogeneous, but not convex. If working with inverse power laws, the formulas for the collision kernels are certainly analytic, but implicit and indirect. Estimating the complex parts in an attempt of the three line lemma (complex interpolation) is not obvious. So I do not know if the idea can be saved.

\bibnotes

A pedagogical introduction to the formulas of changes of variables in diffusion processes can be found in the collective work \cite{toulouse:sobolog}. A pedagogical introduction to classical interpolation theory, together with many references, can be found in Lunardi \cite{lunardi:interpolation}.

In the context of the Boltzmann equation, nonlocal formulas for $\Gamma_1$ under change of functions were introduced by Alexandre \cite{alex:renorm:99} to solve the problem of ``renormalising'' the Boltzmann equation without cutoff; then the two of us simplified and improved the formulas and estimates \cite{ADVW,AV:boltz:02,AV:landau:04}. See also Section \ref{subsechigher} below.

\section{Exploiting gradient variations III: combinations of heat kernels} \label{secheatk}

A powerful principle in harmonic analysis consists in representing functions and operators as combinations of heat kernels. In this way one may perform regularisation and interpolation while take advantage of the good constants and algebra coming from diffusion. This is obviously a core principle of harmonic analysis. Here is a basic example:
\begeq\label{Deltaalpha} 
\int_0^\infty (e^{t\Delta} -1)\,\frac{dt}{t^{1+\alpha}} = - c_\alpha (-\Delta)^{\alpha}
\endeq
for some positive constant $c_\alpha = c(\alpha,d)>0$, $0<\alpha<1$; so nice estimates on $e^{t\Delta}$ may imply nice estimates for the fractional diffusion $(-\Delta)^\alpha$. This point of view will lead to a more precise estimate than Theorem \ref{thmIpos}, which was by comparison to a constant kernel.

To motivate the method in our Boltzmann context, assume that $\beta={\cal K}_t$ is the heat kernel at some time $t>0$ on $\S^{d-1}$: that is
\[ e^{t\Delta} F(k) = \int_{\S^{d-1}} F(\sigma) {\cal K}_t(k\cdot\sigma)\,d\sigma. \]
(Do not be mistaken: $t$ here is no longer the time variable in the Boltzmann equation, it is just a parameter.) The regularising property of the heat equation implies that ${\cal K}_t$ is smooth for any $t>0$. Obviously $\int_{\S^{d-1}} {\cal K}_t(k\cdot\sigma)\,d\sigma =1$ and the linear Boltzmann operator associated with $\beta$ is
\[ L_{{\cal K}_t} F = e^{t\Delta}F - F.\]

Applying successively \eqref{-I'Lf}, \eqref{GILBF}, Fubini, the last bit of Proposition \ref{propPsk}(x), and \eqref{Gamma1beta},
\begin{align}
 \label{-I'FKt} -I'(F) \cdot L_{{\cal K}_t} F 
 & = \int_{\S^{d-1}} \Gamma_{I,L_\beta}(F) \\ \nonumber
 & = \iint _{\S^{d-1}\times\S^{d-1}} 
 F(\sigma)\, \bigl| \nabla \log F(\sigma) -\nabla \log F(k)\bigr|_{k,\sigma}^2\,{\cal K}_t(k\cdot\sigma)\,d\sigma\,dk \\ \nonumber
 & = \int_{\S^{d-1}} F(\sigma) \left( \int_{\S^{d-1}} \bigl| \nabla \log F(\sigma) - \nabla\log F(k)\bigr|_{k,\sigma}^2 {\cal K}_t(k\cdot\sigma)\, dk \right)\,d\sigma\\ \nonumber
 & = \int_{\S^{d-1}} F(\sigma) \left( \int_{\S^{d-1}} \bigl| \nabla \log F(k) - \nabla\log F(\sigma)\bigr|_{\sigma,k}^2 {\cal K}_t(k\cdot\sigma)\, dk \right)\,d\sigma\\ \nonumber
 & = 2 \int F\, \Gamma_{1,{\cal K}_t}(\log F).
 \end{align}
But by convexity of $I$,
\begeq\label{convexI}
-I'(F) \cdot L_{{\cal K}_t}F = -I'(F)\cdot \bigl( e^{t\Delta} F - F\bigr)
\geq I(F) - I(e^{t\Delta}F).
\endeq
Also
\[ -I'(F)\cdot \Delta F = \int F \Gamma_2(\log F) \geq 2 L_* I(F),\]
where $L_*$ is the optimal constant in the differential Bakry--\'Emery inequality, as in \eqref{optDLS}, so by Gronwall (along the heat equation),
\[ I(e^{t\Delta} F) \leq e^{-2L_* t} I(F).\]
Plugging this back in \eqref{convexI} yields
\begeq\label{FG1Ktgeq}
\int_{\S^{d-1}} F\,\Gamma_{1,{\cal K}_t}(\log F) \geq
\left(\frac{1-e^{-2L_*t}}{2}\right) I(F).
\endeq

On the other hand, using successively the definition of $\Gamma_\beta$ (Definition \ref{defcomLB}), the symmetry of ${\cal K}_t$ and the dissipation of $L^2$ norm,
\begin{align}
\int \Gamma_{{\cal K}_t} (\sqrt{F})
& = - \int \sqrt{F} L_{{\cal K}_t} \sqrt{F}  \nonumber \\
& = \int F - \int \sqrt{F} e^{t\Delta} \sqrt{F} \nonumber \\
& = \int F - \int \bigl( e^{(t/2)\Delta} \sqrt{F}\bigr)^2 \nonumber \\
& = \int_0^t \left(- \frac{d}{ds} \int \bigl(e^{(s/2)\Delta}\sqrt{F}\bigr)^2\right)\,ds \nonumber \\
& = \frac12 \int_0^t 2\int \bigl| \nabla e^{(s/2)\Delta} \sqrt{F}\bigr|^2\,ds. \label{intermGKtsup}
\end{align}
From 
\[ \ddt{\pa_t u =\Delta u} \int |\nabla u|^2 = - 2 \int |\nabla^2 u|^2 \leq -2\lambda_1 \int |\nabla u|^2\]
(no confusion: $\lambda_1$ here is the spectral gap for $-\Delta$, not for $L_\beta$), we have
\[ \int  \bigl| \nabla e^{(s/2)\Delta} \sqrt{F}\bigr|^2 \leq e^{-\lambda_1 s} \int |\nabla\sqrt{F}|^2
= \frac{e^{-\lambda_1 s}}{4} I(F).\]
Plugging this back in \eqref{intermGKtsup}, 
\begeq\label{GKtleq}
\int_{\S^{d-1}} \Gamma_{{\cal K}_t}(\sqrt{F}) \leq \left(\frac{1-e^{-\lambda_1 t}}{4\lambda_1}\right)\, I(F).
\endeq
The combination of \eqref{FG1Ktgeq} and \eqref{GKtleq} yields
\[ \int_{\S^{d-1}} \Gamma_{{\cal K}_t}(\sqrt{F}) \leq
\frac{1}{2\lambda_1} \left( \frac{1-e^{-\lambda_1 t}}{1-e^{-2L_*t}}\right)\,
\int_{\S^{d-1}} F\,\Gamma_{1,{\cal K}_t}(\log F),\]
in other words
\[ \beta = {\cal K}_t \Longrightarrow \quad
K_* \geq 2\lambda_1 \left(\frac{1-e^{-2L_*t}}{1-e^{-\lambda_1 t}}\right). \]
As $t$ varies from 0 to $\infty$ this interpolates in a monotone way between $4 L_*(\RP^{d-1})$ (consistent with the fact that $L_{{\cal K}_t}$ for small $t$ is approximately $t$ times $\Delta$) and $2\lambda_1(\RP^{d-1})= 4d$ (consistent with \eqref{inf4d} and the fact that $L_{{\cal K}_t}$ for large $t$ is approximately the averaging operation on the sphere). It follows that the constant $K_*$ is never below the minimum of those two numbers.

\begin{Rk} At the moment it is not known exactly which is the optimal value $L_*$. In these notes I have proven that $L_*\geq 6d/(d+1)$. Sehyun Ji has shown that
\begeq\label{L*RPd} L_*(\RP^d)\geq  d+ 3 - \frac1{d-1}. \endeq
(Compare with the bound by Rothaus \eqref{LRPd}, which is $L \geq d+3-4/d^2$.) In any case it shows that
\[ \min (4L_*,2\lambda_1) = 2\lambda_1 = 4d,\]
and ironically it is now $\lambda_1$ the limiting factor in our inequalities.
\end{Rk}

Let us summarise with the following statement:

\begin{Prop}[Functional inequality for heat kernels]
For any $t>0$, for all even functions $F:\S^{d-1}\to\R_+$,
\[ \int_{\S^{d-1}} F \Gamma_{1,{\cal K}_t} (\log F) \geq 4 d \int_{\S^{d-1}} \Gamma_{{\cal K}_t} (\sqrt{F}).\]
\end{Prop}
 
Once that striking property $K_*(\beta)\geq 4d$ has been established for any $\beta = {\cal K}_t$, it readily extends to any linear combination, discrete or continuous, of such kernels, say
\[\beta(k\cdot\sigma) = \int_0^\infty {\cal K}_t(k\cdot\sigma)\,\lambda(dt)\]
for some (nonnegative) measure $\lambda$ on $\R_+$.
This means that the Boltzmann operator defined by the kernel $\beta$ is 
\[ L_\beta f (k) = -\int_0^{\infty}  \int_{\S^{d-1}} \bigl[ f(k) -{\cal K}_t(k\cdot\sigma) f(\sigma)\bigr]\,\lambda(dt)\,d\sigma  \]
or
\begeq\label{LbetaDelta}
L_\beta = -\int_0^\infty (I-e^{t\Delta})\,\lambda(dt).
\endeq
Warning: A possible source of confusion is that the kernel of an operator $L$ is usually defined via the formula $Lf (x) = \int k(x,y) f(y)\,dy$; but in the context of Boltzmann equation, the formula is $Lf(x) = \int k(x,y) [f(y)-f(x)]\,dy$. This is just a terminology issue. For a kernel of infinite mass, the second convention still makes sense. In any case, I will write informally \eqref{LbetaDelta} as
\begeq\label{LbetaDelta2} \beta = \int_0^\infty {\cal K}_t\,\lambda(dt). \endeq
The above discussion is summarised by the following estimate (Compare Corollary \ref{corNLDLS}):

\begin{Cor}[Regularity-improved nonlocal differential log Sobolev] \label{propheat}
If $\beta:[-1,1]\to\R_+$ satisfies $\int [1-k\cdot\sigma]\, \beta(k\cdot\sigma)\,d\sigma<\infty$ on $\S^{d-1}$ and admits the integral heat kernel representation
\[ \beta(k\cdot\sigma) = \int_0^\infty {\cal K}_t(k\cdot\sigma)\,\lambda(dt) \]
for some measure $\lambda$ on $\R_+$, then for all even functions $F:\S^{d-1}\to\R_+$,
\[ \int_{\S^{d-1}} F \Gamma_{1,\beta} (\log F) \geq 4 d \int_{\S^{d-1}} \Gamma_\beta (\sqrt{F}).\]
In other words, \eqref{eqBcrit} holds with $K_*\geq 4d$.
\end{Cor}

Asking for the heat kernel representation is the same as asking for $L_\beta$ to be of the form $-g(-\Delta)$, where $g(\omega) = \int_0^\infty (I-e^{-\omega t}) \,\lambda(dt)$. So $g$ is the Laplace transform of $\lambda$, up to a constant term. A famous criterion by Bernstein says that this amounts for $g$ to be nonnegative and completely monotonous, that is, $(-1)^n g^{(n)}\geq 0$ for all $n\in\N$. Such functions $g$ are called {\em Bernstein functions}, and Bernstein's representation theorem precisely says that a completely monotonous function $g:\R_+\to\R_+$ can be written
\[ g(\omega) = a \omega + b + \int_0^\infty (1-e^{-\omega t})\,\lambda(dt)\]
for some Borel measure $\lambda$ with $\int \min(1,t)\,\lambda(dt) <\infty$. Note that accordingly
\[ - g(-\Delta) = a \Delta - b I + \int_0^\infty (e^{t\Delta}-I)\,\lambda(dt) \]
is the combination of the usual diffusion, a constant, and a possibly non-cutoff linear Boltzmann operator. In the case under discussion here, only the latter part is relevant, so we may impose $g(0)=0$, and $g(\omega)/\omega \to 0$ (sublinearity) to impose $a=b=0$.

Now the following theorem is a direct consequence of Corollary \ref{propheat} and Theorem \ref{thmcritsyn}:

\begin{Thm} \label{thmFheat}
Let $B(v-v_*,\sigma) = |v-v_*|^\gamma \beta(k\cdot\sigma)$ be a collision kernel in dimension $d$. Assume that the angular kernel $\beta$ is associated with $-g(-\Delta)$ for some $g$ which is nonnegative, completely monotone, sublinear, such that $g(0)=0$. Further assume that $|\gamma| \leq 2\sqrt{d}$. Then $I$ is nonincreasing along solutions of the spatially homogeneous Boltzmann equation with kernel $B$.
\end{Thm}

\begin{Ex} The main cases of application for this theorem are fractional powers of the Laplacian, choosing $g(\omega) = \omega^{\nu/2}$ for some $\nu \in (0,2)$. It is known that such a kernel has the same angular singularity as the Boltzmann kernel for inverse power forces like $r^{-s}$ if $s= 1+ (d-1)/\nu$. Since $2\sqrt{d}\geq \min(d,4)$ for all $d\geq 2$, in this way we have the monotonicity property for a natural family of collision kernels covering all values of $\gamma,\nu$ of interest for the Boltzmann equation (recall Figure \ref{figtaxonomy}).
\end{Ex}

The latter example, as well as motivation from physics, naturally lead to ask whether these power-law induced kernels fall in the range of the theorem. But as recalled in Section \ref{sectax}, these kernels are tricky and defined in an indirect, implicit way. They are in general distinct from fractional heat kernels, and so far nobody knows if they can be represented by heat kernels. But we may use Lemma \ref{lemperturb} to compare such a kernel to a proxy, for instance the kernel of the fractional diffusion with the same singularity, or any other combination of heat kernels.

The following theorem, stated here for collision kernels which are not necessarily in product form, summarises all this.

\begin{Thm}[Curvature-dimension induced decay estimates via heat kernel representation]
Let $B=B(|v-v_*|,\cos\theta)$ be a collision kernel in dimension $d$. Assume that for each $r>0$ there are $m_r,M_r>0$, an angular collision kernel $\beta_{0,r}=\beta_{0,r}(\cos\theta)$ and a measure $\lambda_r$ on $\R_+$ such that
\[ \beta_{0,r} = \int {\cal K}_t\,\lambda_r(dt),\qquad \int_0^\infty (1-e^{-t\omega})\,\lambda_r(dt) <\infty,\qquad
\int \beta_{0,r}(k\cdot\sigma) (1-k\cdot\sigma) <+\infty,\]
and for all $r>0$,
\[ m_r \leq \frac{\beta_r}{\beta_{0,r}} \leq M_r,\]
and
\[  \sup_{0\leq\theta\leq\pi}
\left\{ \frac{r}{B} \left|\derpar{B}{r}\right| (r,\cos\theta)\right\}  \leq 
2 \sqrt{d\, \frac{m_r}{M_r}}.\]
Then $I$ is nonincreasing along solutions of the spatially homogeneous Boltzmann equation with kernel $B$.
\end{Thm}

\begin{Rk}
This theorem uses the rather sharp bound by Ji on $L_*(\RP^{d-1})$, namely \eqref{L*RPd}. But if one applies the cruder bound proven in Subsection \ref{secdiff3d}, namely \eqref{myKstar}, one obtains in the end a similar result only with $d$ replaced by $\min (d,4)$ and as of today that is sufficient for applications to the Boltzmann equation we may think of.
\end{Rk}

Let us see how to apply this Theorem to inverse $s$-power laws in the most important cases left out by Theorem \ref{thmcurvI}, namely 

\bul $d=2$ and $-2\leq \gamma <1$ ($0< \nu\leq1$, or $2\leq s<\infty$);

\bul $d=3$ and $-3\leq \gamma < -2$ ($3/2\leq\nu<2$, or $2<s<7/3$).

There is also a little bit which is outside the ``natural'' mathematical range but still makes physical sense, in principle:

\bul $d=2$ and $-3<\gamma<-2$ ($1<\nu<2$, or $3/2<s<2$).
\sm

In each case, we may work out numerically the angular kernel of the Boltzmann equation for various values of $s$ and compare them to the angular kernel of the fractional diffusion with the same singularity, hoping that the two remain mutually bounded. While this is not a purely mathematical proof, the procedure is quite safe: we are here numerically comparing functions which are well-defined through special functions and other classical tools of numerical analysis, so errors remain under control. Dimension 2 is handled by Fourier series, dimension 3 by classical spherical harmonics, which are numerically addressed in available software. Working along these lines and trying empirically some variations, one obtains:

\bul In dimension $d=2$, the symmetrised Boltzmann kernel for $s$-power law forces stays close to the associated fractional Laplace kernel, with ratio $M/m$ always less than $\sqrt{2}$ (which is obtained for the limit of hard spheres). This yields $\ov{\gamma}\geq 2\cdot 2^{1/4}>2$: already enough to get the range $[-2,2]$. 

\bul Still in dimension 2, refining the analysis, one can see that the criterion applies up to get $\gamma > -2.7$, but fails to get all the way to the physical limit $\gamma=-3$. But by empirically changing the weight function to
\[ \lambda(dt) = \Bigl[ 1 + 2(\nu-1)^2 (1-\exp(-2t))\Bigr]\,t^{-(1+2\nu)} \, dt, \]
one obtains $\ov{\gamma}>3.3$ all throughout the range $s\in [3/2,2]$. So this covers all cases which make physical sense.

\bul In dimension $d=3$, the numerically observed $M/m$ ratio approaches $1.6$ for $\nu=1$. This yields $\ov{\gamma}> \sqrt{7.5}$, not quite sufficient for very singular kernels.  But by empirically changing the weight function to
\[ \lambda(dt) = \frac{t^{-(1+2\nu)}\,dt}{\sqrt{1+(2-\nu)t}} \]
one obtains ratios not larger than 1.1; then this yields $\ov{\gamma}\geq \sqrt{12/1.1}$ which is $>3$ (not by much!). Another more sophisticated weighting, also obtained by trial and error, is
\[ \lambda(dt) = \Bigl[ 1- \min \left(\frac{13}{8} - \frac32 s, \frac25\right) \bigl(1-\exp(-2t)\bigr) \Bigr]\, t^{1-2\nu}\,dt, \]
which yields estimates $\ov{\gamma}\geq 4.3$ all throughout the range $\gamma \in (-3,-2]$, thus a more comfortable margin, robust to numeric errors.
\med

This concludes our investigation of the monotonicity of the Fisher information. All in all, we discovered that the three basic methods presented here (curvature, positivity, Hessian variations), taken together, are enough to establish the monotonicity of the Fisher information for the vast majority of collision kernels of interest, including all power law interactions between Coulomb and hard spheres, in all dimensions. 

\begin{Rk} If one does not want to restrict to potentials decaying at least as fast than Coulomb, then there is actually one little bit still not covered: $\gamma\in (-3,-2\sqrt{2}]$ in dimension $d=4$. There is no particular motivation for that range, except mathematical consistency.
\end{Rk}

\bibnotes

Heat kernel approximation is developed at large in harmonic analysis \cite{FJW:book,stein:harmonic:book}, and combinations of simple elementary kernels are used in the theory of stochastic processes with the concept of ``subordination''. Apart from these elements of context, I am not aware of any precedent in the context of Boltzmann equation of the technique developed in this section, whose basic principle was discovered in a discussion between Luis Silvestre and Cyril Imbert in the Mathemata summer school in Crete (July 2024), explored later by the three of us in \cite{ISV:fisher}.

Bernstein proved his representation theorem in \cite{bernstein:absmon} and there is a modern textbook by Schilling, Song and Vondra\v{c}ek \cite{SSV:bernstein:book}.

The proximity of the heat kernel for power law forces and the fractional Laplace operator was noticed by a number of authors, starting at least with Desvillettes, and quantitatively exploited for regularity issues by Alexandre, Desvillettes, Wennberg and myself \cite{ADVW}. Using the ingredients in the latter work, one easily proves, for instance, that $-\<L_\beta h,h\> \geq K \<(-\Delta)^{\nu/2}h,h\>$ when $\beta$ has the form \eqref{bcostheta}.

Numerics quoted in this section was done by Silvestre \cite{silvestre:numerics} and the results published in \cite{ISV:fisher}.

\section{Regularity of solutions for singular collision kernels} \label{secreg}

Regularity theory for the spatially homogeneous Boltzmann equation classically rests on four types of estimates:

\bul moments

\bul integrability

\bul smoothness

\bul positivity (lower bounds)

\sm

While this is more or less the natural order for the four estimates, they may be combined or related in many ways, and also intertwined with the equilibration problem. It turns out that the Fisher information estimate unlocks all the remaining blanks in the theory of soft potentials, at least in the region of ``conditional regularity'' $\gamma+\nu\geq -2$, $\gamma\geq -d$.

I shall go through all four types of estimates, not searching for exhaustivity or optimality -- a complete exposition of the state of the art would easily fill up a 500-page book. I will focus on the particular case of factorised collision kernels \eqref{collkfact}; this is for simplicity, as more general forms could be handled, and only for very soft potentials (with a high singularity in the relative velocity, like $\gamma < -2$), whose regularity is most tricky, and which is precisely the compartment of the theory which was unlocked by the new Fisher information estimates. I will systematically consider first the Landau equation and then the Boltzmann equation, and most of the time only sketch the proofs, pointing to the existing literature for more information. I will also skip the uniqueness issue, which in this field always follows in practice from good a priori estimates.

Besides the assumption of very soft potential, I shall limit myself to the regime of ``conditional regularity''. So the precise bounds on the singularities in relative velocity and angular variables will be

\bul For Landau's equation,
\begeq\label{VSPLandau}
-4 <\gamma<-2, \qquad \gamma\geq -d;
\endeq

\bul For Boltzmann's equation,
\begeq\label{VSPBoltzmann}
-2 < \gamma+\nu < 0, \qquad \nu < 2, \qquad \gamma \geq -d
\endeq
(of which \eqref{VSPLandau} is obviously a limit case)

The initial datum will be denoted by $f_0$: by assumption it is a probability density; and its initial energy will be $E_0= (1/2)\int f_0(v)|v|^2\,dv$.

\subsection{Integrability} \label{subsecint}

Assume that the initial datum $f_0$ has finite Fisher information and the monotonicity property holds. Then solutions $f(t)= f(t,\cdot)$ satisfy the a priori estimate $I(f(t)) \leq I(f_0)$ for all $t\geq 0$. By Sobolev embedding
\[ I(f) =4 \|\nabla\sqrt{f}\|_{L^2(\R^d)}^2
\geq K(d) \|\sqrt{f}\|_{L^{p^\star}(\R^d)}^2 = K \|f\|_{L^{p^\star/2}(\R^d)}\]
with $p^\star = 2d/(d-2)$ for $d\geq 3$ and any $p^\star<\infty$ for $d=2$. So,
\begin{multline} \label{integr1} 
\sup_{t\geq 0} \|f(t)\|_{L^{p_0}(\R^d)} \leq C(f_0,d)\qquad
p_0 = \frac{d}{d-2} \ \text{if $d\geq 3$},\quad \forall p_0 <\infty \ \text{if $d=2$}.
\end{multline}

This is already a major information, as it rules out, in dimension $d=3$, a behaviour of the type of the quadratic heat equation $\pa_t f = \Delta f + f^2$ (which generically blows up in $L^p$ for all $p>d/2$). Furthermore, by Hardy--Littlewood--Sobolev inequality,
\[ \iint_{\R^d\times\R^d} \frac{f(v) f(v_*)}{|v-v_*|^{\alpha}}\,dv\,dv_* \leq
C(d,\alpha) \|f\|_{L^q(\R^d)}^2 \]
as soon as $\alpha = 2 d(1-1/q)$ and $0<\alpha<d$. Choosing $q$ as close to $q_0$ as these conditions allow, yields the second key a priori estimate
\begeq\label{integr2} \sup_{t\geq 0} 
\int_{\R^d\times\R^d} \frac{f(t,v)\,f(t,v_*)}{|v-v_*|^{\alpha_0}}\,dv\,dv_* \leq C (d,\alpha_0,f_0)\qquad
\forall\alpha_0,\ 0<\alpha_0<\min(4,d).
\endeq

\subsection{Weak solutions}

Various notions of weak solutions have been introduced and used in the context of the Boltzmann and Landau equations. But the above integrability bounds allow to use the most natural one, already considered by Maxwell seventy years before Sobolev and Schwartz would formalize the notion of distributions: For any smooth test function $\varphi=\varphi(v)$,
\[ \frac{d}{dt} \int_{\R^d} f\vphi = \int_{\R^d} Q(f,f)\,\vphi.\]
Let us check that indeed the right-hand side makes perfect sense under the integrability bounds of Subsection \ref{subsecint}.

For Landau, as already noted in Section \ref{secgrazing} (Proposition \ref{propLandau}),
\begin{align}
\int Q_L(f,f)\,\vphi
& = \iint ff_*\, a(v-v_*):\nabla^2 \vphi\,dv\,dv_* + 2 \iint ff_*\, b(v-v_*)\cdot\nabla\vphi \,dv\,dv_*  \nonumber \\
& = \iint ff_* \Bigl( a(v-v_*):\Bigl(\frac{\nabla^2\vphi +(\nabla^2\vphi)_*}2\Bigr)+ b(v-v_*)\cdot \bigl(\nabla\vphi- (\nabla\vphi)_*\bigr)\Bigr)\,dv\,dv_*, \label{weakLandau}
\end{align}
where symmetrisation $v\leftrightarrow v_*$ was used. Now $a = O(|v-v_*|^{\gamma+2})$ and $b= O(|v-v_*|^{\gamma+1})$ so if $\vphi$ is smooth the integrand in the above integral is $O(|v-v_*|^{\gamma+2})$. If $0\geq \gamma\geq-2 $, the integral would converge just by the mass and energy estimates, but for $\gamma<-2$ this is no longer true. (This is where ``very soft'' potentials begin.) But the regularity estimate \eqref{integr2} ensures the convergence of the integral as soon as $-d < \gamma+2< 0$.

Now for Boltzmann: 
\[ \int Q(f,f)\vphi = 
\frac14 \iiint ff_* \bigl( \vphi'+\vphi'_* -\vphi-\vphi_*\bigr)\,B\,d\sigma\, dv\,dv_*.\]
Then if $\vphi$ is smooth, for any $k\in\S^{d-1}$
\[ \int_{\S^{d-2}_{k^\bot}} \bigl(\vphi'+\vphi'_* -\vphi-\vphi_*\bigr)\,d\phi = 
O \bigl( |v-v_*|^2\theta^2 \bigr),\]
so
\[  \int_{\S^{d-1}} \bigl(\vphi'+\vphi'_* -\vphi-\vphi_*\bigr)\,B\,d\sigma
= O\left( |v-v_*|^2 \int B(1-k\cdot\sigma)\,d\sigma\right) = O(|v-v_*|^{\gamma+2}),\]
and again the integral converges thanks to \eqref{integr2}.

\subsection{Moments} \label{submoments}

Moments are the first concern in any mathematical implementation of kinetic theory, so it is worth spending some energy on them.

Let $\<v\> = \sqrt{1+|v|^2}$, so that
\[ \nabla \<v\> = \frac{v}{\<v\>},\qquad
\pa_{ij}\<v\> = \frac1{\<v\>} \left(\delta_{ij} - \frac{v_i v_j}{\<v\>^2}\right), \]
and for $s>0$ let $\vphi_s = \<v\>^s$, so that
\begin{multline} \label{vvphis} \nabla\vphi_s = s\,v\<v\>^{s-2},\qquad
\pa_{ij} \vphi_s = s \<v\>^{s-2}\left(\delta_{ij} - \frac{v_i\,v_j}{\<v\>^2}\right)
+ s(s-1) \<v\>^{s-4} v_i v_j.
\end{multline}
Let $s> 2$ be a real number: Bounding the moment of order $s$, $M_s = \int f |v|^s$, is the same as bounding $\int f \vphi_s$. For that we may use the weak formulation. 
Recall that the weighted Lebesgue spaces, $L^p_s$, are in interpolation: $L^p_s = [L^{p_0}_{s_0}, L^{p_1}_{s_1}]_{\theta}$ if $1/p = (1-\theta)/p_0 + \theta/p_1$, $s= (1-\theta)s_0 + \theta s_1$. I will denote by $C$ various constants changing from time to time and only depending on the parameters involved. 

First attempt for Landau: Using \eqref{weakLandau} and replacing $a,b$ by their expression in terms of $\Psi$,
\[ \frac{d}{dt} \int f\vphi_s
= \iint ff_* A_s(v,v_*)\,dv\,dv_*,\]
where 
\[ A_s(v,v_*) = 
\Psi(|v-v_*|) \left( \Pi_{k^\bot}: \left(\frac{\nabla^2\vphi_s + (\nabla^2\vphi_s)_*}{2}\right)
- (d-1) \frac{v-v_*}{|v-v_*|^2}\cdot \bigl( \nabla\vphi_s-(\nabla\vphi_s)_*\bigr)\right). \]
To fix ideas, assume that $\Psi(|z|) = |z|^{\gamma+2}$.
Let us separate four cases:

(a) $|v| = O(1)$, $|v_*| = O(1)$. Noting that $|\nabla\vphi_s(v) - \nabla\vphi_s(v_*)| = O(|v-v_*|)$, one easily has $|A_s(v,v_*)| = O(|v-v_*|^{\gamma+2})$.

(b) $|v-v_*|\gg 1$ and either $|v|/|v_*| \gg 1$ or $|v|/|v_*| \ll 1$. By symmetry, consider for instance the first situation. Then
\begin{align*} 
A_s(v,v_*)  & = |v|^{\gamma+2} \Bigl( s \frac{(d-1)}{2} |v|^{s-2} - s(d-1) |v|^{s-2}
+ o(|v|^{s-2}) \Bigr)\\
& = |v|^{\gamma+2} \left( -\frac{s (d-1)}{2} |v|^{s-2} + o( |v|^{s-2}) \right)\\
& \leq - \frac{(d-1)}{4} |v|^{s+\gamma} \qquad \text{ for $|v|$ large enough}
\end{align*}

(c) $|v-v_*|\gg 1$ and $|v|, |v_*|$ comparable (and thus both large). Then, assuming for instance $|v-v_*| \geq |v|/2$,
\[ |A_s(v,v_*)| \leq C |v|^{\gamma+2} \<v\>^{s-2}
\leq C \<v\>^{s+\gamma} \leq C \<v\>^{s+\gamma-2}\<v_*\>^2.\]

(d) $|v-v_*| = O(1)$ and $|v|, |v_*|$ both very large. Then $|v|, |v_*|$ are of course comparable and
\[ A_s(v,v_*) \leq C |v-v_*|^{\gamma+2} (|v|^{s-2}+|v_*|^{s-2}) \leq
C |v-v_*|^{\gamma+2} \<v\>^{\frac{s}2-1} \<v_*\>^{\frac{s}2-1}.\]
We end up with 
\begin{multline*} A_s(v,v_*) \leq -K_s (\<v\>^{s+\gamma} + \<v_*\>^{s+\gamma})
+ C (\<v\>^{s+\gamma-2}\<v_*\>^{2} + \<v\>^2\<v_*\>^{s+\gamma-2})\\
+ C |v-v_*|^{\gamma+2} \bigl(1+\<v\>^{\frac{s}2-1} \<v_*\>^{\frac{s}2-1}\bigr).
\end{multline*}

Integrating this against $ff_*\,dv\,dv_*$ yields (possibly changing the constants), after using again the Hardy--Littlewood--Sobolev inequality, and the Young inequality for the last term,
\begeq\label{intAs} \iint ff_*\, A_s(v,v_*)\,dv\,dv_* \leq
- K_s \|f\|_{L^1_{s+\gamma}} + C_s \|f\|_{L^1_{s+\gamma-2}} (1+2E_0)
+ C_s \bigl\| f \<v\>^{\frac{s-2}{2}} \bigr\|^{2}_{L^q} ,\endeq
where
\[ E_0 = \frac12 \int f |v|^2\,dv,\qquad 
\frac1{q} = 1+ \frac{\gamma+2}{2d}.\]
The best we can hope is that the first term in the right-hand side of \eqref{intAs} be dominant in the sense of controlling the contribution of high velocities. Not only will this imply that $M_s$ remains locally bounded, but also, since that first term comes with a negative sign, it will show that actually $dM_s/dt$ remains $O(1)$, hence a linear bound in time, for any $s$. After that, by interpolation with higher order moments it is possible to lower the exponent at leisure. But for that one first has to control the tricky last term in \eqref{intAs}. Let us interpolate to bound that term, which is actually
\[ \|f\|^2_{L^1_{\frac{s-2}{2}}}
\leq \|f\|^{2a}_{L^{p_0}} \|f\|^{2b}_{L^1_2} \|f\|^{2c}_{L^1_\sigma}, \]
where $1/p_0=1-2/d$ (if $d\geq 3$), and $a,b,c \in [0,1]$, $\sigma>0$ satisfy
\[ a+b+c = 1,\quad \frac1{q} = \frac{a}{p_0}+b_c,\quad
\frac{s-2}{2} = 2b+\sigma c.\]
The idea here is that both $\|f\|_{L^{p_0}}$ and $\|f\|_{L^1_2}$ are bounded uniformly, so the last bit of \eqref{intAs} will be bounded by a multiple of $\|f\|_{L^1_\sigma}$ and we wish $\sigma$ to be as small as possible.
A bit of playing shows that our best choice is $a=-(\gamma+2)/4$, $c=1/2$, $b = (\gamma+4)/4$, then $\sigma = s-\gamma-6$. Grand conclusion:
\[ \frac{dM_s}{dt} \leq -K M_{s+\gamma} + C (1+ M_{s-\gamma-6}). \]
But $s-\gamma-6<s+\gamma$ only for $\gamma>-3$, so this misses an important case. Try again!

The way out will come from another weak formulation, and the use of the entropy estimate. First note that
\begin{align*}
\frac{d}{dt}\int f\vphi\,dv
& = \int \nabla\cdot \left( \int a (\nabla-\nabla_*) ff_*\,dv_*\right)\,\vphi\,dv \\
& = - \iint a (\nabla-\nabla_*) ff_* \nabla\vphi\,dv\,dv_*\\
& = - \frac12 \iint a (\nabla-\nabla_*) ff_* \, \bigl(\nabla\vphi -(\nabla\vphi)_*\bigr)\,dv\,dv_*.
\end{align*}
Here $a=a(v-v_*) = \Psi(|v-v_*|) \Pi_{(v-v_*)^\bot}$. On the other hand Landau's dissipation functional (or entropy production functional) is
\begin{align*}
D_L(f) & = \frac12 \iint a ff_* (\nabla-\nabla_*) \log ff_* (\nabla-\nabla_*) \log ff_*\,dv\,dv_* \\
& = 2 \iint a (\nabla-\nabla_*) \sqrt{ff_*} (\nabla-\nabla_*) \sqrt{ff_*}\,dv\,dv_*.
\end{align*}
The point is that here the estimate, contrary to the Fisher information bound, becomes stronger when $\gamma$ becomes more negative and $a$ more singular. There is a cost: along the flow, $D_L(f)$ is bounded only after time-integration. But for the control of moments this will not be a problem. Now, by Cauchy--Schwarz,
\begin{align*}
\frac{d}{dt} \int f\vphi 
& = - \iint a \sqrt{ff_*}\, (\nabla-\nabla_*) \sqrt{ff_*}
\bigl(\nabla\vphi - (\nabla\vphi)_*\bigr)\,dv\,dv_*\\
& \leq C\, D_L(f)^{1/2} \left( \iint ff_* a \bigl(\nabla\vphi - (\nabla\vphi)_*\bigr) \bigl(\nabla\vphi - (\nabla\vphi)_*\bigr)\,dv\,dv_*\right)^{1/2}\\
& = C \, D_L(f)^{1/2} \iint ff_* \Psi(v-v_*) \, 
\Bigl | \Pi_{(v-v_*)^\bot} \bigl[ \nabla\vphi - (\nabla\vphi)_*\bigr] \Bigr|^2\,dv\,dv_*.
\end{align*}
Observe that if $\vphi$ is smooth and in $W^{2,\infty}$ (second derivatives in $L^\infty$), then the integrand in the last integral is $ff_*$ multiplied by $O(|v-v_*|^{\gamma+4})$ and so bounded. Let us do the estimate now for $\vphi_s$ in place of $\vphi$. Let us decompose $\Psi(z)$ into $\Psi(z) 1_{|z|\leq \lambda}+ \Psi(z) 1_{|z|>\lambda}$ (small relative velocities, large relative velocities) where $\lambda$ should be fixed large enough. Then with obvious notation
\[ \frac{d}{dt}\int f\vphi_s 
= \int Q_L^{>\lambda} \vphi_s + \int Q_L^{<\lambda} \vphi_s.
\]
For the first one, the estimate of the beginning of this section is now more favourable, since the last bit of \eqref{intAs}, coming from (a), will not be present:
\begeq\label{QLsup} \int Q_L^{>\lambda} \vphi_s  \leq - K_s M_{s+\gamma} + C_s .
\endeq
For the second one, the estimate just above yields
\[ \int Q_L^{<\lambda} \vphi \leq C \lambda^{\gamma+2} D_L(f)^{1/2} M_{s-2}.\]
By interpolation and using $\gamma\geq -4$,
\begin{align*} 
\int Q_L^{<\lambda} \vphi 
& \leq C \lambda^{\gamma+2} D_L(f)^{1/2} M_{s}^{1/2} M_{s+\gamma}^{1/2}\\
& \leq \var M_{s+\gamma} + C_\var D_L(f) M_s + C.
\end{align*}
Choosing $\var$ small enough and combining this with \eqref{QLsup} yields
\begeq\label{CCL} \frac{d}{dt} M_s \leq -K M_{s+\gamma} + C D_L(f) M_s + C. \endeq
This inequality, first obtained by Carlen, Carvalho and Lu, implies an excellent bound on $M_s$. Indeed, by time-integration, it yields
\[ M_s(t) + K  \int_0^t M_{s+\gamma}(\tau)\,d\tau 
\leq M_s(0) + C(1+t) + \int_0^t D_L(f(\tau)) M_s(\tau)\,d\tau.\]
And then by Gronwall's lemma
\begeq\label{gronwall} M_s(t) + K \int_0^t M_{s+\gamma}(\tau)\,d\tau 
\leq C (1+t) \exp \left( \int_0^\infty D_L(f(\tau))\,d\tau \right).
\endeq
This proves in particular
\[ M_s(t) \leq C(s,\gamma,f_0) (1+t), \]
where $C(s,f_0)$ depends on $f_0$ only through $E_0$, $H(f_0)$ and $M_s(f_0)$.

Now the same strategy for Boltzmann. For this we need to use two weak formulations:
\[ \int Q (f,f)\vphi = 
\frac12 \iint ff_* \bigl(\vphi' + \vphi'_* - \vphi-\vphi_*\bigr)\,dv\,dv_*\, B(v-v_*,\sigma)\, d\sigma\]
(useful for large relative velocities) and 
\[ \int Q (f,f)\vphi = 
-\frac14 \iint (f'f'_* - ff_*) \bigl(\vphi' + \vphi'_* - \vphi-\vphi_*\bigr)\,dv\,dv_*\, B(v-v_*,\sigma)\, d\sigma\]
(useful for small relative velocities).
For large relative velocities we need a good estimate of $\int (\vphi'+\vphi'_*-\vphi-\vphi_*)\,d\sigma$; for small relative velocities it will be a good estimate of $(\vphi+\vphi'_*-\vphi-\vphi_*)$, when $\vphi(v)=\vphi_s(v) = \<v\>^s$, with $s>2$.

Let's go. Write
\[ z = v-v_*,\qquad
c = \frac{v+v_*}{2}, \qquad y = \frac{z}{|c|}, \qquad k = \frac{z}{|z|} = \frac{y}{|y|}, \qquad
e = \frac{c}{|c|}.\]
Then
\[ v = c + \frac{z}{2}, \qquad v_* = c - \frac{z}2, \qquad
v'= c+ \frac{|z|}2 \sigma, \qquad v'_* = c - \frac{|z|}2\sigma.\]
Then
\[ |v|^2 = |c|^2 \left( 1 + e\cdot y + \frac{|y|^2}{4}\right)\]
\[ |v_*|^2 = |c|^2 \left( 1 - e\cdot y + \frac{|y|^2}4 \right) \]
\[ |v'|^2 = |c|^2 \left( 1 + (e\cdot \sigma) |y| + \frac{|y|^2}4\right) \]
\[ |v'_*|^2 = |c|^2 \left( 1 - (e\cdot \sigma) |y| + \frac{|y|^2}4\right).\]
For $\var\ll 1$, 
\[ (1+\var)^{s/2} = 1 + \frac{s}2 \var + \frac{s}2 \left(\frac{s}2-1\right) \frac{\var^2}2
+ \frac{s}2 \left(\frac{s}2-1\right) \left(\frac{s}2-2\right) \frac{\var^3}6 + O(\var^4).\]
It folllows 
\begeq\label{4vs} |v'|^s + |v'_*|^s - |v|^s - |v_*|^s
= |c|^s \left( \frac{s}2 \Bigl(\frac{s}2-1\Bigr) \bigl[ (e\cdot\sigma)^2|y|^2 - (e\cdot y)^2\bigr]
+ O(|y|^4) \right).
\endeq

Again, let us separate small relative velocities $|v-v_*|^{\leq\lambda}$ and larger relative velocities $|v-v_*|\geq \lambda$, informally written $|v-v_*| = O(1)$ and $|v-v_*|\gg 1$, separate the collision kernel, with obvious notation, as $Q^{\leq \lambda} + Q^{>\lambda}$, and apply two different weak formulations for these two parts:
\begin{multline} \label{weak12} \int Q(f,f) \vphi_s
= \frac12 \iint  ff_* \left( \int \left( \vphi_s(v')+\vphi_s(v'_*) - \vphi_s(v)-\vphi_s(v_*) \right) B^{>\lambda} (v-v_*,\sigma) \,d\sigma \right)dv\,dv_*\\
- \frac14 \iiint (f'f'_* - ff_*) \left( \vphi_s(v')+\vphi_s(v'_*) - \vphi_s(v)-\vphi_s(v_*)\right) B^{\leq\lambda} (v-v_*,\sigma)\,d\sigma\,dv\,dv_*.
\end{multline}

\bul For $|v-v_*|\gg 1$ and $|v|, |v_*|$ incomparable, say $|v|\gg |v_*|$, one has
\[ v'\simeq \frac{v+|v|\sigma}{2} = |v| \left( \frac{k+\sigma}2\right),
\qquad v'_* \simeq |v| \left( \frac{k-\sigma}2\right), \]
so 
\begin{align*} & \int \left( \vphi_s(v')+\vphi_s(v'_*) - \vphi_s(v)-\vphi_s(v_*) \right) B^{>\lambda} (v-v_*,\sigma)\,d\sigma \\
& \qquad\qquad\qquad  \simeq - |v|^{s+\gamma} \int_{\S^{d-1}}
\left( 1 - \left|\frac{k+\sigma}2\right|^s - \left| \frac{k-\sigma}2\right|^s\right)\, b(k\cdot\sigma)\,d\sigma \\
&\qquad\qquad\qquad
 = -|v|^{s+\gamma} |\S^{d-2}| \int_0^\pi \left[ 1- \left(\frac{1+\cos\theta}2\right)^{s/2}
- \left(\frac{1-\cos\theta}2\right)^{s/2}\right]\,b(\cos\theta)\,\sin^{d-2}\theta\,d\theta\\
&\qquad\qquad\qquad
 = -|v|^{s+\gamma} |\S^{d-2}| \int_0^\pi \bigl[ 1-(\cos(\theta/2))^s - (\sin (\theta/2))^s\bigr]\,b(\cos\theta)\,\sin^{d-2}\theta\,d\theta\\
&\qquad\qquad\qquad
 \leq -K(s,d) |v|^{s+\gamma} \int_0^{\pi} \theta^2 b(\cos\theta)\,\sin^{d-2}\theta\,d\theta.
\end{align*}

\bul For $|v-v_*|\gg 1$ and $|v|, |v_*|$ comparable, from \eqref{4vs}
\begin{align*}
& \left| \int \bigl( \vphi_s(v')+\vphi_s(v'_*) - \vphi_s(v)-\vphi_s(v_*) \bigr)  B^{>\lambda} (v-v_*,\sigma) \,d\sigma \right| \\
& \qquad\qquad\qquad  \leq C |v-v_*|^{\gamma+2} (1+|v|^{s-2} + |v_*|^{s-2}) \left(\int \sin^2\theta\, b(\cos\theta)\,\sin^{d-2}\theta\,d\theta \right)\\
&\qquad\qquad\qquad  \leq C'(1+|v|)^{s+\gamma-2} (1+|v_*|^2).
\end{align*}

Combining both estimates,
\begin{multline*} 
 \iint ff_* \int \left( \vphi_s(v')+\vphi_s(v'_*) - \vphi_s(v)-\vphi_s(v_*) \right) B^{>\lambda} (v-v_*,\sigma) \,d\sigma\,dv\,dv_* \\  \leq -K(s,d,\gamma,b) \int f \vphi_{s+\gamma} + C(s,d,\gamma,b,E_0) \int f \vphi_{s+\gamma-2}.
 \end{multline*}

Now for smaller relative velocities:
\begin{align}
&  \label{2DBCB} \left| \int Q^{<\lambda} (f,f)\,\vphi_s \right|
 \leq \frac14 \left| 
\iiint (f'f'_*-ff_*) \Bigl[ (\vphi_s)' + (\vphi_s)'_* - \vphi - (\vphi)_* \Bigr]\, B\,dv\,dv_*\,d\sigma\right| \\ \nonumber
&\qquad\qquad \leq \frac14 \left( \iiint \bigl( \sqrt{f'f'_*} - \sqrt{ff_*} \bigr)^2 B\,dv\,dv_*\,d\sigma \right)^{1/2}\\ \nonumber
&\qquad\qquad\qquad\qquad\qquad
\left(\bigl(\sqrt{f'f'_*}+\sqrt{ff_*}\bigr)^2 \Bigl[ (\vphi_s)' + (\vphi_s)'_* - \vphi - (\vphi)_* \Bigr]^2\,B\,dv\,dv_*\,d\sigma\right)^{1/2}\\ \nonumber
&\qquad\qquad \leq 2 D_B(f)^{1/2} \left( \iint ff_*\,{\cal C}(B,\vphi_s)\,dv\,dv_*\right)^{1/2}
\end{align}
where
\[ {\cal C}(B,\vphi_s) = \int \Bigl[ (\vphi_s)' + (\vphi_s)'_* - \vphi_s - (\vphi_s)_* \Bigr]^2\,B\,d\sigma.\]
Playing with \eqref{4vs}, we see that the term within square brackets in the integrand above is $O(|v-v_*|^2\theta) (1+|v|^{s-2} + |v_*|^{s-2})$. When $|v|,|v_*|=O(1)$ the square is $O(|v-v_*|^4\theta^2)$.
When $|v|,|v_*|\gg 1$, then necessarily $|v|,|v_*|$ are comparable to each other and the square is bounded by
$C|v-v_*|^4\theta^2 |v|^{s-2} |v_*|^{s-2}$. All in all,
\begin{align*} & \left| \iint ff_*\,{\cal C}(B,\vphi_s)\,dv\,dv_*\right| \\
 & \leq C(s,d,\gamma) \left(\int \theta^2 b(\cos\theta)\,\sin^{d-2}\theta\,d\theta\right)
\iint ff_* |v-v_*|^{\gamma+4} 1_{|v-v_*|\leq\lambda} \<v\>^{s-2}\<v_*\>^{s-2}\,dv\,dv_*\\
 & \qquad\qquad \leq C \left( \int f \vphi_{s-2}\right)^{2}.
\end{align*}
Plugging back in \eqref{2DBCB} and interpolating $L^1_{s-2}$ between $L^1_{s-4}$ and $L^1_{s}$ yields
\begin{align*}
\left| \int Q^{<\lambda} (f,f)\,\vphi_s \right| 
& \leq C D_B(f)^{1/2} \|f\|_{L^1_s}^{1/2} \|f\|_{L^1_{s-4}}^{1/2}\\
& \leq \var \|f\|_{L^1_{s-4}} + C D_B(f) \|f\|_{L^1_s}.
\end{align*}

Putting the pieces together, we obtain \eqref{CCL} and it follows that $M_s(t)$ remains $O(1+t)$ globally in time.
\med

Once that conclusion is obtained, by interpolation it follows readily that: For any $s$ and any $\var>0$, there is $s_0 = s/\var$ such that if $M_{s_0}(f_0)<\infty$ then $M_s(t)$ remains $O((1+t)^\var)$ for all $t\geq 0$, and the constant can be made explicit in terms of $s,d,\gamma,b,\var, M_{s/\var}(f_0), H(f_0)$. To write it in a synthetic form:
\begeq\label{f0inf}
f_0 \in L^1_{\infty}(\R^d)\cap L\log L(\R^d) \Longrightarrow \qquad \|f\|_{L^1_\infty} = O((1+t)^0).
\endeq

This excellent control on the moments allows to truncate all large velocities, which is a key step in all subsequent developments. The assumption of finite moments of arbitrarily large order is usually quite reasonable in the context of kinetic theory.

\subsection{Higher integrability} \label{subsechigher}

At this stage large velocities are under control and the Landau equation is a quasilinear parabolic equation with diffusion matrix $\ov{a} = a\ast f$. The symmetric matrix $a(z)$ is degenerate in the direction $z$ and has eigenvalues roughly $|z|^{\gamma+2}$ and $0$. When averaged over $f$, which is not concentrated on a subspace (due to the $L^{p_0}$ bound or just the bound on $H(f)$), one finds
\begeq\label{azsup}
\ov{a}(v) \geq K_r |v|^{\gamma}\Pi_v + K_\bot |v|^{\gamma+2} \Pi_{v^\bot}, \qquad
\ov{a}(v) \leq C |v|^{\gamma+2}.
\endeq
The gap between exponents $\gamma$ and $\gamma+2$ could be a problem, but the moment estimates are enough to handle this. Very classically, let us look for energy-type estimates: if $\Phi_p(f) = f^p$, $p>1$, then from
$\pa_t f = \nabla\cdot (\ov{a} \nabla f - \ov{b} f)$, $\ov{b}=\nabla\cdot \ov{a}$, results 
\begin{align} \label{apriorienergy}
\pa_t \Phi_p(f) & = \Phi'_p(f) \nabla\cdot (\ov{a} \nabla f - \ov{b} f)\\
& \nonumber 
=  \nabla\cdot ( \ov{a} \nabla\Phi_p(f) - \ov{b}\Phi(f)) - \Phi''_p(f)\,\ov{a}\nabla f \nabla f 
+ (\nabla\cdot\ov{b}) \bigl[ f\Phi'_p(f) - \Phi_p(f)\bigr].
\end{align}
Writing $\ov{c} = \nabla\cdot\ov{b}$ and integrating, yields
the a priori estimate
\[ \frac{d}{dt} \int \Phi_p(f)
= - \int \Phi''_p(f)\, \ov{a} \nabla f\nabla f
+ \int  \ov{c} \bigl[ f\Phi'_p(f) - \Phi_p(f)\bigr].\]
Then
\[ \Phi''_p(f) = p(p-1) f^{p-2},\qquad
\Phi'_p(f) = p f^{p-1},\qquad f\Phi'_p(f) - \Phi_p(f) = (p-1) f^p.\]
So
\begin{align} \nonumber
 \frac{d}{dt} \int \Phi_p(f) & = - p(p-1) \int f^{p-2} \ov{a} \nabla f\nabla f - (p-1) \int \ov{c} f^p\\
 & \label{Phip/2} = - \frac{p(p-1)}{(p/2)^2}
 \int \ov{a} \nabla \Phi_{p/2}(f) \nabla\Phi_{p/2}(f) - (p-1) \int \ov{c} f^p.
 \end{align}
 The first term on the right-hand side is handled through the positivity estimate \eqref{azsup},
 an easy commutator and Sobolev embedding:
 \begin{align*} \int \bigl| \sqrt{\ov{a}} \nabla \Phi_{p/2}(f)\bigr|^2
 & \geq K(d,\gamma,E_0,H(f_0)) \int \<v\>^\gamma \, |\nabla \Phi_{p/2}(f)|^2\,dv\\
 & \geq K \bigl\| \nabla (\<v\>^{\gamma/2} f^{p/2})\bigr\|^2_{L^2}
 - C \int f^p \<v\>^{\gamma-2}\,dv\\
 & \geq K \|f^{p/2}\|_{L^{2^\star}_{\gamma/p}} - C \|f\|_{L^p_{(\gamma-2)/p}}^2 \qquad (2^\star = (2d)/(d-2))\\
 & = K \|f\|_{L^{\lambda p}_{\gamma/p}} - C \|f\|_{L^p_{(\gamma-2)/p}}^p, \qquad
 \lambda = \frac{d}{d-2}.
 \end{align*}
 (The case $d=2$ should be treated separately, any $\lambda$ will do.)
 
 Let us turn to the second term on the right hand side of \eqref{Phip/2}. First, if $a(z) = |z|^{\gamma+2} \Pi_{z^\bot}$, then $b(z) = -(d-1) |z|^\gamma z$, $c(z) = - (d-1)(\gamma+d)|z|^\gamma$, to be understood as a multiple of the Dirac mass $\delta_0$ in the limit case $d=-\gamma$ (for $d=2,3$). 
From these bounds and Hardy--Littlewood--Sobolev inequality, this second term is controlled by
\begeq\label{2ndtermsuite} C (\gamma+d) \iint |v-v_*|^\gamma f(v_*) f(v)^p\,dv\,dv_*
\leq C \|f\|_{L^{p_0}} \|f^p\|_{L^r},
\endeq
where
\[ 2 + \frac{\gamma}{d} = \frac1{p_0} + \frac1{r}, \]
 and it is required that $-(d+2)<\gamma<-2$, which is implied by \eqref{VSPLandau}. In the limit case $\gamma=-d$, the left-hand side of \eqref{2ndtermsuite} should be replaced by $C \int f^{p+1}$ and the bound still holds true by Lebesgue interpolation. Then, by interpolation again,
\[ \|f^p\|_{L^r}  = \|f\|_{L^{pr}}^p  \leq \bigl( \|f\|_{L^{\lambda p}_{\gamma/p}}^p\bigr)^{r/\lambda}
\|f\|_{L^1_s}^{p(1-(r/\lambda))}\]
for $s= - (\gamma r)/(\lambda -r)$ provided of course that $r<\lambda$, i.e. $1/r<(d-2)/d$, i.e. $\gamma>-4$ (as assumed). All in all, we have identified $s>0$ and $\theta \in (0,1)$ such that
\[ \frac{d}{dt} \|f\|_{L^p}^p \leq -K \|f\|_{L^{\lambda p}_{\gamma/p}}^p + C  \|f\|_{L^p_{(\gamma-2)/p}}^p
+ C (\|f\|^p_{L^{\lambda p}_{\gamma/p}})^\theta \|f\|_{L^1_s}^{1-\theta}.\]
It follows, through Young's inequality and interpolation again,
\begeq\label{ddtLp}
 \frac{d}{dt} \|f\|_{L^p}^p \leq -K \|f\|_{L^{\lambda p}_{\gamma/p}}^p + C \|f\|_{L^1_s}.
 \endeq
Recall that $\|f_0\|^{L^{p_0}}$ is uniformly bounded, that $\|f\|_{L^1_s}$ is $O((1+t)^\var)$ for any given $\var$ if $f$ has finite moments of sufficiently high order. Further note that, by interpolation again, \eqref{ddtLp} implies
\begeq\label{ddtLp'}
 \frac{d}{dt} \|f\|_{L^p}^p \leq -K (\|f\|_{L^p}^p)^{\lambda'} + C \|f\|_{L^1_{s'}}.
 \endeq
 for any $\lambda'<\lambda$, and $s'= s'(s,\lambda,p,\gamma)$.
 Choosing $\lambda'>1$ implies, by a classical reasoning, the short-time appearance of all $L^p$ norms of $f$ for $p = \lambda' p_0$, $(\lambda')^2p_0$, etc. up to any $(\lambda')^\ell p_0$, $\ell\in\N$, and the bounds will be controlled in time like powers of the moments, which in turn are controlled like small powers of $t$. To summarise:
 Given any $p\geq 1$ we have the a priori bound
 \[ \forall t\geq 0 \qquad \|f(t)\|_{L^p} \leq C \left ( \frac1{t^{\alpha_p}} + t^\var\right),\]
 where $\alpha_p$ depends on $p$, $d$, $\gamma$, $\var$ is as small as required, and $C$ depends on $p$, $d$, $\gamma$, $b$, $E_0$, $H(f_0)$, $I(f_0)$, $\var$ and $\|f_0\|_{L^1_s}$ for some $s=s(\var,\gamma,p,d)$.
In short, all $L^p$ norms appear instantly and grow slowly in large time.
\med

Let us repeat the same scheme for the Boltzmann equation. This will work out if the kernel is {\em singular enough} in the angular variable. First a $\Gamma$ formula. From the elementary identity
\[ -p f^{p-1} (f'f'_* - ff_*) + \bigl[ (f')^p f'_* - f^p f_* \bigr] 
= f'_* \bigl[ (f')^p - f^p - p f^{p-1} (f'-f)\bigr] - (p-1) (f'_* -f_*) f^p\]
(for any four numbers $f,f_*,f',f'_*\geq 0$) one finds
\begeq\label{Qffp} Q(f,f^p) - p f^{p-1} Q(f,f)  
= \iint f'_* \Gamma_{\Phi_p}(f,f')\, B\,dv_*\,d\sigma - (p-1) {\cal S}\ast f,
\endeq
where $Q(g,f) = \iint (g'_* f'- g_* f)\,B\,dv_*\,d\sigma$,
\[ \Gamma_{\Phi_p}(f,f') = (f')^p - f^p - p f^{p-1} (f'-f)\]
and
\[ {\cal S}\ast f  = \iint (f'_* - f_*) B(v-v_*,\sigma)\,dv_*\,d\sigma.\]
Changing variable $v'_*\to v_*$, for fixed $\sigma$, shows that ${\cal S}$ is well-defined and $|{\cal S}(z)| \leq C |v-v_*|^\gamma$ for $\gamma>-d$, while in the limit case $\gamma=-d$, ${\cal S}$ is proportional to the Dirac mass $\delta_0$. Also the elementary convexity inequality
\[ x^p - 1 - p(x-1) \geq K_p (x^{p/2}-1)^2\]
yields
\[ \Gamma_{\Phi_p}(f,f') \geq K_p \bigl[ (f')^{p/2} - f^{p/2}\bigr]^2.\]

So integrating \eqref{Qffp} in $\R^d$ yields the a priori estimate
\begin{align*}
& \frac{d}{dt} \int f^p 
 = - \iiint f'_* \Gamma_{\Phi_p} (f,f')\,B\,dv\,dv_*\,d\sigma
+ (p-1) \int f^p ({\cal S}\ast f)\,dv\\
& = - \iiint f_*\Gamma_{\Phi_p}(f',f)\, B\,dv\,dv_*\,d\sigma
+ (p-1) \int f^p ({\cal S}\ast f)\,dv\\
& \leq -K \iiint f_*  \bigl[ f^{p/2}(v') - f^{p/2}(v)\bigr]^2\,B(v-v_*,\sigma)\,dv\,dv_*\,d\sigma
+ C \iint f^p  f_* |v-v_*|^\gamma\,dv\,dv_* \\
& \leq -K \iiint f_* \<v-v_*\>^\gamma \bigl[ f^{p/2}(v') - f^{p/2}(v)\bigr]^2\, b(k\cdot\sigma)\,dv\,dv_*\,d\sigma 
+ C \iint f^p  f_* |v-v_*|^\gamma\,dv\,dv_*,
\end{align*}
where as usual $k=(v-v_*)/|v-v_*|$. 

Let us focus on
\[ {\cal D} (f,F) = \iiint f_* \<v-v_*\>^\gamma \bigl[ F(v') - F(v)\bigr]^2\, b(k\cdot\sigma)\,dv\,dv_*\,d\sigma.\]
The factor $\<v-v_*\>^\gamma$ may be very small compared to $\<v\>^\gamma$ as $|v_*|\to\infty$; so it will be convenient to truncate large velocities $v_*$ in the crudest possible way:
\[ \<v-v_*\>^\gamma \geq K R^\gamma \<v\>^\gamma 1_{|v_*|\leq R}.\]
It will also be convenient to truncate angles $\theta\geq \pi/2$; so let $b_0(\cos\theta) = b(\cos\theta) 1_{1\leq\theta\leq\pi/2}$. Then
\begin{align} \nonumber
{\cal D} (f,F) 
& \geq K(R) \iiint (f_* 1_{|v_*|\leq R}) \<v\>^\gamma \bigl[ F(v') - F(v)\bigr]^2\, b(k\cdot\sigma)\,dv\,dv_*\,d\sigma \\
\label{DfF1}
& \geq \frac{K(R)}2 \iiint (f_* 1_{|v_*|\leq R}) \bigl[ F(v') \<v'\>^\gamma - F(v) \<v\>^\gamma\bigr]^2\, b_0(k\cdot\sigma)\,dv\,dv_*\,d\sigma\\
\label{DfF2}
& \qquad\qquad\qquad
- K(R) \iiint \bigl[ (f_* 1_{|v_*|\leq R}) \bigl[ \<v'\>^{\gamma/2} - \<v\>^{\gamma/2}\bigr]^2 F(v')^2\,b_0(k\cdot\sigma)\,dv\,dv_*\,d\sigma.
\end{align}
Then $(\<v'\>^{\gamma/2} - \<v\>^{\gamma/2})^2 \leq C |v-v_*|^2 \theta^2\leq C |v'-v_*|^2\theta^2$ (because $\theta\leq \pi/2$), so the last integral is crudely bounded above by
\[ C \iiint f_* (1+ |v_*|^2) (1+|v'|^2) \theta^2 b_0(k\cdot\sigma) F(v')^2\,dv\,dv_*\,d\sigma.\]
For given $v_*,\sigma$ the change of variables $v\to v'$ has bounded Jacobian, and transforms $\theta$ into $\theta/2$; so the bound above may be controlled by a multiple of 
\[ \iiint f_* (1+ |v_*|^2) (1+|v|^2) \,b_0(\cos(2\theta))\,\theta^2 F(v)^2\,dv\,dv_*\,d\sigma,\]
which in turn is bounded above by
\[ C \left(\int f_* |v_*|^2\,dv_*\right) \left(\int_{\S^{d-1}} \theta^2\, b(\cos\theta)\,\sin^{d-2}\theta\,d\theta\right)\, \int F^2\<v\>^2\,dv.\]

Now for the other part, \eqref{DfF1}: Here we have reduced to a Maxwellian kernel. A Fourier integral representation going back to Bobylev shows that
\begin{multline*} 
\iiint h(v_*) \bigl[ G(v')  - G(v)\bigr]^2\, b_0(k\cdot\sigma)\,dv\,dv_*\,d\sigma\\
= \int_{\R^{d}} \int_{\S^{d-1}}
b(k\cdot\sigma) \Bigl[ \hat{h}(0) |\hat{G}(\xi)|^2
+ \hat{h} (0) |\hat{G}(\xi^+)|^2
- \hat{h}(\xi^-) \hat{G}(\xi^+) \ov{\hat{G}(\xi)}
- \ov{\hat{h}(\xi^-)}\ov{\hat{G}(\xi^+)} \hat{G}(\xi)\Bigr]\,d\xi\,d\sigma
\end{multline*}
where $\ov{z}$ stands for complex conjugate of $z$, and
\[ \hat{f}(\xi) = \int_{\R^d} e^{-2i\pi\xi\cdot v} f(v)\,dv,\qquad
\xi^+ = \frac{\xi + |\xi|\sigma}2, \qquad
\xi^- = \frac{\xi-|\xi|\sigma}2. \]
It follows, with $h(v)= f(v) 1_{|v|\leq R}$ and $G(v) = F(v) \<v\>^{\gamma/2}$, 
that \eqref{DfF1} is bounded below by
\[ K(R) \iint \bigl( \hat{h}(0) - |\hat{h}(\xi^-)|\bigr) \bigl( |\hat{G}(\xi^+)|^2 + |\hat{G}(\xi)|^2\bigr)\,d\xi\,d\sigma.\]
From the bound on $H(f)$ and $E_0$ follows a non-concentration estimate on $f$ and thus on $h$ (independently of $R$), so that
\[ \hat{h}(0) - |\hat{h}(\xi^-)| \geq K \min (|\xi^-|^2,1),\]
and for $R$ large enough, fixed, the constant $K$ will be bounded bellow.
Eventually \eqref{DfF1} is bounded below by
\[ K \int \min(|\xi^-|^2,1) b\left(\frac{\xi}{|\xi|}\cdot\sigma\right) |\hat{G}(\xi)|^2\,d\xi\,d\sigma
\geq K \int_{\R^d} |\hat{G}(\xi)|^2|\xi|^\nu\,d\xi = \|G\|_{\dot{H}^{\nu/2}}^2.\]

All in all, 
\[ {\cal D}(f,F) \geq K \|F\|_{H^{\nu/2}_{\gamma/2}}^2 - C \|F\|_{L^2_1}^2,\]
where $K,C$ only depend on $E(f_0)$ and $H(f_0)$.

With this at hand, we may reason as in the proof of higher integrability for the Landau equation,
using the Sobolev embedding
\[ \|f\|_{L^{2\lambda}} \leq C \|f\|_{H^{\nu/2}}, \]
where now
\[ \lambda = \frac{d}{d-\nu}.\]
Eventually,
\[ \frac{d}{dt} \int f^p \leq -K \|f^{p/2}\|^2_{H^{\nu/2}_{\gamma/p}} 
+ C \|f\|_{L^p}^{\theta p} + C \|f\|_{L^1_s} \]
for some $\theta\in (0,1)$ and $s$ large enough, as in \eqref{ddtLp}, but now the condition $r<\lambda$ becomes
\[ \gamma+\nu > -2,\]
as assumed in \eqref{VSPBoltzmann}. The end of the reasoning is similar: We find that all $L^p$ norms appear instantaneously like inverse power of $t$, and these bounds are all $O(t^\var)$ as $t\to\infty$, for $\var$ arbitrarily small, provided that $f_0$ has sufficiently many finite moments, and again the bounds and exponents are computable.

\subsection{Smoothness} \label{subsmooth}

The previous subsection was about improving integrability, this one is about improving regularity, working with $L^2$-type spaces now. Again, start with the Landau equation.

The first issue is to evaluate the regularity of the diffusion matrix $\ov{a} = a\ast f$. As $|a|\leq C |z|^{\gamma+2}$, 
\[ \|a\ast f\|_{L^\infty} \leq C\, \|f\|_{L^q}\]
with $\gamma+2+d = d/q$. So, as soon as $\gamma+2+d>0$ (which is implied by \eqref{VSPLandau}), chosing $q$ large enough yields an $L^\infty$ bound on $\ov{a}$, with a norm proportional to $\|f\|_{L^q}$, which is known to be $O(t^\var)$ with $\var$ arbitrarily small, if enough moments are finite. The same reasoning holds for $\nabla \ov{a}$, as soon as $\gamma+1+d>0$. As for $\nabla^2\ov{a}$, as $\gamma+d$ may be equal to 0, we only get a bound on $\|(\nabla^2 a)\ast f\|_{L^p}$, locally, for $p$ arbitrarily large. The conclusion is $\ov{a} \in W^{2,p}$, locally, for all $p\geq 1$, with a local bound that may grow polynomially, and in particular $\ov{a} \in C^{1,\alpha}$ for some $\alpha>0$ and $\ov{c} \in L^p$ locally, for all $p>1$. In the sequel, I shall systematically evaluate $\ov{a}$ in Sobolev norms weighted by a negative power of the velocity -- anyway large velocities are controlled by the moment bounds.

As for the lower bound, nonconcentration implies a lower bound like $\ov{a}(z) \geq K\<z\>^\gamma$.

Now one can go through higher regularity with the classical method of commuting the equation both with the quadratic nonlinearity and the derivation:

\bul Start from $\pa_t f = \nabla\cdot(\ov{a}\nabla f - \ov{b} f)$, $\ov{a} = a\ast f$, $\ov{b} = b\ast f = (\nabla\cdot a)\ast f$.

\bul Then pick up an index $k$, and differentiate with respect to $v_k$, it follows
\[ \pa_t (\pa_k f) = \nabla\cdot \bigl( \ov{a}\nabla \pa_k f - \ov{b}\, \pa_k f \bigr)
+ \pa_i \bigl( \pa_k \ov{a}_{ij} \,\pa_j f\bigr) - \pa_i \bigl(\pa_k \ov{b}_i f\bigr),\]
with the convention of summation over repeated indices.

\bul Then multiply by $\pa_k f$ and work out the commutator in $\Gamma$ style, to get
\begin{multline*} \pa_t \left( \frac{(\pa_k f)^2}2\right)
= \pa_k f \pa_t f \\
= \nabla\cdot \left( \ov{a} \nabla \frac{(\pa_kf)^2}{2} - \ov{b} \frac{(\pa_k f)^2}{2}\right)
- \ov{a} \nabla \pa_k f \nabla \pa_k f - \ov{c} \frac{(\pa_k f)^2}2 \\
+ \pa_k f\, \pa_i \bigl( \pa_k \ov{a}_{ij} \pa_j f\bigr) - \pa_k f \,\pa_i \bigl(\pa_k \ov{b}_i f\bigr).
\end{multline*}
The key term is of course $-\ov{a} \nabla \pa_k f \nabla \pa_k f$ which is like a $\Gamma_2 (f,f)$ term. After that, the last three terms will be treated as error terms. Upon integrating the previous expression over $\R^d$, summing over $k$ and integrating by parts, comes
\begin{multline}
\frac{d}{dt} \int \frac{|\nabla f|^2}2 
= - \int \ov{a} \nabla\nabla f : \nabla\nabla f - \int \ov{c} \frac{|\nabla f|^2}2
- \int \pa_{ik} f\, \pa_k \ov{a}_{ij}\, \pa_j f 
+ \int \pa_{ik} f\, \pa_k \ov{b}_i f.
\end{multline}
Then
\[ \int \ov{a}\nabla\nabla f\nabla\nabla f  \geq K \int \<v\>^\gamma |\nabla^2 f|^2.\]
Also
\begin{align*}
\left| \int \pa_{ik} f\, \pa_k \ov{a}_{ij}\, \pa_j f \right|
& \leq \left( \int \<v\>^\gamma |\nabla^2f|^2\right)^{1/2}
\left( \in \<v\>^{-\gamma} |\nabla \ov{a}|^2 |\nabla f|^2\right)^{1/2}\\
& \leq \var \int \<v\>^\gamma |\nabla^2 f|^2 + C \int \<v\>^{-\gamma} |\nabla\ov{a}|^2 |\nabla f|^2,
\end{align*}
where $\var$ is as small as desired. Similarly,
\[
\left| \int \pa_{ik} f \,\pa_k \ov{a}_{ij}\, \pa_j f \right|
\leq \var \int \<v\>^\gamma |\nabla^2 f|^2 + C \int \<v\>^{-\gamma} |\nabla\ov{b}|^2 |\nabla f|^2.\]
To summarise, at this stage
\begeq\label{wheel} \frac{d}{dt} \int |\nabla f|^2 \leq - K \int \<v\>^\gamma |\nabla^2f|^2
+ C \left( \int \<v\>^{-\gamma} |\nabla \ov{a}|^2 |\nabla f|^2
+ \int \<v\>^{-\gamma} |\nabla \ov{b}|^2 |\nabla f|^2 + \int |\ov{c}| |\nabla f|^2\right).
\endeq

To get the wheel of the estimate go round, one needs to dominate the sum of the three integrals within brackets, at the end of \eqref{wheel}, by a little bit of $\int \<v\>^\gamma |\nabla^2f|^2$. There are two issues to do this:

- powers of $\<v\>$ do not match (the diffusion is weak at large velocities), but this will be handled by moment estimates;

- $\ov{a}$, $\ov{b}$, $\ov{c}$ are just partly regular and their regularity has to rest partly on the regularity of $f$, not just $a,b,c$ (this will be particularly true for higher derivatives). For instance in the important case $d=3$, $\gamma=-3$, we have $|\ov{c}|= f$, which at this stage of the proof is estimated in all $L^p$ spaces but not in $L^\infty$. So the power of $f$ inside the integral will be higher than 2, and Young inequality will be needed to work with indices $2p, 2p'$ with $p\sim \infty$, $p'\sim 1$. For this strategy to work out, we shall use the information that $f$ lies in all weighted $L^p$ spaces. So this will be an interpolation between $W^{k,p}_s$ spaces ($k$ derivatives in $L^p$ and $s$ moments), of the type
\begeq\label{interpWsp}
\|f\|_{W^{1,p}_\sigma} 
\leq C \|f\|^\theta_{W^{2,2}_{\gamma/2}} \|f\|^{1-\theta}_{L^r_s}.
\endeq
Taking $r,s$ large enough will make the exponent $\theta$ lower. In fact the compatibility condition is
\[ -1 + \frac{d}{p} = \theta \left( -2 + \frac{d}{2}\right) + (1-\theta) \frac{d}{r}.\]
So the condition to find suitable $r,s$ (large enough) for \eqref{interpWsp} to hold is $0<\theta<1$, where
\[ \theta = \frac{1-\frac{d}{p}}{2-\frac{d}{2}}.\]

As for the dependence of $\ov{a},\ov{b},\ov{c}$ in $f$, it is as follows: One can find $\sigma$ such that for any $p$ there are $r,s$ such that for all $f$,
\begeq\label{abcf}
\|\ov{a}\|_{W^{2,p}_{-\sigma}} + \|\ov{b}\|_{W^{1,p}_{-\sigma}} + \|\ov{c}\|_{L^p_{-\sigma}}
\leq C \|f\|_{L^r_s}.
\endeq
So, if $p>1$ and $p'= p/(p-1)$,
\begin{align*}
\int \<v\>^{-\gamma} |\nabla\ov{a}|^2 |\nabla f|^2
& \leq C \|\nabla \ov{a}\|^2_{L^{2p}_m} \|\nabla f\|^2_{L^{2p'}_{-\gamma/2}} \\
& \leq C \|f\|^2_{L^{r_1}_{s_1}} \|f\|^{2\theta}_{W^{2,2}_{\gamma/2}} \|f\|^{2(1-\theta)}_{L^{r_2}_{s_2}}.
\end{align*}
and this works for $r_1,r_2,s_1,s_2$ large enough if
\[ 0 < \frac{1-\frac{d}{2p'}}{2-\frac{d}{2}} < 1.\]
Obviously this is true for $p$ large enough, so that $p'$ is close to 1.

In the same way, one shows that
\begin{align*}
\int \<v\>^{-\gamma} |\nabla\ov{b}|^2 |\nabla f|^2 + \int |\ov{c}| |\nabla f|^2 
\leq C \|f\|_{W^{2,2}_{\gamma/2}}^{2\theta} \|f\|^{2(2-\theta)}_{L^r_s}.
\end{align*}
Also note that
\[ \|f\|_{W^{2,2}_{\gamma/2}}^2 \leq \int \<v\>^\gamma |\nabla^2 f|^2 + C \|f\|_{L^2}^2.\]

Then fixing all parameters as above, \eqref{wheel} becomes
\begeq\label{wheel2}
\frac{d}{dt} \int \frac{|\nabla f|^2}2
\leq - K \|f\|_{W^{2,2}_{\gamma/2}}^2 + C \|f\|^4_{L^p_s}.
\endeq
This proves an a priori bound on $\int |\nabla f|^2$, which becomes immediately finite and remains so for all times, with a bound like $O(t^\var)$ in large time, for arbitrarily small $\var$.
\med

Then go for higher derivatives. Differentiating the equation at higher order, working out the commutators, the degenerate coervicity and the error terms, one arrives at
\begin{multline*} \frac{d}{dt}\int |\nabla^k f|^2 \leq -K \int \<v\>^\gamma |\nabla^{k+1}f|^2
+ C \sum_{|r|\leq k, |r|+|\ell|\leq k+1} \int \<v\>^{-\gamma} |\nabla^\ell \ov{a}|^2 |\nabla^r f|^2\\
+ C \sum_{|r|+|\ell|\leq k+1}  \int \<v\>^{-\gamma} |\nabla^\ell \ov{b}|^2 |\nabla^r f|^2
+ C \sum_{|r|+|\ell|\leq k} \<v\>^{-\gamma} |\nabla^\ell \ov{c}|^2 |\nabla^r f|^2,
\end{multline*}
and if $\sigma$ is given we may find $s$ such that
\[ \|\ov{a}\|_{W^{j+2,p}_{-\sigma}}
+ \|\ov{b}\|_{W^{j+1,p}_{-\sigma}}
+\|\ov{c}\|_{W^{j,p}_{-\sigma}} \leq C \|f\|_{W^{j,2}_s},\]
and the relevant interpolation now is
\[ \|f\|_{W^{\ell,p}_\sigma}
\leq C \|f\|^\theta_{W^{L,2}_{\gamma/2}} \|f\|^{1-\theta}_{L^r_s},\]
where
\[ -\ell + \frac{d}{p} = \theta \left( -L + \frac{d}{2}\right) + (1-\theta)\frac{d}{r}.\]
Then error terms are bounded like, for instance,
\begin{align*}
\|\nabla^\ell \ov{b}\|^2_{L^{2p}_m} \|\nabla^r f\|^2_{L^{2p'}_{-\gamma/2}}
& \leq C \|f\|^2_{W^{\ell-1,2p}_{m+\sigma}} \|f\|^2_{W^{r,2p'}_{-\gamma/2}}\\
& \leq C \bigl( \|f\|^{\theta_1}_{W^{k+1,2}_{-\gamma/2}} \|f\|^{1-\theta_1}_{L^m_\sigma}\bigr)^2
\bigl(\|f\|^{\theta_2}_{W^{k+1,2}_{-\gamma/2}} \|f\|^{1-\theta_2}_{L^m_\sigma}\bigr)^2,
\end{align*}
and
\[ \theta_1+\theta_2 = \frac{(\ell-1)- \frac{d}{2p}+ \left(r-\frac{d}{2p'}\right)}{k+1-\frac{d}{2}}
\leq \frac{k-\frac{d}2}{k+1-\frac{d}2}.\]
All in all,
\[ \frac{d}{dt} \int |\nabla^k f|^2 \leq -K \int \<v\>^\gamma |\nabla^{k+1}f|^2
+ C \|f\|_{L^m_s}^A,\]
for some $K>0$ and $m,s,A$ large enough, and there is $\theta \in (0,1)$ such that
\[\frac{d}{dt} \int |\nabla^k f|^2 \leq -K \left( \int |\nabla^k f|^2\right)^{1/\theta}
+ C \|f\|_{L^m_s}^A,\]
which proves both the immediate appearance of Sobolev norms and their good control in large time since all weighted $L^m$ norms are $O(t^{\var})$ for $t\to\infty$, with $\var$ as small as desired (if large enough moments are finite). By Sobolev embedding, this also implies a control of all derivatives like inverse power laws (rapid decay).

More precisely: For arbitrarily large $k, \kappa >0$, for arbitrarily small $\var>0$, there are $\alpha>0$ (controlled from below) and $s>0$ (controlled from above) such that
\begeq\label{nablafkt}
\forall t>0, \qquad |\nabla^k f(t,v) | \leq \frac{C \bigl( t^{-\alpha} + (1+t)^\var\bigr)}{\<v\>^\kappa},
\endeq
and $C$ depends in a computable way on $E_0$, $H(f_0)$, $I(f_0)$ and $\|f_0\|_{L^1_s}$.
\med

For Boltzmann, the strategy to get higher regularity is exactly the same, but now it rests on the properties of the bilinear Boltzmann operator $Q(g,f) = \iint (g'_*f'- g_*f)B\,dv_*\,d\sigma$, where (without loss of generality) $b$ is supported in $(0\leq\theta\leq \pi/2)$. Since $Q(g,\cdot)$ acts like a variant of a fractional diffusion of order $\nu$, one has the estimates
\[ \|Q(g,f) \|_{L^p_\sigma} \leq C \|g\|_{L^q_s} \|f\|_{W^{\ov{\nu},\ov{p}}},\]
where $\ov{\nu}$ is arbitrarily close to $\nu$ and $\ov{p}$ arbitrarily close to $p$. The coercivity estimate remains the same,
\[ \int Q(g,f)f \geq K \|f\|_{H^{\nu/2}_{-\gamma/2}}^2 - C \|f\|_{L^m_s}^2,\]
and going to higher derivatives is easy thanks to the bilinear identity
\[ \pa_\ell Q(g,f) = Q(\pa_\ell g, f) + Q(g,\pa_\ell f).\]
Then all ingredients are here to obtain exactly the same rapid decay estimates \eqref{nablafkt}.

\subsection{What if the initial Fisher information is infinite?}

Bounds and rapid decay of the distribution function are quite classical assumptions, as is the requirement of finite entropy. But the requirement of finite Fisher information, while statistically meaningful, may seem unsatisfactory. It turns out however that this can also be dispended with, thanks to the regularising property of the equation. The scheme of proof is the following:

\bul Control the solution in $L^p$ norms for a short time, as a quasilinear diffusive equation with quadratic nonlinearity, and deduce regularisation during that short time;

\bul Deduce that the Fisher information becomes finite in short time; then it will be nonincreasing, and we are back to the situation treated before.

As a warmup, let us consider the model case
\begeq\label{model}
\pa_t f = \<v\>^\gamma \Delta f + f^2
\endeq
which does not have the same nice conservation properties as the Landau equation, but (for $d=3, \gamma=-3$) satisfies the same estimates for diffusivity and nonlinearity.
The usual estimate yields
\begin{align}  \label{ddtflp}
\frac{d}{dt} \|f\|_{L^p}^p
& \leq -K \| f^{p/2}\|_{H^1_{\gamma/p}}^2 + C \|f\|_{L^{p+1}}^{p+1}\\
\label{-Kflambdap} & \leq 
-K \|f\|_{L^{\lambda p}_{\gamma/p}}^p + C \|f\|_{L^{p+1}}^{p+1},
\end{align}
with $\lambda = d/(d-2) > 1$. Assume $p+1<\lambda p$, ie $p> 1/(\lambda -1) = (d/2)-1$. Then one can interpolate $L^{p+1}$ between $L^p$ and $L^{\lambda p}$, and a little bit of high order moments to compensate for the negative weight in the $L^{\lambda p}$ norm.
This yields
\begeq\label{interpoplambda}
\|f\|_{L^{p+1}} \leq C \|f\|_{L^p}^\theta \|f\|_{L^{\lambda p}_{\gamma/p}}^{1-\ov{\theta}} \|f\|_{L^1_s}^{\ov{\theta}-\theta},
\endeq
where $\ov{\theta}>\theta$ (arbitrarily close to $\theta$) and $s>0$ (large enough), and
\[ \frac1{p+1} = \frac{\theta}{p} + \frac{1-\ov{\theta}}{\lambda p} + (\ov{\theta}-\theta).\]
This can be achieved for any
\[ 1 > \theta > \frac{\frac{p}{p+1}-\frac1{\lambda}}{1-\frac1{\lambda}}.\]
From \eqref{-Kflambdap} and \eqref{interpoplambda} follows
\begeq\label{ddtflpmodel}
\frac{d}{dt} \|f\|_{L^p}^p 
\leq - K \|f\|_{L^{\lambda p}_{\gamma/p}}^p + C \|f\|_{L^p}^r + C \|f\|_{L^1_s}^A,
\endeq
for any
\[ r > \frac{p}{1- \frac{\lambda -1}{(p+1)(\lambda -1)-1}},\]
provided that the exponent of $\|f\|_{L^{\lambda p}_{\gamma/p}}$ in the positive term is smaller than the exponent in the negative term, i.e. $(p+1)(1-\theta)<p$, which means $p> d/2$.
Obviously $r>p$, so \eqref{ddtflpmodel} does not prevent blowup in finite time -- as expected. But it does imply a uniform control of $\|f\|_{L^p}$ in short time. This in turn implies a time-integrated control of $\|f\|_{L^{\lambda p}_{\gamma/p}}$ in short time, and by interpolation with moment bounds, a time-integrated control of $\|f\|_{L^q}$ in short time for any $q<\lambda p$. Then we can repeat the reasoning and obtain a uniform control of $f$ in $L^q$, on a shorter time interval. And by induction, this works out for all $L^p$ norms, and thus for all $L^p_s$ norms, for arbitrary $p$ and $s$. (The control deteriorates as $p$ and $s$ become larger, and the time interval becomes shorter and shorter, but the procedure can be repeated an arbitrary number of times to get arbitrarily large smoothness in short time.) Then we can repeat the procedure for higher regularity. As a conclusion of this step, solutions of \eqref{model} satisfy bounds of rapid decay \eqref{nablafkt}, at least for short time $t\leq t_0 = t_0(k,\kappa, E_0,H(f_0),\|f_0\|_{L^1_s}, \|f_0\|_{L^{q_0}})$, where
$q_0$ is any Lebesgue exponent satisfying $q_0> d/2$.

As a second step, this implies the finiteness of the Fisher information for small $t$, thanks to an estimate of Toscani and myself (see Bibliographical Notes):
\begeq\label{FIH2}
I(f) \leq C(s,d) \|f\|_{H^2_s},\qquad \text{for any $s> \dps \frac{d}2$.}
\endeq
\med

Now working out the same scheme for the Boltzmann equation with very soft potentials yields
\begeq\label{equsualVSP}
\frac{d}{dt} \|f\|_{L^p}^p
\leq - K \|f^{p/2}\|_{H^{\nu/2}_{\gamma/p}}^2 + C \int | ({\cal S}\ast f)| f^p
\endeq
where $|{\cal S}(v)|\leq C |v|^\gamma$.
On the one hand,
\begeq\label{fpsur2}
\|f^{p/2}\|_{H^{\nu/2}_{\gamma/p}}^2 \geq \|f\|_{L^{\lambda p}_{\gamma/p}}^p,\qquad
\lambda = \frac{d}{d-\nu}.
\endeq

On the other hand, we can find $\sigma>0$ such that
\[ \|{\cal S}\ast f\|_{L^r_{-\sigma}} \leq C \|f\|_{L^q}^{\theta_1} \|f\|_{L^1_s}^{1-\theta_1},
\qquad \gamma+d+\frac{d}{r} = \frac{\theta_1 d}{q} + (1-\theta_1)d.\]
So, with $1/r+1/r'=1$,
\begin{align*}
\int |{\cal S}\ast f| f^p
& \leq \|{\cal S}\ast f\|_{L^r_{-\sigma}} \|f^p\|_{L^{r'}_\sigma} \\
& \leq C \|f\|_{L^q_s}^{\theta_1} \|f\|_{L^{p r'}_s}^p\\
& \leq C \|f\|_{L^{\lambda p}_{\gamma/p}}^\alpha \|f\|_{L^p}^\beta \|f\|_{L^1_s}^A,
\end{align*}
where
\[ \frac1{q} = \frac{\eta}{\lambda p} + \frac{1-\ov{\eta}}{p} + (\ov{\eta}-\eta) \qquad (0<\eta<\ov{\eta}<1),\]
\[ \frac{1}{p r'} = \frac{\zeta}{\lambda p} + \frac{1-\ov{\zeta}}{p} + (\ov{\zeta}-\zeta)\qquad
(0<\zeta<\ov{\zeta}<1,\]
\[ \alpha =\eta \theta_1 + \zeta p.\]
Using moments of high order we may choose 
\[ \begin{cases}
\theta_1\simeq 1\\
\frac{\gamma}{d}+1+\frac1{r} \simeq \frac1{q}\\
\frac1{q} \simeq \frac{\eta}{\lambda p} + \frac{1-\eta}{p}\\
\frac1{r'} \simeq \frac{\zeta}{\lambda} + (1-\eta)\\
\frac1{r'}+\frac1{r} = 1\\
\alpha\simeq \eta +\zeta p.
\end{cases} \]

This leads to
\[ \eta\simeq \frac{1- p \left(1 + \frac{\gamma}{d} + \frac1{r}\right)}{1-\frac1{\lambda}} \]
\[ \zeta\simeq \frac{1-\frac1{r'}}{1-\frac1{\lambda}}\]
\[ \alpha\simeq \eta + \zeta p \simeq \frac{1 - p\left(\frac{\gamma+d}{d}\right)}{1-\frac1{\lambda}}.\]
If $\alpha< p$ this leads again to a differential inequality of the form
\[ \frac{d}{dt} \|f\|_{L^p}^p \leq -K \|f\|_{L^p_{\gamma/p}}^p + C \|f\|_{L^p}^r + C \|f\|_{L^1_s}^A\]
for some $r>p$. The condition $\alpha<p$ can always be satisfied if
\[ p> \frac1{2+\frac{\gamma}{d} - \frac1{\lambda}} = \frac{d}{d+\gamma+\nu}.\]
We assumed $\gamma+\nu > -2$, so this estimate is at worse like $f_0 \in L^{d/(d-2)}$ (or $f_0$ in all $L^p$ for $d=2$). Note that this is precisely the integrability estimate implied by the Fisher information estimate! To summarise: 

Using short-time regularisation, we may weaken the assumption of finite Fisher information into the Lebesgue integrability assumption
\begeq\label{weakerLebesgue}
f_0 \in L^{q_0} \qquad \text{ for some $\dps q_0> \frac{d}{d+\gamma+\nu}$},
\endeq
an assumption which is always weaker than the assumption of finite Fisher information. This implies short-time regularisation, and then by \eqref{FIH2} again, Fisher information becomes instantly finite, and we are back into the regime of the previous subsection.
\sm

To summarise: Given any $k,\kappa>0$ and $\var>0$, and any $q_0$ as in \eqref{weakerLebesgue} we may find $\alpha>0$ and $s>0$ such that \eqref{nablafkt} holds true for all times, where the constant $C$ only depends on $k,\kappa,\var,q_0$, the energy $E_0$, the information $H(f_0)$, $\|f_0\|_{L^{q_0}}$ and $\|f_0\|_{L^1_s}$.

\subsection{Lower bound}

Already Carleman, almost hundred years ago, understood the interest of a lower bound estimate on solutions of the Boltzmann equation. This is interesting to control the omnipresent log of the distribution function, but also has intrinsic interest as in any diffusion process. Recall that lower bounds of Gaussian nature also appear in Nash's proof of the continuity of solutions of diffusion equations with discontinuous coefficients.

It turns out that the spatially homogeneous Boltzmann equation always leads to a Gaussian lower bound, as in diffusion processes:
\[ f(t,v) \geq K(t)\, e^{-A(t)|v|^2},\]
where $K,A$ are positive for $t>0$ and remain controlled as $t\to\infty$, as $O(t^\var)$ (through the smoothness bounds). Here I shall not provide such a strong bound but only sketch the proof of a weaker estimate
\[ f(t,v) \geq K(t) e^{-A(t)|v|^q},\]
where $\log K(t)$ and $A(t)$ are both $O(t^\var)$ as $t\to\infty$.
This will be the opportunity to touch the phenomenon of spreading which is the key in proving lower bounds. I shall use the following maximum principle:

If $f=f(t,v)$ and $\vphi=\vphi(t,v)$ are smooth and defined on $[0,t_0)\times\R^d$ with respective initial data $f_0$ and $\vphi_0$,
\[\begin{cases} \label{ifMP}
f_0(v)\geq \vphi_0(v) \qquad \text{ for all $v \in\R^d$} \\
\dps \derpar{f}{t} = Q(f,f)\\
\dps \derpar{\vphi}{t} \leq Q(f,\vphi)\qquad \text{for all $t \in (0,t_0),\ v\in\R^d$},
\end{cases} \]
then $f(t,v)\geq \vphi(t,v)$ for all $(t,v)\in (0,t_0)\times\R^d$.

First consider the Landau equation. By smoothness and energy bound, there is a ball $B(v_0,r)$ with $|v_0|\leq R$ and $r>0$ controlled from below, such that $f$ is bounded below by $K>0$ on $[0,t_0)\times B(v_0,r)$. Let's work only on that time interval: If we get a lower bound proportional to $\exp(-\alpha|v-v_0|^q)$ it will easily follow a lower bound independent of the choice of $v_0$. Let $\vphi(t) = \exp(-\alpha(t) |v-v_0|^q \chi(v-v_0) - B(t))$, where $\alpha(t)\geq 0$ and $B(t)$ large enough will be chosen later, $\chi(v)$ is a cutoff function with value $0$ in $|v|\leq r/2$ and $1$ for $|v|\geq r$. Let us further assume that $\alpha(0)= +\infty$. Then for $B(0)$ large enough, $\vphi_0\leq f_0$ in all of $\R^d$. Consider $f$ solution of the Landau equation, in the form
\[ \derpar{f}{t} = (a \ast f):\nabla^2 f - (c\ast f) f,\]
the goal is
\begeq\label{goalLB} \derpar{\vphi}{t} \geq  (a\ast f): \nabla^2 \vphi - (c\ast f) \vphi.
\endeq
But
\[ \derpar{\vphi}{t} = \Bigl( -\dot{\alpha}(t) |v-v_0|^q \chi(v-v_0) - \dot{B}(t) \Bigr)\vphi,\]
\[ (a\ast f): \nabla^2 \vphi \geq \Bigl( K \<v\>^\gamma \alpha(t)^2 |v-v_0|^{2(q-1)}\chi(v-v_0)^2 - C)\vphi,\]
\[ (c\ast f)\vphi \leq C \vphi,\]
where $C$ is chosen large enough. (There are also lower powers of $|v-v_0|$ and derivatives of $\chi$, but for large $|v|$ these terms are negligible and everything is controlled.) Choosing $B(t) = B t$ for $B$ large enough, $\alpha(t) = \alpha/t^2$ for $\alpha$ small enough, and $2(q-1)+\gamma = q$, i.e. $q=2-\gamma$ solves the problem and yields
\[ f(t,v) \geq K e^{-A|v|^{2-\gamma}}\]
for $t_0/2<t< t_0$ and for all $v\in\R^d$.

It is tempting to believe that the same strategy will work out for the noncutoff Boltzmann equation, but that is not so simple. As a distorted fractional diffusion of order $\nu$, one would hope that $Q(f, e^{-\alpha |v|^q})$ is like $|v|^r e^{-\alpha|v|^q}$ for some power $r$, but then that power would certainly go to 0 as $\nu\to 0$, and that term would never dominate the negative term proportional to $|v|^q e^{-\alpha |v|^q}$ coming from the time derivative of $\vphi$ when $\alpha$ depends on $t$.

Instead, we can go through spreading, with an iterative scheme which I will only sketch. Assume that $f$ is bounded below on $[0,t_0)$ by $K$ on $B(0,R)$ and prove that it will be bounded below by $\beta K$ on $B(0,\lambda R)$ for some $\lambda >1$, and keep track of the dependence of $\beta,\lambda$ on $R$. Choose
\[ \vphi(t,v) = e^{-Bt} \chi \left(\frac{|v|^2}{2 R(t)}\right).\]
Then
\[ \dot{\vphi}(t) = - B \vphi - e^{-Bt} \frac{v}{R(t)}\cdot\nabla \chi \left(\frac{|v|^2}{2 R(t)}\right),\]
\begin{align*}
Q(f,\vphi)  & = \iint f'_* (\vphi'-\vphi)\, B\,dv_*\,d\sigma - ({\cal S}\ast f)\vphi\\
& \geq K R^\gamma \iint f'_* (\vphi'-\vphi)\,dv_*\,d\sigma - C\vphi.
\end{align*}
And so on. Working out the dependencies of the constants will eventually lead to an inequality that can be repeated to get the appearance of a lower bound of stretched exponential form.

\bibnotes

The generic blowup of solutions of the quadratic heat equation in $L^p(\R^d)$ for $d>p/2$ was established by Weissler \cite{weissler:blowup}, see also Quittner and Souplet \cite[Remark 16.2 (iii)]{quittnersouplet:book}. (Thanks to Hatem Zaag for providing these references.)

``Very'' weak solutions for very soft potentials were introduced in my article \cite{vill:new:98} and improved in our joint work \cite{ADVW}, in the region of ``conditional regularity'', into a more classical notion of time-integrated weak solutions. It is only with the new estimates on Fisher information that one can strengthen the notion of weak solution, not having necessarily to time-integrate.

For interpolation I already mentioned the treatise by Lunardi \cite{lunardi:interpolation}. In  the context of the spatially homogeneous Boltzmann equation, interpolation techniques go back at least to Gustafsson \cite{gust:Lp:86} in $L^p$ spaces. From the mid-nineties on, interpolation between weighted Sobolev spaces was heavily used in papers by Desvillettes, Mouhot and myself, see e.g. \cite{DV:landau:1,DV:FP:01,DV:boltz:05,mouvill:04} and many references since then.

Moment estimates have been the first crucial step in the study of the Boltzmann equation, first for hard potentials with cutoff, by Povzner \cite{povz:65}, Arkeryd \cite{ark:I+II:72}, Elmroth \cite{elmr:mom:83}, Desvillettes \cite{desv:mom:93}, Wennberg \cite{wenn:mom:94} and others. A key observation by Desvillettes was that for hard potentials there is appearance of moments: Even if the initial datum is not well localised, suffices to have one finite moment of order higher than~2, to get all moments uniformly bounded. The ultimate results are those of Mischler and Wennberg \cite{mischwenn:hom:99} (see also Lu \cite{lu:unique:97,lu:stability:99}): Under the mere assumption of finite energy of the initial datum, there is instantaneous appearance and uniform boundedness of all moments, and there is also existence and uniqueness. But this ``instant localisation'' is a specific feature of hard potentials: In the regime of Maxwellian or soft potentials there are no such results and immediate appearance of moments does not hold; at best there is propagation. When the collision kernel is Maxwellian, moments equations can be closed in some sense and there is more structure, allowing precise results, see Ikenberry and Truesdell \cite{iktru:56}, Bobylev \cite{bob:theory:88,bob:moment:97}. Propagation estimates for soft potentials were studied by various authors, e.g. Arkeryd \cite{ark:asymptinfi:82} or Desvillettes \cite{desv:mom:93}. For very soft potentials the first estimates appeared in my PhD Thesis \cite{vill:new:98,vill:habil}; there I was already juggling between the two kinds of estimates (with two different symmetrisation formulas) according to the behaviour in $|v-v_*|$, as in Subsection \ref{submoments}. Then came the fine work by Carlen, Carvalho and Lu \cite{CCL:soft:09} who showed how to combine these ideas to get good large-time moment estimates; inequality \eqref{CCL} is established there (for the Boltzmann equation with very soft potentials).

Gronwall's lemma, in the precise form used for \eqref{gronwall}, states that if
\[ u(t) \leq \phi(t) + \int_0^t \lambda(\tau) u(\tau)\,d\tau,\]
then 
\[ u(t) \leq \phi(0) \exp \left(\int_0^t \lambda(\tau)\,d\tau\right)
+ \int_0^t \exp \left(\int_\tau^t \lambda(s)\,ds\right)\,\frac{d\phi}{d\tau}\,d\tau.\]
To prove it, check that 
\begin{multline*}
\vphi(t) \exp\left( -\int_0^t \lambda (\tau)\,d\tau\right)
+ \left(\int_0^t \lambda(\tau) u(\tau)\,d\tau \right) \exp\left( - \int_0^t \lambda(\tau)\,d\tau\right)\\
- \int_0^t \exp\left(-\int_0^\tau \lambda(s)\,ds\right)\,\frac{d\vphi}{d\tau}\,d\tau
\end{multline*}
is a nonincreasing function of $t$.

Integrability and regularity were studied in the context of the Boltzmann equation with hard potentials first, and mainly under a cutoff assumption \cite{gust:Lp:86}. Lions \cite{lions:kyoto1+2:94} first observed that the gain part of the collision operator should act like a dual Radon transform and that the regularity gained in this way would allow theorems of propagation of both regularity and singularity. This important result was revisited by various authors, e.g. Bouchut and Desvillettes \cite{bouchdesv:Q+:98}, or Mouhot and myself \cite{mouvill:04}; in this last work constructive ``hands-on'' estimates of propagation of regularity are deduced. As for propagation of singularities, see for instance Boudin and Desvillettes \cite{boudesv:reg:00} for inhomogeneous solutions of small mass.

But for the Boltzmann equation without cutoff, there is instant regularisation, whenever there is a good theory of Cauchy problem. This was discovered by Desvillettes for the Kac model \cite{desv:kac:95}, and then further generalised to several particular instances of the spatially homogeneous Boltzmann equation by Desvillettes \cite{desv:2D:94,desv:regMaxw:97} and his student Prouti\`ere  \cite{prout:reg:96}. Desvillettes and I also worked out the regularisation for the Landau equation with hard potentials \cite{DV:landau:1}. The study of Boltzmann's entropy production functional motivated the development of efficient tools: after some precursors \cite{lions:sanscutoff:98,vill:noncutoff:99} came the work by Alexandre, Desvillettes, Wennberg and myself \cite{ADVW}, which presented the truncation recipe, the Fourier representation in Bobylev style, the optimal fractional Sobolev estimate which were at the core of Subsections \ref{subsechigher} and \ref{subsmooth}. The combination of these tools in a Moser-type regularity iteration was sketched in my Peccot lecture notes \cite{villani:peccot} and provides a much more efficient approach to the regularisation for the Boltzmann equation without cutoff. From there a long list of works started to systematically cover situations of interest; see for instance Desvillettes--Mouhot \cite{desvmou:noncutoff} (also establishing for the first time uniqueness of the solution for the noncutoff Boltzmann with soft potentials), Alexandre--Morimoto--Ukai--Xu--Yang \cite{AMUXY}, Chen--He \cite{chenhe:noncutoff}, Chen--Desvillettes-He \cite{CDH}, Fournier--Gu\'erin \cite{fourguer:JSP:08}, Fournier--Mouhot \cite{fourmou:CMP:09}, He \cite{he:noncutoff}. Besides the complexity due to the structure of the Boltzmann operator, this method is very much in the flavour of the classical regularity theory by Nash, DeGiorgi and Moser on divergence form parabolic equations: See Mouhot's course in the same volume \cite{mouhot:crete}, yielding much more details and a whole historical perspective. 

By the way, it is interesting to note that Nash's original proof of his celebrated continuity theorem is based on Boltzmann's entropy, and there is a mathematical filiation with the study of the Boltzmann equation here: Indeed, after the death of Carleman, Lennard Carleson (future Abel Prize winner, like Nash) was in charge of completing and editing Carleman's last papers on the subject, and thus became acquainted with the underlying mathematical physics. So when Nash consulted him, Carleson pointed out to him the interest of Boltzmann's $H$ functional, which in those days was hardly known in the community of mathematical analysts (this was told to me by Carleson himself).

Regularity always follows, as a rule, from regularity estimates through usual tools; the main subtlety is if one wants to prove uniqueness of the weak solution in presence of smooth solutions; see \cite{DV:landau:1} for the spatially homogeneous Landau equation with hard potentials, and Fournier \cite{fournier:uniqueness}, Fournier--Gu\'erin \cite{fourguer:JFA:09} for soft or very soft potentials.

While the large majority of regularity estimates were based on time-derivatives of nonlinear functionals, as in the papers by Nash and Moser, Silvestre introduced a genuinely different method to prove regularity estimates, based on Harnack inequalities for nonlocal equations \cite{silvestre:newreg}. See also Chaker--Silvestre \cite{chakersilvestre} for another approach to the regularity induced by the entropy production functional.

All of these results have been, so far, unaccessible to very soft potentials, because of the lack of an adequate $L^p$ bound; there were only partial or conditional results \cite{GIS:partial,GGIV:landau,GIV:landau}. The new Fisher information estimate precisely fills this gap.

Short-time estimates for the spatially homogeneous Landau equation (with Coulomb potential in $d=3$) were first established by Arsen'ev and Peskov \cite{arsenpesk:landau:78} in the seventies. The inequality \eqref{FIH2} appears in \cite{TV:slow:00} and rests on another inequality by Lions and myself \cite{PLLV:95}, controlling (among other estimates) $\sqrt{a}$ in $W^{1,4}$ by $a$ in $W^{2,2}$. The strategy of short-time regularisation to get the finite Fisher information is due to Imbert, Silvestre and myself \cite{ISV:fisher}. Short-time estimates for the Boltzmann equation with very soft potentials were established with full details in weighted $L^\infty$ spaces by Henderson, Snelson and Tarfulea \cite{HST:boltz:irregular,HST:boltz:continuation}, while the same authors \cite{HST:landau:local} and Snelson and Solomon \cite{snelsol:landau} treated the Landau equation. Since the a priori estimates hold in $L^p$ spaces as soon as $p> d/(d+\gamma+\nu)$, it is natural to expect that this is a natural space to achieve short-time regularisation and continuation estimates, as well as uniqueness.

Carleman \cite{carleman} proved a lower bound like $K e^{-|v|^{2+\var}}$ for hard spheres, with arbitrarily small $\var>0$. Much later, unaware of that older work, A. Pulvirenti and Wennberg \cite{pulviwenn:CMPlower:97} refined this into a Maxwellian lower bound for all hard potentials. In the case of Maxwellian potentials, there are similar results by Bobylev \cite{bob:theory:88}. The lower bound through bump and maximum principle was used by Desvillettes and I \cite{DV:landau:1} to prove a Gaussian lower bound for the Landau equation with hard potentials. The proof presented here is just a generalisation, obviously suboptimal. Imbert, Mouhot and Silvestre \cite{IMS:boltzmann:20} obtain the fine Gaussian lower bound for the Boltzmann equation without cutoff, in great generality, as soon as regularity bounds are available; the key to their analysis is a weak Harnack inequality due to Silvestre \cite{silvestre:newreg}. The maximum principle mentioned above, taken from my lecture notes \cite{vill:handbook:02}, is obviously related to the weak Harnack inequality, and was explored at length in recent research. This maximum principle was also used, for instance, in another work, with Gamba and Panferov \cite{GPV:Maxwbound:09}, to get upper bounds for models of granular media. The general arguments of Pulvirenti--Wennberg and Imbert--Mouhot--Silvestre both use an iterative scheme of spreading type, as did Carleman's proof.

Nash's upper and lower Gaussian bounds are in his famous 1958 paper \cite{nash:58}; such bounds are usually called Aronson estimates. See Fabes and Stroock \cite{fabesstroock:86} or Bass \cite[Chapter 7]{bass:diffusionbook} for modern presentations.

Mouhot, Imbert and Silvestre, followed by Cameron, Henderson, Ouyang, Snelson, Turfulea, have worked out an ambitious program showing that the inhomogeneous Landau and Boltzmann equations, with or without cutoff, are regular and well-controlled as long as certain bounds on the macroscopic quantities hold true \cite{CSS:inhomlandau:18,CS:boltzthard:20,HST:lowerbound:20,IMS:boltzmann:20,IMS:lowerbound:20,IS:globalreg:22,mouhot:lowerfull:05,OS:conditional}. This is a whole world with the interplay of homogeneous and inhomogeneous features, in which hypoelliptic regularity and diffusion processes regularity are intertwined; Mouhot's lecture at the 2018 International Congress of Mathematicians \cite{mouhot:ICM:18} was surveying the field at a time when it was starting to achieve some success. The most complete results, putting together Imbert--Silvestre \cite{IS:globalreg:22} and Henderson--Snelson--Tarfulea \cite{HST:lowerbound:20}, roughly say that for moderately soft potentials there is global decay, regularity and positivity conditional to uniform upper bounds on two hydrodynamic fields: density $\rho(t,x)=\int f\,dv$ and energy $E(t,x)=(1/2)\int f |v|^2\,dv$. It is natural to ask how far these results can be pushed with very soft potentials: Existing papers prove the regularity conditional to some uniform local $L^p$ bounds, but it would be more satisfactory to assume $L^p$ or $L^\infty$ estimates on the initial datum and conditional bounds for all times on hydrodynamic quantities. In any case, in the past decade a major step has been achieved in our understanding of the Boltzmann equation as it was progressively proven that regularity is basically conditioned to the non-blowup of hydrodynamic fields.

\section{Equilibration} \label{seceq}

Convergence to equilibrium is one of the fundamental predictions of collisional kinetic theory -- and the word ``prediction'' here is not lightly used, since Maxwellian (Gaussian) equilibrium was put forward by Maxwell and Boltzmann even before the existence of atoms and molecules was mainstream. First I shall explain why the quantitative mathematical study of equilibration is closely connected to the problem of Fisher information estimate; this was already mentioned in Section \ref{secFIK}.

\subsection{From Fisher information monotonicity to entropy inequalities} \label{FItoD}

The main focus of this set of notes was the monotonicity of the Fisher information along solutions of the Boltzmann equation. This is not unrelated to the more classical problem of estimating Boltzmann's classical dissipation functional (entropy production functional). A first way to see the connexion is the following. Consider a collision kernel $\beta(\cos\theta)$ on the sphere. Since the heat and Boltzmann equation with kernel $\beta$ commute, the decay of Fisher information along the Boltzmann flow is the same as the decay of Boltzmann's entropy production along the heat flow. In other words, if $L_\beta$ stands for linear Boltzmann equation with kernel $\beta(\cos\theta)$, 
\begeq\label{commutedecay}
\left. -\frac{d}{dt}\right|_{t=0} D_\beta(e^{t\Delta}F) = \left.-\frac{d}{dt}\right|_{t=0} I(e^{tL_\beta}F).
\endeq
Here is a relation with the inequalities explored in these notes. Given $\beta$, let $\Lsharp$ be the largest constant such that for all even probability densities on $\S^{d-1}$,
\begeq\label{IbetaI}
\int_{\S^{d-1}} F\, \Gamma_{1,\beta}(\log F) \geq \Lsharp I(F).
\endeq
Note that the left hand side is $(-1/2)(d/dt)I(e^{tL_\beta}F)$.
Throughout these notes, we have seen various sufficient conditions for $\Lsharp$ to be positive, and recipes to estimate this constant. For instance,

\bul From \eqref{lowesthanging}, $\Lsharp\geq (d-2)\Sigma(\beta)>0$ if $d>2$;

\bul From \eqref{inegthmposit}, $\Lsharp>0$ if $\min\beta>0$;

\bul From \eqref{FG1Ktgeq}, if $\beta = \int {\cal K}_t \lambda(dt)$, $\Lsharp \geq (1/2) \int (1-e^{-2L_*t})\lambda(dt)>0$.

From \eqref{IbetaI} one may apply Proposition \ref{GbG1} to deduce Criterion \ref{eqBcrit} and the monotonicity of the Fisher information for the nonlinear Boltzmann equation. But we may also keep \eqref{IbetaI} and interpret both sides as time-derivatives along $S_t$: From \eqref{commutedecay} and \eqref{debruijn}, \eqref{IbetaI} is also
\[ \left. -\frac{d}{dt}\right|_{t=0} D_\beta(e^{t\Delta}F) \geq - 2\Lsharp \left.\frac{d}{dt}\right|_{t=0} H(e^{t\Delta}F). \]
Integrating from 0 to $+\infty$ as Stam, it follows the functional inequality
\begeq\label{DgeqH}
D_\beta(F) \geq 2\Lsharp H(F),
\endeq
which itself implies that $H(F(t))$ converges to 0 along the linear Boltzmann flow on the sphere, at least like $O(e^{-\Lsharp t})$. Actually \eqref{DgeqH} admits two notable limit cases:

\bul When $\beta$ is the constant kernel (which corresponds to $t\to\infty$ in \eqref{FG1Ktgeq}, thus $\Ksharp\geq 1/2$), this is the spherical linear version of the entropy -- entropy production inequality conjectured by Cercignani and proven by Toscani and myself for super hard spheres; and the latter inequality, in turn, is the mother of all quantitative equilibration estimates for the nonlinear Boltzmann equation (homogeneous or not) in the large, when the kernel is not Maxwellian;

\bul When $\beta$ concentrates on grazing collisions (which corresponds to the derivative as $t\to 0$ in \eqref{FG1Ktgeq}, thus $\Lsharp\geq L_*$), this reduces to the classical logarithmic Sobolev inequality for even functions on the sphere.
\sm

In other words, inequality \eqref{IbetaI} acts like a bridge between the theory of logarithmic Sobolev inequalities for diffusion processes, and the theory of entropy -- entropy production inequalities in the context of the Boltzmann equation. True, both theories differ significantly when going into details, but at least there is that common core.
To summarise and reformulate it: The central inequality is
\begeq\label{centralineq}
4\Lambda_\# \int_{\S^{d-1}} \Gamma_{\beta}(\sqrt{F}) \leq I(\sqrt{F})
\leq \frac1{L_\#} \int_{\S^{d-1}} F\Gamma_{1,\beta}(\log F).
\endeq
Keeping only the left inequality in \eqref{centralineq} is a Poincar\'e problem for $L_\beta$, or just the spectral domination of $-L_\beta$ by $-\Delta$. Keeping only the right inequality, and integrating along the heat semigroup, \`a la Stam, yields the entropic gap inequality $H(F)\leq D_\beta(F)/(2\Lsharp)$, \`a la Toscani-Villani (except it is on the sphere rather than the whole space). Keeping the whole of \eqref{centralineq} provides the criterion for decay for the Fisher information along the Boltzmann flow, \`a la Imbert--Silvestre--Villani. Finally, taking the AGC, the left inequality becomes an equality (with $\Lambda_\# =1$) and the right one becomes the differential log Sobolev inequality, \`a la Bakry--\'Emery, which thus appears as the diffusive limit case of \eqref{centralineq}.

\subsection{Rates of convergence}

Now I will try and sketch the overall picture for quantitative equilibration. Again, here I shall focus on the spatially homogeneous case. Without loss of generality, the initial probability density will be assumed to have zero mean ($\int f_0(v)\,v\,dv = 0$) and unit temperature (initial energy $E_0=d/2$). While the equilibrium is always the same for all collision kernels (the standard Maxwellian, or centered Gaussian with identity covariance matrix), the estimates and expected rates may depend significantly on the features of the collision kernel. In particular,

\bul In the spectral gap regime ($\gamma+\nu \geq 0$) one expects exponential convergence to equilibrium;

\bul In the regime with no spectral gap but conditional regularity ($-2<\gamma+\nu<0$), convergence should be at best like the exponential of a negative fractional power of time. Such is the case in particular for very soft potentials ($\gamma <-2$).

Whatever the situation, large velocities are one key to the estimates, and obstacles to fast convergence are most often related to velocity tails.

Here is a heuristics for the plausible rate of convergence. In the linearised case, spectral gap is related to the convergence of the nondiagonal part of the covariance matrix, or moments of order 2; so let us consider $Q(v) e^{-|v|^2/2}$, and imagine that there is a distinguished approximate solution of the form
\[ \vphi(t,v) = e^{-\lambda t^\alpha} Q(v) e^{-|v|^2/2}.\]
The action of the operator $Q$ or its linearised version will tend to multiply $\vphi$ by a multiple of $|v|^{\gamma+\nu}$ as $|v|\to\infty$ (each derivative acting on a Gaussian induces a multiplication by $v$, or kind of). But $\pa_t\vphi$ will be like $- t^{\alpha-1}\vphi$, and this should be homogeneous to $|v|^{(2\alpha)/(1-\alpha)}$ since in the definition of $\vphi$, $t$ is homogeneous to $|v|^{2/\alpha}$. So the condition for the good effect of $Q$ to beat the bad effect of the time-derivation, should be $\gamma+\nu\geq 2(\alpha-1)/\alpha$, or equivalently $\alpha \leq 2/(2-(\gamma+\nu))$. Writing $(\gamma+\nu)_-\geq 0$ for the negative part of $\gamma+\nu$, we just arrived at the

\begin{Guess}\label{guess} Solutions of the spatially homogeneous Boltzmann equation with fast enough decay at large velocities converge to equilibrium like $O(\exp(-\lambda t^\alpha))$, where
\[ \alpha = \frac{2}{2+(\gamma+\nu)_-}.\]
\end{Guess}
(So for the Landau--Coulomb equation in dimension 3, the expected rate is $O(\exp(-\lambda t^{2/3}))$.) In a close to equilibrium setting with Gaussian decay, such rate was indeed proven by Guo and Strain for certain cases when $\gamma+\nu<0$. Starting away from equilibrium, the situation is not as good yet, but with the new estimates for the Cauchy problem we can at least prove that there is convergence to equilibrium, and at least like $O(t^{-\infty})$, that is, faster than any inverse power of time. The main steps will be summarised in the next Subsection.

\subsection{Quantitative entropic convergence}

Start again with the spatially homogeneous Boltzmann equation. From Section \ref{secreg} we already have the control of moments and Sobolev norms of arbitrarily high order like $O(t^\var)$, with $\var$ as small as desired. And we also have a lower bound like $K(t) e^{-A(t)|v|^q}$ with $\log K(t)$ and $A(t)$ controlled like $O(t^\var)$. In such a situation there is an entropy -- entropy production estimate of the form
\begeq\label{DBf} D_B(f) \geq K t^\var H(f|M)^{1+\var}, \endeq
where $H(f|M) = H(f)-H(M)=H_M(f)$ is the relative information of $f$ with respect to its associated equilibrium, $\var>0$ is as small as desired, and $K$ depends on $d,B,\var$, and high order moments. This is a Gronwall inequality on $H(t)=H(f(t)|M)$, of the form $-dH/dt \geq K t^\var H^{1+\var}$, which integrates into
\[ H(f(t)|M) \leq \left( \frac1{H(f_0|M)^\var} + \left(\frac{\var K}{1+\var}\right) t^{1+\var}\right)^{-\frac1{\var}},\]
and this yields the announced $O(t^{-\infty})$ estimate.

Then, by the Csisz\'ar--Kullback--Pinsker estimate
\[ H(f|M) \geq \frac12 \|f-M\|_{L^1}^2,\]
actually
\[ \|f(t) -M \|_{L^1} = O(t^{-\infty}).\]
Interpolating this with the bound
\[ \|f(t)-M\|_{L^1_s} \leq \|f(t)\|_{L^1_s} + \|M\|_{L^1_s} = O((1+t)^\delta)\]
with $s$ as large and $\delta$ as small as desired, we further deduce
\[ \|f(t)-M\|_{L^1_\sigma} = O(t^{-\infty}) \]
for any $\sigma>0$. In particular, $\|f\|_{L^1_\sigma}$ remains uniformly bounded as $t\to\infty$, for any $\sigma$, and it follows that all decay and regularity bounds are uniform also. Ultimately, it holds

\begin{Thm} Let $B(v-v_*,\sigma) = |v-v_*|^\gamma b(k\cdot\sigma)$, where $\gamma\geq -d$, $b(\cos\theta)\sin^{d-2}\theta \geq K \theta^{1+\nu}$, $0<\nu<2$. Further assume $-2<\gamma+\nu<0$ and the monotonicity criterion for the Fisher information, as in Theorem \ref{thmcriterion} or \ref{thmcritsyn}. Then if $f_0$ is a probability distribution with zero mean and unit temperature, with finite moments of all orders and $f_0\in L^m(\R^d)$ for some $m> d/(d+(\gamma+\nu))$, then the solution $f=f(t,v)$ of the spatially homogeneous Boltzmann equation with collision kernel $B$ and initial datum $f_0$ becomes instantly smooth and satifies
\[ \forall k\in\N_0, \qquad \forall \kappa>0, \qquad \forall r>0,\qquad
\forall t\geq t_0, \qquad \left( \frac{|\nabla^k (f-M)|}{\<v\>^\kappa} \right) \leq \frac{C}{t^r},\]
where $M=M(v)$ is the centered Maxwellian (Gaussian) with unit temperature, $C$ only depends on $t_0,k,\kappa,r$, m, $\|f_0\|_{L^m}$ and $\|f_0\|_{L^1_s}$ for $s$ large enough. In short, $f$ is uniformly bounded in the Schwartz class of rapidly decaying functions, and converges to $M$ in this class with rate $O(t^{-\infty})$.
\end{Thm}

\subsection{Equilibration through Fisher information} \label{subeqFI}

The decay of $I$ is another approach for the convergence to equilibrium. Actually, the Stam--Gross logarithmic inequality
\[ I_M(f) \geq 2 H_M(f)\]
shows that convergence in the sense of Fisher information is stronger than in the sense of entropy.

Obviously, the decreasing property of Fisher information along the Boltzmann flow does a good part of the way towards a convergence rate. It turns out that we have all the tools for a good discussion of this. Just by working out again the calculations of Section \ref{seccrit}, we have

\begin{Prop}
Consider a kernel $B=B(|v-v_*|,\cos\theta)$, let $\ov{\gamma}=\ov{\gamma}(B,r)$ be defined as in \eqref{ovgammaB} and let $\Ksharp(\ov{\gamma},r)$ be the largest admissible constant in the functional inequality
\begeq\label{IbetaKI}
\int_{\S^{d-1}} F\, \Gamma_{1,\beta_r}(\log F) \geq \ov{\gamma}(B,r)^2 \int_{\S^{d-1}} \Gamma_{\beta_r}(\sqrt{F}) + \Ksharp(\ov{\gamma},r)\, I(F),
\endeq
required to hold for all even functions $F:\S^{d-1}\to\R_+$.
Then, along the spatially homogeneous Boltzmann equation with kernel $B$, one has
\begeq\label{ddtII}
-\frac{dI}{dt} \geq
2 \iint_{\R^d\times\R^d}
ff_*\, \Ksharp(\ov{\gamma},|v-v_*|)\, \left|
\Pi_{k^\bot} \left( \frac{\nabla f}{f} - \Bigl( \frac{\nabla f}{f}\Bigr)_* \right) \right|^2\,dv\,dv_*.
\endeq
\end{Prop}

In other words, the dissipation of Fisher information is at least proportional to $D_L(f)$, Landau's dissipation functional, for a Landau equation whose kernel $\Psi$ is proportional to $\Ksharp (\ov{\gamma},|v-v_*|)$.

How can we exploit this? First, as in Subsection \ref{FItoD}, in all the situations discussed in these notes we can find a convenient function $\Ksharp>0$;  in particular this will be the case for all inverse power law interactions; and then $\Ksharp$ will be proportional to $|v-v_*|^\gamma$, where $\gamma = (s-(2d-1))/(s-1)$. In other words, there will be $K>0$ such that
\begeq\label{dIdL}
-\frac{dI}{dt} \geq K \,D_L^\gamma(f),
\endeq
where
\begeq\label{DLgamma}
D_L^\gamma(f) = \frac12 \iint ff_*\, |v-v_*|^\gamma \left|\Pi_{k^\bot} \left( \frac{\nabla f}{f} - \Bigl( \frac{\nabla f}{f}\Bigr)_* \right) \right|^2\,dv\,dv_*.
\endeq

Then we may estimate $D_L^\gamma$ from below, using non-concentration, localisation, smoothness, interpolation. The projector $\Pi_{k^\bot}$ in \eqref{DLgamma} implies a degradation of the estimates, so in general there is a loss of 2 in the exponent of the relative velocity. Assuming, without loss of generality, that $f_0$ has unit temperature ($E_0=d/2$), and writing $I(f|M) = I_M(f) = I(f)-I(M)$,

\bul There is a constant $K=K(d,\gamma,H(f_0))$ such that
\begeq\label{DL2}
D_L^2 (f) \geq K I(f|M)
\endeq

\bul For any $\var>0$ there is a constant $K_\var = K(d,\gamma,H(f_0),\var)$ and there are $s=s(d,\gamma,\var)$ and $k=k(d,\gamma,\var)$ such that
\begeq\label{DLgammageq}
D_L^\gamma(f) \geq K_\var I(f|M)^{1+\var} \bigl( \|f\|_{L^1_s} + \|f\|_{H^k}\bigr)^{-\var}.
\endeq

This and \eqref{dIdL}, and the moments and smoothness bounds like $O(t^0)$ imply

- for $\gamma=2$, exponential convergence of $f$ to $M$ in the sense of Fisher information, therefore in entropy and in $L^1$, and eventually by interpolation, exponential convergence in Schwartz class;

- for any $\gamma$ (positive or negative, keeping in mind that certain assumptions have to be made to achieve good regularity estimates), convergence of $I(f|M)$ to~0 like $O(t^{-\infty})$, and similarly convergence of $f$ to $M$ like $O(t^{-\infty})$ in Schwartz class.

\begin{Rk} One may conjecture that with a bit more work, it holds $-dI/dt \geq K I$ in the regime of entropic gap, that is, $\gamma+\nu\geq 2$. This may require more specific assumptions on the particular shape of $b(\cos\theta)$; for instance, that $b$ derives from the inverse power laws, or from the fractional Laplace operator on the sphere. It may also require comparison with an integral quantity involving fractional derivatives of higher order than~1.
\end{Rk}

\begin{Rk}  
In the end, under the regime considered here, the Fisher information approach yields basically the same rate $O(t^{-\infty})$ as the entropic approach. There are however two notable differences. The first is that the Fisher information approach may require more structure assumptions: to prove the decay we were led to some specific choices and relations between the angular and velocity dependence of $B$ -- e.g. power law, or perturbation of fractional diffusion, or high dimension... In the same vein, we cannot just be content with a comparison argument: if $B\geq B_0$ then the associated Boltzmann dissipation functionals can be compared, but for Fisher that is not necessarily true. The second difference is that the Fisher dissipation inequality \eqref{DLgammageq} uses moments and smoothness, but does not need any lower bound on $f$ -- contrary to \eqref{DBf}. Both the drawback and the advantage are relatively minor, since on the one hand we saw that the monotonicity for $I$ holds in all the main cases we could think of, and since on the other hand it is often in practice possible to get around the need of lower bound for the convergence to equilibrium, as shown by Carlen, Carvalho and Lu. Of course it might be that for other regimes, the Fisher information approach has some more pronounced advantages.
\end{Rk}

I will close this section on equilibration with two final, more general remarks. The first one is that even in the entropic approach, Fisher information plays a key role in the backstage, as for general collision kernels nobody knows how to prove a good inequality on the entropy production without going through a Fisher information estimate. The other one is that the possibility to prove and estimate equilibration through the rate of decay of Fisher information, even if it does not yet bring anything fundamentally new with respect to the entropic approach, at least reinforces the consistency of the whole area, and especially the relation between regularity theory, equilibration theory, and information theory.

\bibnotes

The discussion in Subsection \ref{FItoD} is taken from \cite{ISV:fisher}. The connexion between the Fisher information decay and the entropy production inequality is implicit in McKean \cite{mck:kac:65} and discussed explicitly by Toscani and I \cite{TV:entropy:99}.

Entropy--entropy production estimate \ref{DBf} is extracted from my paper \cite{vill:cer:03}, which is an improvement of my older work with Toscani \cite{TV:entropy:99}. The strategy of convergence in $O(t^{-\infty})$ from slowly increasing bounds is also from another work with Toscani \cite{TV:slow:00}. Carlo Cercignani played an important role in motivating research in this area.

Carlen, Carvalho and Lu \cite{CCL:soft:09} were the first to obtain convergence results, although in a weaker sense, for the spatially homogeneous Boltzmann equation with very soft potentials away from equilibrium, refining my strategy from \cite{vill:new:98}; their great paper is a model of clarity and efficiency to work out relevant conclusions even in a very rough situation. They also showed how to do without lower bound estimates, applying the entropy -- entropy production inequality not to the true distribution function but to a modified version like $(1-e^{-t}) f(t,v) + e^{-t}M(v)$. Further see Desvillettes \cite{desv:JFA:15} for the spatially homogeneous Landau equation with soft or very soft potentials.

In the 2000's Desvillettes and I worked out a program to extend the $O(t^{-\infty})$ convergence conditionally to regularity bounds, using entropy production bounds, differential inequalities, geometric/analytic functional inequalities, and a lot of interpolation \cite{DV:atlanta,DV:FP:01,DV:korn:02,DV:boltz:05,vill:hypoco}. Two main issues remained open after that: To refine the rate from $O(t^{-\infty})$ to exponential or fractional exponential; and to weaken the smoothness assumptions. In both directions, impressive results have been obtained in the following decade.

Begin with the exponential decay, starting with the regime of hard potentials. Of course it had first been proven for linearised versions of the equation, already back in the fifties, see for instance by Grad \cite{grad:58} or Wang Chang and Uhlenbeck \cite{wchanguhl:spectrum}. Then it was established for the nonlinear equation in a perturbative homogeneous setting, by Arkeryd with nonquantitative arguments \cite{ark:stab:88,arkeryd:str:92}, and in a completely quantitative way by Mouhot \cite{mouhot:sg} and Baranger and Mouhot \cite{barangermouhot:sg:05}. At that point it had all been for the cutoff case; Mouhot and Strain \cite{mouhotstrain:sg} treated noncutoff kernels. Mouhot \cite{mouhot:eqSHBE:06} used quantitative nonsymmetric spectral theory to prove quantitative exponential convergence for the spatially homogeneous nonlinear Boltzmann equation near equilibrium, under regularity assumptions which are compatible with the nonlinear theory in the large. Gualdani, Mischler and Mouhot \cite{GMM:nonsym:17} extended those results to the spatially inhomogeneous equation, both in a perturbative setting and conditional to global regularity bounds; this involved working out a complete new abstract linear theory -- it is actually one nice example of new theory of general interest motivated by problems arising in kinetic theory. Stitching the estimates of \cite{GMM:nonsym:17} (close to equilibrium with the optimal rate) with those that I proved with Desvillettes \cite{DV:boltz:05} (far from equilibrium, with $O(t^{-\infty})$ rate, conditional to global regularity) and the conditional regularity program of Imbert--Mouhot--Silvestre, eventually shows that a uniform control on the density and energy (say in the $d$-dimensional torus) are enough to achieve a complete theory of existence, decay, regularity and exponential convergence to equilibrium. Of course the spatially homogeneous theory is a very special particular case of that general conditional inhomogeneous theorem.

For soft potentials, Caflisch \cite{cafl:soft:80} was the first to obtain decay like $O(\exp(t^{-\alpha}))$ in the cutoff case. This was refined and extended by Guo and Strain \cite{guostrain:almostexp:06,guostrain:exp:08}, who indeed obtained the decay rate of Guess \ref{guess}. They also worked out some less stringent assumptions than Gaussian, e.g. they prove such a decay with a fractional exponential rate, for the Landau--Coulomb equation, if the decay of the distribution function is at least like $e^{-|v|^{1+\var}}$. Refined results were obtained for the inhomogeneous Landau and noncutoff Boltzmann equations with soft potentials by Cao, Carrapataso, Desvillettes, He, Ji, Mischler, Tristani, Wu \cite{CDH:landau:17,CHJ:noncutoff,CM:verysoft:17,CTW:landau:16}.  All this is in a close-to-equilibrium setting. For data away from equilibrium, in the spatially homogeneous setting the $O(t^{-\infty})$ decay rate for very soft potentials was proven in \cite{ISV:fisher} as a consequence of the new regularity estimates coming from the Fisher information monotonicity -- all other ingredients were in place. Stitching together these two theories may be reachable. By the way, the real issue is not so much hard vs soft potentials, than spectral gap vs no spectral gap.

In the special case of Maxwellian kernels ($\gamma=0$) there are several other ways towards the approach to equilibrium: contracting metrics as studied by Tanaka \cite{tanaka:boltz:78} \cite[Section 7.5]{vill:TOT:03}, Toscani and myself \cite{TV:metrics:99}, or Carlen--Gabetta--Toscani \cite{CGT:maxw:99}; central limit theorem, from McKean \cite{mck:kac:65} to Carlen--Gabetta--Toscani again to Dolera--Regazzini \cite{doleraregazzini:13}, or convergence of moments; once again see Bobylev \cite{bob:theory:88}. All those methods eventually lead to exponential convergence if the data is localised enough, see for instance \cite{doleraregazzini:13}, and conversely the convergence can be as slow as desired if the initial datum is not well localised \cite{CGT:maxw:99}.

The new estimates in Subsection \ref{subeqFI} constitute a natural outgrow of my joint work with Imbert and Silvestre \cite{ISV:fisher}. Inequality \eqref{DL2} is from my work with Desvillettes \cite{DV:landau:2} and \eqref{DLgammageq} from my work with Toscani \cite{TV:slow:00}.

\section{Conclusions}

Since Torsten Carleman started a sound mathematical study of the spatially homogeneous Boltzmann equation, one century ago, a number of additional ingredients have been progressively incorporated in that program. In addition to the tools of functional analysis and measure theory which Carleman knew so well, over the years the theory has been enriched by moment estimates, harmonic analysis, maximum principles, spectral analysis, information theory, Moser iteration, Harnack inequalities, and the theory of nonlocal diffusive equations. Fisher information was introduced as a useful tool, already sixty years ago by McKean; then in the nineties its impact in the field grew when it was identified as a central tool in the study of quantitative equilibration; and now its impact has grown even more as it has allowed to unlock the stubborn problem of well-posedness for very soft potentials. These recent developments constitute the last brick to date in the theory, and for the very  first time we see a consistent picture emerging which encompasses all cases of physical and mathematical interest.

This recent outgrow of the theory echoes and eventually answers most of the key questions on the spatially homogeneous Boltzmann equation which agitated me during my PhD: how to handle singularities in the collision kernel, how to quantify the convergence to equilibrium, and more generally how to sort out the global picture of moments, smoothness, lower bound and convergence to equilibrium. In this jungle I was walking in the footsteps of my senior collaborator Laurent Desvillettes, who had first initiated me to the mysteries of soft potentials.

The theory of the Boltzmann equation comes out from this episode much richer than before and with strengthened consistency. It was also during my PhD that the connexion between entropy -- entropy production estimates and information theory was made. This memory is very dear to me, as it was the main outcome of my first scientific stay outside France, in Pavia, on the invitation of Giuseppe Toscani. Our estimate and joint work on Cercignani inequalities stemmed from a shiny coincidence, certainly one of the great lucks of my career. What seemed like a beautiful anomaly or at least a singularity at the time, now seems quite natural and consistent with the way the theory has evolved; so it is a whole architecture which now surrounds this once isolated gem. Finally one may daresay that the collision kernel $Q$, with its complicated dissipative structure, seems well understood, and this was not the case even just  couple of years ago. Frustrating as it may too often be, mathematical physics is at times rewarding. 

Several types of further developments may emerge from the current theory.

The most straightforward is a more systematic development of the Cauchy theory for the spatially homogeneous equation in the large, and also for the conditional theory of the spatially inhomogeneous equation. All the ingredients are available, and I sketched most of the estimates, but some of the theorems have only been written under particular assumptions and it would be quite a task to systematically rework the theory with the best available assumptions.

As far as regularity is concerned, now that the spatially homogeneous theory seems to achieve maturity, and that the conditional regularity theory has achieved some impressive successes, the most famous challenges for the Boltzmann equation -- hydrodynamic limit, derivation from particle systems -- look even more motivating and daunting. Fisher's information may have its role to play there, also. Regularity for the full (inhomogeneous) equation still seems to be the most urgent topic to foray.

But one may also ask questions about the Fisher information itself. Can we be sure that our results are close to optimal, and can one construct a counterexample? that is, a specific collision kernel $B$ for which $I$ is not always decreasing. One may start from the specific kernels yielding a counterexample in Section \ref{secex}, yet things do not seem to come easily from there.

Even if there turn out to be counterexamples for certain kernels, the main striking lesson is the robustness of the Fisher information decay over a very large class of kernels, a striking property which had not been predicted by the community of mathematical physics.

Obviously, this success gives further motivation for the investigation of Fisher's information in the treatment of other equations, either through an integrability a priori estimate, or for the long-term behaviour, or from the point of view of large deviations in Donsker--Varadhan style, a theory which by the way deserves some more quantitative estimates. It also provides a renewed interest for the investigation of the other questions loosely formulated by McKean about higher order behaviour, mentioned at the end of Section \ref{secFIK}. 

One may further enquire whether some of the other functionals known to be monotone for the spatially homogeneous Boltzmann equation with Maxwellian kernel, are actually monotone in greater generality. A case in point is the contractivity property of the Wasserstein distance $W_2$, proven by Tanaka for Maxwell kernels: Various authors have shown that this distance is exponentially stable along the Boltzmann flow for various kernels; but could it actually be nonincreasing?

Finally, let us pause and reflect on that striking discovery: {\em The Fisher information is decaying under the action of Boltzmann collisions.} Lyapunov functionals come with a number of features and consequences, and their knowledge is often incorporated in model equations or numerical schemes. Already 25 years ago it was observed that the decay of Fisher's information seems to be associated with the stability of some deterministic numerical schemes, and now that the theory has matured it is time to re-examine that observation which most researchers (including me) had not taken seriously enough. Of course, the proof of decay for the Fisher information is so much more intricate than the proof of increase for the Boltzmann entropy, that it seems absurd to try and force the decrease within the formulation of the code. But still there is food for thought, thinking of the monotonicity property as a test or a guide.

Taking a more distant look, we may also investigate the physical meaning of the decay of the Fisher information. Think of the influence of the $H$-Theorem, which was mathematically corroborating the already identified Second Law of Thermodynamics, even if just for the specific case of a rarefied gas. Is there another thermodynamical principle which can be formulated with some accuracy and for which a mathematical translation would be precisely the monotonicity of $I$ along the homogeneous Boltzmann equation? In the physics literature there has been some speculation and treatises about the physical meaning of Fisher information, sometimes related to the so-called ``thermodynamical length'', and its potential role in some of the principles of classical or quantum statistical mechanics. This is consistent with the original meaning of Fisher information related to observation (measurement) and statistics. I do not master the physics literature on this topic well enough to recommend precise references, but at least I can mention the names of two physicists who contributed to those lines of thought: Gavin Crooks (University of Berkeley) and Roy Frieden (University of Arizona). Of course, for systems in which the Fisher information is related to the entropy production (recall that this is the case for the basic linear Fokker--Planck equation), there may be a connection to principles of minimisation of entropy production, explored by Ilya Prigogine (1977 Nobel Prize in chemistry). But the control of the Fisher information should hold more generally, as its good behaviour under the Boltzmann equation suggests. Nothing comes easily: a major difference between the variations of $H$ and the variations of $I$ is that the former is preserved under transport ($v\cdot\nabla_x$ in the Boltzmann equation), while the latter is not, as transport implies mixing and growing oscillations in phase space, thus a nontrivial behaviour for any quantity involving regularity. And this enormous difficulty is in the order of things: these growing oscillations are associated to a loss of accuracy in the observation and, for instance, the Landau damping phenomenon which has been the subject of a number of mathematical works since my paper with Cl\'ement Mouhot. All those remarks justify a study in depth in the hope of finding even richer fruits in the scientific continent discovered by Maxwell and Boltzmann.

\bibliographystyle{acm}

\bibliography{./fisher, ./mybiblio}


\bigskip

\signcv

\end{document}